\documentclass[a4paper,12pt,oneside]{book}

\usepackage{amsmath,amsfonts,amssymb,amsthm}
\usepackage{bbm}

\usepackage[parfill]{parskip}

\makeatletter
\def\@makechapterhead#1{%
  \vspace*{50\p@}%
  {\parindent \z@ \raggedright \normalfont
    \interlinepenalty\@M
    \Huge\bfseries  \thechapter \quad #1\par\nobreak
    \vskip 40\p@
  }}
\makeatother

\usepackage{mathtools}
\usepackage{enumitem}
\setlist[enumerate,1]{label={(\arabic*)}}
\setlist[enumerate,3]{label={(\roman*)}}

\usepackage{tikz}
\tikzset{dynnode/.style={circle,fill=white,draw,minimum size=0.2cm,inner sep=0pt}}
\usepackage{tikz-cd}
\usepackage{calc}

\usepackage{hyperref}
\usepackage{cleveref}

\usepackage[english]{isodate}

\usepackage{fancyhdr}

\fancypagestyle{normal}{
    \fancyhf{}
    \fancyhead[L]{\ifodd\value{page}\else\fancyplain{}{\leftmark}\fi}
    \fancyhead[R]{\ifodd\value{page}\fancyplain{}{\rightmark}\fi}
    \fancyfoot[L]{\ifodd\value{page}\else\thepage{}\fi}
    \fancyfoot[R]{\ifodd\value{page}\thepage\fi}
}
\fancypagestyle{plain}{%
    \fancyhf{}
    \fancyfoot[L]{\ifodd\value{page}\else\thepage{}\fi}
    \fancyfoot[R]{\ifodd\value{page}\thepage\fi}
}

\renewcommand{\chaptermark}[1]{\markboth{\thechapter \ #1}{}}

\newcommand{\mathintitleydmod}[1]{\texorpdfstring{$\ydmod{}$}{YD}}

\newcommand{\ot}{\otimes}

\newcommand{\matdot}{\mathrm{.}}
\newcommand{\matcom}{\mathrm{,}}

\newcommand{\scr}[1]{\mathcal{#1}}
\renewcommand{\frak}[1]{\mathfrak{#1}}

\newcommand{\ndN}{\mathbb{N}}
\newcommand{\ndQ}{\mathbb{Q}}
\newcommand{\ndR}{\mathbb{R}}
\newcommand{\ndZ}{\mathbb{Z}}
\newcommand{\fK}{\Bbbk}
\newcommand{\brB}{\mathbb{B}}
\newcommand{\SG}{\mathbb{S}}

\newcommand{\id}{\mathrm{id}}

\newcommand{\rat}{\mathrm{rat}}
\newcommand{\Aut}{\mathrm{Aut}}
\newcommand{\Endo}{\mathrm{End}}
\newcommand{\ad}{\mathrm{ad}}
\newcommand{\coad}{\mathrm{coad}}

\newcommand{\incl}{\iota}
\newcommand{\proj}{\pi}

\newcommand{\coinva}[2]{ {{#1}}^{\mathrm{co} \, {#2}}}
\newcommand{\bimu}{\mu}
\newcommand{\bicomu}{\Delta}
\newcommand{\biunit}{\eta}
\newcommand{\bicounit}{\varepsilon}
\newcommand{\antip}{S}
\newcommand{\act}{\nu}
\newcommand{\qact}{\rhd}
\newcommand{\coact}{\delta}
\newcommand{\quact}{\triangleright}
\newcommand{\brd}{c}

\newcommand{\frobelt}{\lambda}
\newcommand{\integralelt}{\Lambda}

\newcommand{\modcat}[2]{ {}_{{#1}}{{#2}} }
\newcommand{\modcatr}[2]{ {{#2}}_{{#1}} }
\newcommand{\comodcat}[2]{ {}^{{#1}}{{#2}}}
\newcommand{\comodcatr}[2]{ {{#2}}^{{#1}} }
\newcommand{\bimodcat}[2]{ {}_{{#1}}{{#2}}_{{#1}} }
\newcommand{\bicomodcat}[2]{ {}^{{#1}}{{#2}}^{{#1} } }
\newcommand{\ydmod}[1]{ \prescript{{#1}}{{#1}}{\scr{YD}} }
\newcommand{\ydcat}[2]{ \prescript{{#1}}{{#1}}{\scr{YD}}( {#2} ) }
\newcommand{\boso}{\#}
\newcommand{\nich}{\scr{B}}

\newcommand{\roots}[1]{\Delta^{{#1}}}
\newcommand{\rootspos}[1]{\roots{{#1}}_{+}}

\newcommand{\bichar}[1]{{\chi^{{#1}}}}
\newcommand{\indchar}[1]{\lambda^{{#1}}}

\newcommand{\projcomp}[1]{\proj^{{#1}}}
\newcommand{\inclcomp}[1]{\incl^{{#1}}}
\newcommand{\refl}{R}
\newcommand{\funCD}{\Omega}
\newcommand{\funCDcoh}{\omega}

\newcommand{\subalgQ}[1]{\fK [{#1}]}
\newcommand{\filtration}{\scr{F}}

\newcommand{\filtcomod}[1]{\filtration^{#1}}

\newcommand{\Sup}[1]{\mathrm{Sup}({#1})}
\newcommand{\Supind}[1]{N_{{#1}}}
\newcommand{\intind}[1]{N^{\circ}_{{#1}}}
\newcommand{\bouind}[1]{N^{\delta}_{{#1}}}
\newcommand{\inter}[1]{{#1}^{\circ}}
\newcommand{\bound}[1]{\delta{#1}}
\newcommand{\genind}[1]{\gamma_{{#1}}}
\newcommand{\gencomp}[1]{{#1}_0}
\newcommand{\nsalg}[1]{{#1}^0}

\newcommand{\ydindfun}[1]{\scr{I}_{{#1}}}
\newcommand{\ydind}[2]{\ydindfun{{#1}}({#2})}
\newcommand{\dynk}{D}
\newcommand{\dynkcatfld}[2]{\frak{D}_{{#1}}^{{#2}}}
\newcommand{\dynkcat}[1]{\dynkcatfld{\fK}{{#1}}}
\newcommand{\extdynk}[1]{\overline{{#1}}}
\newcommand{\shapodiag}[1]{P^{{#1}}}
\newcommand{\shapoendo}{f}

\newcommand{\antishg}{\S}
\newcommand{\shift}[2]{{#2}^{\uparrow {#1}}}

\newcommand{\hiddenmessage}{}

\theoremstyle{plain}
\newtheorem{thm}{Theorem}[section]
\newtheorem*{thm*}{Theorem}
\newtheorem*{classification*}{Classification theorem}
\newtheorem{algo}[thm]{Algorithm}
\newtheorem{lem}[thm]{Lemma}
\newtheorem{cor}[thm]{Corollary}
\newtheorem{prop}[thm]{Proposition}

\theoremstyle{definition}
\newtheorem{notation}[thm]{Notation}
\newtheorem*{notation*}{Notation}
\newtheorem{defi}[thm]{Definition}

\newtheorem{conjecture}[thm]{Conjecture}
\newtheorem{exa}[thm]{Example}
\theoremstyle{remark}
\newtheorem{rema}[thm]{Remark}
\newtheorem*{rema*}{Remark}

\numberwithin{equation}{chapter}

\newtagform{roman}[]()

\begin{document}

\title{Yetter-Drinfeld modules over Nichols systems and their reflections}

\date{\today}

\author{
    Kevin Wolf\\
    \textit{\footnotesize{Fachbereich Mathematik und Informatik, Philipps-Universität Marburg}} \\
    \textit{\footnotesize{Hans-Meerwein-Straße 6, 35032 Marburg, Germany}}\\
    \textit{\footnotesize{\href{mailto:wolfke@mathematik.uni-marburg.de}{wolfke@mathematik.uni-marburg.de}}}
}

\maketitle

\vspace*{\fill}
\hfill \phantom{(empty page)} \hfill 
\vspace*{\fill}
\thispagestyle{empty}

\tableofcontents
\setcounter{page}{1}

\chapter*{Acknowledgements}
\chaptermark{Acknowledgements}
\addcontentsline{toc}{chapter}{Acknowledgements}

I would like to express my deepest gratitude for my supervisor Istv\'{a}n Heckenberger, for his mentorship, guidance and for his unconditional willingness to support this work with his vast knowledge and extraordinary prescience.

Moreover, I wish to thank Katharina Schäfer, who gave a significant contribution towards describing the structure of the elements $g_{n,k}$ appearing in \cref{sect_nich_sys_group_braided_shuffles}.

I would also like to thank Hiroyuki Yamane for the invitation to Toyama and for the opportunity to learn about his work, 
as well as for his remarkable hospitality.

Finally I wish to thank the Marburg University Research Academy, for their trust and for aiding and financing my work.

\chapter*{Preface}
\chaptermark{Preface}
\addcontentsline{toc}{chapter}{Preface}

The notion of a Nichols system is quite new and appeared naturally while studying Nichols algebras $\mathcal{B}(\oplus_k V_k)$ of direct sums of Yetter-Drinfeld modules $V_k$ (\cite{HeSch}, sections 13.2-13.3). This notion proved to be useful when looking at reflections of such Nichols algebras (\cite{HeSch}, sections 13.3-13.4) and it helps studying right coideal subalgebras of Nichols algebras (\cite{HeSch}, chapter 14).
Our first goal is to give an equivalent definition of a Nichols system, which does not require as many prerequisites (\cref{sect_nich_sys_definition}). We will study reflections of Nichols systems without constructing the necessary functors explicitly, but rather using only some important properties of such functors (\cref{sect_refl_functors}, \cref{sect_nich_sys_reflections}). This approach differs from the one in \cite{HeSch} and interestingly it turns out, that the only required properties will be about how the functors changes $\ndZ$-gradings, see \Cref{prop_funCD_exist_and_properties}.

We can also use these functors to quite intrinsically construct reflection functors for Yetter-Drinfeld modules over Nichols systems. These functors are new and therefore we will discuss their fundamental properties in detail (\cref{sect_ydmod_nich_sys_definition}, \cref{sect_ydmod_nich_sys_reflections}). Using the roots of a Nichols system, we will obtain some properties about the geometry of the support of Nichols systems and their Yetter-Drinfeld modules, by looking at iterated reflections (\cref{sect_nich_sys_iterated_refl}, \cref{sect_ydmod_nich_sys_iterated_reflections}). 
For example we will show, that if a Yetter-Drinfeld module satisfies specific properties, then any proper subobject must reside in the interior of the Yetter-Drinfeld modules support  (\Cref{prop_real_subobject_lies_in_interior}). These geometric properties will be used in the following sections quite often and allow one to develop an intuition about the discussed assertions.

Having done the necessary preparations our next topic will be the description of the maximal subobject of Yetter-Drinfeld modules over Nichols systems (\cref{sect_ydmod_nich_sys_shapovalov}), which also describes the irreducibility of the modules.
Given a Yetter-Drinfeld module we will find a special morphism, that we name Shapovalov morphism (\Cref{defi_shapovalov_morphism}), whose kernel coincides with the maximal subobject of the Yetter-Drinfeld module (\Cref{prop_shapoendo_max_subobject}). This morphism will prove to be incredibly useful, as it behaves well under reflections (\Cref{prop_kernel_shapo_of_reflection}) and it can even be used to characterize properties about the reflections of the Yetter-Drinfeld module, by looking at how it behaves on specific components of the support of the module (\Cref{cor_reflection_sequence_shapo_equiv}, \Cref{cor_reflection_sequence_shapo_equiv_bound}).

We then calculate an explicit formula of the Shapovalov morphism (\Cref{thm_shapoendo_via_gnk}) in the case where the Nichols system is of group type (\cref{sect_nich_sys_group_main}). The formula relies on some special variant of braided shuffles, described in \cref{sect_nich_sys_group_braided_shuffles}, that unexpectedly have a commuting relation with the braided symmetrizer (\Cref{thm_symmetrizer_commutes_gnk}). 
We will use the formula to calculate the maximal subobject of Yetter-Drinfeld modules over Nichols systems of group type, i.e. the kernel of the corresponding Shapovalov morphism, in the components of degree $2$ (\Cref{prop_ker_shapo_deg2}).
Finally, we will use the results of this chapter to calculate the maximal subobject of some examples of Yetter-Drinfeld modules over Nichols algebras of group type (\cref{sect_nich_sys_group_examples}), namely the Fomin-Kirillov algebras of dimension $12$ and $576$ (\cite{FominKirillovAlgebra}), as well as a $1280$-dimensional Nichols algebra of nonabelian group type given by a quandle of $5$ elements appearing in the list of \cite{Zoo}.

With the explicit formula of the Shapovalov morphism, not much more is needed to ascribe the theory of reflections of Yetter-Drinfeld modules over Nichols systems of diagonal type (\cref{sect_nich_sys_diag_main}) with the reflection theory of Dynkin diagrams (\cref{sect_nich_sys_diag_dynkin}).

\thispagestyle{plain}

Throughout we will also apply and specify the theory to the Yetter-Drinfeld modules over Nichols systems that are obtained by inducing comodules of the Nichols systems (\cref{sect_ydmod_nich_sys_induced_yd_module}), a construction that is reminiscent of Verma modules in the representation theory of Lie algebras. In particular, for Nichols systems of diagonal type we will obtain the result, that such an induced objects irreducibility can be characterized by a polynomial that is given by the positive roots of its Nichols system (\Cref{prop_ydind_refl_admit_equiv}, \Cref{cor_ydind_refl_admit_equiv}). This polynomial (\Cref{defi_shapovalov_determinant}) also appears in \cite{MR1684154}, \cite{MR1938453}, \cite{MR2490263}, but most notably in a similar context in \cite{MR2840165}, and is known as the Shapovalov determinant. While in \cite{MR2840165} a similar result was obtained with a lot of technicality and explicit calculations, here, in a more general context, we will obtain this polynomial from the Shapovalov morphism, which is the reason we chose this name. We will explain how to calculate the roots and the polynomial (\Cref{algo_diag_irred}) and in \cref{sect_nich_sys_diag_examples} we will do so for some examples that appear in \cite{MR2379892} and \cite{MR2500361}.

\chapter{Introduction}

We start in \cref{sect_prelim} by laying a foundation of definitions, notations and fundamentals necessary for this theory. \Cref{prop_irred_graded_ydmodule_o} turned out to be especially important, as it is used throughout this work in multiple occasions. 

Then in \cref{sect_refl_functors} we define and discuss reflection functors of categories of Yetter-Drinfeld modules in general. These are the primary building blocks of the following chapters.

\section{Preliminaries} \label{sect_prelim}
Let $\scr{C}$ be a braided strict monoidal category and let $\fK$ be a field.

\begin{notation*}
	Let $\scr{M}$ be any braided monoidal category. We denote $\brd^{\scr{M}}$ for the braiding of $\scr{M}$.
	For a algebra $A$ in $\scr{M}$ we denote $\biunit_A$ for its unit morphism, $\bimu_A$ for its multiplication and $\modcat{A}{\scr{M}}$, $\modcatr{A}{\scr{M}}$ for the category of left and right $A$-modules, respectively.
	For a coalgebra $C$ in $\scr{M}$ we denote $\bicounit_C$ for its counit, $\bicomu_C$ for its comultiplication and $\comodcat{C}{\scr{M}}$, $\comodcatr{C}{\scr{M}}$ for the category of left and right $C$-comodules, respectively.
	The antipode of a Hopf algebra $H$ in $\scr{M}$ is denoted with $\antip_H$.
	
	Moreover for a algebra $A\in \scr{M}$ and an $A$-module $V$ in $\scr{M}$ we denote $\act_V^A$ for its $A$-action. For a coalgebra $C\in \scr{M}$ and a $C$-comodule $W$ in $\scr{M}$ we denote $\coact_W^C$ for its $C$-coaction. 
\end{notation*}

\begin{defi}
    Let $R$ be a Hopf algebra in $\scr{C}$ and $V \in \scr{C}$. Then $V$ is called a \textbf{left Yetter-Drinfeld module} (abbreviated \textbf{$\ydmod{}$-module}) \textbf{over $R$ in $\scr{C}$}, if $V$ is a left $R$-module, $V$ is a left $R$-comodule and
    \begin{align*}
        & (\bimu_R \ot \act_V^R) (\id_R \ot \brd_{R,R}^{\scr{C}} \ot \id_V) (\bicomu_R \ot \coact_V^R) 
        \\=& (\bimu_R \ot \id_V) (\id_R \ot \brd_{V,R}^{\scr{C}}) (\coact_V^R \act_V^R \ot \id_R) (\id_R \ot \brd_{R,V}^{\scr{C}}) (\bicomu_R \ot \id_V) \matdot
    \end{align*}
    We denote $\ydcat{R}{\scr{C}}$ for the category where objects are left $\ydmod{}$-modules and where morphisms are $R$-module and $R$-comodule morphisms. Additionally if $H$ is a Hopf algebra over $\fK$, then we denote $\ydmod{H} := \ydcat{H}{\scr{M}_{\fK}}$, where $\scr{M}_{\fK}$ is the category of vector spaces over $\fK$.
\end{defi}
\begin{rema}
    Let $R$ be a Hopf algebra in $\scr{C}$. Then $\ydcat{R}{\scr{C}}$ is a braided category with braiding given by
    \begin{align*}
        \brd_{V,W}^{\ydcat{R}{\scr{C}}} := (\act_W^R \ot \id_V) (\id_R \ot \brd_{V,W}^{\scr{C}}) (\coact_V^R \ot \id_W) : V \ot W \rightarrow W \ot V
    \end{align*}
    for objects $V,W \in \ydcat{R}{\scr{C}}$, see \cite{HeSch}, Section 3.4.
\end{rema}

\begin{lem} \label{lem_defining_relation_ydmodule_equivalence}
	Let $R$ be a Hopf algebra in $\scr{C}$ and let $V \in \scr{C}$ be a left $R$-module and a left $R$-comodule. Then $V$ is an object in $\ydcat{R}{\scr{C}}$ if and only if 
	\begin{align*}
		\coact_V^R \act_V^R = &(\bimu_R \ot \act_V^R)(\bimu_R \ot \brd_{R\ot V, R}^{\scr{C}}) (\id_R \ot \brd_{R,R}^{\scr{C}} \ot \id_V \ot \antip_R) \\& (\bicomu_R \ot \brd_{R, R\ot V}^{\scr{C}})(\bicomu_R \ot \coact_V^R) \matdot
	\end{align*}
\end{lem}
\begin{proof}
	Assume $V$ is an object in $\ydcat{R}{\scr{C}}$. Let
	\begin{align*}
	    f :=& (\bimu_R \ot \id_V) (\id_R \ot \brd_{V,R}^{\scr{C}}) (\coact_V^R \act_V^R \ot \id_R) (\id_R \ot \brd_{R,V}^{\scr{C}}) (\bicomu_R \ot \id_V)
        \\=& (\bimu_R \ot \act_V^R) (\id_R \ot \brd_{R,R}^{\scr{C}} \ot \id_V) (\bicomu_R \ot \coact_V^R) \matdot
	\end{align*}
	Then starting with the left hand side of $f$ and then replacing it with the right hand side we obtain
	\begin{align*}
	    \coact_V^R \act_V^R 
	    =& (\bimu_R \ot \id_V) (\id_R \ot \brd_{V,R}^{\scr{C}}) \left( f \ot \antip_R \right) (\id_R \ot \brd_{R,V}^{\scr{C}}) (\bicomu_R \ot \id_V)
	    \\ =& (\bimu_R \ot \act_V^R)(\bimu_R \ot \brd_{R\ot V, R}^{\scr{C}}) (\id_R \ot \brd_{R,R}^{\scr{C}} \ot \id_V \ot \antip_R) \\& (\bicomu_R \ot \brd_{R, R\ot V}^{\scr{C}})(\bicomu_R \ot \coact_V^R)
	    \matdot
	\end{align*}
	Now conversely if this relation holds for $V$, it is straight forward to prove
	\begin{align*}
	    & (\bimu_R \ot \id_V) (\id_R \ot \brd_{V,R}^{\scr{C}}) (\coact_V^R \act_V^R \ot \id_R) (\id_R \ot \brd_{R,V}^{\scr{C}}) (\bicomu_R \ot \id_V) 
        \\=& (\bimu_R \ot \act_V^R) (\id_R \ot \brd_{R,R}^{\scr{C}} \ot \id_V) (\bicomu_R \ot \coact_V^R) \matcom
	\end{align*}
	and thus $V \in \ydcat{R}{\scr{C}}$.
\end{proof}
\begin{lem} \label{lem_induced_Q_comodule_is_YD_module}
	Let $V \in \ydcat{R}{\scr{C}}$, and let $U \in \scr{C}$ be a sub $R$-comodule of $V$. Then $R \cdot U$ is a subobject of $V$ in $\ydcat{R}{\scr{C}}$.
\end{lem}
\begin{proof}
	Clearly $R \cdot U$ is a sub $R$-module of $V$. By \Cref{lem_defining_relation_ydmodule_equivalence} we obtain that it is also a sub $R$-comodule of $V$.
\end{proof}

\begin{defi} \label{defi_bosonisation}
    Let $H \in \scr{C}$ and $R \in \ydcat{H}{\scr{C}}$ be Hopf algebras in the specific categories. We denote $R \boso H$ for the Hopf algebra in $\scr{C}$ with unit $\biunit_{R\boso H} := \biunit_R \ot \biunit_H$, counit $\bicounit_{R\boso H} := \bicounit_R \ot \bicounit_H$, multiplication and comultiplication
    \begin{align*}
        \bimu_{R\boso H} :=& (\bimu_R \ot \id_H) (\id_R \ot \act_{R \ot H}^{H}) \matcom \\
        \bicomu_{R\boso H} :=&  (\id_R \ot \coact_{R \ot H}^H) (\bicomu_R \ot \id_H)
    \end{align*}
    and antipode $ \bimu_{R \ot H} ( \biunit_R \bicounit_R \ot \antip_H \ot \antip_R \ot \biunit_H \bicounit_H ) \bicomu_{R \ot H} $.
\end{defi}
\begin{rema}
    To see that $R \boso H$ from \Cref{defi_bosonisation} is indeed a Hopf algebra in $\scr{C}$ refer to \cite{HeSch}, Theorem 3.8.10.
\end{rema}

\begin{prop} \label{prop_boso_is_ydcat}
	Let $H \in \scr{C}$ and $R \in \ydcat{H}{\scr{C}}$ be Hopf algebras in the specific categories. Then the functor
	\begin{align*}
	\ydcat{R}{\ydcat{H}{\scr{C}}} &\rightarrow \ydcat{R \boso H}{\scr{C}} \matcom \\
	\left( \left( V, \act^H, \coact^H \right),\act^R, \coact^R \right) &\mapsto \left(V, \act^R (\id \ot \act^H), (\id \ot \coact^H) \coact^R \right)
	\end{align*}
	and where morphisms are mapped onto oneself, is a braided strict mo\-noidal isomorphism. The inverse is given by the functor
	\begin{align*}
		(V, \act^{R \boso H}, \coact^{R \boso H}) \mapsto \left( \left( V, \act^H, \coact^H \right),\act^R, \coact^R \right)
	\end{align*}
	for all $(V, \act^{R \boso H}, \coact^{R \boso H}) \in \ydcat{R \boso H}{\scr{C}}$, where
	\begin{align*}
		\act^H &= \act^{R \boso H} (\biunit_R \ot \id_H \ot \id_V) \matcom & \coact^H &= (\bicounit_R \ot \id_H \ot \id_V ) \coact^{R \boso H} \matcom 
		\\ \act^R &= \act^{R \boso H} (\id_R \ot \biunit_H \ot \id_V) \matcom & \coact^R &= (\id_R \ot \bicounit_H  \ot \id_V) \coact^{R \boso H}
		\matdot
	\end{align*}
\end{prop}
\begin{proof}
	Refer to \cite{HeSch}, Theorem 3.8.7.
\end{proof}

\begin{prop} \label{prop_funCD_induces_wtfunCD}
    Let $\scr{D}$ be a braided strict monoidal category. Moreover let $(\funCD, \funCDcoh) : \scr{C} \rightarrow \scr{D}$ be a braided monoidal functor and let $R \in \scr{C}$ be a Hopf algebra. The induced functor
	\begin{align*}
		(\widetilde{\funCD}, \widetilde{\funCDcoh}) : \ydcat{R}{\scr{C}} \rightarrow \ydcat{\funCD(R)}{\scr{D}}
	\end{align*}
	is braided monoidal. If $\funCD$ is an isomorphism, then so is $\widetilde{\funCD}$.
\end{prop}
\begin{proof}
	Indeed $\widetilde{\funCD}$ is well defined and for $X,Y \in  \ydcat{R}{\scr{C}}$ canonically $\funCDcoh_{X,Y}$ is a morphism in $\ydcat{\funCD(R)}{\scr{D}}$ and thus $(\widetilde{\funCD}, \widetilde{\funCDcoh}) $ is a monoidal functor. One can than check that
	\begin{align*}
		\funCD( \brd_{X,Y}^{\ydcat{R}{\scr{C}}} ) \funCDcoh_{X,Y} = \funCDcoh_{Y,X} \brd_{\funCD(X), \funCD(Y)}^{\ydcat{\funCD(R)}{\scr{D}}} \matcom
	\end{align*}
	hence $(\widetilde{\funCD}, \widetilde{\funCDcoh}) $ is also a braided functor.
\end{proof}

\begin{defi}
    Let $A$ be an algebra in $\scr{C}$ and $C$ be a coalgebra in $\scr{C}$.
    An \textbf{$A$-bimodule} is a triple $(V,\act_l,\act_r)$, such that $V \in \scr{C}$, $\act_l : A \ot V \rightarrow V$ is a left $A$-module structure for $V$, $\act_r : V \ot A \rightarrow V$ is a right $A$-module structure for $V$ and $\act_r (\act_l \ot \id_A)=\act_l (\id_A \ot \act_r)$. 
    
    Similarly a \textbf{$C$-bicomodule} is a triple $(W,\coact_l,\coact_r)$, such that $W \in \scr{C}$, $\coact_l:W\rightarrow C \ot W$ is a left $C$-comodule structure for $W$, $\coact_r: W \rightarrow W \ot C$ is a right $C$-comodule structure for $W$ and $(\id_C \ot \coact_r)\coact_l=(\coact_l \ot \id_C)\coact_r$. 
    
    The category of $A$-bimodules is denoted $\bimodcat{A}{\scr{C}}$ and the category of $C$-bicomo\-dules is denoted $\bicomodcat{R}{\scr{C}}$.
\end{defi}

\begin{defi}
    Let $R$ be a Hopf algebra in $\scr{C}$.
    For an $R$-bimodule $(V,\act_l, \act_r) \in \bimodcat{R}{\scr{C}}$ and an $R$-bicomodule $(W,\coact_l,\coact_r) \in \bicomodcat{R}{\scr{C}}$ define
    \begin{align*}
		\ad_{V} := & \act_r (\act_l \ot \antip_R) (\id_R \ot \brd^{\scr{C}}_{R,V})(\bicomu_R \ot \id_V) : R\ot V \rightarrow V \matcom
		\\ \coad_{W} := & (\bimu_R \ot \id_W) (\id_R \ot \brd^{\scr{C}}_{W,R}) (\coact_l \ot \antip_R) \coact_r : W \rightarrow R \ot W \matdot
	\end{align*} 
\end{defi}

\begin{prop} \label{prop_ad_coad_functor}
    Let $R$ be a Hopf algebra in $\scr{C}$. The following are well-defined functors, where Morphisms are mapped onto oneself:
    \begin{align*}
        \comodcat{R}{(\bimodcat{R}{\scr{C}} ) } &\rightarrow \ydcat{R}{\scr{C}} \matcom & ((V,\act_l,\act_r),\coact) &\mapsto (V, \ad_{V}, \coact) \matcom
        \\ \modcat{R}{(\bicomodcat{R}{\scr{C}} ) } &\rightarrow \ydcat{R}{\scr{C}} \matcom & ((W,\coact_l,\coact_r),\act) &\mapsto (W, \act, \coad_{W}) \matdot
    \end{align*}
\end{prop}
\begin{proof}
    Refer to \cite{HeSch}, Proposition~3.7.8 and Proposition~3.7.9.
\end{proof}

\begin{rema} \label{rema_prop_ad_coad_functor}
    Let $R$ be a Hopf algebra in $\scr{C}$. Let $V \in \modcat{R}{\scr{C}}$ and $W \in \comodcat{R}{\scr{C}}$. Since $(R,\bimu_R,\bimu_R) \in \bimodcat{R}{\scr{C}}$ and since $(V,\act^R_V,(\id_V \ot \bicounit_R)) \in \bimodcat{R}{\scr{C}}$ we have an $R$-bimodule structure on $R \ot V$, which is an $R$-comodule via $\bicomu_R \ot \id_V$.
    Similarly since $(R,\bicomu_R,\bicomu_R) \in \bicomodcat{R}{\scr{C}}$ and $(W,\coact^R_W,(\id_W \ot \biunit_R))  \in \bicomodcat{R}{\scr{C}}$ we get an $R$-bicomodule structure on $R \ot W$, that is an $R$-module via $\bimu_R \ot \id_W$.
    Hence \Cref{prop_ad_coad_functor} induces functors
    \begin{align*}
        \modcat{R}{\scr{C}}  &\rightarrow \ydcat{R}{\scr{C}} \matcom & (V,\act_l,\act_r) &\mapsto (R\ot V, \ad_{R\ot V}, \bicomu_R \ot \id_V) \matcom
        \\ \comodcat{R}{\scr{C}} &\rightarrow \ydcat{R}{\scr{C}} \matcom & (W,\coact_l,\coact_r) &\mapsto (R\ot W, \bimu_R \ot \id_W, \coad_{R\ot W}) \matcom
    \end{align*}
    where morphisms $f$ are mapped to $\id_R \ot f$ and where
    \begin{align*}
		\ad_{R\ot V} := & 
		(\bimu_R \ot \act^R_V) (\bimu_R \ot \brd_{R\ot V,R}^{\scr{C}}) (\id_R \ot \brd^{\scr{C}}_{R,R} \ot \id_V \ot \antip_R) \\
	&	(\bicomu_R \ot \brd_{R,R \ot V}^{\scr{C}}) (\bicomu_R \ot \id_{R\ot V}) \matcom
		\\ \coad_{R \ot W} := &  (\bimu_R \ot \id_{R\ot V})  (\bimu_R \ot \brd_{R \ot V,R}^{\scr{C}}) (\id_R \ot \brd^{\scr{C}}_{R,R} \ot \id_V \ot \antip_R) 
	\\& (\bicomu_R \ot \brd_{R,R\ot V}^{\scr{C}}) (\bicomu_R \ot \coact^R_V) \matdot
	\end{align*} 
\end{rema}

\begin{defi}
    Let $\Gamma$ be a monoid. An object $X \in \scr{C}$ together with objects $X(\gamma)$, $\gamma \in \Gamma$ is called \textbf{$\Gamma$-graded}, if $X = \oplus_{\gamma \in \Gamma} X(\gamma)$. For two $\Gamma$-graded objects $X,Y \in \scr{C}$ a morphism $f:X\rightarrow Y$ is called \textbf{graded}, if for all $\gamma \in \Gamma$ there exists a morphism $f_\gamma : X(\gamma) \rightarrow Y(\gamma)$, such that the diagram
    \begin{align*}\begin{tikzcd}[ampersand replacement=\&]
	X  \arrow{r}{f} \& 	Y \arrow{d}{\proj_\gamma} \\
	X(\gamma) \arrow{u}{\incl_\gamma} \arrow{r}{f_\gamma} \& Y(\gamma) 
	\end{tikzcd}\end{align*}
    commutes, where $\incl_\gamma : X(\gamma) \rightarrow X$ is the canonical inclusion and where $\proj_\gamma : Y \rightarrow Y(\gamma)$ is the canonical projection.
    
    An algebra $A \in \scr{C}$ is called \textbf{$\Gamma$-graded algebra}, if $A$ is a graded object in $\scr{C}$ such that $\bimu_A$ and $\biunit_A$ are graded morphisms.
    If $A \in \scr{C}$ is a $\Gamma$-graded algebra, then $V \in \modcat{A}{\scr{C}}$ is called \textbf{$\Gamma$-graded $A$-module}, if $V$ is a graded object in $\scr{C}$ such that $\act_V^A$ is a graded morphism.
    
    A coalgebra $C \in \scr{C}$ is called \textbf{$\Gamma$-graded coalgebra}, if $C$ is a graded object in $\scr{C}$ such that $\bicomu_C$ and $\bicounit_C$ are graded morphisms.
    If $C \in \scr{C}$ is a $\Gamma$-graded coalgebra, then $W \in \comodcat{C}{\scr{C}}$ is called \textbf{$\Gamma$-graded $C$-comodule}, if $W$ is a graded object in $\scr{C}$ such that $\coact_W^C$ is a graded morphism.
    
    Moreover a bialgebra $B \in \scr{C}$ is a \textbf{$\Gamma$-graded bialgebra}, if $B$ is both a $\Gamma$-graded algebra and coalgebra. A Hopf algebnra $R \in \scr{C}$ is a \textbf{$\Gamma$-graded Hopf algebra}, if $R$ is a graded bialgebra and $\antip_R$ is graded.
    
    Finally for a $\Gamma$-graded Hopf algebra $R \in \scr{C}$ an object $V \in \ydcat{R}{\scr{C}}$ is called \textbf{$\Gamma$-graded $\ydmod{}$-module}, if $V$ is both a $\Gamma$-graded module and comodule in $\scr{C}$.

\end{defi}
\begin{defi}
    Let $\Gamma$ be a monoid and $V \in \scr{C}$ be a $\Gamma$-graded object.
	For $\Gamma' \subset \Gamma$ let
	\begin{align*}
	    V(\Gamma') := \oplus_{\gamma \in \Gamma'} V(\gamma) \matdot
	\end{align*}
	Moreover let $\Sup{V} := \lbrace \gamma \in \Gamma \, | \, V(\gamma) \ne 0 \rbrace$ denote the \textbf{support of $V$}.
\end{defi}
\begin{rema}
    Let $\Gamma, \Gamma'$ be monoids and let $f:\Gamma \rightarrow \Gamma'$ be a monoidal morphism. If $V \in \scr{C}$ is a $\Gamma$-graded object, then
    $V$ is also a $\Gamma'$-graded object via $V(\gamma)=V(f^{-1}(\gamma))$ for all $\gamma \in f(\Gamma)$ and $V(\gamma)=0$ for all $\gamma \in \Gamma' \setminus f(\Gamma)$.
\end{rema}

The rest of this section assumes that $\scr{C}$ is a subcategory of the category where objects are sets and morphisms are maps between sets.

\begin{defi}
    Let $\Gamma$ be a partial ordered monoid with partial order denoted by $\le$.
    Let $R$ be a $\Gamma$-graded Hopf algebra in $\scr{C}$ and let $V$ be a left $R$-module. We call $V$ a \textbf{rational $R$-module}, if for all $v \in V$ there exists $\gamma_0 \in \Gamma$, such that $R(\gamma) v = 0$ for all $\gamma \in \Gamma$, $\gamma_0 \le \gamma$.
    Moreover we denote $\ydcat{R}{\scr{C}}_{\rat}$ for the category of $\ydmod{}$-modules which are rational $R$-modules.
\end{defi}

\begin{defi} 
	Let $\theta \in \ndN$, $C \in \scr{C}$ be a $\ndN_0^\theta$-graded coalgebra and $Y \in \scr{C}$ a left $C$-comodule. We denote
	\begin{align*}
	     \filtcomod{0}(Y) := {\coact^C_{Y}}^{-1}(C(0) \ot Y)=\lbrace y \in Y \, | \, \coact^C_Y(y) \in C(0) \ot Y \rbrace \matdot
	\end{align*}
\end{defi}

\begin{lem} \label{lem_irred_graded_ydmodule}
    Let $\theta \in \ndN$, $C \in \scr{C}$ be a $\ndN_0^\theta$-graded coalgebra and $Y \in \scr{C}$ a left $C$-comodule. Let $\Gamma$ be an abelian group, such that $\ndZ^\theta \subset \Gamma$. Then
    \begin{enumerate}
        \item $\coact^C_Y (\filtcomod{0}(Y)) \subset C(0) \ot \filtcomod{0}(Y)$.
        \item If $Y \ne 0$, then $\filtcomod{0}(Y) \ne 0$.
        \item $\filtcomod{0}(Y)$ is a sub $C$-comodule of $Y$.
        \item If $Y$ is $\Gamma$-graded $C$-comodule, then $\filtcomod{0}(Y)$ is a graded sub $C$-comodule of $Y$.
    \end{enumerate}
\end{lem}
\begin{proof}
    (3) is implied by (1). (4) is then given by definition, since $\coact^C_Y$ is a graded map and $C(0) \ot Y$ is a graded subobject of $C \ot Y$ in $\scr{C}$.
    
    (1): By definition we have $\coact^C_Y (\filtcomod{0}(Y)) \subset C(0) \ot Y$. Since $\bicomu_C$ is graded, we have $\bicomu_C(C(0)) \subset C(0) \ot C(0)$. Then
    \begin{align*}
        (\id_{C(0)} \ot \coact^C_Y) \coact^C_Y (\filtcomod{0}(Y)) = (\bicomu_C \ot \id_Y) \coact^C_Y(\filtcomod{0}(Y)) = C(0) \ot C(0) \ot Y \matdot
    \end{align*}
    Applying $\bicounit_C$ yields $\coact^C_Y (\filtcomod{0}(Y)) \subset C(0) \ot \filtcomod{0}(Y)$.
    
    (2): Let $p: \ndZ^\theta \rightarrow \ndZ$, $(n_1,\ldots,n_\theta) \mapsto \sum_{i=1}^{\theta} n_i$.
    We view $C$ as a $\ndN_0$-graded coalgebra via $C(n) = C(p^{-1}(n))$ for all $n \in \ndN_0$.
    Let $0 \ne y \in Y$. Moreover let $n \in \ndN_0$ and let $y_i \in C(i) \ot Y$ for all $0 \le i \le n$, such that $y_n \ne 0$ and $\coact^C_Y(y) = \sum_{i = 1}^n y_i$. Since $\bicomu_C$ is graded, we have
    \begin{align*}
       (\bicomu_C \ot \id_Y) (y_i) \in \oplus_{i_1,i_2 \in \ndN_0}^{i_1+i_2=i} C(i_1) \ot C(i_2) \ot Y
    \end{align*}
    for all $0 \le i \le n$.
    Thus
    \begin{align*}
        & \sum_{i=1}^n (\id_C \ot \coact^C_Y)(y_i) = (\id_C \ot \coact^C_Y) \coact^C_Y (y)
        = (\bicomu_C \ot \id_Y) \coact^C_Y (y)
        \\ =& \sum_{i=0}^n(\bicomu_C \ot \id_Y)(y_i) \in \oplus_{i=1}^n \oplus_{i_1,i_2 \in \ndN_0}^{i_1+i_2 =i} C(i_1) \ot C(i_2) \ot Y
    \end{align*}
    Since by definition $(\id_C \ot \coact^C_Y)(y_k) \in C(k) \ot C \ot Y$ for all $0 \le k \le n$ we conclude
    \begin{align*}
        (\id_C \ot \coact^C_Y)(y_k) \in \oplus_{i=k}^n C(k) \ot C(i-k) \ot Y
    \end{align*}
    In particular $(\id_C \ot \coact^C_Y)(y_n) \in C(n) \ot C(0) \ot Y$. As $y_n \ne 0$, let $m \in \ndN$, $d_1,\ldots,d_m \in C(n)$ be linear independent, $z_1,\ldots,z_m \in Y \setminus \lbrace 0\rbrace $, such that $y_n = \sum_{i=1}^m d_i \ot z_i$. Then by above $\coact^C_Y(z_i) \in C(0) \ot Y$ for all $1\le i \le m$, i.e. $z_1,\ldots,z_m \in \filtcomod{0}(Y)$.
\end{proof}

\begin{defi} \label{defi_well_graded}
    Let $\theta \in \ndN$, $\Gamma$ be an abelian group such that $\ndZ^\theta \subset \Gamma$ and let $R$ be a $\ndN_0^\theta$-graded Hopf algebra in $\scr{C}$. Let $V \in \ydcat{R}{\scr{C}}$ be a $\Gamma$-graded object, where there exists $n_0 \in \Gamma$ such that $V(n_0)$ generates $V$ as an $R$-module. We say $V$ is \textbf{well graded}, if $V(n_0) = \filtcomod{0}(V)$.
\end{defi}

\begin{rema} \label{rema_defi_well_graded}
	In \Cref{defi_well_graded} we have $\Sup{V} \subset n_0 + \ndN_0^\theta$ and thus always $V(n_0) \subset \filtcomod{0}(V)$, since the $R$-coaction $\coact_V^R$ is graded. Moreover if $R$ is connected, i.e. $R(0)$ coincides with the unit object of $\scr{C}$, then
	\begin{align*}
		\filtcomod{0}(V) = \lbrace v \in V \, | \, \coact_V^R(v) = 1 \ot v \rbrace\matcom
	\end{align*}
	that is $\filtcomod{0}(V)$ coincides with the coinvariant elements of $V$.
\end{rema}

\begin{rema}
    Let $\theta \in \ndN$, $\Gamma$ be an abelian group such that $\ndZ^\theta \subset \Gamma$ and let $R$ be a $\ndN_0^\theta$-graded Hopf algebra in $\scr{C}$. Let $V \in \ydcat{R}{\scr{C}}$ be a $\Gamma$-graded object. Then for all $\gamma \in \Gamma$, $V(\gamma) \in \ydcat{R(0)}{\scr{C}}$, where the action is the by $\act^R_V$ induced action and the coaction is the by $\coact^R_V$ projected coaction. This is well defined, since $\act^R_V$ and $\coact^R_V$ are graded.
\end{rema}

\begin{prop} \label{prop_irred_graded_ydmodule_o}
	Let $\theta \in \ndN$, $\Gamma$ be an abelian group such that $\ndZ^\theta \subset \Gamma$ and let $R$ be a $\ndN_0^\theta$-graded Hopf algebra over $\fK$. Let $V \in \ydmod{R}$ be a $\Gamma$-graded object.
	Then the following are equivalent:
	\begin{enumerate}
		\item $V$ is irreducible in $\ydmod{R}$.
		\item $V$ is irreducible as an object of the category of $\Gamma$-graded objects of $\ydmod{R}$.
		\item There exists $n_0 \in \Gamma$ such that \begin{enumerate}
			\item $V(n_0)$ is irreducible in $\ydmod{R(0)}$.
			\item $V(n_0)$ generates $V$ as an $R$-module.
			\item $V$ is well graded.
		\end{enumerate}
	\end{enumerate}
\end{prop}
\begin{proof}
	The following is a generalisation of the proof of \cite{HeSch}, Proposition~13.1.2.
	
	Clearly (1)$\implies$(2) holds. (2)$\implies$(3): With \Cref{lem_irred_graded_ydmodule}(2),~(3) and~(4) we obtain that $0\ne \filtcomod{0}(V)$ is a graded sub $R$-comodule of $V$. Let $n_0 \in \Gamma$, such that $\filtcomod{0}(V)(n_0) \ne 0$. By \Cref{lem_irred_graded_ydmodule}(1) we know that 
	\begin{align*}
	\coact^R_V(\filtcomod{0}(V)(n_0)) \subset H(0) \ot \filtcomod{0}(V) \matdot
	\end{align*}
	Since $\coact^R_V$ is graded, this implies $\coact^R_V(\filtcomod{0}(V)(n_0)) \subset H(0) \ot \filtcomod{0}(V)(n_0)$, i.e. that $\filtcomod{0}(V)(n_0)$ is a sub $R$-comodule of $V$. Then \Cref{lem_induced_Q_comodule_is_YD_module} implies that $R \cdot \filtcomod{0}(V)(n_0)$ is a graded subobject of $V$ in $\ydmod{R}$. Thus by (2) we obtain $R \cdot \filtcomod{0}(V)(n_0) = V$, in particular $V(\gamma) = R(\gamma - n_0) \cdot \filtcomod{0}(V)(n_0)$ for all $\gamma \in \Gamma$. This means that $V$ is generated as an $R$-module by $\filtcomod{0}(V)(n_0)$ and that 
	\begin{align*}
	    \Sup{V} \subset n_0 + \ndN_0^\theta \matdot
	\end{align*}
	Morover $n_0 \in \Sup{V}$, as $V(n_0) = R(0) \cdot \filtcomod{0}(V)(n_0) \ne 0$.
	Since $n_0 \in \Gamma$ was an arbitrary element, such that $\filtcomod{0}(V)(n_0)\ne 0$ we can conclude that if $n_1 \in \Gamma$, such that $\filtcomod{0}(V)(n_1)\ne 0$, then 
	$n_1 \in \Sup{V} \subset n_0 + \ndN_0^\theta$ and $n_0 \in n_1 + \ndN_0^\theta$, implyingt $n_0 = n_1$. 
	Hence $\filtcomod{0}(V)(\gamma) = 0$ for all $n_0 \ne \gamma \in \Gamma$ and thus $\filtcomod{0}(V)=\filtcomod{0}(V)(n_0) \subset V(n_0)$. Considering \Cref{rema_defi_well_graded} we conclude $\filtcomod{0}(V)=V(n_0)$, implying (3)(b) and (3)(c). 
	Finally let $0 \ne X$ be a subobject of $V(n_0)$ in $\ydmod{R(0)}$. By (3)(b) We have
	$
	    \coact^R_V(V(n_0)) \subset R(0) \ot V(n_0)
	$,
	i.e. the $R(0)$-comodule structure of $V(n_0)$ coincides with the $R$-comodule structure of $V$ restricted to $V(n_0)$. It follows that $X$ is a sub $R$-comodule of $V$, hence $R \cdot X$ is a subobject of $V$ in $\ydmod{R}$ by \Cref{lem_induced_Q_comodule_is_YD_module}.
    Then $ R \cdot X = V$ by (2). In particular $V(n_0) = R(0) \cdot X \subset X$, i.e. $V(n_0)=X$. This implies (3)(a).
	
	(3)$\implies$(1): Let $0\ne X \subset V$ be a subobject in $\ydmod{R}$.
	Then we get $\filtcomod{0}(X) \ne 0$ by \Cref{lem_irred_graded_ydmodule}(2). 
	By \Cref{lem_irred_graded_ydmodule}(1) we know that $\coact^{R}_V (\filtcomod{0}(X)) \subset R(0) \ot \filtcomod{0}(X)$. Let $x \in \filtcomod{0}(X)$ and $r \in R(0)$. Then by \Cref{lem_defining_relation_ydmodule_equivalence}, since $\bicomu_R$, $\antip_R$ are graded and since $\coact^R_V (x) \in R(0) \ot X$, we obtain
	\begin{align*}
	    \coact^R_V( r \cdot x) = r_{(1)} x_{(-1)} \antip_R ( r_{(3)} ) \ot r_{(2)} \cdot x_{(0)} \in R(0) \ot X \matdot
	\end{align*}
	This implies $R(0) \cdot \filtcomod{0}(X) \subset \filtcomod{0}(X)$, hence $0\ne \filtcomod{0}(X) $ is a subobject of $\filtcomod{0}(V)=V(n_0)$ in $\ydmod{R(0)}$, using (3)(c). Thus $V(n_0)=\filtcomod{0}(X) \subset X$ by (3)(a). With (3)(b) it follows that $X=V$.
\end{proof}

For the rest of this section let $H$ be a Hopf algebra over $\fK$. We view $H$ as $\ndN_0$-graded with the trivial grading $H=H(0)$.

\begin{cor} \label{cor_irred_graded_ydmodule_o}
    Let $\theta \in \ndN$, $\Gamma$ be an abelian group such that $\ndZ^\theta \subset \Gamma$ and let $R$ be a $\ndN_0^\theta$-graded Hopf algebra in $\ydmod{H}$. Let $V \in \ydcat{R}{\ydmod{H}}$ be a $\Gamma$-graded object. Then the following are equivalent:
	\begin{enumerate}
		\item $V$ is irreducible in $\ydcat{R}{\ydmod{H}}$.
		\item $V$ is irreducible as an object of the category of $\Gamma$-graded objects of $\ydcat{R}{\ydmod{H}}$.
		\item There exists $n_0 \in \ndZ$ such that \begin{enumerate}
			\item $V(n_0)$ is irreducible in $\ydmod{R(0)\boso H}$.
			\item $V(n_0)$ generates $V$ as an $R$-module.
			\item $V$ is well graded.
		\end{enumerate}
	\end{enumerate}
\end{cor}
\begin{proof}
	The claim follows with \Cref{prop_boso_is_ydcat}, \Cref{prop_irred_graded_ydmodule_o} and the fact that $(R \boso H)(n) = R(n) \boso H$ for all $n \in \ndN_0^\theta$.
\end{proof}

\begin{defi}
    An algebra $A$ over $\fK$ equipped with a linear form $\frobelt : A \rightarrow  \fK$ is called \textbf{Frobenius algebra with Frobenius element $\frobelt$}, if $\frobelt \bimu_A : A \ot A \rightarrow \fK$ is a non-degenerate pairing.
\end{defi}
\begin{rema}
    Let $A$ be a finite dimensional Frobenius algebra over $\fK$ with Frobenius element $\frobelt$. Then
    \begin{align*}
        \phi: A \rightarrow A^*, \, \phi(a)(b) = \frobelt(ba), \text{ for all } a,b \in A \matcom
    \end{align*}
    is an isomorphism of left $A$-modules, where $A^*$ has $A$-module structure given by $af(b) = f(ba)$ for all $a,b \in A$, $f \in A^*$. In particular $\frobelt=\phi(1)$ is a generator of $A^*$ as an $A$-module.
\end{rema}

\begin{prop} \label{prop_fin_dim_hopf_alg_is_frobenius}
    Let $R$ be a finite-dimensional $\ndN_0$-graded Hopf algebra in $\ydmod{H}$. Let $N \in \ndN_0$, such that $R(N) \ne 0$ and $R(n)=0$ for all $n > N$ and $\proj_N : R \rightarrow R(N)$ be the canonical projection. Moreover let $0\ne \integralelt \in R(N)$.
    Then $R(N)= \fK \integralelt$ and $R$ is a Frobenius algebra with Frobenius element $\frobelt : R \rightarrow \fK$ given by
    \begin{align*}
        \proj_N(r) = \frobelt(r) \integralelt 
    \end{align*}
    for all $r \in R$.
\end{prop}
\begin{proof}
    This is proven in \cite{HeSch}, Theorem 4.4.13(3).
\end{proof}

\begin{cor} \label{cor_fin_dim_hopf_alg_is_frobenius}
    Let $R$ be a finite-dimensional $\ndN_0$-graded Hopf algebra in $\ydmod{H}$. Let $N \in \ndN_0$, such that $R(N) \ne 0$ and $R(n)=0$ for all $n > N$ and let $0\ne \integralelt \in R(N)$. Then the following hold:
    \begin{enumerate}
        \item $R(N) = \fK \integralelt$.
        \item For all $0 \ne x \in R$ there exists some $y \in R$, such that $yx=\integralelt$.
        \item $R(N)$ is contained in every non-zero left ideal of $R$.
    \end{enumerate}
\end{cor}
\begin{proof}
By \Cref{prop_fin_dim_hopf_alg_is_frobenius} we obtain (1) and that $R$ is a Frobenius algebra with Frobenius element $\frobelt : R \rightarrow \fK$ given by
    \begin{align*}
        \proj_N(r) = \frobelt(r) \integralelt \matcom
    \end{align*}
for all $r \in R$, where $\proj_N: R \rightarrow R(N)$ is the canonical projection. Let $0 \ne x \in R$ and let $m \in \ndN_0$, $x_i \in R(i)$ for all $m \le i \le N$, such that $x_m \ne 0$ and $x = \sum_{i=m}^N x_i$.
Since $\frobelt \bimu_R$ is non-degenrate, there exists $y \in R$, such that $\frobelt(y x_m) \ne 0$. Let $y_i \in R(i)$ for all $0 \le i \le R$, such that $y=\sum_{i=0}^N y_i$. Then
\begin{align*}
    y_{N-m} x_{m} = \proj_N(y x_m) = \frobelt(y x_m) \integralelt \matdot
\end{align*}
Observe that $y_{N-m} x_i \in R(N-m+i)$, in particular $y_{N-m} x_i=0$ for all $i > m$. Thus
\begin{align*}
    \frac{y_{N-m}}{\frobelt(y x_m)} x = \frac{y_{N-m} x_m}{\frobelt(y x_m)} = \integralelt
\end{align*}
and (2) is proven.
Finally (3) is implied by (1) and (2).
\end{proof}
\section{Reflection functors} \label{sect_refl_functors}

We define reflection functors in \Cref{defi_funCD_induced_funct}. They are constructed by using a functor between the categories of rational Yetter-Drinfeld modules of dual pairs of locally finite $\ndN_0$-graded Hopf algebras (\Cref{prop_funCD_exist_and_properties}), that reverses $\ndZ$-gradings. We will discuss in detail how gradings are preserved when applying these functors.

Let $H$ be a Hopf algebra over some field $\fK$ and denote $\scr{C} := \ydmod{H}$. Let $A,B$ be a dual pair of locally finite $\ndN_0$-graded Hopf algebras in the category $\scr{C}$. 

\begin{prop} \label{prop_funCD_exist_and_properties}
	There is a braided monoidal isomorphism
	\begin{align*}
		(\funCD, \funCDcoh): \ydcat{B}{\scr{C}}_{\rat} \rightarrow \ydcat{A}{\scr{C}}_{\rat} \matcom
	\end{align*}
	such that
	\begin{enumerate}
		\item For each $\ndZ$-graded object $V \in \ydcat{B}{\scr{C}}_{\rat}$ the object $\funCD(V)$ is $\ndZ$-graded with $\funCD(V)(n) = V(-n)$ for all $n \in \ndZ$. In particular $\funCD(V) = V$ as vector spaces.
		\item If $V,V' \in \ydcat{B}{\scr{C}}_{\rat}$ are $\ndZ$-graded and $f:V \rightarrow V'$ is a $\ndZ$-graded morphism, then $\funCD(f) = f$ as linear maps. In particular $\funCD(f): \funCD(V) \rightarrow \funCD(V')$ is a $\ndZ$-graded morphism.
	\end{enumerate}
\end{prop}
\begin{proof}
	The functor discussed in \cite{HeSch}, Section 12.3 has those properties, see Corollary 12.3.6.
\end{proof}

\begin{rema}
    Remarkably the properties of \Cref{prop_funCD_exist_and_properties} are the only ones which we need in this theory. At no point we require to know explicitly how the functor changes the action and coaction.
\end{rema}

Let $(\funCD, \funCDcoh)$ be a functor with the properties from \Cref{prop_funCD_exist_and_properties}.

\begin{lem} \label{prop_funCD_properties_irred}
	Let $V \in \ydcat{B}{\scr{C}}_{\rat}$ be an irreducible object.
	Then $\funCD(V)$ is irreducible in $\ydcat{A}{\scr{C}}_{\rat}$.
\end{lem}
\begin{proof}
	If $U$ is a subobject of $\funCD(V)$ in $\ydcat{C}{\scr{C}}_{\rat}$, then since $\funCD$ is an isomorphism, $U$ can be viewed as a subobject of $V$ in $\ydcat{B}{\scr{C}}_{\rat}$. Since $V$ is irreducible this implies $U=0$ or $U=V$.
\end{proof}

Let $\Gamma$ be an abelian group where there exists some $\alpha \in \Gamma$ and a subgroup $\Gamma' \subset \Gamma$, such that $\Gamma = \ndZ \alpha \oplus \Gamma'$. We view $B$ as a $\Gamma$-graded Hopf algebra with $B(n \alpha) := B(n)$ for all $n \in \ndN_0$ and $B(\gamma) = 0$ for all $\gamma \in \Gamma \setminus \ndN_0 \alpha$. Moreover we view $A$ as a $\Gamma$-graded Hopf algebra with $A(-n \alpha ) := A(n)$ for all $n \in \ndN_0$ and $A(\gamma) = 0$ for all $\gamma \in \Gamma \setminus \ndN_0 (-\alpha)$.

\begin{prop} \label{prop_gam_grad_funCD}
	 Let $V \in \ydcat{B}{\scr{C}}_{\rat}$ be a $\Gamma$-graded object. Then $\funCD(V) \in \ydcat{A}{\scr{C}}_{\rat}$ is $\Gamma$-graded with $\funCD(V)(\gamma) = V(\gamma)$ for $\gamma \in \Gamma$. Moreover if $K$ is a $\Gamma$-graded Hopf algebra in $\ydcat{B}{\scr{C}}$, then $\funCD(K)$ is a $\Gamma$-graded Hopf algebra in $\ydcat{A}{\scr{C}}$ with $\funCD(K)(\gamma) = K(\gamma)$ for $\gamma \in \Gamma$.
\end{prop}
\begin{proof}
	Let $\gamma \in \Gamma$. By assumption $\gamma = m \alpha + \gamma'$ for some $m \in \ndZ$, $\gamma' \in \Gamma'$.
	Now $V \in \ydcat{B}{\scr{C}}_{\rat}$ is $\ndZ$-graded with 
	\begin{align*}
		V(n) = V(n \alpha + \gamma')
	\end{align*}
	for $n \in \ndZ$. By \Cref{prop_funCD_exist_and_properties}(1) we obtain that $\funCD(V) \in \ydcat{A}{\scr{C}}_{\rat}$ is $\ndZ$-graded with
	\begin{align*}
	\funCD(V)(n) = V(-n \alpha + \gamma')
	\end{align*}
	for $n \in \ndZ$. This implies for $n \in \ndN_0$
	\begin{align*}
		& A(-n \alpha) V (\gamma) = A(n) V(m \alpha + \gamma') = A(n) \funCD(V)(-m) 
		\\ & \subset  \funCD(V)(n-m) = V( (-n+m) \alpha + \gamma') = V( -n \alpha + \gamma ) \matcom
	\end{align*}
	hence by construction $A(\beta) V (\gamma) \subset V(\beta + \gamma)$ for all $\beta \in \Gamma$.
	Now the above also implies
	\begin{align*}
		& \coact_{\funCD(V)}^A (V(\gamma)) = \coact_{\funCD(V)}^A(V(m\alpha + \gamma')) = \coact_{\funCD(V)}^A ( \funCD(V)(-m) ) 
		\\ &\subset \oplus_{m_1,m_2 \in \ndZ}^{m_1+m_2 = -m} A(m_1) \ot \funCD(V)(m_2)
		\\ &= \oplus_{m_1,m_2 \in \ndZ}^{m_1+m_2 = -m} A(-m_1 \alpha) \ot V(-m_2 \alpha + \gamma')
		\\ &\subset \oplus_{\gamma_1,\gamma_2 \in \Gamma}^{\gamma_1+\gamma_2 = \gamma} A(\gamma_1) \ot V(\gamma_2) \matcom
	\end{align*}
	hence the first claim is correct. Now if $K$ is a $\Gamma$-graded Hopf algebra in $\ydcat{B}{\scr{C}}$, then since $\funCD$ is braided monoidal, $\funCD(K)$ is a Hopf algebra in $\ydcat{A}{\scr{C}}$, with the same unit, multiplication, counit and comultiplication by \Cref{prop_funCD_exist_and_properties}(2). Since the grading is also the same, as discussed above, these structure morphisms remain $\Gamma$-graded.
\end{proof}

\begin{prop} \label{prop_grading_on_funCDK}
	Let $K \in \ydcat{B}{\scr{C}}_{\rat}$ be a $\Gamma$-graded Hopf algebra. Then
	\begin{enumerate}
		\item $K \boso B$ is a $\Gamma$-graded Hopf algebra in $\scr{C}$ with 
		\begin{align*}
			(K\boso B)(\gamma) &= \oplus_{\gamma_1,\gamma_2 \in \Gamma}^{\gamma_1+\gamma_2=\gamma} K(\gamma_1) \ot B(\gamma_2) \\
			&= \oplus_{n_1 \in \ndZ, n_2 \in \ndN_0}^{n_1+n_2=n} K(n_1 \alpha + \gamma') \ot B(n_2) \matcom
		\end{align*}
		where $n \in \ndZ$, $\gamma' \in \Gamma'$, $\gamma = n\alpha + \gamma' \in \Gamma$.
		\item $\funCD(K) \boso A$ is a $\Gamma$-graded Hopf algebra in $\scr{C}$ with 
		\begin{align*}
		(\funCD(K) \boso A)(\gamma) &= \oplus_{\gamma_1,\gamma_2 \in \Gamma}^{\gamma_1+\gamma_2=\gamma} K(\gamma_1) \ot A(\gamma_2) \\
		&= \oplus_{n_1 \in \ndZ, n_2 \in \ndN_0}^{n_1-n_2=n} K(n_1 \alpha + \gamma') \ot A(n_2) \matcom
		\end{align*}
		where $n \in \ndZ$, $\gamma' \in \Gamma'$, $\gamma = n\alpha + \gamma' \in \Gamma$.
	\end{enumerate}
\end{prop}
\begin{proof}
	(1) follows quickly when using that the $B$-action and $B$-coaction of $K$ are $\Gamma$-graded and by the construction of the $\Gamma$-grading of $B$. (2)~is implied by \Cref{prop_gam_grad_funCD} and (1), as well as considering the construction of the $\Gamma$-grading of $A$.
\end{proof}

\begin{defi} \label{defi_funCD_induced_funct}
	Let $K \in \ydcat{B}{\scr{C}}_{\rat}$ be a Hopf algebra. We denote 
	\begin{align*}
		(\funCD_{K,B},\funCDcoh_{K,B}) : \ydcat{K \boso B}{\scr{C}}_{\rat} \rightarrow \ydcat{\funCD(K) \boso A}{\scr{C}}_{\rat}
	\end{align*}
	for the braided monoidal isomorphism defined such that the diagram
	\begin{align*}\begin{tikzcd}[ampersand replacement=\&]
	\ydcat{K \boso B}{\scr{C}}_{\rat}	 \arrow{d}{\cong} \arrow{r}{(\funCD_{K,B},\funCDcoh_{K,B})} \& 	\ydcat{\funCD(K) \boso A}{\scr{C}}_{\rat}  \\
	\ydcat{K}{\ydcat{B}{\scr{C}}_{\rat}} \arrow{r}{(\widetilde{\funCD},\widetilde{\funCDcoh})} \& \ydcat{\funCD(K)}{\ydcat{A}{\scr{C}}_{\rat}} \arrow{u}{\cong}
	\end{tikzcd}\end{align*}
	commutes, where $(\widetilde{\funCD},\widetilde{\funCDcoh})$ is the by $(\funCD,\funCDcoh)$ induced functor of \Cref{prop_funCD_induces_wtfunCD} and the two vertical isomorphisms are given by \Cref{prop_boso_is_ydcat}. We call $(\funCD_{K,B},\funCDcoh_{K,B})$ the \textbf{reflection functor induced by $(\funCD,\funCDcoh)$}.
\end{defi}

\begin{rema} \label{rema_explicit_action_coaction_funcdboso}
	Let $K \in \ydcat{B}{\scr{C}}_{\rat}$ be a Hopf algebra and $V \in \ydcat{K \boso B}{\scr{C}}_{\rat}$. Then the $\funCD(K) \boso A$-action and coaction of $\funCD_{K,B} (V) \in \ydcat{\funCD(K) \boso A}{\scr{C}}_{\rat}$ are
	\begin{align*}
		\act_{\funCD_{K,B} (V)}^{\funCD(K) \boso A} &= \act_V^{K \boso B} (\id_K \ot \biunit_B \ot \act_{\funCD(V)}^A) \matcom \\
		\coact_{\funCD_{K,B} (V)}^{\funCD(K) \boso A} &=  (\id_K \ot \bicounit_B \ot \coact_{\funCD(V)}^A) \coact_V^{K \boso B} \matcom
	\end{align*}
	where $\act_{\funCD(V)}^A : A \ot V \rightarrow V$ and $\coact_{\funCD(V)}^A :V \rightarrow A \ot V$ are the $A$-action and coaction of $\funCD(V) \in \ydcat{A}{\scr{C}}_{\rat}$, where $V \in \ydcat{B}{\scr{C}}_{\rat}$ has $B$-action and coaction 
	\begin{align*}
		\act_V^B &= \act_V^{K \boso B} (\biunit_K \ot \id_B \ot \id_V) \matcom & \coact_V^B &= (\bicounit_K \ot \id_B \ot \id_V) \coact_V^{K \boso B}
		\matdot
	\end{align*}
\end{rema}

\begin{lem} \label{lem_funcd_keeps_irreducibility}
	Let $K \in \ydcat{B}{\scr{C}}_{\rat}$ be a Hopf algebra and $V \in \ydcat{K\boso B}{\scr{C}}_{\rat}$.
	Let $U \subset V$ be a subset, such that $U$ is also a subobject of $\funCD_{K,B} (V)$ in $\ydcat{\funCD(K) \boso A}{\scr{C}}_{\rat}$.
	Then $U$ is a subobject of $V$ in $\ydcat{K \boso B}{\scr{C}}_{\rat}$.
	
	In particular if $V$ is an irreducible object in $\ydcat{K\boso B}{\scr{C}}_{\rat}$, then $\funCD_{K,B} (V)$ is irreducible in $\ydcat{\funCD(K) \boso A}{\scr{C}}_{\rat}$.
\end{lem}
\begin{proof}
    The claim follows since $\funCD_{K,B}$ is an isomorphism.
\end{proof}

\begin{prop} \label{prop_funCD_module_grading}
	Let $K \in \ydcat{B}{\scr{C}}_{\rat}$ be a $\Gamma$-graded Hopf algebra. Moreover let $V \in \ydcat{K\boso B}{\scr{C}}_{\rat}$ be a $\Gamma$-graded object. Then the object $\funCD_{K,B} (V) \in \ydcat{\funCD(K) \boso A}{\scr{C}}_{\rat}$ is $\Gamma$-graded via
	$
		\funCD_{K,B} (V) (\gamma) := V(\gamma)
	$
	for $\gamma \in \Gamma$.
\end{prop}
\begin{proof}
	Let $\gamma,\gamma',\gamma'' \in \Gamma$. Using \Cref{rema_explicit_action_coaction_funcdboso} and \Cref{prop_gam_grad_funCD} we obtain 
	\begin{align*}
		 \left( \funCD(K)(\gamma) \boso A(\gamma') \right) V(\gamma'')
		&= (K(\gamma) \boso \fK 1_B) ( A(\gamma') V(\gamma'') )
		\\ &\subset K(\gamma) V(\gamma' + \gamma'') \subset K(\gamma + \gamma' + \gamma'') \matdot
	\end{align*}
	Moreover, since $\bicounit_B : B \rightarrow \fK$ is $\Gamma$-graded, where $\fK(0)=\fK$, $\fK(\beta)=0$ for all $\beta \in \Gamma \setminus \lbrace 0 \rbrace$, considering \Cref{rema_explicit_action_coaction_funcdboso} we obtain
	\begin{align*}
		&\coact_{\funCD_{K,B} (V)}^{\funCD(K) \boso A} (V(\gamma)) = (\id_K \ot \bicounit_B \ot \coact_{\funCD(V)}^A) \coact_V^{K\boso B} (V(\gamma))
		\\ \subset& (\id_K \ot \bicounit_B \ot \coact_{\funCD(V)}^A) \left( \oplus_{\gamma_1,\gamma_2,\gamma_3 \in \Gamma}^{\gamma_1+\gamma_2+\gamma_3=\gamma} K(\gamma_1) \ot B(\gamma_2) \ot V(\gamma_3) \right)
		\\ =& \oplus_{\gamma_1,\gamma_3 \in \Gamma}^{\gamma_1+\gamma_3=\gamma} K(\gamma_1) \ot \fK \ot \coact_{\funCD(V)}^A(V(\gamma_3))
		\\ \subset & \oplus_{\gamma_1,\gamma_2,\gamma_3 \in \Gamma}^{\gamma_1+\gamma_2+\gamma_3=\gamma} \funCD(K)(\gamma_1) \ot A(\gamma_2) \ot V(\gamma_3) \matdot
	\end{align*}
	The claim follows.
\end{proof}

\begin{rema} \label{rema_shifting_grading_by_autom}
	Let $s \in \Aut(\Gamma)$. If $K \in \scr{C}$ is a $\Gamma$-graded Hopf algebra, then $K(\gamma) := K(s(\gamma))$ for $\gamma \in \Gamma$, defines a new $\Gamma$-grading for $K$. 
	Moreover if $V \in \ydcat{K}{\scr{C}}$ is $\Gamma$-graded in regards to the standard $\Gamma$-grading of $K$, then $V(\gamma):= V(s(\gamma))$ for $\gamma \in \Gamma$, defines a $\Gamma$-grading for $V$ in regards to the by $s$ shifted $\Gamma$-grading of $K$.
\end{rema}

\chapter{Nichols systems}

In \cref{sect_nich_sys_definition} we define Nichols systems. They also appear in \cite{HeSch}, section~13.5, however our definition is a bit different, as it requires less prerequisites. We will also discuss how the two definitions are equivalent. 

Then in \cref{sect_nich_sys_reflections} we discuss how and when we can reflect a Nichols system. An important difference between this work and \cite{HeSch}, section 13.5, is the underlying functor that we obtain by \Cref{prop_funCD_exist_and_properties}. In \cite{HeSch} this functor is defined explicitly and proofs make use of its explicit structure. The goal here was to prove similar results, using only the properties of the functor in \Cref{prop_funCD_exist_and_properties}. The main result of this section is \Cref{prop_RiQ_is_nich_sys}, where we discuss when a reflection of a Nichols system is again a Nichols system.

Finally in \cref{sect_nich_sys_iterated_refl} we look at iterated reflections of Nichols systems. Doing so, we will also define the set of roots of a Nichols system, which coincides with a set of real roots of an associated semi-Cartan graph. We will use the roots to discuss the geometry of the support of a Nichols system, i.e. the set of indices of non-zero components regarding the grading.

Let $\theta \in \ndN$ and $\mathbb{I} = \lbrace 1, \ldots, \theta \rbrace$. Let $\alpha_1, \ldots, \alpha_\theta \in \ndN_0^\theta$ be the unit vectors.
Let $H$ be a Hopf algebra over some field $\fK$ with bijective antipode and $\scr{C} := \ydmod{H}$. We view $H$ as $\ndN_0^\theta$-graded with the trivial grading $H=H(0)$.

\section{Nichols systems} \label{sect_nich_sys_definition}

\begin{defi} \label{defi_pre_nich_sys}
	Let $Q \in \scr{C}$ be a Hopf algebra and $N_1, \ldots, N_\theta \in \scr{C}$ be finite-dimensional subobjects of $Q$ in $\scr{C}$ and $N:=(N_1,\ldots,N_\theta)$. The tuple $\scr{N} := (Q,N)$ is called a \textbf{pre-Nichols system}, if
	\begin{enumerate}
		\item The algebra $Q$ is generated by $N_1,\ldots,N_\theta$.
		\item $Q$ is an $\ndN_0^\theta$-graded Hopf algebra in $\scr{C}$ with $Q(\alpha_i)=N_i$ for $i \in \mathbb{I}$.
	\end{enumerate}
	For $i \in \mathbb{I}$ we denote
	\begin{align*}
	\nsalg{\scr{N}} := Q \matcom && \scr{N}_i = N_i \matdot
	\end{align*}
	A \textbf{morphism $f :\scr{N}\rightarrow \scr{N}'$ of pre-Nichols systems $\scr{N}$ and $\scr{N}'$} is a Hopf algebra morphism $f: \nsalg{\scr{N}}\rightarrow \nsalg{\scr{N}'}$, such that $f$ induces an isomorphism $f : \scr{N}_j \rightarrow \scr{N}'_j$ for all $j \in \mathbb{I}$.
\end{defi}

\begin{rema}
    Some categorial properties of the category of Nichols systems are discussed in \cite{HeSch}, Remark~13.5.23.
\end{rema}

\begin{notation}
	Let $\scr{N}=(Q,(N_1,\ldots,N_\theta))$ be a pre-Nichols system and $i \in \mathbb{I}$. We denote $\subalgQ{N_i}$ for the subalgebra of $Q$ generated by $N_i$. 
	Moreover denote
	\begin{align*}
	\inclcomp{Q}_i : \subalgQ{N_i} \rightarrow Q \matcom && \projcomp{Q}_i : Q \rightarrow \subalgQ{N_i}
	\end{align*}
	for the canonical $\ndN_0^\theta$-graded maps, which are the identity on $N_i$. Finally let 
	\begin{align*}
	K_i := \coinva{Q}{\subalgQ{N_i}} \matcom
	\end{align*}
	where the right $\subalgQ{N_i}$-comodule structure on $Q$ is given by $(\id_Q \ot \projcomp{Q}_i)\bicomu_Q : Q \rightarrow Q \ot \subalgQ{N_i}$.
\end{notation}

\begin{rema} \label{rema_structure_of_Ki}
	Let $\scr{N}=(Q,(N_1,\ldots,N_\theta))$ be a pre-Nichols system and $i \in \mathbb{I}$. 
	Since $\bicomu_{Q}$ is graded, we obtain that $\subalgQ{N_i}$ is a $\ndN_0^\theta$-graded sub-Hopf algebra of $Q$ in $\scr{C}$, where the grading is such that $\subalgQ{N_i}(n \alpha_i)=N_i^n$ for $n \in \ndN_0$ and all other components are zero. 
	Now $K_i$ is a Hopf algebra in $\ydcat{\subalgQ{N_i}}{\scr{C}}$, where the action is the adjoint action
	\begin{align*}
	\ad = \bimu_Q (\bimu_Q \ot \antip_Q) (\id_{\subalgQ{N_i}} \ot \brd^{\scr{C}}_{\subalgQ{N_i},K_i}) (\bicomu_{\subalgQ{N_i}} \ot \id_{K_i}) : \subalgQ{N_i} \ot K_i \rightarrow K_i
	\end{align*}
	and coaction
	\begin{align*}
		(\projcomp{Q}_i \ot \id) \bicomu_{Q} : K_i \rightarrow \subalgQ{N_i} \ot K_i \matdot
	\end{align*}
	Moreover we have a Hopf algebra isomorphy
	\begin{align*}
	K_i \boso \subalgQ{N_i} \rightarrow Q \matcom \,\, x \ot a \mapsto xa \matcom && x \in K_i, \, a \in \subalgQ{N_i} \matdot
	\end{align*}
	For more details refer to \cite{HeSch}, Corollary~4.3.1.
	Since the algebra $Q$ is generated by $N_1, \ldots, N_\theta$, it follows that the algebra $K_i$ is generated by
	\begin{align*}
	(\ad \, \subalgQ{N_i}) (N_j) \matcom && j \in \mathbb{I}\setminus \lbrace i  \rbrace \matcom
	\end{align*}
	see \cite{HeSch}, Lemma~2.6.25.
	Moreover $K_i$ inherits the $\ndN_0^\theta$-grading of $Q$ and thus is an $\ndN_0^\theta$-graded Hopf algebra in $\ydcat{\subalgQ{N_i}}{\scr{C}}$, where
	\begin{align*}
		(\ad N_i)^n (N_j) = K_i(\alpha_j + n \alpha_i)
	\end{align*}
	for $ j \in \mathbb{I}\setminus \lbrace i \rbrace$, $n \in \ndN_0$. In particular $K_i(n \alpha_i) = 0$ for all $n \in \ndN$.
\end{rema}

\begin{prop} \label{prop_ad_kNi_Nj_is_subobj}
	Let $\scr{N}=(Q,(N_1,\ldots,N_\theta))$ be a pre-Nichols system, $i \in \mathbb{I}$, $j \in \mathbb{I} \setminus \lbrace i \rbrace$ and $V := (\ad \, \subalgQ{N_i})(N_j)$. 
	Then $V$ is a $\ndN_0^\theta$-graded subobject of $K_i$ in $\ydcat{\subalgQ{N_i}}{\scr{C}}$. Moreover if we view $\subalgQ{N_i}$ as $\ndN_0$-graded Hopf algebra in $\scr{C}$, then $V$ is an $\ndN_0$-graded object in $\ydcat{\subalgQ{N_i}}{\scr{C}}$ via 
	\begin{align*}
	V(k) =	(\ad N_i)^k (N_j) \matcom && k \in \ndN_0 \matcom
	\end{align*}
	generated as a $\subalgQ{N_i}$-module by $V(0) = N_j$.
\end{prop}
\begin{proof}
	Let $\coact_{K_i} : K_i \rightarrow \subalgQ{N_i} \ot K_i$ be the graded $\subalgQ{N_i}$-coaction of $K_i$. Observe that
	\begin{align*}
		V= \oplus_{k \in \ndN_0} (\ad \, N_i)^k (N_j) = \oplus_{k \in \ndN_0}  K_i (\alpha_j + k \alpha_i) \matdot
	\end{align*}
	Then for $k \in \ndN_0$ we have
	\begin{align*}
		\coact_{K_i} ( (\ad N_i)^k (N_j) ) &\subset \oplus_{\beta_1,\beta_2 \in \ndN_0^\theta}^{\beta_1+\beta_2=\alpha_j + k \alpha_i} \subalgQ{N_i}(\beta_1) \ot K_i(\beta_2)
		\\&= \oplus_{k_1,k_2 \in \ndN_0}^{k_1+k_2=k} \subalgQ{N_i}(k_1 \alpha_i) \ot K_i(k_2 \alpha_i + \alpha_j)
		\\&= \oplus_{k_1,k_2 \in \ndN_0}^{k_1+k_2=k} \subalgQ{N_i}(k_1 \alpha_i) \ot (\ad \, N_i)^{k_2} (N_j) \matdot
	\end{align*}
	In particular, $\coact_{K_i}(V) \subset \subalgQ{N_i} \ot V$, hence $V$ is a subobject of $K_i$ in $\ydcat{\subalgQ{N_i}}{\scr{C}}$. Clearly the $\ndN_0^\theta$-grading of $V$ restricts to a $\ndN_0$-grading as claimed.
\end{proof}

\begin{lem} \label{lem_primitive_elts_of_adNi_Nj}
	Let $\scr{N}=(Q,(N_1,\ldots,N_\theta))$ be a pre-Nichols system, $i \in \mathbb{I}$, $j \in \mathbb{I} \setminus \lbrace i \rbrace$ and $V := (\ad \, \subalgQ{N_i})(N_j)$. Then the primitive elements of $Q$ in $V$ are given by $\filtcomod{0}(V)$.
\end{lem}
\begin{proof}
	By \Cref{rema_defi_well_graded} we have
	\begin{align*}
	\filtcomod{0}(V) = \lbrace v \in V \, | \,  (\projcomp{Q}_i \ot \id) \bicomu_Q (v) = 1 \ot v \rbrace \matdot
	\end{align*}
	Assume $v \in V$ is a primitive element. Since $\projcomp{Q}_i(v)=0$ we have 
	\begin{align*} 
		(\projcomp{Q}_i \ot \id) \bicomu_Q (v) = 1 \ot v \matcom
	\end{align*}
	hence $v \in \filtcomod{0}(V)$.
	
	Conversely assume $v \in \filtcomod{0}(V)$.	Let $k \in \ndN_0$, $u \in (\ad \, N_i)^k(N_j)$. Since $\bicomu_Q$ is graded we have
	\begin{align*}
	\bicomu_Q (u) \in \oplus_{\beta_1,\beta_2 \in \ndN_0^\theta}^{\beta_1+\beta_2= \alpha_j + k \alpha_i} Q(\beta_1) \ot Q(\beta_2) \matdot
	\end{align*}
	For $\beta_1,\beta_2 \in \ndN_0^\theta$, such that $\beta_1+\beta_2= \alpha_j + k \alpha_i$, there exist $k_1,k_2 \in \ndN_0$, such that $k_1+k_2=k$ and we either have $\beta_1 = \alpha_j + k_1 \alpha_i$, $\beta_2=k_2 \alpha_i$ or $\beta_1 = k_1 \alpha_i$, $\beta_2=\alpha_j + k_2\alpha_i$. Let $Q' = \oplus_{l \in \ndN_0} Q(\alpha_j + l \alpha_i)$.
	Then 
	\begin{align*}
	\bicomu_Q(u) &\in \oplus_{ k_1, k_2 \in \ndN_0}^{k_1+k_2=k}  Q(\alpha_j + k_1 \alpha_i) \ot N_i^{k_2} \oplus N_i^{k_1} \ot Q(\alpha_j + k_2 \alpha_i)
	\\&\subset Q' \ot \subalgQ{N_i}  \oplus \subalgQ{N_i} \ot Q' \matdot
	\end{align*}
	Since $v$ is a sum of such elements $u$, we obtain 
	\begin{align*} \bicomu_Q(v) \in Q' \ot \subalgQ{N_i} \oplus \subalgQ{N_i} \ot Q' \matdot \end{align*}
	Since $K_i = \coinva{Q}{\subalgQ{N_i}}$ we have $\bicomu_Q(v) \in \bicomu_Q(K_i) \subset Q \ot K_i$ and we can conclude $\bicomu_Q(v) \in Q' \ot \fK \oplus \subalgQ{N_i} \ot V$, as $\subalgQ{N_i} \cap K_i = \fK$ and $Q' \cap K_i = V$.
	Hence there exist $w \in Q'$, $x \in \subalgQ{N_i} \ot V$, such that $\bicomu_Q(v)=w \ot 1 + x$.
	As $\bicounit_Q$ is graded we get $(\id \ot \bicounit_Q) (x) = 0$. Thus
	\begin{align*}
	w = (\id\ot \bicounit_Q) (w \ot 1 + x) = (\id\ot \bicounit_Q) \bicomu_Q(v) = v \matdot
	\end{align*}
	Since $\projcomp{Q}_i(w) = 0$, $(\projcomp{Q}_i\ot \id)(x) = x$ and $v \in \filtcomod{0}(V)$ we have
	\begin{align*}
	x = (\projcomp{Q}_i \ot \id) (w \ot 1 + x)= (\projcomp{Q}_i \ot \id) \bicomu_Q (v) = 1 \ot v \matdot
	\end{align*}
	Combining the above we get $\bicomu_Q(v)=v \ot 1 + 1 \ot v$.
\end{proof}

\Cref{prop_ad_kNi_Nj_is_subobj} allows the following definition.

\begin{defi}
	Let $\scr{N}$ be a pre-Nichols system, $i \in \mathbb{I}$.
	For $j \in \mathbb{I} \setminus \lbrace i \rbrace$ we call $\scr{N}$ \textbf{$j$-well graded over $i$}, if $(\ad \, \subalgQ{\scr{N}_i})(\scr{N}_j) \in \ydcat{\subalgQ{\scr{N}_i}}{\scr{C}}$ is well graded.
\end{defi}

\begin{rema} \label{rema_adNi_Nj_irred_implies_j_graded}
	Let $\scr{N}$ be a pre-Nichols system, $i \in \mathbb{I}$ and $j \in \mathbb{I} \setminus \lbrace i \rbrace$. If $(\ad \, \subalgQ{\scr{N}_i})(\scr{N}_j) \in \ydcat{\subalgQ{\scr{N}_i}}{\scr{C}}$ is irreducible, then 
	$\scr{N}$ is $j$-well graded over $i$ by \Cref{cor_irred_graded_ydmodule_o}.
\end{rema}

\begin{prop} \label{prop_j_well_graded_charact}
	Let $\scr{N}=(Q,(N_1,\ldots,N_\theta))$ be a pre-Nichols system, $i \in \mathbb{I}$ and $j \in \mathbb{I} \setminus \lbrace i \rbrace$. The following are equivalent
	\begin{enumerate}
		\item $\scr{N}$ is $j$-well graded over $i$.
		\item The primitive elements of $Q$ in $\oplus_{k \ge 0} (\ad \, N_i)^k(N_j)$ are given by $N_j$.
	\end{enumerate}
\end{prop}
\begin{proof}
	By definition $\scr{N}$ is $j$-well graded over $i$ if and only if 
	\begin{align*}N_j = \filtcomod{0}((\ad \, \subalgQ{N_i})(N_j)). \end{align*} With \Cref{lem_primitive_elts_of_adNi_Nj} the claim follows.
\end{proof}

\begin{defi} \label{defi_nich_sys}
	Let $\scr{N}$ be a pre-Nichols system and $i \in \mathbb{I}$.
	We call $\scr{N}$ a \textbf{Nichols system over $i$} if
	\begin{enumerate}
		\item $\subalgQ{\scr{N}_i}$ is strictly graded.
		\item $\scr{N}$ is $j$-well graded over $i$ for all $j \in \mathbb{I} \setminus \lbrace i \rbrace$.
	\end{enumerate}
\end{defi}
\begin{rema} \label{rema_nich_sys_hesch}
	\Cref{defi_nich_sys}(1) holds if and only if $\subalgQ{N_i} \cong \nich(N_i)$. In \Cref{prop_j_well_graded_charact} we showed that our definition of a Nichols system is equivalent to the one in \cite{HeSch}, see also \cite{HeSch}, Lemma~13.5.5.
\end{rema}

\begin{lem} \label{lem_ad_Ni_Nj_is_irred}
	Let $i \in \mathbb{I}$ and $\scr{N}$ be a Nichols system over $i$ and let $j \in \mathbb{I} \setminus \lbrace i \rbrace$. The following are equivalent:
	\begin{enumerate}
		\item $\scr{N}_j$ is irreducible in $\scr{C}$.
		\item $(\ad \, \subalgQ{\scr{N}_i})(\scr{N}_j)$ is irreducible in $\ydcat{\subalgQ{\scr{N}_i}}{\scr{C}}$.
	\end{enumerate}
\end{lem}
\begin{proof}
	By assumption $(\ad \, \subalgQ{\scr{N}_i})(\scr{N}_j) \in \ydcat{\subalgQ{\scr{N}_i}}{\scr{C}}$ is well graded. Thus by \Cref{cor_irred_graded_ydmodule_o} the claim is implied.
\end{proof}
\section{Reflections of Nichols systems} \label{sect_nich_sys_reflections}

Let $\scr{N}$ be a pre-Nichols system. Moreover let $Q := \nsalg{\scr{N}}$ and $N_j := \scr{N}_j$ for all $j \in \mathbb{I}$.

We will construct the reflection of $\scr{N}$ step by step using the functor we obtained in \Cref{prop_funCD_exist_and_properties}. To do so, we must first talk about the rationality of $K_i \in \ydcat{\subalgQ{N_i}}{\scr{C}}$.

\begin{defi}
    Let $i \in \mathbb{I}$.
	We denote for $j \in \mathbb{I} \setminus \lbrace i \rbrace$
	\begin{align*}
	a^{\scr{N}}_{ii} &:= 2 \matcom & a^{\scr{N}}_{ij} &:= - \max \lbrace n \in \ndN_0 \, | \, (\ad N_i)^m (N_j) \ne 0 \rbrace \matdot
	\end{align*}
	We say $\scr{N}$ is \textbf{$i$-finite} if for all $j \in \mathbb{I} \setminus \lbrace i \rbrace$ we have $-a^{\scr{N}}_{ij} \in \ndN_0$, that is there is a number $m \in \ndN_0$, such that $(\ad N_i)^m (N_j) = 0$. In this case let $s^{\scr{N}}_i \in \Aut(\ndZ^\theta)$ be such that for all $j \in \mathbb{I}$ we have
	\begin{align*}
	s^{\scr{N}}_i (\alpha_j) = \alpha_j - a^{\scr{N}}_{ij} \alpha_i \matdot
	\end{align*}
\end{defi}


\begin{prop}
    Let $i \in \mathbb{I}$. We have that $\scr{N}$ is $i$-finite if and only if $K_i \in \ydcat{\subalgQ{N_i}}{\scr{C}}_{\rat}$.
\end{prop}
\begin{proof}
	Follows with \Cref{rema_structure_of_Ki} and since $N_j$ is finite-dimensional for all $j \in \mathbb{I}$.
\end{proof}

Let $i \in \mathbb{I}$ and assume $\scr{N}$ is a Nichols system over $i$.

\begin{rema}  \label{rema_mii_equals_mi}
    We have $N_i^* \in \scr{C}$ with $H$-action and $H$-coaction given by
    \begin{align*}
        (h \cdot f) (v) = f( \antip_H (h) \cdot v) \matcom && f_{(-1)}f_{(0)}(v) = \antip_H^{-1}(v_{(-1)}) f(v_{(1)})
    \end{align*}
    for all $h\in H$, $f\in N_i^*$, $v \in N_i$ for details refer to \cite{HeSch}, Lemma~4.2.2.
	Denote $\subalgQ{N_i^*}$ for the Nichols algebra of $N_i^*$, i.e. the $\ndN_0$-graded Hopf algebra generated as an algebra by $N_i^*$ with primitive elements $N_i^*$.
	Then $\subalgQ{N_i^*}$ and $\subalgQ{N_i}$ are a dual pair of locally finite $\ndN_0$-graded Hopf algebras in $\scr{C}$. For details refer to \cite{HeSch}, Corollary~7.2.8.
	In particular for $k \in \ndN$ we have $N_i^k=0$ if and only if ${N_i^*}^k=0$.
\end{rema}

\begin{defi} \label{defi_refl_yd_mod}
	Assume $\scr{N}$ is $i$-finite. We denote
	\begin{align*}
		(\funCD_i, \funCDcoh_i) : \ydcat{\subalgQ{N_i}}{\scr{C}}_{\rat} \rightarrow \ydcat{\subalgQ{N_i^*}}{\scr{C}}_{\rat}
	\end{align*}
	for the braided monoidal isomorphism from \Cref{prop_funCD_exist_and_properties}. Moreover let $\ad^{\funCD_i} : \subalgQ{N_i^*} \ot \funCD_i(K_i) \rightarrow \funCD_i(K_i)$ be the $\subalgQ{N_i^*}$-action on $\funCD_i(K_i)$.
	Finally denote 
	\begin{align*}
		\refl_i(Q) := \funCD_i(K_i) \boso \subalgQ{N_i^*}
	\end{align*}
	and let
	\begin{align*}
		(\funCD_{\refl_i},\funCDcoh_{\refl_i})  := (\funCD_{K_i, \subalgQ{N_i}},\funCDcoh_{K_i, \subalgQ{N_i}}) : \ydcat{Q}{\scr{C}}_{\rat} \rightarrow \ydcat{\refl_i(Q)}{\scr{C}}_{\rat}
	\end{align*}
	be the reflection functor induced by $(\funCD_i, \funCDcoh_i)$, see \Cref{defi_funCD_induced_funct}.	
\end{defi}

\begin{lem} \label{lem_structure_funCDiKi}
	We view $\subalgQ{N_i^*}$ as $\ndZ^\theta$-graded via $\subalgQ{N_i^*}(-n \alpha_i) = {N_i^*}^n$ for $n \in \ndN_0$ and where all other components are zero.
	\begin{enumerate}
		\item $\funCD_i(K_i) \in \ydcat{\subalgQ{N_i^*}}{\scr{C}}_{\rat}$ is a $\ndN_0^\theta$-graded Hopf algebra with components $\funCD_i(K_i)(\alpha)=K_i(\alpha) $
		for all $\alpha \in \ndN_0^\theta$.
		\item The algebra $\funCD_i(K_i)$ is generated by the spaces
		\begin{align*}
			K_i(\alpha_j + n\alpha_i) = (\ad N_i)^n (N_j) \matcom && j \in \mathbb{I} \setminus \lbrace i \rbrace,\, 0 \le n \le -a^{\scr{N}}_{ij} \matdot
		\end{align*}
	\end{enumerate}
\end{lem}
\begin{proof}
	(1) is implied by \Cref{prop_gam_grad_funCD}.
	(2) is implied by (1),  \Cref{prop_funCD_exist_and_properties}(2) and \Cref{rema_structure_of_Ki}. 
\end{proof}

\begin{notation}
	If $\scr{N}$ is $i$-finite, then denote $\refl_i(N_i) := N_i^*$ and for $j \in \mathbb{I} \setminus \lbrace i \rbrace$ let 
	\begin{align*}
	\refl_i(N_j) := (\ad N_i)^{-a^{\scr{N}}_{ij}} (N_j) \matdot
	\end{align*}
	Finally let $\refl_i(\scr{N}) := (\refl_i(Q),(\refl_i(N_1), \ldots, \refl_i(N_\theta) ))$.
\end{notation}

\begin{lem} \label{lem_structure_funCDiKi2}
	Assume $\scr{N}$ is $i$-finite and let $j \in \mathbb{I} \setminus \lbrace i \rbrace$.
	If $N_j$ is irreducible in $\scr{C}$, then $\refl_i(N_j)$ is irreducible in $\scr{C}$ and for $n \in \ndN_0$ we have
		\begin{align*}
			(\ad^{\funCD_i}{N_i^*} )^n ( \refl_i(N_j) ) =
			\begin{cases}
			(\ad N_i)^{-a^{\scr{N}}_{ij}-n} (N_j), &\text{if $n \le -a^{\scr{N}}_{ij}$,} \\
			0, &\text{if $n > -a^{\scr{N}}_{ij}$.}
			\end{cases}
		\end{align*}
\end{lem}
\begin{proof}
	Let $V := (\ad \, \subalgQ{N_i})(N_j) \subset K_i$.
	By \Cref{lem_ad_Ni_Nj_is_irred} we obtain that $V \in \ydcat{\subalgQ{N_i}}{\scr{C}}$ is irreducible. Then $\funCD_i(V) \in \ydcat{\subalgQ{N_i^*}}{\scr{C}} \cong \ydmod{\subalgQ{N_i^*} \boso H}$ is irreducible and $\ndZ$-graded via $\funCD_i(V)(n)=V(-n)$ for all $n \in \ndZ$ by \Cref{prop_funCD_exist_and_properties} and \Cref{prop_funCD_properties_irred}. Since the component of smallest degree is $\funCD_i(V)(a^{\scr{N}}_{ij})=V(-a^{\scr{N}}_{ij})=\refl_i(N_j)$, \Cref{cor_irred_graded_ydmodule_o}(1)$\implies$(3) yields that $\refl_i(N_j)$ is irreducible in $\scr{C}$ and  
	\begin{align*}
		(\ad^{\funCD_i}{N_i^*} )^n ( \refl_i(N_j) ) = \funCD_i(V)(a^{\scr{N}}_{ij}+n) =
			(\ad N_i)^{-a^{\scr{N}}_{ij}-n} (N_j)
	\end{align*}
	for all $0 \le n \le -a^{\scr{N}}_{ij}$.
\end{proof}

\begin{prop} \label{prop_RiQ_is_nich_sys}
	Assume $\scr{N}$ is $i$-finite. Moreover assume that $N_j$ is irreducible in $\scr{C}$ for all $j \in \mathbb{I}\setminus \lbrace i \rbrace$.
	\begin{enumerate}
		\item $\refl_i(N_j)$ is irreducible in $\scr{C}$ for all $j \in \mathbb{I}\setminus \lbrace i \rbrace$.
		\item $\refl_i(\scr{N})$ is a pre-Nichols system, that is
		\begin{enumerate}
			\item The algebra $\refl_i(Q)$ is generated by $\refl_i(N_1), \ldots, \refl_i(N_\theta)$.
			\item $\refl_i(Q)$ is a $\ndN_0^\theta$-graded Hopf algebra in $\scr{C}$, where the grading is given by 	$\refl_i(Q)(\alpha_j) = \refl_i(N_j)$ for $j \in \mathbb{I}$.
		\end{enumerate}
		\item $\refl_i(\scr{N})$ is a Nichols system over $i$.
		\item $\refl_i(\scr{N})$ is $i$-finite.
	\end{enumerate}
\end{prop}
\begin{proof}
	(1) and (4) are implied by \Cref{lem_structure_funCDiKi2}.
	(2)(a): Let $\widetilde{Q}$ be the subalgebra of $\refl_i(Q)$ generated by $\refl_i(N_1), \ldots, \refl_i(N_\theta)$. Since $N_i^* = \refl_i(N_i) \subset \widetilde{Q}$ and for $a,b \in \subalgQ{N_i^*}$ we have
	\begin{align*}
		\bimu_{\refl_i(Q)} (1 \ot a, 1 \ot b) = \bicounit_{\subalgQ{N_i^*}} (a_{(1)}) 1 \ot a_{(2)} b = 1 \ot ab \matcom
	\end{align*}
	we obtain that $\subalgQ{N_i^*} \subset \widetilde{Q}$. Now it remains to show that $\funCD_i(K_i) \subset \widetilde{Q}$. By \Cref{lem_structure_funCDiKi}(2) this narrows to showing $(\ad N_i)^n (N_j) \subset \widetilde{Q}$ for all $j \in \mathbb{I} \setminus \lbrace i \rbrace$, $0 \le n \le -a^{\scr{N}}_{ij}$. For $n=-a^{\scr{N}}_{ij}$ this is given, as $\refl_i(N_j) \subset \widetilde{Q}$. 
	So assume $( \ad N_i)^n (N_j) \subset \widetilde{Q}$ for some $1\le n \le -a^{\scr{N}}_{ij}$. We show $(\ad N_i)^{n-1} (N_j) \subset \widetilde{Q}$.
	Now for $x \in (\ad N_i)^n (N_j) \subset \widetilde{Q}$, $a \in N_i^*$ we obtain since $a$ is primitive
	\begin{align*}
		\widetilde{Q} \ni \bimu_{\refl_i(Q)} (1 \ot a, x \ot 1) = (\ad^{\funCD_i} a)(x) \ot 1 + a_{(-1)} \cdot x \ot a_{(0)} \matdot
	\end{align*}
	Since $a_{(-1)} \cdot x \ot a_{(0)} \in \widetilde{Q}$ we obtain $(\ad^{\funCD_i} a)(x) \in \widetilde{Q}$. Then \Cref{lem_structure_funCDiKi2} yields
	\begin{align*}
		(\ad N_i)^{n-1} (N_j) 
		=& (\ad^{\funCD_i} N_i^*)^{-a^{\scr{N}}_{ij}-n+1} (\refl_i(N_j))
		\\=& (\ad^{\funCD_i} N_i^*) \left( (\ad^{\funCD_i} N_i^*)^{-a^{\scr{N}}_{ij}-n} (\refl_i(N_j)) \right)
		\\=& (\ad^{\funCD_i} N_i^*) \left( \ad(N_i)^n (N_j) \right) \subset \widetilde{Q} \matcom
	\end{align*}
	implying (1) inductively.
	
	(2)(b): $K_i$ is an $\ndZ^\theta$-graded Hopf algebra in $\scr{C}$ by extending the $\ndN_0^\theta$-grading with zero components. Hence using \Cref{prop_grading_on_funCDK}(2) we obtain that $\refl_i(Q)$ is a $\ndZ^\theta$-graded Hopf algebra in $\scr{C}$ by
	\begin{align*}
		\refl_i(Q) ( \gamma) = \oplus_{m_1 \in \ndZ, m_2 \in \ndN_0}^{m_1-m_2=n_i} K_i (m_1 \alpha_i + \sum_{j \in \mathbb{I}\setminus \lbrace i \rbrace} n_j \alpha_j) \ot {N_i^*}^{m_2} \matcom
	\end{align*}
	where $n_1, \ldots, n_\theta \in \ndZ$, $\gamma = \sum_{j \in \mathbb{I}} n_j \alpha_j \in \ndZ^\theta$ (observe that this is not $\ndN_0^\theta$-graded, as $\refl_i(Q)(-\alpha_i)=N_i^* \ne 0$). Observe that
	\begin{align*}
	    s^{\scr{N}}_i ( \gamma ) = -n_i \alpha_i + \sum_{j \in \mathbb{I}\setminus \lbrace i \rbrace} n_j (\alpha_j - a^{\scr{N}}_{ij} \alpha_i)) \matdot
	\end{align*}
	Hence by applying $s^{\scr{N}}_i$ to that grading as in \Cref{rema_shifting_grading_by_autom} we get that $\refl_i(Q)$ is a $\ndZ^\theta$-graded Hopf algebra in $\scr{C}$ by
	\begin{align*}
		\refl_i(Q) ( \gamma) = \oplus_{m_1 \in \ndZ, m_2 \in \ndN_0}^{m_1-m_2=-n_i - \sum_{j \in \mathbb{I}\setminus \lbrace i \rbrace} n_j  a^{\scr{N}}_{ij}} K_i (m_1 \alpha_i + \sum_{j \in \mathbb{I}\setminus \lbrace i \rbrace} n_j \alpha_j ) \ot {N_i^*}^{m_2} \matdot
	\end{align*}
	By substituting $m_1$ in the sum above with $-m_1 - \sum_{j \in \mathbb{I}\setminus \lbrace i \rbrace} n_ja_{ij}$ we obtain that $\refl_i(Q)$ is a $\ndZ^\theta$-graded Hopf algebra in $\scr{C}$ by
	\begin{align*}
		\refl_i(Q) ( \gamma) = \oplus_{m_1 \in \ndZ, m_2 \in \ndN_0}^{m_1+m_2=n_i} K_i (-m_1 \alpha_i + \sum_{j \in \mathbb{I}\setminus \lbrace i \rbrace} n_j (\alpha_j - a^{\scr{N}}_{ij} \alpha_i)) \ot {N_i^*}^{m_2} \matdot
	\end{align*}
	For this grading we have $\refl_i(N_i) = K_i(0) \ot N_i^* \subset \refl_i(Q)(\alpha_i)$ and for $j \in \mathbb{I} \setminus \lbrace i \rbrace$ we have 
	\begin{align*}
		\refl_i(N_j) = (\ad N_i)^{-a^{\scr{N}}_{ij}} (N_j) = K_i(\alpha_j - a^{\scr{N}}_{ij} \alpha_i) \ot {N_i^*}^0 \subset \refl_i(Q)(\alpha_j) \matdot
	\end{align*}
	Thus by (2)(a) we obtain $\refl_i(N_j) = \refl_i(Q)(\alpha_j)$ for all $j \in \mathbb{I}$ and that this $\ndZ^\theta$-grading on $\refl_i(Q)$ is in fact a $\ndN_0^\theta$-grading on $\refl_i(Q)$.
	
	(3): The primitive elements of $\subalgQ{\refl_i(N_i)} = \subalgQ{N_i^*}$ are given by $N_i^*$ by definition. Now let $j \in \mathbb{I} \setminus \lbrace i \rbrace $.
	By \Cref{lem_ad_Ni_Nj_is_irred} we know that the object $V := (\ad \, \subalgQ{N_i})(N_j) \in \ydcat{\subalgQ{N_i}}{\scr{C}}$ is irreducible. In \Cref{lem_structure_funCDiKi2} we see that $(\ad^{\funCD_i}{\subalgQ{N_i^*}} ) ( \refl_i(N_j) ) = \funCD_i(V)$.
	Now $\funCD_i(V) \in \ydcat{\subalgQ{N_i^*}}{\scr{C}}$ is irreducible by \Cref{prop_funCD_properties_irred} and as discussed in \Cref{rema_adNi_Nj_irred_implies_j_graded} we obtain that $\refl_i(\scr{N})$ is $j$-well graded over $i$.
\end{proof}

\begin{defi}
	If $\scr{N}$ is $i$-finite and $N_j$ is irreducible in $\scr{C}$ for all $j \in \mathbb{I}\setminus \lbrace i \rbrace$, then we call the Nichols system $\refl_i(\scr{N})$ over $i$ the \textbf{$i$-th reflection of $\scr{N}$}.
\end{defi}

\begin{rema}
	The notations for Nichols systems and its reflections are consistent: Indeed if $\scr{N}$ is $i$-finite and $N_j$ is irreducible in $\scr{C}$ for all $j \in \mathbb{I}\setminus \lbrace i \rbrace$, then
	$\nsalg{\refl_i(\scr{N})} = \refl_i(Q) = \refl_i (\nsalg{\scr{N}})$ and ${\refl_i(\scr{N})}_j = \refl_i(N_j) = \refl_i ( {\scr{N}}_j)$ for all $j \in \mathbb{I}$.
\end{rema}

\begin{notation}
    For $j \in \mathbb{I}$ let 
	\begin{align*}
	    m^{\scr{N}}_j := \max \lbrace m \in \ndN_0 \, | \, 
		Q(m \alpha_j)=\scr{N}_{j}^m \ne 0 \rbrace \matdot
	\end{align*}
\end{notation}

\begin{prop} \label{prop_sup_of_reflection}
    Assume $\scr{N}$ is $i$-finite, $m^{\scr{N}}_i \in \ndN_0$ and that $N_j$ is irreducible in $\scr{C}$ for all $j \in \mathbb{I}\setminus \lbrace i \rbrace$. Then 
    \begin{align*}
        \Sup{Q} = m^{\scr{N}}_i \alpha_i + s^{\scr{N}}_i (\Sup{\refl_i(Q)})\matdot
    \end{align*}
    In particular $\Sup{Q} \subset m^{\scr{N}}_i \alpha_i + s^{\scr{N}}_i (\ndN_0^\theta) $.
\end{prop}
\begin{proof}
    Let $n_1,\ldots,n_\theta \in \ndN_0$, $\gamma = \sum_{j \in \mathbb{I}} n_j \alpha_j \in \ndN_0^\theta$ and let 
    \begin{align*}
        n_i' := m^{\scr{N}}_i-n_i - \sum_{j \in \mathbb{I}\setminus \lbrace i \rbrace} n_j  a^{\scr{N}}_{ij} 
    \end{align*}
    Then 
    \begin{align*}
         m^{\scr{N}}_i \alpha_i + s^{\scr{N}}_i (\gamma) = n_i' \alpha_i + \sum_{j \in \mathbb{I}\setminus \lbrace i \rbrace} n_j  \alpha_j
    \end{align*}
    and by construction we have
    \begin{align*}
        Q ( m^{\scr{N}}_i \alpha_i + s^{\scr{N}}_i (\gamma) ) = \oplus_{m_1 \in \ndZ, 0 \le m_2 \le m^{\scr{N}}_i }^{m_1+m_2=n_i'} K_i (m_1 \alpha_i + \sum_{j \in \mathbb{I}\setminus \lbrace i \rbrace} n_j  \alpha_j ) \ot N_i^{m_2} \matdot
    \end{align*}
    Now substituting $m_1$ with $ -m_1  - \sum_{j \in \mathbb{I}\setminus \lbrace i \rbrace} n_j  a^{\scr{N}}_{ij}  $ and $m_2$ with $m^{\scr{N}}_i - m_2$ yields
    \begin{align*}
        &Q ( m^{\scr{N}}_i \alpha_i + s^{\scr{N}}_i (\gamma) ) 
        \\=& \oplus_{m_1 \in \ndZ, 0 \le m_2 \le m^{\scr{N}}_i }^{m_1 + m_2=n_i} K_i (-m_1 \alpha_i + \sum_{j \in \mathbb{I}\setminus \lbrace i \rbrace} n_j  (\alpha_j- a^{\scr{N}}_{ij} \alpha_i) ) \ot N_i^{m^{\scr{N}}_i - m_2} \matdot
    \end{align*}
    Moreover as we worked out in the proof of \Cref{prop_RiQ_is_nich_sys}(2)(b) we have
    \begin{align*}
		\refl_i(Q) ( \gamma) = \oplus_{m_1 \in \ndZ, m_2 \in \ndN_0}^{m_1+m_2=n_i} K_i (-m_1 \alpha_i + \sum_{j \in \mathbb{I}\setminus \lbrace i \rbrace} n_j (\alpha_j - a^{\scr{N}}_{ij} \alpha_i)) \ot {N_i^*}^{m_2} \matdot
	\end{align*}
	Considering \Cref{rema_mii_equals_mi}, ${N_i^*}^{m_2} \ne 0$ if and only if $0 \le m_2 \le m^{\scr{N}}_i$. We conclude that $Q ( m^{\scr{N}}_i \alpha_i + s^{\scr{N}}_i (\gamma) ) = 0$ if and only if $\refl_i(Q) ( \gamma)=0$. We obtain
	\begin{align*}
	    m^{\scr{N}}_i \alpha_i + s^{\scr{N}}_i (\Sup{\refl_i(Q)}) \subset \Sup{Q}. 
	\end{align*}
	If $\gamma \in \Sup{Q}$, then since $\gamma = m^{\scr{N}}_i \alpha_i + s^{\scr{N}}_i ( m^{\scr{N}}_i \alpha_i + s^{\scr{N}}_i ( \gamma) )$ we obtain that $m^{\scr{N}}_i \alpha_i + s^{\scr{N}}_i ( \gamma) \in \Sup{\refl_i(Q)}$ and $\gamma \in m^{\scr{N}}_i \alpha_i + s^{\scr{N}}_i (\Sup{\refl_i(Q)})$.
\end{proof}

\section{Iterated reflections} \label{sect_nich_sys_iterated_refl}
    In this section we will study what happens if we reflect a Nichols system multiple times. Here we will come across the set of roots of a Nichols system and will discuss how they can be used to describe the geometric shape of the support of a Nichols system. In particular we will describe the edges of that shape.

	Let $\scr{N}$ be a pre-Nichols system, $Q := \nsalg{\scr{N}}$ and assume that $\scr{N}_j$ is irreducible in $\scr{C}$ for all $j \in \mathbb{I}$.
	
	To look at iterated reflections of $\scr{N}$ we first have to make sure the preconditions of \Cref{prop_RiQ_is_nich_sys} are met after each reflection.
\begin{defi}
	Let $k \in \ndN_0$, $i_1, \ldots, i_k \in \mathbb{I}$.
	\begin{enumerate}
		\item We say $\scr{N}$ \textbf{admits the reflection sequence $(i_1,\ldots, i_k)$}, if $k=0$ or $\scr{N}$ is a Nichols-system over $i_1$, $\scr{N}$ is $i_1$-finite and $\refl_{i_1}(\scr{N})$ admits the reflection sequence $(i_2,\ldots, i_k)$.
		\item We say $\scr{N}$ \textbf{admits all reflections}, if for all $n \in \ndN_0$, $i \in \mathbb{I}^n$,
		$\scr{N}$ admits the reflection sequence $i$.
	\end{enumerate}
\end{defi}

\begin{rema}
	Let $k \in \ndN$, $i = (i_1,\ldots, i_k) \in \mathbb{I}^k$ and assume $\scr{N}$ admits the reflection sequence $i$. Then by definition we have
	\begin{align*}
		\refl_{i}(\scr{N})_{j} &= \begin{cases}
			{\refl_{(i_1,\ldots,i_{k-1})}(\scr{N})_{i_k}}^* & \text{if $j = i_k$,} \\
			(\ad \refl_{(i_1,\ldots,i_{k-1})}(\scr{N})_{i_k})^{-a^{\scr{N}}_{i,j}} (\refl_{(i_1,\ldots,i_{k-1})}(\scr{N})_{j})	& \text{if $j \ne i_k$.} 
		\end{cases} \matdot
	\end{align*}
\end{rema}

We extend our notation to handle iterated reflections.

\begin{notation}
	Let $k \in \ndN_0$, $i = (i_1,\ldots, i_k) \in \mathbb{I}^k$. 
	We abbreviate
	\begin{align*} \refl_i = \refl_{i_k} \cdots \refl_{i_1} \end{align*}
	to describe the iterated reflections of Nichols systems, starting with the $i_1$-th reflection, up to the $i_k$-th reflection. 
	Assume $\scr{N}$ admits the reflection sequence $(i_1,\ldots,i_{k-1})$. 
	Denote $m^{\scr{N}}_{()}=0$ and if $k \ge  1$ denote
	\begin{align*}
		m^{\scr{N}}_{i} := m^{\refl_{(i_1,\ldots,i_{k-1})}(\scr{N})}_{i_k}  = \max \lbrace m \in \ndN_0 \, | \, 
		{\refl_{(i_1,\ldots,i_{k-1})}(\scr{N})_{i_k}^m }
		 \ne 0 \rbrace \matdot
	\end{align*}
	Moreover let $a^{\scr{N}}_{i,i_k} = 2$ and for $j \in \mathbb{I}\setminus \lbrace i_k \rbrace$ define
	\begin{align*}
		a^{\scr{N}}_{i,j} :=& \, a^{\refl_{(i_1,\ldots,i_{k-1})}(\scr{N})}_{i_k j} \\ =& -\max \lbrace a \in \ndN_0 \, | \, (\ad \refl_{(i_1,\ldots,i_{k-1})}(\scr{N})_{i_k})^{a} (\refl_{(i_1,\ldots,i_{k-1})}(\scr{N})_{j}) \ne 0 \rbrace 
		\matdot
	\end{align*}
	If $\refl_{(i_1,\ldots,i_{k-1})}(\scr{N})$ is $i_k$-finite, 
	then let $s^{\scr{N}}_{(i_1,\ldots,i_{k-1}),i_k} \in \Aut(\ndZ^\theta)$ be such that for all $j \in \mathbb{I}$
	\begin{align*}
		s^{\scr{N}}_{(i_1,\ldots,i_{k-1}),i_k} ( \alpha_j) = \alpha_j - a^{\scr{N}}_{i,j} \alpha_{i_k} \matdot
	\end{align*}
	Finally define $s^{\scr{N}}_i=s^{\scr{N}}_{(),i_1} s^{\scr{N}}_{(i_1),i_2} \cdots s^{\scr{N}}_{(i_1,\ldots,i_{k-1}),i_k}$ and $s^{\scr{N}}_{()}=\id_{\ndZ^\theta}$.
\end{notation}

\begin{defi}
    Let $k \in \ndN_0$, $i \in \mathbb{I}^k$. If $\scr{N}$ admits the reflection sequence $i$, then for $j \in \mathbb{I}$ the element $s_i^{\scr{N}}(\alpha_j) \in \ndZ^\theta$ is called a \textbf{root of~$\scr{N}$}. The set of all roots is denoted $\roots{\scr{N}}$. 
    Furthermore we denote $\rootspos{\scr{N}} = \roots{\scr{N}} \cap \ndN_0^\theta$ for the set of \textbf{positive roots of $\scr{N}$}. 
\end{defi}

\begin{rema}
    In \Cref{lem_ti_explicit}(2) we will give an explicit formula for the roots of $\scr{N}$. In the proof of \Cref{prop_roots_pos_and_neg} we will see, that the set of roots coincides with the set of real roots of a semi-Cartan graph.
\end{rema}

\begin{prop} \label{prop_roots_pos_and_neg}
    Assume $\scr{N}$ admits all reflections. Then 
    \begin{align*}
        \roots{\scr{N}} = \rootspos{\scr{N}} \cup - \rootspos{\scr{N}} \matdot 
    \end{align*}
\end{prop}
\begin{proof} 
    Let $M= \oplus_{i=1}^\theta N_i \in \scr{C}$.
    As discussed in \Cref{rema_nich_sys_hesch}, we know that $(\ad_Q N_i)^m (N_j) \cong (\ad_{\nich(M)} N_i)^m (N_j)$ for all $i,j \in \mathbb{I}$, $i \ne j$, where $\ad_Q : \subalgQ{N_i} \ot K_i \rightarrow K_i$ and $\ad_{\nich(M)} : \nich(N_i) \ot \coinva{\nich(M)}{\nich(N_i)} \rightarrow \coinva{\nich(M)}{\nich(N_i)}$ are the adjoint actions. In particular
    \begin{align*}
        a^{\scr{N}}_{ij} =& -\max \lbrace a \in \ndN_0 \, | \, (\ad_{\nich(M)} N_{i})^{a} (N_j) \ne 0 \rbrace \matdot
    \end{align*}
    Inductively repeating this argument for all $k \in \ndN$, $(i_1,\ldots,i_k) \in \mathbb{I}^k$ and $\refl_i(M)= \oplus_{j=1}^\theta \refl_i (\scr{N})_j \in \scr{C}$ we obtain for $j \in \mathbb{I}$, $j \ne i_k$
    \begin{align*}
        &a^{\scr{N}}_{i,j} = \\ & -\max \lbrace a \in \ndN_0 \, | \, (\ad_{\nich(\refl_i(M))} \refl_{(i_1,\ldots,i_{k-1})}(\scr{N})_{i_k})^{a} (\refl_{(i_1,\ldots,i_{k-1})}(\scr{N})_{j}) \ne 0 \rbrace \matdot
    \end{align*}
    Hence the Cartan matrices of the semi-Cartan graph of $M$ (\cite{HeSch}, Definition 13.6.3) are given by $(a^{\refl_i(\scr{N})}_{jk})_{1\le j,k \le \theta}$, $i \in \mathbb{I}^k$, $k \in \ndN_0$ and thus the set of real roots of this semi-Cartan graph at $M$ coincides with $\roots{\scr{N}}$.
    Then \cite{HeSch}, Theorem 14.2.12 implies the claim.
\end{proof}

\begin{lem} \label{lem_sup_of_iterated_refl}
    Let $k \in \ndN$, $i = (i_1,\ldots, i_k) \in \mathbb{I}^k$ and assume $\scr{N}$ admits the reflection sequence $i$ and $m^{\scr{N}}_i \in \ndN_0$. Then 
    \begin{align*}
        \Sup{\refl_{(i_1,\ldots,i_{k-1})}(Q)}= m^{\scr{N}}_i \alpha_{i_k} + s^{\scr{N}}_{(i_1,\ldots,i_{k-1}),i_k} (\Sup{\refl_i(Q)}) \matdot
    \end{align*}
\end{lem}
\begin{proof}
	Follows from \Cref{prop_sup_of_reflection}.
\end{proof}

\begin{lem} \label{lem_ti_explicit}
    Let $k \in \ndN_0$, $i = (i_1,\ldots, i_k) \in \mathbb{I}^k$ and assume $\scr{N}$ admits the reflection sequence $i$.
	For $1 \le r < r'$ let
	\begin{align*}
		C_r^{r'} := \lbrace (b_1,\ldots,b_a) \, | \, a \ge 2, r=b_1 < \ldots < b_a = r'  \rbrace \matdot
	\end{align*}
	Then for $1\le l \le k-1$ we have
	\begin{align*}
	s^{\scr{N}}_{(i_1,\ldots,i_l)} (\alpha_{i_k}) &= \alpha_{i_k} + \sum_{k'=1}^{l} (-a^{\scr{N}}_{(i_1,\ldots,i_{k'}),i_{k}}) s^{\scr{N}}_{(i_1,\ldots,i_{k'-1})} (\alpha_{i_{k'}}) \tag{1} \matcom
	\\
	s^{\scr{N}}_{(i_1,\ldots,i_{k-1})} (\alpha_{i_k})  &= \alpha_{i_k} + \sum_{r=1}^{k-1} \left( \sum_{(b_1, \ldots, b_{a}) \in C_r^k} 
	\prod_{c=1}^{a-1} -a^{\scr{N}}_{(i_1,\ldots,i_{b_{c}}),i_{b_{c+1}}} \right) \alpha_{i_r} \tag{2} \matdot
	\end{align*}
\end{lem}
\begin{proof}
	(1): Follows inductively for $l$: The claim is given if $l=1$ by definition. Moreover
	\begin{align*}
	s^{\scr{N}}_{(i_1,\ldots,i_l)} (\alpha_{i_k}) &= s^{\scr{N}}_{(i_1,\ldots,i_{l-1})} ( \alpha_{i_k} - a^{\scr{N}}_{(i_1,\ldots,i_{l}),i_k} \alpha_{i_l} ) 
	\\ &= s^{\scr{N}}_{(i_1,\ldots,i_{l-1})} ( \alpha_{i_k} ) - a^{\scr{N}}_{(i_1,\ldots,i_{l}),i_k} s^{\scr{N}}_{(i_1,\ldots,i_{l-1})}(\alpha_{i_l}) \matdot
	\end{align*}
	
	(2): We do induction on $k$. The claim is given for $k=1$. For $k \ge 2$ using (1) we obtain
	\begin{align*}
	s^{\scr{N}}_{(i_1,\ldots,i_{k-1})} (\alpha_{i_k}) &= \alpha_{i_k} + \sum_{k'=1}^{k-1} (-a^{\scr{N}}_{(i_1,\ldots,i_{k'}),i_{k}}) s^{\scr{N}}_{(i_1,\ldots,i_{k'-1})} (\alpha_{i_{k'}}) \matdot
	\end{align*}
	Using the induction hypothesis we get
	\begin{align*}
	&\sum_{k'=1}^{k-1} (-a^{\scr{N}}_{(i_1,\ldots,i_{k'}),i_{k}}) s^{\scr{N}}_{(i_1,\ldots,i_{k'-1})} (\alpha_{i_{k'}})
	\\=& \sum_{k'=1}^{k-1} (-a^{\scr{N}}_{(i_1,\ldots,i_{k'}),i_{k}}) \left( \alpha_{i_{k'}} + \sum_{r=1}^{k'-1} \left( \sum_{(b_1, \ldots, b_{a}) \in C_r^{k'}} 
	\prod_{c=1}^{a-1} -a^{\scr{N}}_{(i_1,\ldots,i_{b_{c}}),i_{b_{c+1}}} \right) \alpha_{i_r} \right) \matdot
	\end{align*}
	The coefficients for $\alpha_{i_r}$, $1\le r \le k-1$ in this sum (by differentiating the indices $i_1,\ldots,i_k$, even if they are not necessarily pairwise distinct) are
	\begin{align*}
		&  -a^{\scr{N}}_{(i_1,\ldots,i_{r}),i_{k}} + \sum_{k'=r+1}^{k-1} -a^{\scr{N}}_{(i_1,\ldots,i_{k'}),i_{k}} \left( \sum_{(b_1, \ldots, b_{a}) \in C_r^{k'}} 
		\prod_{c=1}^{a-1} -a^{\scr{N}}_{(i_1,\ldots,i_{b_{c}}),i_{b_{c+1}}} \right)
		\\=& \sum_{(b_1, \ldots, b_{a}) \in C_r^k} 
		\prod_{c=1}^{a-1} -a^{\scr{N}}_{(i_1,\ldots,i_{b_{c}}),i_{b_{c+1}}} \matdot
	\end{align*}
	This finishes the induction.
\end{proof}

\begin{notation} \label{nota_bety_sys}
    Let $k \in \ndN_0$, $i = (i_1,\ldots, i_k) \in \mathbb{I}^k$ and assume $\scr{N}$ admits the reflection sequence $i$.
	For $m=(m_1,\ldots,m_k) \in \ndN_0^k$ let
	\begin{align*}
	\beta^{\scr{N}}_{i,m} =& - \sum_{r=1}^{k} m_r s^{\scr{N}}_{(i_1,\ldots,i_r)}(\alpha_{i_r})   \in \ndZ_0^\theta \matdot
	\end{align*}
	Moreover if $m^{\scr{N}}_{(i_1,\ldots,i_r)} \in \ndN_0$ for all $1 \le r \le k-1$, then for $0 \le m \le m^{\scr{N}}_i$ let $\beta^{\scr{N}}_{i,m} := \beta^{\scr{N}}_{i,(m^{\scr{N}}_{(i_1)},\ldots,m^{\scr{N}}_{(i_1,\ldots,i_{k-1})}),m)}$, where $\beta^{\scr{N}}_{(),0}=0$. If also $m^{\scr{N}}_i \in \ndN_0$, then let $\beta^{\scr{N}}_i := \beta^{\scr{N}}_{i,m^{\scr{N}}_i}$.
\end{notation}


\begin{rema}
    Let $k \in \ndN_0$, $i = (i_1,\ldots, i_k) \in \mathbb{I}^k$ and assume $\scr{N}$ admits the reflection sequence $i$ and $m^{\scr{N}}_{(i_1,\ldots,i_r)} \in \ndN_0$ for all $1 \le r \le k$.
	By \Cref{lem_ti_explicit}(2), the vectors $\beta^{\scr{N}}_{i,m} \in \ndZ_0^\theta$, where $m \in \ndN_0^k$, rely only on the variables
	$a^{\scr{N}}_{(i_1,\ldots,i_{r}),j}$, $1 \le r \le k-1$, $j \in \lbrace i_{r+1},\ldots,i_k \rbrace$. In \Cref{prop_sup_iterated_reflection} we will show that $\beta^{\scr{N}}_{i} \in \ndN_0^\theta$ and in \Cref{prop_beta_explicit} we will give an explicit description of $\beta^{\scr{N}}_{i,m}$.
\end{rema}

\begin{lem} \label{lem_rels_beta_i}
    Let $k \in \ndN_0$, $i = (i_1,\ldots, i_k) \in \mathbb{I}^k$ and assume $\scr{N}$ admits the reflection sequence $i$.
	For $m \in \ndN_0^k$ the following relation holds
	\begin{align*}
		\beta^{\scr{N}}_{i,m} = \beta^{\scr{N}}_{(i_1,\ldots,i_{k-1}),(m_1,\ldots,m_{k-1})} - m_k s^{\scr{N}}_i(\alpha_{i_k}) \matdot
	\end{align*}
	In particular if $m^{\scr{N}}_{(i_1,\ldots,i_r)} \in \ndN_0$ for all $1 \le r \le k$, then 
	\begin{align*}
	    \beta^{\scr{N}}_{i} = \beta^{\scr{N}}_{(i_1,\ldots,i_{k-1})} - m^{\scr{N}}_i s^{\scr{N}}_i(\alpha_{i_k}) \matdot
	\end{align*}
\end{lem}
\begin{proof}
	Follows directly from the definition.
\end{proof}

\begin{lem} \label{lem_rels_beta_i2_sys}
    Let $k \in \ndN_0$, $i = (i_1,\ldots, i_k) \in \mathbb{I}^k$, assume $\scr{N}$ admits the reflection sequence $i$ and $m^{\scr{N}}_{(i_1,\ldots,i_r)} \in \ndN_0$ for all $1 \le r \le k$.
	The following relations hold
	\begin{align*}
	m^{\scr{N}}_{(i_1,\ldots,i_k,i_k)} = m^{\scr{N}}_{(i_1,\ldots,i_k)} \matcom &&
	\beta^{\scr{N}}_{(i_1,\ldots,i_k,i_k)} = \beta^{\scr{N}}_{(i_1,\ldots,i_{k-1})} \matdot
	\end{align*}
\end{lem}
\begin{proof}
	The first relation is discussed in \Cref{rema_mii_equals_mi}. Using \Cref{lem_rels_beta_i} we get
	\begin{align*}
		\beta^{\scr{N}}_{(i_1,\ldots,i_k,i_k)} = \beta^{\scr{N}}_{i} - m^{\scr{N}}_{(i_1,\ldots,i_k,i_k)} s^{\scr{N}}_{(i_1,\ldots,i_k,i_k)}(\alpha_{i_k}) = \beta^{\scr{N}}_i + m^{\scr{N}}_i s^{\scr{N}}_i(\alpha_{i_k}) 
	\end{align*}
	and $\beta^{\scr{N}}_i  = \beta^{\scr{N}}_{(i_1,\ldots,i_{k-1})}-m^{\scr{N}}_i s^{\scr{N}}_i(\alpha_{i_k})$.
\end{proof}

\begin{prop} \label{prop_beta_explicit}
    Let $k \in \ndN_0$, $i = (i_1,\ldots, i_k) \in \mathbb{I}^k$ and assume $\scr{N}$ admits the reflection sequence $i$.
	For $1 \le r < r'$ let
	\begin{align*}
	C_r^{r'} := \lbrace (b_1,\ldots,b_a) \, | \, a \ge 2, r=b_1 < \ldots < b_a = r'  \rbrace \matdot
	\end{align*}
	Moreover let $m =(m_1,\ldots,m_k) \in \ndN_0^k$. Then
	\begin{align*}
		\beta^{\scr{N}}_{i,m} = \sum_{r=1}^k \left( {m}_r + \sum_{l=r+1}^{k} {m}_l \sum_{(b_1, \ldots, b_{a}) \in C_{r}^{l}} \prod_{c=1}^{a-1} -a^{\scr{N}}_{(i_1,\ldots,i_{b_{c}}),i_{b_{c+1}}} \right) \alpha_{i_r}
		\matdot
	\end{align*}
\end{prop}
\begin{proof}
	By \Cref{lem_ti_explicit}(2) we obtain
	\begin{align*}
		 \beta^{\scr{N}}_{i,m} &= -\sum_{l=1}^k {m}_{l} s^{\scr{N}}_{(i_1,\ldots,i_{l})}(\alpha_{i_{l}})
		 = \sum_{l=1}^k {m}_{l} s^{\scr{N}}_{(i_1,\ldots,i_{l-1})}(\alpha_{i_{l}})
		 \\ &= \sum_{l=1}^k {m}_{l} \left( \alpha_{i_{l}} + \sum_{r=1}^{{l}-1} \left( \sum_{(b_1, \ldots, b_{a}) \in C_r^{l}} 
		 \prod_{c=1}^{a-1} -a^{\scr{N}}_{(i_1,\ldots,i_{b_{c}}),i_{b_{c+1}}} \right) \alpha_{i_r} \right) \matdot
	\end{align*}
	The coefficients for $\alpha_{i_r}$, $1\le r \le k$ in this sum (by differentiating the indices $i_1,\ldots,i_k$, even if they are not necessarily pairwise distinct) are
	\begin{align*}
		{m}_r + \sum_{l=r+1}^{k} {m}_l \sum_{(b_1, \ldots, b_{a}) \in C_{r}^{l}} \prod_{c=1}^{a-1} -a^{\scr{N}}_{(i_1,\ldots,i_{b_{c}}),i_{b_{c+1}}} \matcom
	\end{align*}
	implying the claim.
\end{proof}

\begin{prop} \label{prop_sup_iterated_reflection}
    Let $k \in \ndN_0$, $i = (i_1,\ldots, i_k) \in \mathbb{I}^k$, assume $\scr{N}$ admits the reflection sequence $i$ and $m^{\scr{N}}_{(i_1,\ldots,i_r)} \in \ndN_0$ for all $1 \le r \le k$.
    Then 
    \begin{align*} \Sup{Q} = \beta^{\scr{N}}_i + s^{\scr{N}}_i ( \Sup{\refl_i(Q)}) . \end{align*}
    In particular we obtain $\beta^{\scr{N}}_i \in \Sup{Q} \subset \ndN_0^\theta$.
\end{prop}
\begin{proof}
    We do induction over $k$: For $k=0$ the claim is trivial. Now assume $k \ge 1$ and
    \begin{align*}
        \Sup{Q} = \beta^{\scr{N}}_{(i_1, \ldots, i_{k-1})} + s^{\scr{N}}_{(i_1, \ldots, i_{k-1})} ( \Sup{\refl_{(i_1, \ldots, i_{k-1})}(Q)}) \matdot
    \end{align*}
    By \Cref{lem_sup_of_iterated_refl} we have
    \begin{align*}
        \Sup{\refl_{(i_1,\ldots,i_{k-1})}(Q)}= m^{\scr{N}}_i \alpha_{i_k} + s^{\scr{N}}_{(i_1,\ldots,i_{k-1}),i_k} (\Sup{\refl_i(Q)}) \matdot
    \end{align*}
    Now 
    $
    \alpha_{i_k}  = -s^{\scr{N}}_{(i_1,\ldots,i_{k-1}),i_k} (\alpha_{i_k})
    $ 
    and $s^{\scr{N}}_{(i_1, \ldots, i_{k-1})} s^{\scr{N}}_{(i_1,\ldots,i_{k-1}),i_k} = s^{\scr{N}}_i$. Hence
    \begin{align*}
        \Sup{Q} = \beta^{\scr{N}}_{(i_1, \ldots, i_{k-1})} - m^{\scr{N}}_i s^{\scr{N}}_{i} ( \alpha_{i_k} ) + s^{\scr{N}}_{i} (\Sup{\refl_i(Q)}) \matdot
    \end{align*}
    Finally $\beta^{\scr{N}}_{(i_1, \ldots, i_{k-1})} - m^{\scr{N}}_i s^{\scr{N}}_{i} ( \alpha_{i_k} ) = \beta^{\scr{N}}_i$ by \Cref{lem_rels_beta_i}.
\end{proof}

\begin{cor} \label{cor_sup_iterated_reflection}
    Let $k \in \ndN_0$, $i = (i_1,\ldots, i_k) \in \mathbb{I}^k$, assume $\scr{N}$ admits the reflection sequence $i$ and $m^{\scr{N}}_{(i_1,\ldots,i_r)} \in \ndN_0$ for all $1 \le r \le k$. Then we have $\beta^{\scr{N}}_{i,m} \in \Sup{Q}$ for all $0 \le m \le m^{\scr{N}}_i$.
\end{cor}
\begin{proof}
    By \Cref{lem_rels_beta_i2_sys} we have $m^{\scr{N}}_i = m^{\scr{N}}_{(i_1,\ldots,i_k,i_k)}$, hence by definition $(m^{\scr{N}}_i - m) \alpha_{i_k} \in \Sup{\refl_i(Q)}$. Now \Cref{lem_rels_beta_i} and \Cref{lem_rels_beta_i2_sys} imply that
    \begin{align*}
        \beta^{\scr{N}}_{i,m} =& \beta^{\scr{N}}_{(i_1,\ldots,i_{k-1})} - m s^{\scr{N}}_i(\alpha_{i_k})
         = \beta^{\scr{N}}_{(i_1,\ldots,i_{k},i_k)} - m s^{\scr{N}}_i(\alpha_{i_k})
        \\ =& \beta^{\scr{N}}_{i} - m^{\scr{N}}_{(i_1,\ldots,i_{k},i_k)}  s^{\scr{N}}_{(i_1,\ldots,i_{k},i_k)} (\alpha_{i_k}) - m s^{\scr{N}}_i(\alpha_{i_k})
        \\ =& \beta^{\scr{N}}_{i} + m^{\scr{N}}_{i}  s^{\scr{N}}_{i} (\alpha_{i_k}) - m s^{\scr{N}}_i(\alpha_{i_k})
        \\ =& \beta^{\scr{N}}_i + s^{\scr{N}}_i( (m^{\scr{N}}_i - m) \alpha_{i_k} )
    \end{align*}
    Finally $\beta^{\scr{N}}_i + s^{\scr{N}}_i( (m^{\scr{N}}_i - m) \alpha_{i_k} ) \in \Sup{Q} $ by \Cref{prop_sup_iterated_reflection}.
\end{proof}

\begin{cor} \label{cor_sup_iterated_reflection2}
    Let $k,l \in \ndN_0$. Moreover let $i = (i_1,\ldots, i_k) \in \mathbb{I}^k$, $j=(j_1,\ldots,j_l) \in \mathbb{I}^l$ and assume $\scr{N}$ admits both reflection sequences $i$ and $j$. Moreover assume $m^{\scr{N}}_{(i_1,\ldots,i_r)} \in \ndN_0$ for $1 \le r \le k$ as well as $m^{\scr{N}}_{(j_1,\ldots,j_r)} \in \ndN_0$ for $1 \le r \le l$.
    If $s^{\scr{N}}_i ( \Sup{\refl_j(Q)})=s^{\scr{N}}_j ( \Sup{\refl_j(Q)})$, then we have $\beta^{\scr{N}}_i = \beta^{\scr{N}}_j$.
    In particular if $s_i^{\scr{N}}(\Sup{Q}) = \Sup{Q}$, then $\beta^{\scr{N}}_i=0$.
\end{cor}
\begin{proof}
By assumption and \Cref{prop_sup_iterated_reflection} we have 
\begin{align*}
    \beta^{\scr{N}}_i + s^{\scr{N}}_i ( \Sup{\refl_i(Q)}) &= \Sup{Q} = \beta^{\scr{N}}_j + s^{\scr{N}}_j ( \Sup{\refl_j(Q)}) 
    \\&= \beta^{\scr{N}}_j + s^{\scr{N}}_i ( \Sup{\refl_j(Q)})  \matdot
\end{align*} 
Applying ${s^{\scr{N}}_i }^{-1}$ to that equation yields
\begin{align*}
    {s^{\scr{N}}_i }^{-1}(\beta^{\scr{N}}_i-\beta^{\scr{N}}_j) + \Sup{\refl_i(Q)} = \Sup{\refl_j(Q)} \matdot
\end{align*} 
Since $0 \in \Sup{\refl_i(Q)}$ and $\Sup{\refl_j(Q)} \subset \ndN_0^\theta$ we obtain ${s^{\scr{N}}_i }^{-1}(\beta^{\scr{N}}_i-\beta^{\scr{N}}_j) \in \ndN_0^\theta$. Since $0 \in \Sup{\refl_j(Q)}$ as well as $\Sup{\refl_i(Q)} \subset \ndN_0^\theta$ we can conclude ${s^{\scr{N}}_i }^{-1}(\beta^{\scr{N}}_i-\beta^{\scr{N}}_j)=0$, i.e. $\beta^{\scr{N}}_i=\beta^{\scr{N}}_j$.
\end{proof}

\begin{prop} \label{prop_support_is_spanned_by_beta_i_sys}
	Assume $\scr{N}$ admits all reflections and $m^{\scr{N}}_i \in \ndN_0$ for all $i \in \mathbb{I}^k$, $k \in \ndN_0$. Let $M' \subset \ndR^\theta$ be the convex hull spanned by $\beta^{\scr{N}}_i \in \Sup{Q}$, $i \in \mathbb{I}^k$, $k \in \ndN_0$ and let $M:= M' \cap \ndN_0^\theta$. Then $\Sup{Q} \subset M$. In particular the convex hull of $\Sup{Q}$ equals $M'$.
\end{prop}
\begin{proof}
	For $k \in \ndN_0$, $i \in \mathbb{I}^k$ let $M_i = \beta^{\scr{N}}_i + s^{\scr{N}}_i(\ndN_0^\theta)$.
	By \Cref{prop_sup_iterated_reflection} we have $\Sup{Q} \subset M_i$.
	Now for all $j \in \mathbb{I}$ we have
	\begin{align*}
		\beta^{\scr{N}}_{(i_1,\ldots,i_{k},j)} = \beta^{\scr{N}}_i - m^{\scr{N}}_{(i_1,\ldots,i_k,j)} s^{\scr{N}}_{(i_1,\ldots,i_k,j)}(\alpha_j)
		=  \beta^{\scr{N}}_i + m^{\scr{N}}_{(i_1,\ldots,i_k,j)} s^{\scr{N}}_{i}(\alpha_j) \matdot
	\end{align*}
	In particular for $j=i_k$ we obtain $\beta^{\scr{N}}_{(i_1,\ldots,i_{k-1})} = \beta^{\scr{N}}_i + m^{\scr{N}}_{i} s^{\scr{N}}_{i}(\alpha_{i_k})$, using \Cref{lem_rels_beta_i2_sys}. 
	Thus we obtain that $M$ is the intersection of all $M_i$, $i \in \mathbb{I}^k$, $k \in \ndN_0$. Since $\Sup{Q} \subset M_i$ for all $i \in \mathbb{I}^k$, $k \in \ndN_0$, we can conclude that $\Sup{Q} \subset M$.
\end{proof}
\chapter{\mathintitleydmod{}-modules over Nichols systems}

In \cref{sect_ydmod_nich_sys_definition} we will quickly discuss some properties of $\ydmod{}$-modules over Nichols systems that are important when it comes to reflections.

In \cref{sect_ydmod_nich_sys_reflections} we will then use our previously defined reflection functors to the full extend to define reflections of such $\ydmod{}$-modules. We will discuss how some important properties are preserved when reflecting, the most important result being \Cref{prop_funCD_refl_generator}.

Moreover in \cref{sect_ydmod_nich_sys_iterated_reflections} we look at iterated reflections. We will again describe the geometric shape of the support of a reflectible $\ydmod{}$-module using the roots of the Nichols system. We will see how the vertices of that shape are actually the generating components of given reflections. Finally in \Cref{cor_real_subobject_lies_in_interior} we will discuss how there is a graded maximal subobject in the interior of that shape of such an $\ydmod{}$-module.

Now in \cref{sect_ydmod_nich_sys_induced_yd_module} we will do a construction of induced $\ydmod{}$-modules over Nichols systems that is similar to Verma modules in representation theory of Lie algebras. We will also study how reflections behave on these induced $\ydmod{}$-modules. In particular in \Cref{cor_ydind_irred_all_refl} we will characterize irreducibility via reflectiveness, if the Nichols system if finite dimensional.

Finally in \cref{sect_ydmod_nich_sys_shapovalov} we will study the maximal subobject of homogeneously generated $\ydmod{}$-modules over Nichols systems. We will do so by defining a special morphism, whose kernel gives the maximal subobject. We will also see that this morphism behaves well with reflections and actually tells us a lot about what reflection sequences a $\ydmod{}$-module admits. 

Let $H$ be a Hopf algebra over some field $\fK$ with bijective antipode and $\scr{C} := \ydmod{H}$. Let $\scr{N}$ be a pre-Nichols system and let $Q := \nsalg{\scr{N}}$, $N_j := \scr{N}_j$ for all $j \in \mathbb{I}$. Let $\Gamma$ be an abelian group, such that $\ndZ^\theta \subset \Gamma$. We view $Q$ as $\Gamma$-graded by continuing the $\ndN_0^\theta$-grading with zero components.

\section{\mathintitleydmod{}-modules over Nichols systems} \label{sect_ydmod_nich_sys_definition}

\begin{defi}
	Let $V \in \ydcat{Q}{\scr{C}}$ be a $\Gamma$-graded object.
	We say $V$ is \textbf{homogeneously generated}, if there exists an $\gamma \in \Gamma$, such that $V$ is generated by $V(\gamma)$ as a $Q$-module. We denote $\genind{V} := \gamma$ and $\gencomp{V} := V(\gamma)$.
\end{defi}

\begin{rema}
	Let $V \in \ydcat{Q}{\scr{C}}$ be a homogeneously generated $\Gamma$-graded object. Then $\Sup{V} \subset \genind{V} + \ndN_0^\theta$ and
	\begin{align*}
	V(\genind{V} + m \alpha_j) = Q(m \alpha_j) \cdot V(\genind{V}) = N_j^m \cdot \gencomp{V}
	\end{align*}
	for all $j \in \mathbb{I}$, $m \in \ndN_0$ by definition of $Q$.
\end{rema}

\begin{prop} \label{prop_irred_module_is_hom_gen}
	If $V \in \ydcat{Q}{\scr{C}}$ is a irreducible $\Gamma$-graded object, then $V$ is homogeneously generated and $\gencomp{V}$ is irreducible in $\scr{C}$.
\end{prop}
\begin{proof}
    By \Cref{cor_irred_graded_ydmodule_o}(1)$\implies$(3) there exists $n_0 \in \Gamma$ such that $V(n_0)$ is irreducible in $\scr{C}$ and generates $V$ as a $Q$-module.
\end{proof}

\begin{lem} \label{lem_Ni_generated_submodule_is_yd}
	Let $i \in \mathbb{I}$ and let $V \in \ydcat{Q}{\scr{C}}$ be homogeneously generated $\Gamma$-graded. Then 
	$W := \subalgQ{N_i} \cdot \gencomp{V}$
	is a $\ndN_0$-graded object in $\ydcat{\subalgQ{N_i}}{\scr{C}}$, where the $\subalgQ{N_i}$-coaction is the restriction $\coact^Q_V$ onto $W$ and with $W(n) = N_i^n \cdot \gencomp{V}$ for $n \in \ndN_0$.
\end{lem}
\begin{proof}
Indeed for $0 \le k \le m_i$ we have since $\coact_V^Q$ is $\Gamma$-graded
\begin{align*}
\coact_V^{Q} (N_i^k \cdot \gencomp{V}) =& \coact_V^Q (V(\genind{V}+ k \alpha_i)) 
 \subset  \oplus_{\gamma_1,\gamma_2 \in \Gamma}^{\gamma_1 + \gamma_2  = \genind{V} + k \alpha_i} Q(\gamma_1) \ot V(\gamma_2)
\\=& \oplus_{n_1,n_2 \in \ndN_0}^{n_1 + n_2 = k } \subalgQ{N_i}(n_1 \alpha_i) \ot V(n_2 \alpha_i + \genind{V})
\\=& \oplus_{n_1,n_2 \in \ndN_0}^{n_1 + n_2 = k } {N_i}^{n_1} \ot {N_i}^{n_2} \cdot \gencomp{V}
\subset  \subalgQ{N_i} \ot W \matcom
\end{align*}
hence $\coact_V^{Q}(W) \subset \subalgQ{N_i} \ot W$ and we obtain $W \in \ydcat{\subalgQ{N_i}}{\scr{C}}$.
\end{proof}

\Cref{lem_Ni_generated_submodule_is_yd} allows the following definition.

\begin{defi}
	Let $i \in \mathbb{I}$ and let $V \in \ydcat{Q}{\scr{C}}$ be homogeneously generated $\Gamma$-graded. If $\subalgQ{N_i} \cdot \gencomp{V} \in \ydcat{\subalgQ{N_i}}{\scr{C}}$ is well graded, then we call $V$ \textbf{$i$-well graded}.
\end{defi}

\begin{rema} \label{rema_i_well_graded_characterisation}
	Let $i \in \mathbb{I}$ and $V \in \ydcat{Q}{\scr{C}}$ be homogeneously generated $\Gamma$-graded. As discussed in \Cref{rema_defi_well_graded} the following are equivalent:
	\begin{enumerate}
		\item $V$ is $i$-well graded.
		\item $\lbrace v \in \subalgQ{N_i} \cdot \gencomp{V} \, | \, \coact_V^Q(v) = 1 \ot v \rbrace = \gencomp{V} $.
	\end{enumerate}
\end{rema}

\begin{lem} \label{lem_gencompV_irred_iff_Ni_V_irred}
	Let $i \in \mathbb{I}$ and $V \in \ydcat{Q}{\scr{C}}$ be homogeneously generated $i$-well $\Gamma$-graded. The following are equivalent:
	\begin{enumerate}
		\item $\gencomp{V}$ is irreducible in $\scr{C}$.
		\item $\subalgQ{N_i}\cdot \gencomp{V} $ is irreducible in $\ydcat{\subalgQ{N_i}}{\scr{C}}$.
	\end{enumerate}
\end{lem}
\begin{proof}
Implied by \Cref{cor_irred_graded_ydmodule_o}.
\end{proof}

\begin{prop} \label{prop_irreducible_module_is_well_graded}
	Let $i \in \mathbb{I}$ and $V \in \ydcat{Q}{\scr{C}}$ be a $\Gamma$-graded irreducible object. Then $V$ is $i$-well $\Gamma$-graded.
\end{prop}
\begin{proof}
By \Cref{prop_irred_module_is_hom_gen} we obtain that $V$ is homogeneously generated.
Regarding \Cref{cor_irred_graded_ydmodule_o} it is enough to show that $\subalgQ{N_i}\cdot \gencomp{V} $ is irreducible in $\ydcat{\subalgQ{N_i}}{\scr{C}}$.
Let $p : \Gamma \rightarrow \ndZ$ be a group homomorphism, such that $p(\alpha_i)=0$ and $p(\alpha_j) = 1$ for all $j \in \mathbb{I} \setminus \lbrace i \rbrace$. Then $Q$ is $\ndN_0$-graded by $Q(n) = Q(p^{-1}(n))$ for $n \in \ndN_0$ and $V \in \ydcat{Q}{\scr{C}}_{\rat}$ is $\ndZ$-graded by $V(n) = V(p^{-1}(n))$ for $n \in \ndZ$.
In particular $Q(0) = \subalgQ{N_i}$ and the component of $V$ of smallest degree is $V(p(\genind{V})) =\subalgQ{N_i}\cdot \gencomp{V}$.
Then \Cref{cor_irred_graded_ydmodule_o}(1)$\implies$(3) implies that $\subalgQ{N_i}\cdot \gencomp{V}$ is irreducible in $\ydcat{Q(0)}{\scr{C}} = \ydcat{\subalgQ{N_i}}{\scr{C}}$.
\end{proof}
\section{Reflections of \mathintitleydmod{}-modules over Nichols systems} \label{sect_ydmod_nich_sys_reflections}

We start this section with the definition of a reflection of a $\ydmod{}$-module over a Nichols system and by discussing how that changes the grading, where the main work was already done in \cref{sect_refl_functors}. We will continue by discussing how the important properties regarding reflections are preserved in \Cref{prop_funCD_refl_generator} and what happens if we do the same reflection twice in \Cref{lem_mii_equals_mi}.

Let $i \in \mathbb{I}$ and $\scr{N}$ be a Nichols system over $i$. Moreover let $Q := \nsalg{\scr{N}}$ and $N_j := \scr{N}_j$ for all $j \in \mathbb{I}$. Assume that $\scr{N}$ is $i$-finite and that all $N_j$, $j \in \mathbb{I} \setminus \lbrace i \rbrace$ are irreducible in $\scr{C}$. Let $\Gamma$ be an abelian group, such that $\ndZ^\theta \subset \Gamma$. Finally let $t_i \in \Aut(\Gamma)$ be such that $t_i (\alpha) = s^{\scr{N}}_i(\alpha)$ for all $\alpha  \in \ndZ^\theta$ and $t_i^2 = \id_{\Gamma}$. We view both $Q$ and $\refl_i(Q)$ as $\Gamma$-graded by continuing the $\ndN_0^\theta$-grading with zero components.

If needed recall \Cref{defi_refl_yd_mod}, where the reflection functor $\funCD_{\refl_i}$ is defined.

\begin{defi}
	For $V \in \ydcat{Q}{\scr{C}}_{\rat}$ we denote 
	\begin{align*}
		\refl_i(V) := \funCD_{\refl_i}(V) \in \ydcat{\refl_i(Q)}{\scr{C}}_{\rat} \matdot
	\end{align*}
		$\refl_i(V)$ is called the \textbf{$i$-th reflection of $V$}.
\end{defi}

\begin{prop} \label{prop_refl_of_irreducible_module_is_irreducible}
	If $V \in \ydcat{Q}{\scr{C}}_{\rat}$ is irreducible, then $\refl_i(V)$ is irreducible in $\ydcat{\refl_i(Q)}{\scr{C}}_{\rat}$.
\end{prop}
\begin{proof}
	This is implied by \Cref{lem_funcd_keeps_irreducibility}.
\end{proof}

\begin{rema}
	If $V \in \ydcat{Q}{\scr{C}}_{\rat}$, then $\refl_i(V) = V$ as vector spaces, hence $\funCD_{\refl_i}$ defines a $\refl_i(Q)$-module and comodule structure on $V$.
\end{rema}

\begin{prop} \label{prop_grading_of_reflected_yd_module}
	Let $V \in \ydcat{Q}{\scr{C}}_{\rat}$ be a $\Gamma$-graded object.
	Then $\refl_i (V) \in \ydcat{\refl_i(Q)}{\scr{C}}_{\rat}$ is $\Gamma$-graded with $\refl_i (V) (\gamma) = V(t_i(\gamma))$ for $\gamma \in \Gamma$.
\end{prop}
\begin{proof}
	By \Cref{prop_funCD_module_grading}, $\refl_i (V) \in \ydcat{\refl_i(Q)}{\scr{C}}_{\rat}$ is $\Gamma$-graded with grading $\refl_i (V) (\alpha) = V(\alpha)$ for $\alpha \in \Gamma$, where the $\Gamma$-grading of $\refl_i(Q)$ is given by $Q$ as in \Cref{prop_grading_on_funCDK}(2). 
	Now the $\ndN_0^\theta$-grading on $\refl_i(Q)$ that we use is given by shifting this grading by the the automorphism $s^{\scr{N}}_i$, see the proof of \Cref{prop_RiQ_is_nich_sys}(2)(b).
	Hence with concerning \Cref{rema_shifting_grading_by_autom}, the claim is implied.
\end{proof}

\begin{notation}
    For $j \in \mathbb{I}$ and $V \in \ydcat{Q}{\scr{C}}$ we denote
    \begin{align*}
        m^{V}_j = \max \lbrace m \in \ndN_0 \, | \,  N_j^{m} \cdot \gencomp{V} \ne 0 \rbrace \matdot
    \end{align*}
\end{notation}

\begin{rema}
    Let $j \in \mathbb{I}$ and $V \in \ydcat{Q}{\scr{C}}$. Clearly $m^{V}_j \le m^{\scr{N}}_j$ for all $j  \in \mathbb{I}$. If $V \in \ydcat{Q}{\scr{C}}_{\rat}$, then by definition we have $m^V_j \in \ndN_0$.
\end{rema}

\begin{lem} \label{lem_Nidual_act_on_Ni_act_on_V}
	Let $V \in \ydcat{Q}{\scr{C}}_{\rat}$ be homogeneously generated $i$-well $\Gamma$-graded. Moreover let $m_i := m_i^V$ and let 
	\begin{align*}
	W := \subalgQ{N_i} \cdot \gencomp{V} \in \ydcat{\subalgQ{N_i}}{\scr{C}} \matdot
	\end{align*}
	Assume $\gencomp{V}$ is irreducible in $\scr{C}$. Then
	\begin{enumerate}
		\item $N_i^{m_i} \cdot \gencomp{V}$ is irreducible in $\scr{C}$.
		\item $\funCD_i(W) \in \ydcat{\subalgQ{N_i^*}}{\scr{C}}_{\rat}$ is generated by $N_i^{m_i} \cdot \gencomp{V}$, that is we have for $n \in \ndN_0$
		\begin{align*}
		{N_i^*}^n \cdot (N_i^{m_i} \cdot \gencomp{V})= 
		\begin{cases}
		N_i^{m_i-n} \cdot \gencomp{V} & \text{if $n \le m_i$,} \\
		0 & \text{if $n > m_i$.}
		\end{cases}
		\end{align*}
		\item $\funCD_i(W) = \subalgQ{N_i^*} \cdot \left( N_i^{m_i} \cdot \gencomp{V} \right) $ is well graded.
	\end{enumerate}
\end{lem}
\begin{proof}
	By \Cref{lem_gencompV_irred_iff_Ni_V_irred} we know that $W$ is irreducible in $\ydcat{\subalgQ{N_i}}{\scr{C}}$. Hence $\funCD_i (W) \in \ydcat{\subalgQ{N_i^*}}{\scr{C}}$ is irreducible and $\ndZ$-graded by \Cref{prop_funCD_exist_and_properties} and \Cref{prop_funCD_properties_irred} via $\funCD_i (W)(n)=W(-n)$, $n \in \ndZ$.
	Observe that the component of smallest degree is
	\begin{align*}
	\funCD_i (W)(-m_i)=W(m_i)=N_i^{m_i} \cdot \gencomp{V} \matdot
	\end{align*}
	Thus \Cref{cor_irred_graded_ydmodule_o}(1)$\implies$(3) implies (1), (2) and~(3).
\end{proof}

\begin{prop} \label{prop_funCD_refl_generator}
	Let $V \in \ydcat{Q}{\scr{C}}_{\rat}$ be homogeneously generated $i$-well $\Gamma$-graded and let $m_i = m_i^V$. Assume that $\gencomp{V}$ is irreducible.
	Then the object $\refl_i(V) \in \ydcat{\refl_i(Q)}{\scr{C}}_{\rat}$ is homogeneously generated $i$-well $\Gamma$-graded, such that $\genind{\refl_i(V)} = t_i(\genind{V}) - m_i \alpha_i$ and
	\begin{align*} 
		 \gencomp{\refl_i(V)}= V(\genind{V} + m_i \alpha_i) = {N_i}^{m_i} \cdot \gencomp{V} \matdot
	\end{align*}
	Moreover $\gencomp{\refl_i(V)}$ is irreducible in $\scr{C}$.
\end{prop}
\begin{proof}
	Recall the $\Gamma$-grading of $\refl_i(V)$ as described in \Cref{prop_grading_of_reflected_yd_module} and observe
		\begin{align*} 
	\refl_i(V)( t_i(\genind{V}) - m_i \alpha_i ) = V(\genind{V} + m_i \alpha_i) = {N_i}^{m_i} \cdot \gencomp{V} \matdot
	\end{align*}
	Below we are going to show for $\beta \in \Gamma$
	\begin{align*} 
		\refl_i(V)(\beta) = \refl_i(Q)(\beta - t_i(\genind{V}) + m_i \alpha_i) \cdot \refl_i(V)(t_i(\genind{V}) - m_i \alpha_i) \matdot \tag{i}
	\end{align*}
	Then $\refl_i(V)$ is homogeneously generated with $\genind{\refl_i(V)} = t_i(\genind{V}) - m_i \alpha_i$ and $\gencomp{\refl_i(V)}={N_i}^{m_i} \cdot \gencomp{V}$. From here \Cref{lem_Nidual_act_on_Ni_act_on_V}(3) implies that $\refl_i(V)$ is $i$-well graded and
	\Cref{lem_Nidual_act_on_Ni_act_on_V}(1) implies that $\gencomp{\refl_i(V)}$ is irreducible in $\scr{C}$, finishing the proof.
	
	It is enough to show (i) for $\beta \in \Sup{\refl_i (V) }$. Let $\gamma := t_i(\beta)$. Now to show (i) is to show
	\begin{align*}
		V(\gamma) = \refl_i(Q)(t_i(\gamma) - t_i(\genind{V}) + m_i \alpha_i) \cdot \left( {N_i}^{m_i} \cdot \gencomp{V} \right) \matdot
	\end{align*}
	As $0 \ne \refl_{i}(V)(\beta) = V(\gamma)= Q(\gamma-\genind{V}) \gencomp{V}$ and $Q(\gamma') = 0$ for $\gamma' \in \Gamma \setminus \ndN_0^\theta$, we can conclude $\gamma = \genind{V} + \sum_{j\in \mathbb{I}} n_j \alpha_j$ for some $n_1, \ldots, n_\theta \in \ndN_0$.
	Now recall that the $\ndN_0^\theta$-grading of $Q$ is given by \Cref{prop_grading_on_funCDK}(1), i.e.
	\begin{align*}
		Q(\gamma - \genind{V}) = \oplus_{ k_1, k_2 \in \ndN_0}^{k_1+k_2 = n_i}  K_i(k_1 \alpha_i + \sum_{j \in \mathbb{I} \setminus \lbrace i \rbrace} n_j \alpha_j ) \ot  {N_i}^{k_2} \matdot
	\end{align*}
	Observe that 
	\begin{align*}
	t_i (t_i(\gamma) - t_i(\genind{V}) + m_i \alpha_i) = \gamma-\genind{V} - m_i \alpha_i = \sum_{j\in \mathbb{I}} n_j \alpha_j - m_i \alpha_i
	\end{align*}
	and recall from the proof of \Cref{prop_RiQ_is_nich_sys}(2)(b) the $\ndN_0^\theta$-grading of $\refl_i(Q)$, which yields
	\begin{align*}
		&\refl_i(Q)(t_i(\gamma) - t_i(\genind{V}) + m_i \alpha_i)
		\\ =& \oplus_{k_1 , k_2 \in \ndN_0}^{k_1-k_2 = n_i - m_i} K_i(k_1 \alpha_i + \sum_{j \in \mathbb{I} \setminus \lbrace i \rbrace} n_j \alpha_j ) \ot {N_i^*}^{k_2}  \matdot
	\end{align*}
	Thus by \Cref{lem_Nidual_act_on_Ni_act_on_V}(2) we obtain
	\begin{align*}
		&\refl_i(Q)(t_i(\gamma) - t_i(\genind{V}) + m_i \alpha_i) \cdot \left( {N_i}^{m_i} \cdot \gencomp{V} \right)
		\\ =& \oplus_{k_1, k_2 \in \ndN_0}^{k_1-k_2 = n_i - m_i} K_i(k_1 \alpha_i + \sum_{j \in \mathbb{I} \setminus \lbrace i \rbrace} n_j \alpha_j ) \cdot \left( {N_i^*}^{k_2} \cdot \left( {N_i}^{m_i} \cdot \gencomp{V} \right) \right)
		\\ =& \oplus_{ k_1 \in \ndN_0, 0\le k_2 \le m_i}^{k_1-k_2 = n_i - m_i} K_i(k_1 \alpha_i + \sum_{j \in \mathbb{I} \setminus \lbrace i \rbrace} n_j \alpha_j ) \cdot  \left( {N_i}^{m_i-k_2} \cdot \gencomp{V} \right) 
		\\ =&  \oplus_{ k_1 \in \ndN_0, 0\le k_2 \le m_i}^{k_1+k_2 = n_i}   K_i(k_1 \alpha_i + \sum_{j \in \mathbb{I} \setminus \lbrace i \rbrace} n_j \alpha_j ) \cdot  \left( {N_i}^{k_2} \cdot \gencomp{V} \right)
		\\ =& Q(\gamma - \genind{V}) \cdot \gencomp{V} = V(\gamma) \matcom
	\end{align*}
	where the second last equality uses $N_i^k \cdot \gencomp{V} = 0$ for all $k > m_i$.
\end{proof}

\begin{lem} \label{lem_mii_equals_mi}
	Let $V \in \ydcat{Q}{\scr{C}}_{\rat}$ be a homogeneously generated $i$-well $\Gamma$-graded object and assume $\gencomp{V}$ is irreducible. 
	Then we have $m^V_i=m^{\refl_i(V)}_i$, $\genind{\refl_i (\refl_i(V))} = \genind{V}$ and
	\begin{align*} 
		 \gencomp{\refl_i (\refl_i(V))}= \gencomp{V} \matdot
	\end{align*}
\end{lem}
\begin{proof}
Since with \Cref{prop_funCD_refl_generator} and \Cref{lem_Nidual_act_on_Ni_act_on_V}(2) for $m \in \ndN_0$ we have
\begin{align*}
{\refl_i(N)_i}^{m} \cdot \gencomp{ \refl_i(V)} = {N_i^*}^m \cdot \left( {N_i}^{m^V_i} \cdot \gencomp{V} \right) 
	= \begin{cases}
		N_i^{m^V_i-m} \cdot \gencomp{V} & \text{if $m \le m^V_i$,} \\
		0 & \text{if $m > m^V_i$.}
	\end{cases}
\end{align*}
we obtain $m^V_i=m^{\refl_i(V)}_i$. Thus by \Cref{prop_funCD_refl_generator}
\begin{align*}
\genind{\refl_i (\refl_i(V))} = t_i(\genind{\refl_i(V)}) - m^V_i \alpha_i = t_i(t_i(\genind{V}) - m^V_i \alpha_i) - m^V_i \alpha_i  = \genind{V}
\end{align*}
and hence $\gencomp{\refl_i (\refl_i(V))}= \gencomp{V}$.
\end{proof}
\section{Iterated reflections} \label{sect_ydmod_nich_sys_iterated_reflections}

Similar to \cref{sect_nich_sys_iterated_refl} we will study what happens if we reflect a $\ydmod{}$-module of a Nichols system several times. We study the grading and the geometric shape of the support of such a reflectible $\ydmod{}$-module. As it turns out the edges and in particular the vertices of that shape play an important role, as the latter describes the generating components of given reflections of the $\ydmod{}$-module. We will show in \Cref{prop_real_subobject_lies_in_interior} and \Cref{cor_real_subobject_lies_in_interior}, that a proper $\ydmod{}$-subobject of an object admitting all reflections cannot have an area on any of the edges and therefore there must exist a maximal graded subobject in the interior of that shape. 

Let $\scr{N}$ be a pre-Nichols system, $Q := \nsalg{\scr{N}}$ and assume that $\scr{N}_j$ is irreducible in $\scr{C}$ for all $j \in \mathbb{I}$. Let $0 \ne V \in \ydcat{Q}{\scr{C}}_{\rat}$ be homogeneously generated $\Gamma$-graded and assume $\gencomp{V}$ is irreducible in~$\scr{C}$. Again let $\Gamma$ be an abelian group, such that $\ndZ^\theta \subset \Gamma$.
	
To look at iterated reflections of $V$ we have to make sure the preconditions of \Cref{prop_funCD_refl_generator} are met after each reflection.
\begin{defi}
	Let $k \in \ndN_0$, $i_1, \ldots, i_k \in \mathbb{I}$.
	\begin{enumerate}
		\item Assume $\scr{N}$ admits the reflection sequence $(i_1,\ldots, i_k)$. We say $V$ \textbf{admits the reflection sequence $(i_1,\ldots, i_k)$}, if $k=0$ or $V$ is $i_1$-well graded and $\refl_{i_1}(V) \in \ydcat{\refl_{i_1}(Q)}{\scr{C}}_{\rat}$ admits the reflection sequence $(i_2,\ldots, i_k)$.
		\item Assume $\scr{N}$ admits all reflections. We say $V$ \textbf{admits all reflections}, if for all $n \in \ndN_0$, $i \in \mathbb{I}^n$, $V$ admits the reflection sequence $i$.
	\end{enumerate}
\end{defi}

\begin{prop} \label{prop_irred_admits_all_reflections}
    Let $k \in \ndN$, $i=(i_1,\ldots,i_k) \in \mathbb{I}^k$. If $\scr{N}$ admits the reflection sequence $i$ and $V$ is an irreducible object in $\ydcat{Q}{\scr{C}}$, then $V$ admits the reflection sequence $i$. 
\end{prop}
\begin{proof}
	By \Cref{prop_irreducible_module_is_well_graded} $V$ is $i_1$-well graded and by \Cref{prop_refl_of_irreducible_module_is_irreducible} $\refl_{i_1}(V)$ is irreducible in $\ydcat{\refl_{i_1}(Q)}{\scr{C}}$, hence the claim follows inductively.
\end{proof}

\begin{notation}
	Let $k \in \ndN$, $i = (i_1,\ldots, i_k) \in \mathbb{I}^k$. 
	Assume $\scr{N}$ admits the reflection sequence $i$ and assume $V$ admits the reflection sequence $(i_1,\ldots,i_{k-1})$.
	Let $m^V_{()}=0$ and denote
	\begin{align*}
		m^V_{i} :=&\, m^{\refl_{(i_1,\ldots,i_{k-1})}(V)}_{i_k} \\ =& \max \lbrace m \in \ndN_0 \, | \, 
		{\refl_{(i_1,\ldots,i_{k-1})}(\scr{N})_{i_k}}^m 
		\cdot \gencomp{\refl_{(i_1,\ldots,i_{k-1})}(V)} \ne 0 \rbrace \matdot
	\end{align*}
    For $0 \le m \le m^V_i$ let 
    \begin{align*}
        \beta^{V}_{i,m} := \beta^{\scr{N}}_{i,(m^V_{(i_1)},\ldots,m^V_{(i_1,\ldots,i_{k-1})}),m)} \matcom
    \end{align*}
    where $\beta^V_{(),0}=0$ and let $\beta^{V}_i := \beta^V_{i,m^V_i}$.

Moreover let $t_{(i_1,\ldots,i_{k-1}),i_k} \in \Aut(\Gamma)$ be such that $t_{(i_1,\ldots,i_{k-1}),i_k}^2 = \id_{\Gamma}$ and for all $j \in \mathbb{I}$
\begin{align*}
	t_{(i_1,\ldots,i_{k-1}),i_k} ( \alpha_j) = s^{\scr{N}}_{(i_1,\ldots,i_{k-1}),i_k} ( \alpha_j) \matdot
\end{align*}
Also let $t_i=t_{(),i_1} t_{(i_1),i_2} \cdots t_{(i_1,\ldots,i_{k-1}),i_k}$ and $t_{()}=\id_\Gamma$.
\end{notation}

\begin{prop} \label{prop_grading_of_iterated_refl_mod}
    Let $k \in \ndN$, $i \in \mathbb{I}^k$ and assume $V$ admits the reflection sequence $i$.
	The reflection $\refl_i(V) \in \ydcat{\refl_i(Q)}{\scr{C}}_{\rat}$ is $\Gamma$-graded with $\refl_i(V)(\gamma) = V(t_i(\gamma))$. 
\end{prop}
\begin{proof}
	Follows inductively from \Cref{prop_grading_of_reflected_yd_module}.
\end{proof}

\begin{lem} \label{lem_generator_index_explicit}
    Let $k \in \ndN$, $i = (i_1,\ldots, i_k) \in \mathbb{I}^k$ and assume $V$ admits the reflection sequence $i$.
	The following hold:
	\begin{align*}
		\genind{\refl_i(V)} &= t_{(i_1,\ldots,i_{k-1}),i_k} (\genind{\refl_{(i_1,\ldots,i_{k-1})}(V)}) - m^V_i \alpha_{i_k} \tag{1} \matdot \\
		\gencomp{\refl_i(V)} &= \refl_{(i_1,\ldots,i_{k-1})}(V)(\genind{\refl_{(i_1,\ldots,i_{k-1})}(V)} + m^V_i \alpha_{i_k}) \tag{2}  \\ &=\refl_{(i_1,\ldots,i_{k-1})}(\scr{N})_{i_k}^{m^V_i} 
		\cdot \gencomp{\refl_{(i_1,\ldots,i_{k-1})}(V)} \matdot 
	\end{align*}
\end{lem}
\begin{proof}
	Both (1) and (2) follow from \Cref{prop_funCD_refl_generator}.
\end{proof}

\begin{lem} \label{lem_rels_beta_i2}
    Let $k \in \ndN$, $i = (i_1,\ldots, i_k) \in \mathbb{I}^k$ and assume $V$ admits the reflection sequence $i$.
	The following relations hold:
	\begin{align*}
	m^V_{(i_1,\ldots,i_k,i_k)} = m^V_{(i_1,\ldots,i_k)} \matcom &&
	\beta^V_{(i_1,\ldots,i_k,i_k)} = \beta^V_{(i_1,\ldots,i_{k-1})} \matdot
	\end{align*}
\end{lem}
\begin{proof}
	The first relation is given by \Cref{lem_mii_equals_mi}. Then similar to the proof of \Cref{lem_rels_beta_i2_sys} the second relation follows from \Cref{lem_rels_beta_i}.
\end{proof}

\begin{prop} \label{prop_generator_index_explicit}
    Let $k \in \ndN$, $i = (i_1,\ldots, i_k) \in \mathbb{I}^k$ and assume $V$ admits the reflection sequence $i$.
	We have $\genind{\refl_i(V)} = t_i^{-1}(\genind{V}+\beta^V_i)$ and
	\begin{align*}
		\gencomp{\refl_i(V)} = V(\genind{V}+\beta^V_i) = Q(\beta^V_i) \cdot \gencomp{V}.
	\end{align*}
	In particular, $\beta^V_i \in \ndN_0^\theta$.
\end{prop}
\begin{proof}
	It is enough to show $\genind{\refl_i(V)} = t_i^{-1}(\genind{V}+\beta^V_i)$, since then with \Cref{prop_grading_of_iterated_refl_mod} we have
	\begin{align*}
		0 \ne \gencomp{\refl_i(V)} = \refl_i(V)(t_i^{-1}(\genind{V}+\beta^V_i)) = V(\genind{V}+\beta^V_i) \matdot
	\end{align*}
	We do this by induction on $k$: If $k=0$, then $\beta^V_i = 0$, $t_i^{-1}=\id_{\Gamma}$ and the claim holds. Assume $k \ge 1$.	
	By \Cref{lem_generator_index_explicit}(1) and by induction hypothesis we obtain
	\begin{align*}
		\genind{\refl_i(V)} 
		=& t_{(i_1,\ldots,i_{k-1}),i_k} (\genind{\refl_{(i_1,\ldots,i_{k-1})}(V)}) - m^V_i \alpha_{i_k} 
		\\=& t_{(i_1,\ldots,i_{k-1}),i_k} ( t_{(i_1,\ldots,i_{k-1})}^{-1}(\genind{V}+\beta^V_{(i_1,\ldots,i_{k-1})}) ) - m^V_i \alpha_{i_k} 
		\\ =& t_i^{-1} ( \genind{V}+\beta^V_{(i_1,\ldots,i_{k-1})} - m^V_i t_i( \alpha_{i_k}) )
		= t_i^{-1} ( \genind{V}+\beta^V_i )
		\matcom
	\end{align*}
	where the last equality is given by \Cref{lem_rels_beta_i}.
\end{proof}

\begin{cor} \label{cor_generator_index_explicit}
    Let $k \in \ndN$, $i = (i_1,\ldots, i_k) \in \mathbb{I}^k$, assume $\scr{N}$ admits the reflection sequence $i$ and assume $V$ admits the reflection sequence $(i_1,\ldots,i_{k-1})$.
	For $0 \le m \le m^V_i$ we have
	\begin{align*}
	\refl_{(i_1,\ldots,i_{k-1})}(\scr{N})_{i_k}^{m} \cdot \gencomp{\refl_{(i_1,\ldots,i_{k-1})}(V)} = V(\genind{V}+\beta^V_{i,m}) \matdot
	\end{align*}
	In particular, $\beta^V_{i,m} \in \ndN_0^\theta$.
\end{cor}

\begin{proof}
	We conclude this from \Cref{prop_generator_index_explicit} and \Cref{prop_grading_of_iterated_refl_mod}:
	\begin{align*}
	&{\refl_{(i_1,\ldots,i_{k-1})}(\scr{N})_{i_k}}^{m} 
	\cdot \gencomp{\refl_{(i_1,\ldots,i_{k-1})}(V)}
	\\=& \refl_{(i_1,\ldots,i_{k-1})}(V) (\genind{\refl_{(i_1,\ldots,i_{k-1})}(V)} + m \alpha_{i_k})
	\\=& \refl_{(i_1,\ldots,i_{k-1})}(V) (t_{(i_1,\ldots,i_{k-1})}^{-1}(\genind{V} + \beta^V_{(i_1,\ldots,i_{k-1})}) + m \alpha_{i_k})
	\\=&
	V(\genind{V}+ \beta^V_{(i_1,\ldots,i_{k-1})} + m t_{(i_1,\ldots,i_{k-1})}(\alpha_{i_k})) = V(\genind{V}+\beta^V_{i,m})
	\matcom
	\end{align*}
	since $t_{(i_1,\ldots,i_{k-1})}(\alpha_{i_k}) = -t_i (\alpha_{i_k})$.
\end{proof}

\begin{cor} \label{cor_generator_index_explicit2}
    Let $k \in \ndN$, $i \in \mathbb{I}^k$ and assume $V$ admits the reflection sequence $i$.
    The following are equivalent:
    \begin{enumerate}
        \item $s_i^{\scr{N}}(\beta^V_i) = \beta^V_i$.
        \item $\beta^V_i = 0$.
        \item $\genind{\refl_i(V)}=t_i^{-1}(\genind{V})$ and $\gencomp{\refl_i(V)} = \gencomp{V}$.
    \end{enumerate}
\end{cor}
\begin{proof}
Clearly (2) implies (1). By \Cref{prop_generator_index_explicit} we have 
\begin{align*}
    \genind{\refl_i(V)} = t_i^{-1}(\genind{V}+\beta^V_i) =  t_i^{-1}(\genind{V}) + {s^{\scr{N}}_i}^{-1}(\beta^V_i) \matdot
\end{align*} 
Hence by \Cref{prop_grading_of_iterated_refl_mod} (2) and (3) are equivalent. Now 
\begin{align*}
    0 \ne V(\genind{V}) = \refl_i(V)(t_i^{-1}(\genind{V})) = \refl_i(V) (\genind{\refl_i(V)} - {s^{\scr{N}}_i}^{-1}(\beta^V_i)) \matdot
\end{align*}
Since $\Sup{\refl_i(V)} \subset \genind{\refl_i(V)} + \ndN_0^\theta$ it follows that $-{s^{\scr{N}}_i}^{-1}(\beta^V_i) \in \ndN_0^\theta$. If (1) holds, then this means $-\beta^V_i \in \ndN_0^\theta$. We also have $\beta^V_i \in \ndN_0^\theta$ by \Cref{prop_generator_index_explicit}, hence $\beta^V_i=0$.
\end{proof}

\begin{defi}
	Assume $V$ admits all reflections. For $k \in \ndN_0$, $i = (i_1,\ldots, i_k) \in \mathbb{I}^k$ we call $\gencomp{\refl_i(V)}=V(\genind{V}+\beta^V_i)$ the \textbf{$i$-th vertex of $V$} and $\oplus_{m=0}^{m_i^V} V(\genind{V}+\beta^V_{i,m})$ the \textbf{$i$-th edge of $V$}. Moreover we define the following subsets of $\ndN_0^\theta$
	\begin{align*}
	\Supind{V} :=& \lbrace n \in \ndN_0^\theta \, | \, \genind{V}+n \in \Sup{V} \rbrace \matcom \\
	\bouind{V} :=& \lbrace  \beta^V_{i,m} \, | \, k \in \ndN, i=(i_1,\ldots,i_k) \in \mathbb{I}^k, 0 \le m \le m^V_i \rbrace \matcom \\
	\intind{V} :=& \Supind{V} \setminus \bouind{V} \matdot
	\end{align*}
	Now define $\bound{V} := V(\genind{V} + \bouind{V})$, called the \textbf{edges of $V$} and $\inter{V} := V(\genind{V} + \intind{V})$, called the \textbf{interior of $V$}.
\end{defi}

\begin{rema}
    Geometrically speaking the interior of $V$ contains all points except the edges of $\Supind{V}$, i.e. it also contains the inside points of the faces of $\Supind{V}$.
\end{rema}

\begin{rema}
	Assume $V$ admits all reflections.
	The following hold:
	\begin{enumerate}
		\item $\Sup{V} = \genind{V} + \Supind{V}$,
		\item $\bouind{V} \subset \Supind{V}$,
		\item $V = \inter{V} \oplus \bound{V}$.
	\end{enumerate}
	Indeed: While (1) and (3) are given by definition, (2) follows from \Cref{cor_generator_index_explicit}: For $k \in \ndN$, $i=(i_1,\ldots,i_k) \in \mathbb{I}^k$, $0 \le m \le m^V_i$ we have by definition of $m^V_i$:
	\begin{align*}
		0 \ne {\refl_{(i_1,\ldots,i_{k-1})}(\scr{N})_{i_k}}^{m} 
		\cdot \gencomp{\refl_{(i_1,\ldots,i_{k-1})}(V)}
		=
		V(\genind{V}+ \beta^V_{i,m}) \matdot
	\end{align*}
\end{rema}

\begin{prop} \label{prop_support_is_spanned_by_beta_i}
	Assume $V$ admits all reflections, let $M' \subset \ndR^\theta$ be the convex hull spanned by the points $\beta^V_i \in \bouind{V}$, $i \in \mathbb{I}^k$, $k \in \ndN_0$ and let $M:= M' \cap \ndN_0^\theta$. Then $\Supind{V} \subset M$. In particular, the convex hull of $\Supind{V}$ coincides with $M'$.
\end{prop}
\begin{proof}
    The proof is similar to the proof of \Cref{prop_support_is_spanned_by_beta_i_sys}:
	For $k \in \ndN_0$, $i \in \mathbb{I}^k$ let $M_i = \beta^V_i + s^{\scr{N}}_i(\ndN_0^\theta)$.
	Considering \Cref{prop_generator_index_explicit} and since for $n \in \ndN_0^\theta$ we have 
	\begin{align*}
	\refl_i(Q)(n) \cdot \gencomp{\refl_i(V)} = \refl_i(V)(t_i^{-1}(\genind{V} + \beta^V_i) + n) = V(\genind{V}+\beta^V_i +s^{\scr{N}}_i(n))
	\end{align*}
	we obtain $\Supind{V} \subset M_i$.
	Now for all $j \in \mathbb{I}$ we have
	\begin{align*}
		\beta^V_{(i_1,\ldots,i_{k},j)} = \beta^V_i - m^V_{(i_1,\ldots,i_k,j)} s^{\scr{N}}_{(i_1,\ldots,i_k,j)}(\alpha_j)
		=  \beta^V_i + m^V_{(i_1,\ldots,i_k,j)} s^{\scr{N}}_{i}(\alpha_j) \matdot
	\end{align*}
	In particular for $j=i_k$ we obtain $\beta^V_{(i_1,\ldots,i_{k-1})} = \beta^V_i + m^V_{i} s^{\scr{N}}_{i}(\alpha_{i_k})$, using \Cref{lem_rels_beta_i2}. Thus we obtain that $M$ is the intersection of all $M_i$, $i \in \mathbb{I}^k$, $k \in \ndN_0$. Since $\Supind{V} \subset M_i$ for all $i \in \mathbb{I}^k$, $k \in \ndN_0$, we can conclude that $\Supind{V} \subset M$.
\end{proof}

\begin{lem} \label{lem_iterated_refl_edges_are_irred}
	Let $k \in \ndN$, $i = (i_1,\ldots, i_k) \in \mathbb{I}^k$, assume $V$ admits the reflection sequence $i$.
	Then the $i$-th edge of $V$
	\begin{align*}
	\oplus_{m=0}^{m^V_i} V(\genind{V}+\beta^V_{i,m}) \in \ydcat{\subalgQ{\refl_{(i_1,\ldots,i_{k-1})}(\scr{N})_{i_k}}}{\scr{C}}
	\end{align*}
	is irreducible.
\end{lem}
\begin{proof}
	This is directly implied by \Cref{cor_generator_index_explicit} and \Cref{lem_gencompV_irred_iff_Ni_V_irred}.
\end{proof}

\begin{prop} \label{prop_real_subobject_lies_in_interior}
	Assume $V$ admits all reflections. 
	Let $U \subsetneq V$ be a graded subobject in $\ydcat{Q}{\scr{C}}$. Then $U \subset \inter{V}$.
\end{prop}
\begin{proof}
	Assume $U \cap \bound{V} \ne 0$. We show that then $U=V$. 
	Let $k \in \ndN$, $i=(i_1, \ldots, i_k) \in \mathbb{I}^k$, $0 \le m \le m^V_i$, such that $U\cap V(\genind{V}+\beta^V_{i,m}) \ne 0$ and denote $i' = (i_1,\ldots,i_{k-1})$. Observe that $\refl_{i'}(U) \subset \refl_{i'}(V)$ is a subobject in $\ydcat{\refl_{i'}(Q)}{\scr{C}}_{\rat}$.	
	Moreover $\refl_{i'}(U) \in \ydcat{\subalgQ{\refl_{i'}(\scr{N})_{i_k}}}{\scr{C}}$ with induced $\subalgQ{\refl_{i'}(\scr{N})_{i_k}}$-action and projected $\subalgQ{\refl_{i'}(\scr{N})_{i_k}}$-coaction.
	Now 
	$
	U \cap \oplus_{m=0}^{m^V_i} V(\genind{V}+\beta^V_{i,m}) \ne 0 \matcom
	$
	i.e.
	\begin{align*}
	    0 \ne \refl_{i'}(U) \cap \oplus_{m=0}^{m^V_i} V(\genind{V}+\beta^V_{i,m}) \in \ydcat{\subalgQ{\refl_{i'}(\scr{N})_{i_k}}}{\scr{C}}
	    \matcom
	\end{align*}
	hence by \Cref{lem_iterated_refl_edges_are_irred} $\oplus_{m=0}^{m^V_i} V(\genind{V}+\beta^V_{i,m}) \subset \refl_{i'}(U)$. Now with \Cref{cor_generator_index_explicit} this implies $\gencomp{\refl_{i'}(V)} = V(\genind{V}+\beta^V_{i,0}) \subset  \refl_{i'}(U)$. Now since $\gencomp{\refl_{i'}(V)}$ generates $\refl_{i'}(V)$ as a $\refl_i(Q)$-module, we can conclude that $\refl_{i'}(U) = \refl_{i'} (V)$, i.e. $U=V$.
\end{proof}

\begin{cor} \label{cor_real_subobject_lies_in_interior}
	Assume $V$ admits all reflections. There exists a unique graded subobject $U \subset \inter{V}$ of $V$ in $\ydcat{Q}{\scr{C}}$, such that every proper graded subobject of $V$ in $\ydcat{Q}{\scr{C}}$ is contained in $U$.
	Moreover $V/U$ is irreducible in $\ydcat{Q}{\scr{C}}$.
\end{cor}
\begin{proof}
	Let $U$ be the sum of all proper graded subobjects of $V$ in $\ydcat{Q}{\scr{C}}$. By \Cref{prop_real_subobject_lies_in_interior} we obtain $U \subset \inter{V}$. Clearly $V/U$ is irreducible in the category of $\Gamma$-graded objects in $\ydcat{Q}{\scr{C}}$. Then $V/U$ is also irreducible in $\ydcat{Q}{\scr{C}}$ by \Cref{cor_irred_graded_ydmodule_o}(2)$\implies$(1).
\end{proof}

\begin{rema}
    We will discuss more about the maximal subobject of objects in $\ydcat{Q}{\scr{C}}$ in \cref{sect_ydmod_nich_sys_shapovalov}.
\end{rema}

\begin{prop}
Assume $V$ admits all reflections. Moreover let $U,W \in \ydcat{Q}{\scr{C}}_{\rat}$ be $\Gamma$-graded objects such that
\begin{align*}
0 \rightarrow U \rightarrow V \rightarrow W \rightarrow 0
\end{align*}
is an exact sequence in the category of $\Gamma$-graded objects in $\ydcat{Q}{\scr{C}}_{\rat}$ and $W \ne 0$.
Then $\Supind{U} \subset \intind{V}$, $\bouind{V} \subset \Supind{W}$ and $\Supind{W} = \Supind{V} \setminus M$, where 
\begin{align*}M = \lbrace n \in \Supind{U} \, | \, U(n) \rightarrow V(n) \text{ is bijective} \rbrace. \end{align*}
\end{prop}

\begin{proof}
First since $W \ne 0$, $U$ can be viewed as a proper graded subobject of $V$ and thus by \Cref{prop_real_subobject_lies_in_interior} we obtain $\Supind{U} \subset \intind{V}$. As the sequence is exact, we know that $\Supind{W} = \Supind{V} \setminus M$. Since $\Supind{U} \cap \bouind{V} = 0$ we can conclude $\bouind{V} \subset \Supind{W}$.
\end{proof}
\section{The induced \mathintitleydmod{}-module} \label{sect_ydmod_nich_sys_induced_yd_module}

We now construct and study a special set of $\ydmod{}$-modules over Nichols systems that are obtained by comodules of the Nichols systems. 
These induced objects can also be found in a different context in \cite{HeSch}, Proposition~4.5.1(1).
We will see in \Cref{prop_action_induces_epimorphism_from_induced_module}, that every $\ydmod{}$-modules over a Nichols system can be realized as a quotient of such an induced $\ydmod{}$-module. Finally in \Cref{cor_ydind_irred_all_refl} we will study how reflections can be used to characterize irreducibility of induced $\ydmod{}$-modules, if the Nichols system is finite-dimensional. We will study these induced objects in more detail in the case where the Nichols system is of diagonal type in \cref{sect_nich_sys_diag_main}.

Let $\scr{N}$ be a pre-Nichols system and let $Q := \nsalg{\scr{N}}$ and $N_j := \scr{N}_j$ for all $j \in \mathbb{I}$. 
Let $\Gamma$ be an abelian group, such that $\ndZ^\theta \subset \Gamma$.

Recall from \Cref{rema_prop_ad_coad_functor} that we have a functor
\begin{align*}
    \ydindfun{Q} : \comodcat{Q}{\scr{C}} \rightarrow \ydcat{Q}{\scr{C}}, 
\end{align*}
where a $Q$-comodule $U \in \scr{C}$ gets mapped to $\ydind{Q}{U}=Q \ot U \in \ydcat{Q}{\scr{C}}$ with $Q$-action $\bimu_Q \ot \id_U$ and $Q$-coaction
\begin{align*}
	\coad_{Q\ot U} = & (\bimu_Q \ot \id_{Q \ot U})  (\bimu_Q \ot \brd_{Q \ot U,Q}^{\scr{C}}) (\id_Q \ot \brd^{\scr{C}}_{Q,Q} \ot \id_U \ot \antip_Q) \\& 
	(\bicomu_Q \ot \brd_{Q,Q\ot U}^{\scr{C}}) (\bicomu_Q \ot \coact^Q_U)  \matcom
\end{align*}
and where a $Q$-comodule morphism $f$ gets mapped to $\ydind{Q}{f}=\id_Q \ot f$.

\begin{rema}
    The above construction of an induced $\ydmod{}$-module is very reminiscent of Verma modules in the representation theory of Lie algebras.
\end{rema}

\begin{rema} \label{rema_induced_yd_structure_boso}
	Let $U \in \scr{C}$ and denote $\widetilde{Q} := Q \boso H$.
	In \cite{HeSch}, Proposition~4.5.1(1) it is shown that $\widetilde{Q} \ot_H U \in \ydmod{\widetilde{Q}}$, where the $\widetilde{Q}$-action is $\act_{\widetilde{Q} \ot_H U}^{\widetilde{Q}} := \bimu_{\widetilde{Q}} \ot \id_U$ and the $\widetilde{Q}$-coaction $\coact_{\widetilde{Q} \ot_H U}^{\widetilde{Q}}$ is given by 
	\begin{align*}
		\coact_{\widetilde{Q} \ot_H U}^{\widetilde{Q}}((q \boso h) \ot u) = (q \boso h)_{(1)} (1 \boso u_{(-1)}) \antip_{\widetilde{Q}}((q \boso h)_{(3)}) \ot (q \boso h)_{(2)} \ot u_{(0)}
	\end{align*}
	for all $q \in Q, h \in H, u \in U$. With \Cref{prop_boso_is_ydcat} we obtain that $\widetilde{Q} \ot_H U \in \ydcat{Q}{\scr{C}}$ with $Q$-action $\act_{\widetilde{Q} \ot_H U}^Q = \act_{\widetilde{Q} \ot_H U}^{\widetilde{Q}} (\id_Q \ot \biunit_H \ot \id_{\widetilde{Q} \ot U})$ and $Q$-coaction $\coact_{\widetilde{Q} \ot_H U}^Q = (\id_Q \ot \bicounit_H \ot \id_{\widetilde{Q} \ot U}) \coact_{\widetilde{Q} \ot_H U}^{\widetilde{Q}}$.
	Below we show that $\ydind{Q}{U}$ and $\widetilde{Q} \ot_H U$ are isomorphic in $\ydcat{Q}{\scr{C}}$, where $U$ has trivial left $Q$-coaction $\biunit_Q \ot \id_U$.
\end{rema}

\begin{prop}
	Let $U \in \scr{C}$ and denote $\widetilde{Q} := Q \boso H$. We view $U$ as a left $Q$-comodule via coaction $\biunit_Q \ot \id_U$. Then 
	\begin{align*}
	\phi : \widetilde{Q} \ot_H U \rightarrow \ydind{Q}{U}, \, (q \boso h) \ot u \mapsto q \ot h \cdot u \matcom && q \in Q, h \in H, u \in U
	\end{align*}
	is a well defined isomorphism in the category $\ydcat{Q}{\scr{C}}$. The inverse is given by $\phi^{-1}(q \ot u)= ( q \boso 1) \ot u$ for all $q\in Q, u\in U$.
\end{prop}
\begin{proof}
It is straight forward to verify that $\phi$ is a well defined isomorphism in $\scr{C}$ with claimed inverse. Now since
\begin{align*}
	&\phi (\bimu_{\widetilde{Q}} \ot \id_U) ((q\boso 1) \ot (q' \boso h) \ot u)
	= \phi ((q q' \boso h) \ot u) 
	\\&= qq' \ot h \cdot u 
	 = (\bimu_Q \ot \id_U) (\id_Q \ot \phi) (q \ot (q' \boso h) \ot u)
\end{align*}
for all $ q,q' \in Q$, $h \in H$, $u \in U $, we obtain that $\phi$ is a $Q$-module morphism. Observe that
\begin{align*}
(\id_Q \ot \bicounit_H)\antip_{\widetilde{Q}}(q \boso 1) = \antip_H ( q_{(-1)} ) \cdot \antip_R (q_{(0)}) \matdot
\end{align*}
Hence for $q \in Q$, $u \in U$ we have
\begin{align*}
	&\coact_{\widetilde{Q} \ot_H U}^Q \phi^{-1} (q \ot u) = \coact_{\widetilde{Q} \ot_H U}^Q (q \boso 1 \ot u)
	\\=& q_{(1)} \left( \left( q_{(2)(-1)} q_{(3)(-3)} u_{(-1)} \antip_H(q_{(3)(-1)}) \right) \cdot \antip_Q(q_{(3)(0)}) \right)  \ot (q_{(2)(0)} \boso q_{(3)(-2)} ) \ot u_{(0)}
	\\=& q_{(1)} \left( \left( q_{(2)(-1)} q_{(3)(-3)} u_{(-1)} \antip_H(q_{(3)(-1)}) \right) \cdot \antip_Q(q_{(3)(0)}) \right) \\& \ot (q_{(2)(0)} \boso 1 ) \ot q_{(3)(-2)} \cdot u_{(0)}
	\\=& q_{(1)} \left( \left( q_{(2)(-1)} (q_{(3)(-1)} \cdot u)_{(-1)} \right) \cdot \antip_Q(q_{(3)(0)} \right) \ot (q_{(2)(0)} \boso 1) \ot (q_{(3)(-1)} \cdot u)_{(0)}
	\\=& (\id_Q \ot \phi^{-1}) \coact_{\ydind{Q}{U}}^Q (q \ot u) \matdot
\end{align*}
This proves that $\phi^{-1}$, and thus $\phi$, is a $Q$-comodule morphism.
\end{proof}

\begin{prop}
	Let $U \in \comodcat{Q}{\scr{C}}$ and $\gamma \in \Gamma$. $\ydind{Q}{U}$ is homogeneously generated $\Gamma$-graded via $\ydind{Q}{U}(\gamma + n) = Q(n) \ot U$ for $n \in \ndN_0^\theta$ and $V(\gamma')=0$ for all $\gamma' \in \Gamma \setminus \gamma + \ndN_0$. Moreover $\genind{\ydind{Q}{U}} = \gamma$ and $\gencomp{\ydind{Q}{U}} = U$.
\end{prop}
\begin{proof}
	Clearly $\ydind{Q}{U}$ is $\Gamma$-graded with the claimed grading. For $n \in \ndN_0^\theta$ we have $\ydind{Q}{U}(\gamma + n) = Q(n) \ot U = Q(n) (\fK \ot U) = Q(n) \ydind{Q}{U}(\gamma)$.
\end{proof}

\begin{lem} \label{lem_action_is_yd_morphism}
    Let $V \in \ydcat{Q}{\scr{C}}$ and let $U$ be a $Q$-subcomodule of $V$. Then the restriction of $\act_V^Q$ to $\ydind{Q}{U} \subset Q \ot V$, i.e.
	\begin{align*}
	\pi : \ydind{Q}{U} \rightarrow V, \, x \mapsto \act_V^Q(x) \matcom
	\end{align*}
	is a morphism in $\ydcat{Q}{\scr{C}}$.
\end{lem}
\begin{proof}
It is elementary to verify that $\pi$ is a $Q$-module morphism in $\scr{C}$. 
For $q \in Q, v \in U$ we have by \Cref{lem_defining_relation_ydmodule_equivalence} 
\begin{align*}
	\coact_V^Q \pi (q \ot v) =&(\bimu_Q \ot \act_V^Q)(\bimu_Q \ot \brd_{Q\ot V, Q}^{\scr{C}}) (\id_Q \ot \brd_{Q,Q}^{\scr{C}} \ot \id_V \ot \antip_Q) \\& (\bicomu_Q \ot \brd_{Q, Q\ot V}^{\scr{C}})(\bicomu_Q \ot \coact_V^Q) (q \ot v)
	\\=&(\id_Q \ot \act_V^Q) \coact_{\ydind{Q}{U}}^Q (q \ot v) = (\id_Q \ot \pi ) \coact_{\ydind{Q}{U}}^Q  (q \ot v) \matdot
\end{align*}
Hence $p$ is a $Q$-comodule morphism.
\end{proof}

\begin{rema}
    Let $V \in \ydcat{Q}{\scr{C}}$ be a homogeneously generated $\Gamma$-graded object.
    Observe that since $\coact_V^Q$ is graded we have for $v \in \gencomp{V}$
    \begin{align*}
        \coact_V^Q(v) \in Q(0) \ot V_0 = \fK \ot V_0 \matcom
    \end{align*}
    hence $\coact_V^Q(v) = 1 \ot v$. This implies that $\gencomp{V}$ is a $Q$-subcomodule of $V$.
\end{rema}

\begin{prop} \label{prop_action_induces_epimorphism_from_induced_module}
	Let $V \in \ydcat{Q}{\scr{C}}$ be a homogeneously generated $\Gamma$-graded object. Then the restriction of $\act_V^Q$ to $\ydind{Q}{\gencomp{V}} \subset Q \ot V$, i.e.
	\begin{align*}
	\pi : \ydind{Q}{\gencomp{V}} \rightarrow V, \, q \ot v \mapsto \act_V^Q(q \ot v) \matcom
	\end{align*}
	is a graded epimorphism in $\ydcat{Q}{\scr{C}}$. In particular $V \cong \ydind{Q}{\gencomp{V}} / \ker \pi $.
\end{prop}
\begin{proof}
\Cref{lem_action_is_yd_morphism} implies that $\pi$ is a morphism in $\ydcat{Q}{\scr{C}}$. By definition of $V$, $\pi$ is a graded epimorphism.
\end{proof}

\begin{rema}
    Let $V \in \ydcat{Q}{\scr{C}}$ be a homogeneously generated $\Gamma$-graded object. If $\ydind{Q}{\gencomp{V}}$ is irreducible in $\ydcat{Q}{\scr{C}}$, then $\pi: \ydind{Q}{\gencomp{V}} \rightarrow V$ from \Cref{prop_action_induces_epimorphism_from_induced_module} is an isomorphism.
\end{rema}

\begin{lem} \label{lem_ydind_mi}
    Let $k \in \ndN$, $i = (i_1,\ldots,i_k) \in \mathbb{I}^k$ and assume $\scr{N}$ admits the reflection sequence $i$. Let $U \in \comodcat{Q}{\scr{C}}$ be irreducible, assume $\ydind{Q}{U} \in \ydcat{Q}{\scr{C}}_{\rat}$ and that $\ydind{Q}{U}$ admits the reflection sequence $(i_1,\ldots,i_{k-1})$. Then  $m^{\scr{N}}_{(i_1,\ldots,i_r)} \in \ndN_0$ for all $1 \le r \le k$ and
    \begin{align*}
        m^{\ydind{Q}{U}}_i &= m^{\scr{N}}_i
        \matdot
    \end{align*}
    In particular $\beta^{\ydind{Q}{U}}_{i,m} = \beta^{\scr{N}}_{i,m}$ for all $0\le m \le m_i^{\scr{N}}$.
\end{lem}
\begin{proof}
    We do induction on $k$. If $k=1$, then
    \begin{align*}
        \ndN_0 \ni m_{(i_1)}^{\ydind{Q}{U}} = \max \lbrace m \in \ndN_0 \mid N_{i_1}^m \ot U \ne 0 \rbrace = m_{(i_1)}^{\scr{N}} \matdot
    \end{align*}
    Assume $k \ge 2$ and that the claim holds for $k-1$.
    As seen in the proof of \Cref{cor_generator_index_explicit} we have for $m \in \ndN_0$
    \begin{align*}
        &\refl_{(i_1,\ldots,i_{k-1})}(\scr{N})_{i_k}^m \cdot \gencomp{\refl_{(i_1,\ldots,i_{k-1})}(\ydind{Q}{U})} 
        \\=& Q\left( \beta_{(i_1,\ldots,i_{k-1})}^{\ydind{Q}{U}} + m s^{\scr{N}}_{(i_1,\ldots,i_{k-1})} (\alpha_{i_k}) \right) \ot U
    \end{align*}
    By induction hypothesis we know $\beta_{(i_1,\ldots,i_{k-1})}^{\ydind{Q}{U}} = \beta_{(i_1,\ldots,i_{k-1})}^{\scr{N}}$.
    Moreover by \Cref{prop_sup_iterated_reflection} we have $\beta_{(i_1,\ldots,i_{k-1})}^{\scr{N}} + m s^{\scr{N}}_{(i_1,\ldots,i_{k-1})} (\alpha_{i_k}) \in \Sup{Q}$ if and only if $m \alpha_{i_k} \in \Sup{\refl_{(i_1,\ldots,i_{k-1})}(Q)}$, i.e. if and only if $m \le m^{\scr{N}}_i$.
    We conclude that $m^{\scr{N}}_i= m^{\ydind{Q}{U}}_i  \in \ndN_0$.
\end{proof}

\begin{rema}
    Let $p : \ndZ^\theta \rightarrow \ndZ$ be the group homomorphism given by $p(\alpha_j)=1$ for all $j \in \mathbb{I}$. $Q$ is $\ndN_0$-graded via $Q(n) := Q(p^{-1}(n))$. Let 
    \begin{align*}
        N := \max \lbrace n \in \ndN_0 \, | \, Q(n) \ne 0 \rbrace = \max \lbrace p(\gamma) \, | \, \gamma \in \Sup{Q} \rbrace \matdot
    \end{align*} 
    If $Q$ is finite-dimensional, then by \Cref{cor_fin_dim_hopf_alg_is_frobenius} we know that $Q(N)$ is one dimensional. In particular, there exists a an index we denote $N_Q^{\max} \in \ndN_0^\theta$, such that 
    \begin{align*}
        \lbrace \gamma \in \Sup{Q} \, |\, p(\gamma) \ge N \rbrace = \lbrace  N_Q^{\max} \rbrace \matdot
    \end{align*}
    and $Q(N_Q^{\max}) = Q(N)$ is one-dimensional.
\end{rema}

Let $0 \ne \integralelt \in Q(N_Q^{\max})$, i.e. $Q(N_Q^{\max}) = \fK \integralelt$.

\begin{lem} \label{lem_ydind_frobenius_elt}
    If $Q$ is finite-dimensional, then the following hold:
    \begin{enumerate}
        \item For all $0 \ne x \in Q$ there exists $y \in Q$, such that $yx = \integralelt$.
        \item $Q(N_Q^{\max})$ is contained in every left ideal of $Q$.
    \end{enumerate}
\end{lem}
\begin{proof}
    Implied by \Cref{cor_fin_dim_hopf_alg_is_frobenius}.
\end{proof}

\begin{prop} \label{prop_ydind_integral_comp_is_edge}
    Let $U \in \scr{C}$ be an irreducible object. Assume that $Q$ is finite-dimensional and that $\ydind{Q}{U} \in \ydcat{Q}{\scr{C}}_{\rat}$ admits all reflections. Then there is some $k \in \ndN_0$, $i \in \mathbb{I}^k$, such that $N_Q^{\max} = \beta^{\ydind{Q}{U}}_i$ and
    \begin{align*}
    \gencomp{\refl_i (\ydind{Q}{U})} = \fK \integralelt \ot U \matdot
    \end{align*}
    In particular $\fK \integralelt \ot U \subset \bound{\ydind{Q}{U}}$.
\end{prop}
\begin{proof}
    Let $M$ be the convex hull spanned by $\Supind{\ydind{Q}{U}}$. Since $Q$ is finite dimensional, $\Supind{\ydind{Q}{U}}$ is a finite set. By construction, $N_Q^{\max}$ lies on a vertex of $M$. By \Cref{prop_support_is_spanned_by_beta_i} $M$ coincides with the convex hull spanned by the vectors $\beta_i^{\ydind{Q}{U}}$, $k \in \ndN_0$, $i \in \mathbb{I}^k$. Hence there exists some $k \in \ndN_0$, $i \in \mathbb{I}^k$, such that $N_Q^{\max} = \beta_i^{\ydind{Q}{U}}$. Then by \Cref{prop_generator_index_explicit} we obtain
    \begin{align*}
        \gencomp{\refl_i (\ydind{Q}{U})} = \ydind{Q}{U}(\genind{\ydind{Q}{U}} + N_Q^{\max}) = Q(N_Q^{\max}) \ot U = \fK \integralelt \ot U 
    \end{align*}
    By definition this component lies on an edge of $\ydind{Q}{U}$.
\end{proof}

\begin{lem} \label{lem_ydind_highest_graded_vector_lies_in_every_submodule}
    Let $U \in \comodcat{Q}{\scr{C}}$ and assume $Q$ is finite-dimensional. Let $V$ be a non-zero sub $Q$-module of $\ydind{Q}{U}$. Then there exists $0 \ne u \in U$, such that $\integralelt \ot u \in V$. In particular $V \cap (\fK \integralelt \ot U) \ne 0$.
\end{lem}
\begin{proof}
    Let $0 \ne x \in V$ and for all $n \in \ndN_0^\theta$ let $x_n \in Q(n) \ot U$, such that $x=\sum_{n \in \ndN_0^\theta} x_n$. Let $p : \ndZ^\theta \rightarrow \ndZ$ be the group homomorphism given by $p(\alpha_j)=1$ for all $j \in \mathbb{I}$. Let us fix one $m \in \ndN_0^\theta$, such that $x_m \ne 0$ and
    \begin{align*}
        p(m) = \min \lbrace p(n) \,|\, n \in \ndN_0, x_n \ne 0 \rbrace \matdot
    \end{align*}
    Let $k \in \ndN$, $u_1, \ldots, u_k \in U$ be linear independent elements and $q_1,\ldots q_k \in Q(m) \setminus \lbrace 0 \rbrace$, such that $x_m = \sum_{j=1}^k q_j \ot u_j$.
    By \Cref{lem_ydind_frobenius_elt} there exists some $y \in Q$, such that $y q_1 = \integralelt$. Since $Q$ is graded, we can achieve that $y \in Q(N_Q^{\max} - m)$. Thus $y q_j \in Q(N_Q^{\max}) =\fK \integralelt$ for all $2 \le j \le k$ and hence there exist $c_2, \ldots, c_k \in \fK$, such that $yq_j = c_j \integralelt$ for all $2 \le j \le k$. In particular $u := u_1 + \sum_{j=2}^k c_j u_j \ne 0$ and 
    \begin{align*}
        y x_m =  \sum_{j=1}^k yq_j \ot u_j = \integralelt \ot u_1 + \sum_{j=2}^k c_j \integralelt \ot u_j = \integralelt \ot u \matdot
    \end{align*}
    Recall that by definition of $N_Q^{\max}$ we have 
    \begin{align*}
        \lbrace n' \in \Sup{Q} \, | \, p(n') \ge p(N_Q^{\max}) \rbrace = \lbrace N_Q^{\max} \rbrace \matdot
    \end{align*}
    Now let $m\ne n \in \ndN_0^\theta$, such that $x_n \ne 0$. By definition of $m$ we have $p(m) \le p(n)$, hence
    \begin{align*}
        p(N_Q^{\max}-m + n) = p(N_Q^{\max}) -p(m) + p(n) \ge p(N_Q^{\max})
    \end{align*}
    and thus $N_Q^{\max}-m + n \notin \Sup{Q}$, since $n \ne m$. Then we have $y x_n \in Q(N_Q^{\max}-m + n) \ot U = 0$, i.e. $y x_n = 0$. In total this gives $V \ni  y x = y x_m = \integralelt \ot u$.
\end{proof}

\begin{cor} \label{cor_ydind_irred_all_refl}
	Assume that $Q$ is finite dimensional and that $\scr{N}$ admits all reflections. Let $U \in \comodcat{Q}{\scr{C}}$ be an irreducible object in $\scr{C}$. Then the following are equivalent:
	\begin{enumerate}
		\item $\ydind{Q}{U}$ admits all reflections.
		\item $\ydind{Q}{U}$ is irreducible in $\ydcat{Q}{\scr{C}}$.
	\end{enumerate}
\end{cor}
\begin{proof}
	(2)$\implies$(1) is given \Cref{prop_irred_admits_all_reflections}. Assume $\ydind{Q}{U}$ admits all reflections. Let $V \in \ydcat{Q}{\scr{C}}$ be a graded non-zero subobject of $\ydind{Q}{U}$. By \Cref{lem_ydind_highest_graded_vector_lies_in_every_submodule} we know that $V \cap (\fK \integralelt \ot U) \ne 0$ and by \Cref{prop_ydind_integral_comp_is_edge} we have $\fK \integralelt \ot U \subset \bound{\ydind{Q}{U}}$, thus $V \cap \bound{\ydind{Q}{U}} \ne 0$. 
	Since by \Cref{prop_real_subobject_lies_in_interior} we know that proper graded subobjects of $\ydind{Q}{U}$ are contained in $\inter{\ydind{Q}{U}}$, we obtain that $V=\ydind{Q}{U}$. Hence $\ydind{Q}{U}$ has no proper graded non-zero subobjects. Then \Cref{cor_real_subobject_lies_in_interior} gives that $\ydind{Q}{U}$ is irreducible in $\ydcat{Q}{\scr{C}}$.
\end{proof}

\section{Shapovalov morphism and maximal subobject} \label{sect_ydmod_nich_sys_shapovalov}

The goal of this section is to study the maximal subobject of a $\ydmod{}$-module over a Nichols system. We have already seen in \Cref{cor_real_subobject_lies_in_interior}, that if the Nichols system admits all reflections, this subobject lies in the interior of the $\ydmod{}$-module. As it turns out in general the subobject coincides with the kernel of a special morphism we call Shapovalov morphism (\Cref{prop_shapoendo_max_subobject}). Studying this morphism we find that it behaves well under reflections (\Cref{prop_kernel_shapo_of_reflection}) and we can also use it to determine which reflection sequences a $\ydmod{}$-module admits (\Cref{cor_reflection_sequence_shapo_equiv} and \Cref{cor_reflection_sequence_shapo_equiv_bound}). In \cref{chapt_nich_sys_group} and \cref{chapt_nich_sys_diag} we will calculate the kernel of the Shapovalov morphism in specific application areas.

Let $\scr{N}$ be a pre-Nichols system and let $Q := \nsalg{\scr{N}}$ and $N_j := \scr{N}_j$ for all $j \in \mathbb{I}$. 
Let $\Gamma$ be an abelian group, such that $\ndZ^\theta \subset \Gamma$.
Let $V \in \ydcat{Q}{\scr{C}}$ be a homogeneously generated $\Gamma$-graded object and let $\proj_0 : V \rightarrow \gencomp{V}$ be the canonical projection.

We view $Q \ot \gencomp{V}$ as a $\Gamma$-graded object with $(\genind{V}+n)$-th component $Q(n) \ot \gencomp{V}$ for all $n \in \ndN_0^\theta$.
\begin{defi} \label{defi_shapovalov_morphism}
    Let $\shapoendo_V$ be the graded morphism in the category $\scr{C}$ that is defined as follows: 
    \begin{align*}
        \shapoendo_V & :  V \xrightarrow{\coact_V^Q} Q \ot V \xrightarrow{\id_Q \ot \proj_0} Q \ot \gencomp{V} \matdot 
    \end{align*}
    $\shapoendo_V$ is called the \textbf{Shapovalov morphism of $V$}.
\end{defi}

\begin{rema}
    The name was chosen because in \cref{sect_nich_sys_diag_main} we will obtain a polynomial that is known as the \textit{Shapovalov determinant}, by calculating the Kernel of $\shapoendo_V$ in the case where $V$ is given as an induced module and the Nichols system is of diagonal type.
\end{rema}

\begin{rema}
    In general $\shapoendo_V : V \rightarrow \ydind{Q}{\gencomp{V}}$ is not a morphism in the category $\ydcat{Q}{\scr{C}}$: 
    Assume $Q$ is a Nichols algebra of diagonal type with generators $x_1,\ldots,x_\theta$ and let $N_j = \fK x_j$ for all $j \in \mathbb{I}$. 
    Let $V \in \ydcat{Q}{\scr{C}}$ be homogeneously generated $\Gamma$-graded, such that $\gencomp{V}$ is one-dimensional and let $0\ne v \in \gencomp{V}$. Finally for $j \in \mathbb{I}$ let $r_j \in \fK$, such that
    \begin{align*}
        \brd_{\gencomp{V},N_j}^{\scr{C}} \brd_{N_j,\gencomp{V}}^{\scr{C}} ( x_j \ot v) = r_j x_j \ot v \matdot
    \end{align*}
    Let $j \in \mathbb{I}$, such that $r_j \ne 1$.
    Then since $\antip_Q(x_j) = -x_j$, $\coact_V^Q (v) = 1 \ot v$, $(\id_Q \ot \bicomu_Q) \bicomu_Q (x_j) = 1\ot 1 \ot x_j + 1 \ot x_j \ot 1 + x_j \ot 1 \ot 1 $ and by \Cref{lem_defining_relation_ydmodule_equivalence} we have
    \begin{align*}
       & \shapoendo_V \act_V^Q (x_j \ot v) = (\id_Q \ot \proj_0)\coact_V^Q(x_j \cdot v) 
        \\ &=  (\id_Q \ot \proj_0) \left( 1 \ot x_j \cdot v + (1-r_j) x_j \ot v \right)
        = (1-r_j) x_j \ot v
    \end{align*}
    Conversely $(\bimu_Q \ot \id_{\gencomp{V}})(\id_Q \ot \shapoendo_V) (x_j \ot v) = x_j \ot v$. Hence $\shapoendo_V$ is no $Q$-module morphism.
    Similarly
    \begin{align*}
        (\id_Q \ot \shapoendo_V) \coact_V^Q (x_j \cdot v) = (1-r) ( 1 \ot x_j \ot v + x_j \ot 1 \ot v)
    \end{align*}
    and 
    \begin{align*}
        \coad_{Q\ot \gencomp{V}} \shapoendo_V (x_j \cdot v) = (1-r) \left( 1 \ot x_j \ot v + (1-r) x_j \ot 1 \ot v\right) \matcom
    \end{align*}
    hence in general $\shapoendo_V$ is no $Q$-comodule morphism.
\end{rema}

\begin{lem} \label{lem_proj0_act_is_bicounit}
    We have
    \begin{align*}
        \proj_0 \act^Q_V = \bicounit_Q \ot \proj_0 \matdot
    \end{align*}
\end{lem}
\begin{proof}
    Since $\act^Q_V$ is graded we have for all $n,n' \in \ndN_0$, such that $n\ne 0$ or $n' \ne 0$, that
    \begin{align*}
        \proj_0 \act^Q_V (Q(n) \ot V(\genind{V}+n')) \subset \proj_0 ( V(\genind{V}+n+n') ) = 0 \matdot 
    \end{align*}
    Similarly $(\bicounit_Q \ot \proj_0)(Q(n) \ot V(\genind{V}+n')) = 0$. Finally for $\lambda \in \fK=Q(0)$, $v \in \gencomp{V}$ we have
    \begin{align*}
        \proj_0 \act^Q_V (\lambda \ot v) = \lambda v = (\bicounit \ot \proj_0) (\lambda \ot v) \matdot
    \end{align*}
    Hence the two morphisms coincide on all components of $Q \ot V$.
\end{proof}

\begin{lem} \label{lem_shapovalov_endo_product}
    Let $x,y \in Q$, $v \in \gencomp{V}$. Then we have
    \begin{align*}
        \shapoendo_V \left( (xy) \cdot v \right) =& (\bimu_Q \ot \id_{\gencomp{V}})(\bimu_Q \ot \brd^{\scr{C}}_{\gencomp{V},Q}) (\id_Q \ot \brd^{\scr{C}}_{Q, Q\ot \gencomp{V}}) \\&(\id_Q \ot \antip_Q \ot \id_Q \ot \id_{\gencomp{V}}) \left(\bicomu_Q(x) \ot \shapoendo_V(y \cdot v) \right)
        \matdot
    \end{align*}
    In particular $\ker \shapoendo_V$ is a sub $Q$-module of $V$.
\end{lem}
\begin{proof}
    Since $\proj_0 \act^Q_V = \bicounit_Q \ot \proj_0$ by \Cref{lem_proj0_act_is_bicounit} we obtain by \Cref{lem_defining_relation_ydmodule_equivalence} that
    \begin{align*}
        \shapoendo_V \left( x \cdot (y \cdot v) \right) =& (\id_Q \ot \proj_0) \coact_V^Q \act^Q_V (x \ot y\cdot v) 
        \\=& (\bimu_Q \ot \proj_0 \act_V^Q)(\bimu_Q \ot \brd_{Q\ot V, Q}^{\scr{C}}) (\id_Q \ot \brd_{Q,Q}^{\scr{C}} \ot \id_V \ot \antip_Q) \\&  (\bicomu_Q \ot \brd_{Q, Q\ot V}^{\scr{C}}) (\bicomu_Q(x) \ot \coact_V^Q(y\cdot v)) 
        \\=& (\bimu_Q \ot \id_{\gencomp{V}}) (\bimu_Q \ot \brd_{\gencomp{V}, Q}^{\scr{C}}) (\id_Q \ot \brd_{Q, Q\ot \gencomp{V}}^{\scr{C}})  \\&  (\id_Q \ot \antip_Q \ot \id_Q \ot \id_{\gencomp{V}})(\bicomu_Q(x) \ot (\id_Q\ot \proj_0) \coact_V^Q(y\cdot v)) \matdot
    \end{align*}
    This proves the relation.
    Finally if $y \cdot v \in \ker \shapoendo_V$, then this relation implies $x (y \cdot v) \in \ker \shapoendo_V$, hence $\ker f$ is a sub $Q$-module of $V$.
\end{proof}

\begin{prop} \label{prop_shapoendo_max_subobject}
    Assume $\gencomp{V}$ is irreducible in $\scr{C}$. Then $\ker \shapoendo_V$ is a graded subobject of $V \in \ydcat{Q}{\scr{C}}$, such that for every graded subobject $U\ne V$ of $V \in \ydcat{Q}{\scr{C}}$ we have $U \subset \ker \shapoendo_V$. Moreover $V$ is irreducible in $\ydcat{Q}{\scr{C}}$ if and only if $\ker \shapoendo_V = 0$.
\end{prop}
\begin{proof}
    Let $U \ne V$ be a graded subobject of $V \in \ydcat{Q}{\scr{C}}$. Then $U(\genind{V}) = 0$, since $\gencomp{V}=V(\genind{V})$ is irreducible, generates $V$ as a $Q$-module and $U \ne V$.
    We have
    \begin{align*}
        \shapoendo_V(U) = (\id_Q \ot \proj_0) \coact_V^Q (U) \subset Q \ot U(\genind{V}) = 0 \matcom
    \end{align*}
    hence  $U \subset \ker \shapoendo_V$.

    $\ker  \shapoendo_V$ is graded in $\scr{C}$, since $\shapoendo_V$ is graded, i.e.
    \begin{align*}
        \ker \shapoendo_V= \oplus_{n\in \ndN_0^\theta} \ker ( \shapoendo_V(\genind{V}+n) )
    \end{align*}
    and $\ker ( \shapoendo_V(\genind{V}+n) ) \subset V(\genind{V}+n)$ for all $n \in \ndN_0^\theta$.
    Now by definition we have
    \begin{align*}
        &(\id_Q \ot \shapoendo_V) \coact_V^Q (\ker \shapoendo_V)
        = (\id_Q \ot (\id_Q \ot \proj_0)\coact_V^Q) \coact_V^Q \left( \ker \shapoendo_V \right)
        \\&= (\bicomu_Q \ot \proj_0) \coact_V^Q \left( \ker \shapoendo_V \right)
        = (\bicomu_Q \ot \id_{\gencomp{V}}) \shapoendo_V \left( \ker \shapoendo_V \right) = 0
        \matdot
    \end{align*}
    Hence $\coact_V^Q \left( \ker \shapoendo_V \right) \subset Q \ot (\ker \shapoendo_V)$ and $\ker \shapoendo_V$ is a sub $Q$-comodule of $V$. 
    Now by \Cref{lem_shapovalov_endo_product} we obtain that $\ker \shapoendo_V$ is a sub $Q$-module of~$V$.
    This implies that $\ker \shapoendo_V$ is a graded subobject of $V$ in $\ydcat{Q}{\scr{C}}$.
    
    Overall we have also proven that $V / \ker \shapoendo_V$ is irreducible in the category of $\Gamma$-graded objects in $\ydcat{Q}{\scr{C}}$, hence by \Cref{cor_irred_graded_ydmodule_o} it is also irreducible in $\ydcat{Q}{\scr{C}}$. This implies the last claim.
\end{proof}

\begin{cor} \label{cor_induced_irreducible_shapoendo}
	Assume that $Q$ is finite dimensional and that $\scr{N}$ admits all reflections. Let $U \in \comodcat{Q}{\scr{C}}$ be an irreducible object in $\scr{C}$ and let $\Lambda \in Q$ be defined as in \Cref{lem_ydind_highest_graded_vector_lies_in_every_submodule}. Then the following are equivalent:
	\begin{enumerate}
		\item $\ydind{Q}{U}$ is irreducible in $\ydcat{Q}{\scr{C}}$.
		\item $\shapoendo_{\ydind{Q}{U}}( \Lambda \ot u) \ne 0$ for all $0 \ne u \in U$.
	\end{enumerate}
\end{cor}
\begin{proof}
    By \Cref{prop_shapoendo_max_subobject} (1) is equivalent to $\ker \shapoendo_{\ydind{Q}{U}} = 0$. Hence clearly (1) implies (2). Assume $V$ is a proper non-zero subobject of $\ydind{Q}{U}$. Then by \Cref{lem_ydind_highest_graded_vector_lies_in_every_submodule} there is $0\ne u \in U$, such that $\Lambda \ot u \in V$. Then $\Lambda \ot u \in \ker \shapoendo_{\ydind{Q}{U}}$ by \Cref{prop_shapoendo_max_subobject}. Hence (2) does not hold.
\end{proof}

\begin{cor} \label{cor_i_well_graded_shapo_kernel}
    Assume $\gencomp{V}$ is irreducible in $\scr{C}$. Let $i \in \mathbb{I}$. The following are equivalent.
    \begin{enumerate}
        \item $V$ is a $i$-well graded.
        \item $\ker \shapoendo_V (\genind{V}+ \ndN \alpha_i) = 0$.
    \end{enumerate}
\end{cor}
\begin{proof}
    Recall that by \Cref{lem_Ni_generated_submodule_is_yd} we know $\subalgQ{N_i} \cdot \gencomp{V} \in \ydcat{\subalgQ{N_i}}{\scr{C}}$ and $\shapoendo_V(\genind{V} + n \alpha_i) = \shapoendo_{\subalgQ{N_i} \cdot \gencomp{V}} (n) : N_i^n \cdot \gencomp{V} \rightarrow N_i^n \ot \gencomp{V}$ for all $n \in \ndN_0$.

    By \Cref{cor_irred_graded_ydmodule_o} $V$ is $i$-well graded if and only if $\subalgQ{N_i} \cdot \gencomp{V}$ is irreducible in $\ydcat{\subalgQ{N_i}}{\scr{C}}$. By \Cref{prop_shapoendo_max_subobject} this is equivalent to
    \begin{align*}
        \ker \shapoendo_{\subalgQ{N_i} \cdot \gencomp{V}} = 0 \matdot
    \end{align*}
    By the argument above, this is equivalent to (2).
\end{proof}

Assume $V\in \ydcat{Q}{\scr{C}}_{\rat}$ and that $\gencomp{V}$ and $N_j$ are irreducible in~$\scr{C}$ for all $j\in \mathbb{I}$.

\begin{prop} \label{prop_kernel_shapo_of_reflection}
    Let $i \in \mathbb{I}$, assume $\scr{N}$ is an $i$-finite Nichols system over $i$ and that $V$ is a $i$-well graded. Then as objects in $\ydcat{\refl_i(Q)}{\scr{C}}_{\rat}$ we have
    \begin{align*}
        \ker \shapoendo_{\refl_i(V)} = \refl_i( \ker \shapoendo_V) \matdot
    \end{align*}
    In particular $\ker \shapoendo_{\refl_i(V)} =  \ker \shapoendo_V $ as objects in $\scr{C}$
\end{prop}
\begin{proof}
    $\refl_i( \ker \shapoendo_V)$ is graded subobject of $\refl_i(V)$ in $\ydcat{\refl_i(Q)}{\scr{C}}_{\rat}$. Then by \Cref{prop_shapoendo_max_subobject} we have $ \refl_i( \ker \shapoendo_V) \subset \ker \shapoendo_{\refl_i(V)}$. Hence it is enough to show, that $\ker \shapoendo_{\refl_i(V)} =  \ker \shapoendo_V $ as objects in $\scr{C}$. Now by \Cref{lem_funcd_keeps_irreducibility} we obtain that $\ker \shapoendo_{\refl_i(V)}$ is a subobject of $V$ in $\ydcat{Q}{\scr{C}}_{\rat}$. Hence by \Cref{prop_shapoendo_max_subobject} we have $\ker \shapoendo_{\refl_i(V)} \subset \ker \shapoendo_V$, which together with the above implies equality.
\end{proof}

\begin{cor} \label{cor_reflection_sequence_shapo_equiv}
    Let $k \in \ndN$, $i=(i_1,\ldots,i_k)\in \mathbb{I}^k$. Assume $\scr{N}$ admits the reflection sequence $i$. The following are equivalent.
    \begin{enumerate}
        \item $V$ admits the reflection sequence $i$.
        \item $V$ admits the reflection sequence $(i_1,\ldots,i_{k-1})$ and for $0 \le m \le m_i^V$ we have
        \begin{align*}
            \ker \shapoendo_V (\genind{V} + \beta^V_{i,m}) = 0 \matdot
        \end{align*}
    \end{enumerate}
\end{cor}
\begin{proof}
    By \Cref{cor_i_well_graded_shapo_kernel} we know that (1) holds if and only if $V$ admits the reflection sequence $(i_1,\ldots,i_{k-1})$ and
    \begin{align*}
        \ker \shapoendo_{\refl_{(i_1,\ldots,i_{k-1})}(V)} (\genind{\refl_{(i_1,\ldots,i_{k-1})}(V)}+ \ndN \alpha_{i_k}) = 0 \matdot
    \end{align*}
    By \Cref{prop_kernel_shapo_of_reflection} the latter is equivalent to    
    \begin{align*}
       \refl_{(i_1,\ldots,i_{k-1})}( \ker \shapoendo_{V} ) (\genind{\refl_{(i_1,\ldots,i_{k-1})}(V)}+ \ndN \alpha_{i_k}) = 0 \matdot
    \end{align*}
    Finally by \Cref{cor_generator_index_explicit} this is equivalent to
    \begin{align*}
         \ker \shapoendo_V (\genind{V} + \beta^V_{i,m}) = 0
    \end{align*}
    for all $0 \le m \le m_i^V$.
\end{proof}

\begin{cor} \label{cor_reflection_sequence_shapo_equiv_bound}
    Assume $\scr{N}$ admits all reflections. Let $M$ be the convex hull of $\Supind{V}$ and let $\bound{M}$ be the set of points in $\ndN_0^\theta$, that lie on an edge of $M$. The following are equivalent.
    \begin{enumerate}
        \item $V$ admits all reflections.
        \item $\ker \shapoendo_V (\genind{V}+ \bound{M}) = 0$.
    \end{enumerate}
\end{cor}
\begin{proof}
    If (1) holds, then $\bound{M} = \bouind{V}$ by \Cref{prop_support_is_spanned_by_beta_i}. Hence (2) holds by \Cref{cor_reflection_sequence_shapo_equiv}. Now assume (2) holds.  Let $k \in \ndN$, $i=(i_1,\ldots,i_k)\in \mathbb{I}^k$ and assume $V$ admits the reflection sequence $(i_1,\ldots,i_{k-1})$. Let $0 \le m \le m_i^V$. Then indeed $\beta_{i,m}^V \in \bound{M}$ by \Cref{cor_generator_index_explicit}. Hence (2) implies that $\ker \shapoendo_V (\genind{V} + \beta^V_{i,m}) = 0$, which in combination with \Cref{cor_reflection_sequence_shapo_equiv} implies (1).
\end{proof}

\begin{rema}
    Let $k \in \ndN_0$.
    Intuitively speaking, \Cref{cor_reflection_sequence_shapo_equiv} says that $V$ admitting a reflection sequence $i \in \mathbb{I}^k$ in addition to admitting the prior reflection sequence in $\mathbb{I}^{k-1}$, is given if and only if $\shapoendo_V$ does not vanish on the corresponding edge of $V$. In addition \Cref{cor_reflection_sequence_shapo_equiv_bound} says
    $V$ admits all reflections, if and only if $\shapoendo_V$ does not vanish on any edge of V. Compare this also to \Cref{prop_shapoendo_max_subobject}, where it is stated that $V$ is irreducible in $\ydcat{Q}{\scr{C}}$ if and only if $\shapoendo_V$ does not vanish anywhere.
    
    Then \Cref{cor_induced_irreducible_shapoendo} states that for induced $\ydmod{}$-modules $\ydind{Q}{U}$, $U \in \comodcat{Q}{\scr{C}}$, where $Q$ is finite-dimensional, that $\shapoendo_{\ydind{Q}{U}}$ does not vanish anywhere, if and only if it does not vanish on the maximal degree component of $\ydind{Q}{U}$. Since this component lies on an edge, this explains intuitively why \Cref{cor_ydind_irred_all_refl}(1)$\implies$(2) holds.
\end{rema}

\chapter{Nichols systems of group-type} \label{chapt_nich_sys_group}

In this chapter we want to calculate the Shapovalov morphism and its kernel of $\ydmod{}$-modules, i.e. the maximal subobject of the $\ydmod{}$-module, for Nichols systems of group type. To realize that, we must first look at the theory of braided shuffles in \cref{sect_nich_sys_group_braided_shuffles}. Here we introduce a particular kind of shuffle elements in the braid group and show in \cref{sect_nich_sys_group_main}, \Cref{thm_shapoendo_via_gnk}, that we can give an explicit formula for the Shapovalov morphism on specific components using these shuffle elements. We use this formula to calculate the Kernel in degree $2$ under some conditions, see \Cref{prop_ker_shapo_deg2}. Finally we calculate the kernel for some specific examples of Nichols algebras in \cref{sect_nich_sys_group_examples}.

Let $\theta \in \ndN$ and $\mathbb{I}=\lbrace 1,\ldots, \theta \rbrace$.
Let $\fK$ be a field, $H$ a Hopf algebra over $\fK$ with bijective antipode and let $\scr{C}=\ydmod{H}$.

\section{Braided shuffles} \label{sect_nich_sys_group_braided_shuffles}

After discussing some fundamentals, the focus of this section lies on the elements $g_{n,k} \in \fK \brB_{n+1}$ for given $0 \le k \le n$, introduced in \Cref{nota_gnk}. In \Cref{thm_symmetrizer_commutes_gnk} will show that they have a commuting relation with the braided symmetrizer, which is a rare and unexpected property.
These elements turn out to be essential when calculating the Shapovalov morphism in \cref{sect_nich_sys_group_main}.

For $n \in \ndN$ let $\SG_n$ be the symmetric group and $\brB_{n}$ the group given by generators $c_1,\ldots,c_{n-1}$ and relations
\begin{align*}
    c_i c_{i+1} c_i = c_{i+1} c_i c_{i+1} \matcom && c_i c_j = c_j c_i \matcom
\end{align*}
for all $1 \le i \le n-2$ and $i+2 \le j \le n-1$. Here $\brB_{1}$ is the trivial group with one element.
For $n'>n$ we canonically view $\SG_n$ and $\brB_n$ as subgroups of $\SG_{n'}$ and $\brB_{n'}$, respectively.
For $1 \le i \le n-1$ denote $s_i = (i (i+1)) \in \SG_n$. 

\begin{rema}
    For all $\pi \in \SG_n$ there exists a minimal $k \in \ndN_0$ and $(i_1,\ldots,i_k)$ with $1\le i_1,\ldots,i_k \le n-1$, such that $\pi=s_{i_1} \cdots s_{i_k}$. Such $(i_1,\ldots,i_k)$ is called \textbf{reduced decomposition of $\pi$}. 
\end{rema}

\begin{prop} \label{prop_matsumoto_section}
    Let $n \in \ndN$. Then
    \begin{align*}
        \SG_n \rightarrow \brB_{n}, \, \pi \mapsto c_{\pi} \matcom
    \end{align*}
    where $c_{\pi}=c_{i_1} \cdots c_{i_k}$, where $(i_1,\ldots,i_k)$, $k \in \ndN_0$ is a reduced decomposition of $\pi$, is a well-defined map.
\end{prop}
\begin{proof}
    Refer to \cite{HeSch}, Theorem 1.7.2.
\end{proof}
\begin{rema}
    The map in \Cref{prop_matsumoto_section} is called \textbf{Matsumoto section}. It is a section for the surjective group homomorphism given by
    \begin{align*} \brB_{n} \rightarrow \SG_n, \, c_i \mapsto s_i \matdot \end{align*}
\end{rema}

\begin{defi}
    For $n\in \ndN$ the element
    \begin{align*}
        S_n := \sum_{\pi \in \SG_n} c_{\pi} \in \fK \brB_n
    \end{align*}
    is called the \textbf{braided symmetrizer}. Moreover for $0\le k \le n$ a permutation $\pi \in \SG_n$ is called \textbf{$k$-shuffle}, if
    \begin{align*}
        &\pi(1)< \ldots < \pi(k) \matcom  &\pi(k+1) < \ldots < \pi(n) \matdot
    \end{align*}
    The set of all $k$-shuffles is denoted $\SG_{k,n-k}$. Finally denote
    \begin{align*}
        S_{k,n-k} = \sum_{\pi \in \SG_{k,n-k}} c_{\pi^{-1}}  \in \fK \brB_n \matdot
    \end{align*}
\end{defi}

\begin{notation}
    For $n \in \ndN$, $i \in \ndN_0$ let $\shift{i}{\cdot} : \brB_n \rightarrow \brB_{n+i}$ be the injevtive group homomorphism, such that for all $1\le j \le n-1$ we have
    \begin{align*}
       \shift{i}{c_j} = c_{j+i} \matdot
    \end{align*}
\end{notation}

\begin{lem} \label{lem_symmetrizer_inductive_formula}
    Let $n \in \ndN$ and $1 \le k \le n$. Let $\varphi : \fK \brB_{n} \rightarrow \fK \brB_{n}$ be the algebra antimorphism such that we have $\varphi (c_i) = c_i$ for all $1\le i\le n-1$. The following relations hold:
    \begin{align*}
            \varphi (S_n) &= S_n \matcom \tag{1} \\
            S_{n-k+1,k} &= S_{n-k+1,k-1} + c_{n-k+1} \cdots c_n S_{n-k,k} \matcom \tag{2} \\
            S_{n+1} &= S_{n-k+1} \shift{{n-k+1}}{S_{k}} S_{n-k+1,k} \matcom \tag{3} \\
            S_{n+1} &= \varphi(S_{n-k+1,k}) S_{n-k+1} \shift{{n-k+1}}{S_{k}}  \matdot \tag{4}
    \end{align*}
\end{lem}
\begin{proof}
    For all $\pi \in \SG_n$ we have $\varphi(c_{\pi}) = c_{\pi^{-1}}$ and thus
    \begin{align*}
        \varphi (S_n) = \sum_{\pi \in \SG_n} \varphi(c_{\pi}) = \sum_{\pi \in \SG_n} c_{\pi^{-1}} = \sum_{\pi \in \SG_n} c_{\pi} = S_n \matdot
    \end{align*}
    Now (2) and (3) are proven in \cite{HeSch}, Corollary 1.8.8. (4) is implied by (1), (3) and the fact that $S_{n-k+1} \shift{{n-k+1}}{S_{k}} =  \shift{{n-k+1}}{S_{k}} S_{n-k+1} $.
\end{proof}

\begin{lem} \label{lem_reduced_decomposition_shifts}
    Let $n \in \ndN$, $\pi \in \SG_{n}$ and let $(i_1,\ldots,i_l)$, $l\in \ndN_0$ be a reduced decomposition of $\pi$.
    The following holds:
    \begin{enumerate}
        \item Let $1\le j \le n-1$. $(i_1,\ldots,i_l,j)$ is a reduced decomposition of $\pi s_j$ if and only if $\pi(j) < \pi(j+1) $.
        \item Let $1 \le k \le n$. $(i_1,\ldots,i_l,n,n-1,\ldots,k)$ is a reduced decomposition of the permutation $\pi s_{n} s_{n-1} \cdots s_k \in \SG_{n+1}$.
        \item Let $1 \le k \le n-1$ and assume $\pi(j) < \pi(n)$ for all $k \le j\le n-1$.
        Then $(i_1,\ldots,i_l,n-1,\ldots,k)$ is a reduced decomposition of the permutation $\pi s_{n-1} \cdots s_k$.
        \item Let $1 \le k \le n$ and assume $\pi \in \SG_{n-k,k}$. Then 
        \begin{align*}
            (i_1,\ldots,i_l,n, \ldots, n-k+1, n, \ldots, n-k+2, \ldots,n,n-1, n)
        \end{align*}
        is a reduced decomposition of the permuation
        \begin{align*}
            \pi (s_{n} \cdots s_{n-k+1}) (s_n \cdots s_{n-k+2}) \cdots (s_n s_{n-1}) (s_n) \matdot
        \end{align*}
    \end{enumerate}
\end{lem}
\begin{proof}
    (1) is proven in \cite{HeSch}, Theorem 1.7.2(1).
    
    (2): Clearly $(i_1,\ldots,i_l,n)$ is a reduced decomposition of $\pi s_{n}$, since $\pi \in \SG_n$. Assume for $1 \le j \le n-k$ that $(i_1,\ldots,i_l,n,\ldots,n-j+1)$ is a reduced decomposition of $\pi s_{n} s_{n-1} \cdots s_{n-j+1}$. Then
    \begin{align*} 
         &\pi s_{n} s_{n-1} \cdots s_{n-j+1} (n-j) = \pi(n-j) \\
        &< n+1 = \pi s_{n} s_{n-1} \cdots s_{n-j+1} (n-j+1) \matcom
    \end{align*}
    hence by (1) we get that $(i_1,\ldots,i_l,n,\ldots,n-j)$ is a reduced decomposition of $\pi s_{n} s_{n-1} \cdots s_{n-j}$.
    
    (3): Since $\pi(n-1) < \pi(n)$ using (1) we obtain that $(i_1,\ldots,i_l,n-1)$ is a reduced decomposition of $\pi s_{n-1}$. Assume for $1 \le j \le n-k-1$ that $(i_1,\ldots,i_l,n-1,\ldots,n-j)$ is a reduced decomposition of $\pi s_{n-1} \cdots s_{n-j}$. Then since $k \le n-j-1 \le n-2$
    \begin{align*} 
         &\pi s_{n-1} \cdots s_{n-j} (n-j-1) = \pi(n-j-1) \\
        &< \pi(n) = \pi s_{n-1} \cdots s_{n-j} (n-j) \matcom
    \end{align*}
    hence by (1) we get that $(i_1,\ldots,i_l,n,\ldots,n-j-1)$ is a reduced decomposition of $\pi s_{n-1} \cdots s_{n-j-1}$.
    
    (4): By (2) we know that $(i_1,\ldots,i_l,n,n-1,\ldots,n-k+1)$ is a reduced decomposition of $\pi s_{n} s_{n-1} \cdots s_{n-k+1}$. Assume that for $1\le l \le k-1$ we have that
    \begin{align*}
            (i_1,\ldots,i_l,n, \ldots, n-k+1, n, \ldots, n-k+2, \ldots,n,\ldots,n-k+l)
    \end{align*}
    is a reduced decomposition of 
    \begin{align*} 
        \pi' = \pi (s_{n} \cdots s_{n-k+1}) (s_n \cdots s_{n-k+2}) \cdots (s_n \cdots s_{n-k+l}) \matdot
    \end{align*}
    We show that 
    \begin{align*}
            (i_1,\ldots,i_l,n, \ldots, n-k+1, n, \ldots, n-k+2, \ldots,n,\ldots,n-k+l+1)
    \end{align*}
    is a reduced decomposition of $\pi' (s_n \cdots s_{n-k+l+1})$, finishing the proof. We have for $1\le l \le n+1$
    \begin{align*}
        \pi'(j) = \begin{cases} \pi(j) & \text{if $1 \le j \le n-k$,} \\ 
        \pi(n+1 -(j-(n-k+1))) & \text{if $n-k+1\le j \le n-k+l$,} \\
        \pi(j-l) & \text{if $n-k+l+1 \le j \le n+1$.}
        \end{cases}
    \end{align*} 
    In particular since $\pi \in \SG_{n-k,k}$ we have $\pi(n-k+1) < \cdots < \pi(n-l+1)$ and thus
    \begin{align*}
        \pi'(n-k+l+1) < \pi'(n-k+l+2) < \cdots < \pi'(n+1) \matdot
    \end{align*}
    Using (3) the above claim is implied.
\end{proof}

\begin{notation}
    Let $\omega_{1}=\id_{\brB_1}$ and for $n \in \ndN$ let
    \begin{align*}
        \omega_{n+1} = (c_n) (c_{n-1} c_n) \cdots (c_{1} \cdots c_n) \in \brB_{n+1} \matdot
    \end{align*}
\end{notation}

\begin{rema} \label{rema_longest_word_permutation}
    Let $n \in \ndN$. The element $\omega_{n}$ is equal to the image $c_{\pi}$ of the Matsumoto section of the permutation $\pi \in \SG_n$, such that $\pi(k) = n-k+1$ for all $1\le k \le n$. $\omega_{n}$ is the longest occurring product of generators in the image under the Matsumoto section of $\SG_n$.
\end{rema}

\begin{lem} \label{lem_longest_word_rels}
    Let $n\in \ndN$ and $\varphi : \fK \brB_{n} \rightarrow \fK \brB_{n}$ be the algebra antimorphism such that $\varphi (c_i) = c_i$ for all $1\le i\le n-1$. Let $\psi : \fK \brB_n \rightarrow \fK \brB_n$ be the algebra morphism, such that $\psi(c_j) = c_{n-j}$ for all $1\le j\le n-1$. Finally let $a \in \fK \brB_n$. The following relations hold:
    \begin{align*}
        \omega_{n+1} &= \shift{1}{\omega_{n}} c_1 \cdots c_n  \matcom \tag{1}\\
        \varphi(\omega_{n}) &= \omega_{n} \matcom \tag{2} \\
        \omega_{n} a &= \psi(a) \omega_{n} \matcom \tag{3} \\
        \omega_{n} S_n &= S_n \omega_{n}  \matdot \tag{4}
    \end{align*}
\end{lem}
\begin{proof}
    (1) is directly implied by definition. (2): Let $\pi \in \SG_n$, such that $\pi(k) = n-k+1$ for all $1\le k \le n$. Then $\pi = \pi^{-1}$ and thus
    \begin{align*}
        \varphi(\omega_{n}) = \varphi(c_{\pi}) = c_{\pi^{-1}} = c_{\pi} = \omega_{n} \matdot
    \end{align*}
    (3): It is enough to show the relation for $a=c_j$, $1\le j \le n-1$.
    Since for $1\le k \le j \le n-2$ we have
    \begin{align*}
        (c_k \cdots c_{n-1}) c_{j} &= c_k \cdots c_{j-1} c_{j} c_{j+1} c_{j} c_{j+2} \cdots c_{n-1} 
        \\&= c_k \cdots c_{j-1} c_{j+1} c_{j} c_{j+1} c_{j+2} \cdots c_{n-1} 
        =  c_{j+1} (c_k \cdots c_{n-1})\matcom
    \end{align*}
    we obtain for $1\le j \le n-1$
    \begin{align*}
         \omega_n c_j 
        = (c_{n-1}) \cdots (c_{n-j} \cdots c_{n-1}) c_{n-1} (c_{n-1-j} \cdots c_{n-1}) \cdots (c_1 \cdots c_{n-1})
    \end{align*}
    Below we inductively will show:
    \begin{align*}
        (c_{n-1}) \cdots (c_{n-j} \cdots c_{n-1}) c_{n-1} = c_{n-j} (c_{n-1}) \cdots (c_{n-j} \cdots c_{n-1}) \matcom \tag{a}
    \end{align*}
    implying $\omega_n c_j = \psi(c_j) \omega_n $, proving (3).
    Clearly (a) holds for $j=1$. Assume $j \ge 2$ and
    \begin{align*}
        (c_{n-1}) \cdots (c_{n-j+1} \cdots c_{n-1}) c_{n-1} = c_{n-j+1} (c_{n-1}) \cdots (c_{n-j+1} \cdots c_{n-1}) \matdot
    \end{align*}
    Observe that $(c_{n-1}) \cdots (c_{n-j+1} \cdots c_{n-1}) = \shift{n-j}{\omega_{j}}$, hence \begin{align*}\shift{n-j}{\omega_{j}} c_{n-1} = c_{n-j+1} \shift{n-j}{\omega_{j}} \matdot \end{align*}
    Using (1) and~(2) we obtain
    \begin{align*} 
        &(c_{n-1}) \cdots (c_{n-j} \cdots c_{n-1}) = \shift{n-j-1}{\omega_{j+1}} = \varphi(\shift{n-j-1}{\omega_{j+1}} ) 
        \\&= \varphi( \shift{n-j}{\omega_{j}} c_{n-j-1} \cdots c_{n-1}) = c_{n-1} \cdots c_{n-j-1} \shift{n-j}{\omega_{j}} \matdot
    \end{align*}
    Thus we conclude
    \begin{align*} 
        &(c_{n-1}) \cdots (c_{n-j} \cdots c_{n-1}) c_{n-1} = c_{n-1} \cdots c_{n-j-1} c_{n-j+1} \shift{n-j}{\omega_{j}} 
        \\&= c_{n-1} \cdots c_{n-j+2} c_{n-j+1} c_{n-j} c_{n-j+1} c_{n-j-1}  \shift{n-j}{\omega_{j}} 
        \\&= c_{n-1} \cdots c_{n-j+2} c_{n-j} c_{n-j+1} c_{n-j} c_{n-j-1}  \shift{n-j}{\omega_{j}}
        \\&= c_{n-j} c_{n-1} \cdots c_{n-j-1}  \shift{n-j}{\omega_{j}} = c_{n-j} (c_{n-1}) \cdots (c_{n-j} \cdots c_{n-1}) \matdot
    \end{align*}
    (4): Let $\tilde{\omega} \in \SG_n$, such that $\tilde{\omega}(j) = n-j+1$ for all $1\le j \le n$. Since for all $\pi \in \SG_n$ we have $\psi(c_{\pi}) = c_{\tilde{\omega} \pi \tilde{\omega}}$, we obtain that $\psi(S_n)=S_n$. Hence~(4) is implied by (3).
\end{proof}

\begin{notation} \label{nota_gnk}
    For $n \in \ndN_0$, $0\le k \le n$ let $\antishg_{n,k} \subset \SG_{n+1}$ be the subset of all $\pi \in \SG_{n+1}$, such that
    \begin{align*} 
        &\pi(1)< \ldots < \pi(n-k) \matcom \\& \pi(n-k+1)=n+1 \matcom \\ &\pi(n-k+2) > \ldots > \pi(n+1) \matdot
    \end{align*}
    Moreover let
    \begin{align*}
        g_{n,k} = \sum_{\pi \in \antishg_{n,k}} c_{\pi^{-1}} \in \fK \brB_{n+1} \matdot
    \end{align*}
    For $k < 0$ and $k > n$ let $g_{n,k}=0 \in \fK \brB_{n+1}$.
\end{notation}

\begin{rema}
    We will see that the elements $g_{n,k}$ interchange with the braided symmetrizer (\Cref{thm_symmetrizer_commutes_gnk}), which is a rare property. Also they play an important role when calculating the Shapovalov endomorphism (\Cref{thm_shapoendo_via_gnk}).
\end{rema}

\begin{prop} \label{prop_gnk_by_hnk}
    Let $n \in \ndN$. For $0 \le k \le n$ define the permutation $\tilde{\omega} = (s_{n} \cdots s_{n-k+1}) (s_{n} \cdots s_{n-k+2}) \cdots (s_n)$ (if $k=0$, then $\tilde{\omega}=\id_{\SG_{n+1}}$). We have
    \begin{align*}
        \antishg_{n,k} &= \SG_{n-k,k} \tilde{\omega} \matcom \\
        g_{n,k} &= \shift{n-k}{\omega_{k+1}} S_{n-k,k} \matdot 
    \end{align*}
\end{prop}
\begin{proof}
    We have for $1\le j \le n+1$
    \begin{align*}
        \tilde{\omega}(j) = \begin{cases}
        j & \text{if $1\le j \le n-k$,} \\
        n+1-(j-(n-k+1)) &\text{if $n-k+1 \le j \le n+1 $.}
        \end{cases}
    \end{align*}
    Let $\pi \in \SG_{n-k,k}$. Then $\pi \tilde{\omega}(n-k+1) = n+1$ and for $1 \le i \le n-k-1$ and $n-k+2 \le j \le n$ we have since $n-k+2 \le 2n-k+2-j \le n$
    \begin{align*}
        \pi \tilde{\omega} (i) &= \pi(i) < \pi(i+1) = \pi \tilde{\omega} (i+1) \matcom \\
        \pi \tilde{\omega} (j) &= \pi(2n-k+2-j) > \pi(2n-k+2-j-1) = \pi \tilde{\omega} (j+1)
        \matcom
    \end{align*}
    hence $\pi \tilde{\omega} \in \antishg_{n,k}$.
    
    Now let $\pi \in \antishg_{n,k}$. Then we have $\pi \tilde{\omega} (n+1) = \pi (n-k+1) = n+1$, hence $\pi \tilde{\omega} \in \SG_{n}$. Moreover for $1\le i \le n-k-1$ and $n-k+1 \le j \le n-1$ we obtain since $n-k+3 \le 2n-k+2-j \le n+1$
    \begin{align*}
        \pi \tilde{\omega} (i) &= \pi(i) < \pi(i+1) = \pi \tilde{\omega} (i+1) \matcom \\
        \pi \tilde{\omega} (j) &= \pi(2n-k+2-j) < \pi(2n-k+2-j-1) = \pi \tilde{\omega} (j+1) \matcom
    \end{align*}
    hence $\pi \tilde{\omega} \in \SG_{n-k,k}$.
    Since $\tilde{\omega}^2=\id_{\SG_{n+1}}$ we get $\pi = (\pi \tilde{\omega}) \tilde{\omega} \in \SG_{n-k,k} \tilde{\omega}$ and thus we can conclude $\antishg_{n,k}= \SG_{n-k,k} \tilde{\omega} $.
    
    Now for $\pi \in \SG_{n-k,k}$ by \Cref{lem_reduced_decomposition_shifts}(4) we obtain that
    \begin{align*}
        c_{(\pi \tilde{\omega})^{-1}} = (c_n) (c_{n-1} c_n) \cdots (c_{n-k+1} \cdots c_n) c_{\pi^{-1}} = \shift{n-k}{\omega_{k+1}} c_{\pi^{-1}} \matdot
    \end{align*}
    Combining this with the above result we obtain
    \begin{align*}
        g_{n,k} &= \sum_{\pi \in \SG_{n-k,k}} c_{(\pi \tilde{\omega})^{-1}} 
        = \shift{n-k}{\omega_{k+1}}  \sum_{\pi \in \SG_{n-k,k}} c_{\pi^{-1}} = \shift{n-k}{\omega_{k+1}}  S_{n-k,k} \matcom
    \end{align*}
    finishing the proof.
\end{proof}

\begin{rema} \label{rema_gnk}
    We have $\antishg_{n,0} = \lbrace \id_{\SG_{n+1}} \rbrace$ and $g_{n,0} = 1$ for all $n \in \ndN_0$. Moreover $\# \antishg_{n,k} = \# \SG_{n-k,k} = \binom{n}{k}$ and for $\pi \in \antishg_{n,k}$ we have $\pi(1)=1$ or $\pi(n+1)=1$: Indeed if $\pi(1) \ne 1$, then $1 < \pi(1) < \ldots < \pi(n-k)$. Thus we must have $\pi(n+1)=1$.
\end{rema}

\begin{exa} \label{exa_gnk}
    We have $\antishg_{1,1} = \lbrace (12) \rbrace$, hence $g_{1,0} = 1$ and $g_{1,1} = c_1$. Moreover we have
    \begin{align*}
         \antishg_{2,1} = \lbrace (23), (123) \rbrace \matcom && \antishg_{2,2} = \lbrace (13) \rbrace \matcom
    \end{align*}
    hence we have $g_{2,0}=1$, $g_{2,1} = c_2 + c_2 c_1$ and $g_{2,2} = c_2 c_1 c_2$. 
    Observe that
    \begin{align*}
        \SG_{2,1} = \lbrace \id_{\SG_3}, (23), (123) \rbrace, && \SG_{1,2} = \lbrace \id_{\SG_3} , (12), (132) \rbrace, && \SG_{0,3} = \lbrace \id_{\SG_3} \rbrace \matdot
    \end{align*}
    Using \Cref{prop_gnk_by_hnk} we obtain $g_{3,0}=1$, $g_{3,1} = c_3(1+c_2 + c_2 c_1)$, $g_{3,2} = c_3c_2c_3 (1+c_1 + c_1 c_2)$ and $g_{3,3} = c_3 c_2 c_3 c_1 c_2 c_3$. 
\end{exa}
    
\begin{lem} \label{lem_gnk_inductive_formulas}
    For $n \in \ndN$, $1\le k \le n$ we have
    \begin{align*}
        g_{n,k} + g_{n,k-1} = \shift{n-k+1}{\omega_{k}} S_{n-k+1,k} \matcom \tag{1} \\
        g_{n,k} =  g_{n-1,k-1} c_n \cdots c_1 + \shift{1}{g_{n-1,k}} \matdot \tag{2}
    \end{align*}
\end{lem}
\begin{proof}
    (1): Using \Cref{prop_gnk_by_hnk} and \Cref{lem_longest_word_rels}(1) we obtain
    \begin{align*}
        g_{n,k} = \shift{n-k}{\omega_{k+1}} S_{n-k,k} = \shift{n-k+1}{\omega_{k}} c_{n-k+1} \cdots c_n S_{n-k,k} 
        \matdot 
    \end{align*}
    Now applying \Cref{lem_symmetrizer_inductive_formula}(2) this simplifies to
    \begin{align*}
        \shift{n-k+1}{\omega_{k}} S_{n-k+1,k} - \shift{n-k+1}{\omega_{k}} S_{n-k+1,k-1} = \shift{n-k+1}{\omega_{k}} S_{n-k+1,k} - g_{n,k-1} \matdot
    \end{align*}
    
    (2): For $\pi \in \SG_n$ let $\shift{1}{\pi} \in \SG_{n+1}$ be the permutation that fixes $1$ and maps $2\le i \le n+1$ to $\pi(i-1)+1$.
    Assuming $\antishg_{n-1,k}$ is the empty set for $k<0$ or $k\ge n$, we will show the following:
    \begin{enumerate}[label=(\alph*)]
        \item $\antishg_{n,k} = s_{1} \cdots s_n \antishg_{n-1,k-1}  \cup \shift{1}{\antishg_{n-1,k}}$.
        \item The union in (1) is disjoint.
    \end{enumerate}
    Then (2) is implied by (a), (b) and \Cref{lem_reduced_decomposition_shifts}(2).
    
    (a): It is straight forward to check $\shift{1}{\antishg_{n-1,k}} \subset \antishg_{n,k}$. If $k \ge 1$ and $\pi \in \antishg_{n-1,k-1}$ then
    \begin{align*}
        s_1 \cdots s_n \pi (n-k+1) = s_1 \cdots s_n (n) = n+1
    \end{align*}
    and for $1\le i \le n-k-1$ and $n-k+2 \le j \le n$ we have
    \begin{align*}
         s_1 \cdots s_n \pi (i) & = \pi (i) +1 < \pi (i+1) +1 =  s_1 \cdots s_n \pi (i+1)
         \matcom 
         \\ s_1 \cdots s_n \pi (j) &= \pi(j) +1 > \pi(j+1) +1 = s_1 \cdots s_n \pi(j+1)
        \matcom
    \end{align*}
    hence $ s_{1} \cdots s_n \antishg_{n-1,k-1} \subset \antishg_{n,k}$. Now if $\pi \in \antishg_{n,k}$ and $\pi(1)=1$, then the permutation that maps $1\le i\le n$ to $\pi(i+1)-1$ is an element in $\antishg_{n-1,k}$, hence $\pi \in \shift{1}{\antishg_{n-1,k}}$. If $\pi(1) \ne 1$, then $\pi(n+1) = 1$, see \Cref{rema_gnk}. In particular $\pi(i) \ne 1$ for all $1\le i \le n$. We also have $s_n \cdots s_1 \pi  (n+1) = n+1$, hence $s_n \cdots s_1 \pi \in \SG_n$. Moreover we get $
    s_n \cdots s_1 \pi (n-k+1) = s_n \cdots s_1 (n+1) = n
    $
    and for $1\le i \le n-k-1$ and $n-k+2 \le j \le n-1$ we have
    \begin{align*}
        s_n \cdots s_1 \pi (i) & = \pi (i) - 1 < \pi (i+1) -1 =  s_n \cdots s_1 \pi (i+1)
         \matcom 
         \\ s_n \cdots s_1 \pi (j) &= \pi (j) - 1 > \pi (j+1) - 1 = s_n \cdots s_1 \pi (j+1)
        \matcom
    \end{align*}
    thus $s_n \cdots s_1 \pi \in \antishg_{n-1,k-1}$ and $\pi = s_1 \cdots s_n s_n \cdots s_1 \pi \in s_1 \cdots s_n \antishg_{n-1,k-1}$.
    
    (b): Let $\pi \in \antishg_{n-1,k-1}$. Then $s_1 \cdots s_n \pi (1) = \pi(1) + 1 \ne 1$.
    Since any permutation in $\shift{1}{\antishg_{n-1,k}}$ fixes $1$, the union is disjoint.
\end{proof}

\begin{thm} \label{thm_symmetrizer_commutes_gnk}
    Let $n \in \ndN$ and let $\varphi : \fK \brB_{n+1} \rightarrow \fK \brB_{n+1}$ be the algebra antimorphism such that we have $\varphi (c_i) = c_i$ for all $1\le i\le n$.
    Then for $0 \le k \le n$ we have
    \begin{align*}
        S_{n+1} g_{n,k} = \varphi(g_{n,k}) S_{n+1} \matdot
    \end{align*}
\end{thm}
\begin{proof}
    Clearly the claim holds for $k=0$, since $g_{n,0} = 1$. Assume $k \ge 1$ and that $S_{n+1} g_{n,k-1} = \varphi(g_{n,k-1}) S_{n+1}$.
    In the following calculation we use \Cref{lem_gnk_inductive_formulas}(1) in the first equation, \Cref{lem_symmetrizer_inductive_formula}(4) in the second equation, \Cref{lem_longest_word_rels}(2) and (4) in the third equation, \Cref{lem_symmetrizer_inductive_formula}(3) in the fourth equation and again \Cref{lem_gnk_inductive_formulas}(1) in the last equation:
    \begin{align*}
        &  S_{n+1} g_{n,k} = S_{n+1} \shift{n-k+1}{\omega_{k}} S_{n-k+1,k} - S_{n+1} g_{n,k-1} \\
        &= \varphi(S_{n-k+1,k} ) S_{n-k+1} \shift{n-k+1}{S_k} \shift{n-k+1}{\omega_{k}} S_{n-k+1,k} - S_{n+1} g_{n,k-1}  \\
        &= \varphi( S_{n-k+1,k} ) \varphi(\shift{n-k+1}{\omega_{k}} ) S_{n-k+1} \shift{n-k+1}{S_k} S_{n-k+1,k} - S_{n+1} g_{n,k-1} \\
        &= \varphi( \shift{n-k+1}{\omega_{k}} S_{n-k+1,k} ) S_{n+1} - S_{n+1} g_{n,k-1} \\
        &= \varphi( g_{n,k} ) S_{n+1} + \varphi( g_{n,k-1} ) S_{n+1} - S_{n+1} g_{n,k-1}
        \matdot
    \end{align*}
    This simplifies to $\varphi( g_{n,k} ) S_{n+1}$ by assumption.
\end{proof}

\section{Nichols systems of group-type} \label{sect_nich_sys_group_main}

We are now ready to give a explicit formula for the Shapovalov morphism on specific components, if some conditions are met, see \Cref{thm_shapoendo_via_gnk}. In particular, most of these conditions are met if the corresponding component of the Nichols system is of group-type, see \Cref{exa_group_type}. At the end of the section in \Cref{prop_ker_shapo_deg2} we use the formula to calculate parts of the maximal subobject in the componenent of the $\ydmod{}$-module of degree $2$. The explicit formula will also prove very useful when we go over to Nichols systems of diagonal type in \cref{sect_nich_sys_diag_main}.

Denote $\fK^{\times} = \fK \setminus \lbrace 0 \rbrace$.
Let $\scr{N}$ be a pre-Nichols system, let $Q := \nsalg{\scr{N}}$ and $N_j := \scr{N}_j$ for all $j \in \mathbb{I}$. 
Let $\Gamma$ be an abelian group, such that $\ndZ^\theta \subset \Gamma$.
Let $V \in \ydcat{Q}{\scr{C}}$ be a homogeneously generated $\Gamma$-graded object. 
We fix $i \in \mathbb{I}$, assume that $\subalgQ{N_i}$ is strictly graded and that there exists $\lambda_i \in \fK^{\times}$, such that
\begin{align} \label{equ_braiding_squared_lambda_id}
    \brd^{\scr{C}}_{\gencomp{V}, N_i} \brd^{\scr{C}}_{N_i, \gencomp{V}} = \lambda_i \, \id_{N_i \ot \gencomp{V}} \matdot
\end{align}

\begin{exa} \label{exa_group_type}
    Such $\lambda_i$ exists in the following setting:
    Assume $N_i$ is irreducible in $\scr{C}$ and that $\gencomp{V}$ is one-dimensional. Let $X$ be a basis of $N_i$ and $0\ne v \in \gencomp{V}$. 
    We assume that $N_i$ is of group-type, that is for all $x \in X$, there exists a group-like element $g_x \in H$, such that
    \begin{align*}
        \coact_{N_i}^H (x) = g_x \ot x \matdot
    \end{align*}
    Moreover assume that $H$ is generated by group-like elements and that the elements $g_x$, $x \in X$ are linearly independent.
    Observe that if $a \in N_i$, such that $a\ne 0$ and $\coact_{N_i}^H (a) = g \ot a$ for some group-like element $g \in H$, then $g=g_x$ and $a \in \fK x$ for some $x \in X$.
    By \Cref{lem_induced_Q_comodule_is_YD_module} we obtain that for $x \in  X$ the set $Hx$ is a subobject of $N_i$. Since $N_i$ is irreducible we obtain $Hx=N_i$.
    $Hx$ is spanned by elements $g\cdot x$, where $g$ is group-like. Such $g$ is invertible with inverse $S(g)$ and we have
    \begin{align*}
       \coact_{N_i}^H (g \cdot x) = g g_x g^{-1} \ot g\cdot x.
    \end{align*}
    By the above we conclude that there exists some $y \in X$, such that $g g_x g^{-1} = g_y$ and $g\cdot x \in \fK y$. Since $Hx = N_i$ we conclude that all $g_z$, $z \in X$ are conjugate.
    Now there exists some $\lambda' \in \fK^{\times}$, such that $g_x \cdot v = \lambda' v$. Let $y \in X$, $g \in H$ and $\mu \in \fK^{\times}$, such that $g$ is group-like, $g_y = g g_x g^{-1}$ and $g\cdot x = \mu y$. 
    We obtain $g_y \cdot v = g g_x g^{-1} \cdot v = \lambda' v$, hence $g_z \cdot v = \lambda' v$ for all $z \in X$. Now let $g_v \in H$ be the group-like element with $\coact_{\gencomp{V}} (v) = g_v \ot v$. For each group-like element $h \in H$ we have $h \cdot v \in \fK^{\times} v$ and thus
    \begin{align*}
        g_v \ot v = \coact_{N_i}^H(v) 
        = h g_v h^{-1} \ot v \matcom
    \end{align*}
    i.e. $h g_v = g_v h$. Since $g_v \cdot x \in N_i$ and
    \begin{align*}
        \coact_{N_i}^H (g_v \cdot x) = g_v g_x g_v^{-1} \ot g_v \cdot x = g_x \ot g_v \cdot x,
    \end{align*}
    we conclude that $g_v \cdot x \in \fK^{\times} x$. Let $\lambda'' \in \fK^{\times}$, such that $g_v \cdot x = \lambda'' x$. Since
    \begin{align*}
        g_v \cdot y = \mu^{-1} g_v g \cdot x = \mu^{-1} g g_v \cdot x = \lambda'' \mu^{-1}g \cdot x = \lambda'' y
    \end{align*}
    we have $g_v \cdot z = \lambda'' z$ for all $z \in X$.
    This implies 
    \begin{align*}
        \brd^{\scr{C}}_{\gencomp{V}, N_i} \brd^{\scr{C}}_{N_i, \gencomp{V}} (z \ot v) = \lambda' \brd^{\scr{C}}_{\gencomp{V}, N_i} (v \ot z) = \lambda' \lambda'' z \ot v
    \end{align*}
    for all $z \in X$, i.e. $\brd^{\scr{C}}_{\gencomp{V}, N_i} \brd^{\scr{C}}_{N_i, \gencomp{V}} = \lambda' \lambda'' \, \id_{N_i \ot \gencomp{V}}$.
\end{exa}

Let $n \in \ndN$. We view $\fK \brB_n$ as a subset of $\Endo_{\scr{C}}(N_i^{\ot n})$ by identifying $c_j$ with the endomorphism
\begin{align*}
    \id_{N_i^{\ot_{j-1}}} \ot \brd^{\scr{C}}_{N_i,N_i} \ot \id_{N_i^{\ot n-j-1 }} 
\end{align*}
for all $1\le j\le n-1$. Since $\subalgQ{N_i}$ is strictly graded, we know that
\begin{align*}
    N_i^n \cong N_i^{\ot n} / \ker S_n \matcom
\end{align*}
for details refer to \cite{HeSch}, Corollary 1.9.7. By \Cref{thm_symmetrizer_commutes_gnk} we obtain that $g_{n,k}(\ker S_{n+1}) \subset \ker S_{n+1}$, hence $g_{n,k}\in \fK \brB_{n+1}$ yields a well defined endomorphism
\begin{align*}
    g_{n,k} : N_i^{n+1} \rightarrow N_i^{n+1}
\end{align*}
for all $k \in \ndZ$.

\begin{lem} \label{lem_shapovalov_endo_product_rack_case}
    Let $x \in N_i$, $y \in Q$, $v \in \gencomp{V}$. Then we have
    \begin{align*}
        \shapoendo_V \left( (xy) \cdot v \right) =& \left(
        \bimu_Q \ot \id_{\gencomp{V}}
        - \lambda_i (\bimu_Q \brd^{\scr{C}}_{Q, Q}  \ot \id_{\gencomp{V}}) \right) \left(x \ot \shapoendo_V(y \cdot v) \right)
        \matdot
    \end{align*}
\end{lem}
\begin{proof}
    Combining \Cref{lem_shapovalov_endo_product} and the fact that $x$ is primitive we obtain that $\shapoendo_V \left( (xy) \cdot v \right)$ simplifies to
    \begin{align*}
        & (\bimu_Q \ot \id_{\gencomp{V}})(\bimu_Q \ot \brd^{\scr{C}}_{\gencomp{V},Q}) (\id_Q \ot \brd^{\scr{C}}_{Q, Q\ot \gencomp{V}}) \left( (x \ot 1 - 1 \ot x) \ot \shapoendo_V(y \cdot v) \right)
        \\=& \left( \bimu_Q \ot \id_{\gencomp{V}}  - (\bimu_Q \ot \id_{\gencomp{V}})(\id_Q \ot \brd^{\scr{C}}_{\gencomp{V},Q}) \brd^{\scr{C}}_{Q, Q\ot \gencomp{V}} \right) \left( x \ot \shapoendo_V(y \cdot v) \right)
    \end{align*}
    Observe that $\brd^{\scr{C}}_{Q,Q} ( N_i \ot Q) \subset Q \ot N_i$. Hence \cref{equ_braiding_squared_lambda_id} implies the claim.
\end{proof}

\begin{thm} \label{thm_shapoendo_via_gnk}
    Let $n \in \ndN_0$. For $x \in N_i^{n+1}$ and $v \in \gencomp{V}$ we have
    \begin{align*}
        \shapoendo_V (x \cdot v) = (1-\lambda_i) \sum_{k=0}^n (-\lambda_i)^k g_{n,k}(x) \ot v \matdot
    \end{align*}
\end{thm}
\begin{proof}
    We do induction on $n$. If $n=0$, then $x \in N_i$ and thus by setting $y=1$ in  \Cref{lem_shapovalov_endo_product_rack_case} and the fact that $\shapoendo_V(v) = 1 \ot v$ we obtain
    \begin{align*}
        \shapoendo_V (x \cdot v) = (1-\lambda_i) x \ot v \matdot
    \end{align*}
    Now assume $n \ge 1$. Let $x' \in N_i$, $y \in N_i^n$, such that $x=x' y$. 
    We have
    \begin{align*}
        x' g_{n-1,k}(y) = \shift{1}{g_{n-1,k}}(x'y) = \shift{1}{g_{n-1,k}}(x)
    \end{align*}
    as well as
    \begin{align*}
        \bimu_Q \brd^{\scr{C}}_{Q, Q} (x' \ot g_{n-1,k}(y) ) = (g_{n-1,k} c_n \cdots c_1)(x'y) = (g_{n-1,k} c_n \cdots c_1)(x) \matdot
    \end{align*}
    Then by \Cref{lem_shapovalov_endo_product_rack_case} and induction hypothesis we obtain
    \begin{align*}
        &\shapoendo_V (x \cdot v) = \left(
        \bimu_Q \ot \id_{\gencomp{V}}
        - \lambda_i (\bimu_Q \brd^{\scr{C}}_{Q, Q}  \ot \id_{\gencomp{V}}) \right) \left(x' \ot \shapoendo_V(y \cdot v) \right)
        \\=& (1-\lambda_i) \sum_{k=0}^{n-1}  (-\lambda_i)^k \left(
        \bimu_Q \ot \id_{\gencomp{V}}
        - \lambda_i (\bimu_Q \brd^{\scr{C}}_{Q, Q}  \ot \id_{\gencomp{V}}) \right)  \left(x' \ot g_{n-1,k}(y) \ot v \right)
        \\=& (1-\lambda_i) \sum_{k=0}^{n-1}  (-\lambda_i)^k \left( \shift{1}{g_{n-1,k}}
        - \lambda_i g_{n-1,k} c_n \cdots c_1 \right)(x) \ot v 
        \matdot
    \end{align*}
    Since $g_{n-1,n}=0$ and $g_{n-1,-1}=0$ this further simplifies to
    \begin{align*}
       \shapoendo_V (x \cdot v) = (1-\lambda_i) \sum_{k=0}^{n}  (-\lambda_i)^k \left( \shift{1}{g_{n-1,k}} + g_{n-1,k-1} c_n \cdots c_1 \right)(x) \ot v \matdot
    \end{align*}
    Finally \Cref{lem_gnk_inductive_formulas} yields $\shift{1}{g_{n-1,k}} + g_{n-1,k-1} c_n \cdots c_1 = g_{n,k}$.
\end{proof}

\begin{rema}
    If $\lambda_i = 1$, then by \Cref{thm_shapoendo_via_gnk} we get that 
    \begin{align*}
        \ker \shapoendo_V \cap \subalgQ{N_i} \cdot \gencomp{V} = \oplus_{n \ge 1} N_i^n \cdot \gencomp{V} \matdot 
    \end{align*}
    If $\lambda_i \ne 1$, then
    \begin{align*}
        \ker \shapoendo_V \cap \subalgQ{N_i} \cdot \gencomp{V} \subset \oplus_{n \ge 2} N_i^n \cdot \gencomp{V} \matdot 
    \end{align*}
\end{rema}

Assume $\lambda_i \ne 1$. We now want to calculate $\ker \shapoendo_V \cap N_i^2 \cdot \gencomp{V}$ in the case where there is a basis of $N_i$, such that for all $x,y \in N_i$ in that basis there exists a minimal number $m \in \ndN_0$, such that the elements $x\ot y, c_1(x \ot y), \ldots, c_1^m(x\ot y)$ are linearly independent in $N_i^{\ot 2}$ and
    \begin{align*}c_1^{m+1} (x \ot y) \in \fK x \ot y \matdot \end{align*}
    
\begin{exa} \label{exa_rack_cocycle}
    Assume $N_i$ is defined via quandle and cocycle (for details refer to \cite{MR1994219}), that is that there exists a basis $X$ of $N_i$ as well as an action $\qact: X \times X \rightarrow X$, such that $X$ is a quandle and there exists a 2-cocycle $q : X \ot X \rightarrow \fK^\times$, such that that for $x,y \in X$ we have that
    \begin{align*}
        \brd^{\scr{C}}_{N_i,N_i} (x \ot y) = q(x,y) x \qact y \ot x \matdot
    \end{align*}
    Then such $m$ exists for all $x,y \in X$, since $N_i$ is finite-dimensional. Here $N_i$ is of group-type like in \Cref{exa_group_type}, where $H$ could be the group algebra over the enveloping group of $X$.
\end{exa}

Let $x,y \in N_i$. Assume there is a minimal number $m \in \ndN_0$, such that the elements $x\ot y, c_1(x \ot y), \ldots, c_1^m(x\ot y) \in N_i^{\ot 2}$ are linearly independent and that $c_1^{m+1} (x \ot y) \in \fK x \ot y$. Let $q \in \fK$ be such that 
\begin{align*}
    c_1^{m+1} (x \ot y) = q x \ot y 
    \matdot
\end{align*}
Observe that since $c_1 S_2 = S_2 c_1$ we have a well-defined endomorphism
\begin{align*}
    c_1 : N_i^2 \rightarrow N_i^2 \matdot
\end{align*}
We obtain $c_1^{m+1}(xy) = q xy$.


\begin{lem} \label{lem1_ker_shapo_deg2}
    If $q = (-1)^{m+1}$ then $c_1(xy),\ldots, c_1^{m}(xy)$ are linearly independent and
    \begin{align*}
        \sum_{k=1}^m (-1)^k c_1^k(xy) = - xy \matdot 
    \end{align*} 
    If $q \ne (-1)^{m+1}$ then $xy, c_1(xy),\ldots, c_1^{m}(xy)$ are linearly independent.
\end{lem}
\begin{proof}
    Assume $q = (-1)^{m+1}$. Then
    \begin{align*}
        S_2 \left( \sum_{k=0}^m (-1)^k c_1^k(x\ot y)\right) &= \sum_{k=0}^m (-1)^k (1+c_1) c_1^k(x\ot y) = 0
    \end{align*}
    since $(-1)^m c_1^{m+1}(x\ot y) = (-1)^{2m+1} (x \ot y) = -x \ot y$. This implies that $\sum_{k=0}^m (-1)^k c_1^k(xy) = 0$. Now let $\mu_1,\ldots, \mu_m \in \fK$ and assume
    \begin{align*}
        \sum_{k=1}^m \mu_i c_1^k(xy) = 0 \matdot
    \end{align*}
    We obtain
    \begin{align*}
        0 &= S\left(\sum_{k=1}^m \mu_k c_1^k(x \ot y)\right) = \sum_{k=1}^m \mu_k (1+c_1) c_1^k(x\ot y) \\&= (-1)^{m+1} \mu_m (x\ot y) + \mu_{1} c_1(x\ot y) + \sum_{k=2}^{m} (\mu_{k-1} + \mu_{k}) c_1^k(x\ot y)
    \end{align*}
    This implies $\mu_m=\mu_1=\mu_{k-1}+\mu_k = 0$ for all $2\le k \le m$, i.e. $\mu_k=0$ for all $1\le k \le m$. Similarly if $q \ne (-1)^{m+1}$ and $\mu_0,\ldots,\mu_m \in \fK$ are given such that $\sum_{k=0}^m \mu_i c_1^k(xy) = 0$,
    then similarly to the above we obtain
    \begin{align*}
        0 = (\mu_0 + q\mu_m) (x\ot y) +\sum_{k=1}^{m} (\mu_{k-1} + \mu_{k}) c_1^k(x\ot y) \matdot
    \end{align*}
    This implies $\mu_0 = -q \mu_m$, as well as $\mu_{k-1}=-\mu_k$ for all $1 \le k \le m$. Hence $q \mu_m = -\mu_0 = (-1)^{m+1} \mu_m$ and since $q\ne (-1)^{m+1}$, this implies $\mu_0 = \mu_1 = \ldots = \mu_m = 0$.
\end{proof}

\begin{lem} \label{lem2_ker_shapo_deg2}
    For $k \in \ndN_0$ and $v \in \gencomp{V}$ we have
    \begin{align*}
        \shapoendo_V ( c_1^k(xy) \cdot v ) = (1-\lambda_i) (c_1^k(xy) - \lambda_i c_1^{k+1}(xy)) \ot v \matdot
    \end{align*}
\end{lem}
\begin{proof}
    The claim is a direct implication of \Cref{thm_shapoendo_via_gnk}.
\end{proof}

\begin{lem} \label{lem3_ker_shapo_deg2}
    Assume $q \ne (-1)^{m+1}$. Then for 
    \begin{align*}
        a = \sum_{k=0}^m \lambda_i^k c_1^k(xy) \in N_i^2
    \end{align*}
    and $v \in \gencomp{V}$ we have $a \ne 0$ and
    \begin{align*}
        \shapoendo_V(a \cdot v) = (1 - \lambda_i)(1 - q \lambda_i^{m+1}) x y \ot v
    \end{align*}
\end{lem}
\begin{proof}
    $a$ is non-zero by \Cref{lem1_ker_shapo_deg2}. Using \Cref{lem2_ker_shapo_deg2} we get
    \begin{align*}
         \shapoendo_V(a \cdot v) = (1-\lambda_i) \sum_{k=0}^m \lambda_i^k  (c_1^k(xy) - \lambda_i c_1^{k+1}(xy)) \ot v
    \end{align*}
    This is a telescoping sum and simplifies to $(1 - \lambda_i)(1 - q \lambda_i^{m+1}) x y \ot v$.
\end{proof}

\begin{lem} \label{lem4_ker_shapo_deg2}
    Assume $q = (-1)^{m+1}$. Then for 
    \begin{align*}
        a = \sum_{k=1}^m (\sum_{l=0}^{k-1} (-1)^{k-1-l} \lambda_i^l) c_1^k(xy) \in N_i^2
    \end{align*}
    and $v \in \gencomp{V}$ we have $a \ne 0$ and
    \begin{align*}
        \shapoendo_V(a \cdot v) = (1 - \lambda_i)(\sum_{l=0}^m (-1)^{l} \lambda_i^l) x y \ot v
    \end{align*}
\end{lem}
\begin{proof}
    $a$ is non-zero by \Cref{lem1_ker_shapo_deg2}.
    Using \Cref{lem2_ker_shapo_deg2} we get
    \begin{align*}
         \shapoendo_V(a \cdot v) = (1-\lambda_i) \sum_{k=1}^m (\sum_{l=0}^{k-1} (-1)^{k-1-l} \lambda_i^l) (c_1^k(xy) - \lambda_i c_1^{k+1}(xy)) \ot v
    \end{align*}
    Now rewriting the sums, most terms will cancel each other out:
    \begin{align*}
        &\sum_{k=1}^m (\sum_{l=0}^{k-1} (-1)^{k-1-l} \lambda_i^l) (c_1^k(xy) - \lambda_i c_1^{k+1}(xy))
        \\ =& \sum_{k=1}^m (\sum_{l=0}^{k-1} (-1)^{k-1-l} \lambda_i^l) c_1^k(xy) + \sum_{k=2}^{m+1} (\sum_{l=1}^{k-1} (-1)^{k-l} \lambda_i^{l}) c_1^{k}(xy) 
        \\ =& c_1(xy) + \sum_{k=2}^m (-1)^{k-1} c_1^k(xy) + (\sum_{l=1}^{m} (-1)^{m+1-l} \lambda_i^{l}) c_1^{m+1}(xy) 
        \\ =& \sum_{k=1}^m (-1)^{k-1} c_1^k(xy) + (\sum_{l=1}^{m} (-1)^{l} \lambda_i^{l}) xy  \matdot
    \end{align*}
    By \Cref{lem1_ker_shapo_deg2} the left summand is $xy$, hence the whole sum simplifies to $(\sum_{l=0}^m (-1)^{l} \lambda_i^l) x y$.
\end{proof}

Let $O \subset N_i^2$ be the subspace spanned by $xy, c_1(xy),\ldots, c_1^{m}(xy)$.


\begin{prop} \label{prop_ker_shapo_deg2}
    We have $\ker \shapoendo_V \cap O \cdot \gencomp{V} \ne 0$ if and only if $q \ne (-1)^{m+1}$ and $q \lambda_i^{m+1} = 1$ or if $q = (-1)^{m+1}$ and $\sum_{l=0}^{m} (-1)^{l} \lambda_i^l=0$. 
    In this case $\ker \shapoendo_V \cap O \cdot \gencomp{V} = a \cdot \gencomp{V}$, where 
    \begin{align*}
        a = \begin{cases}
            \sum_{k=0}^{m} \lambda_i^k c_1^k(xy) \matcom & \text{if $q \ne (-1)^{m+1}$,} \\
            \sum_{k=1}^{m} (\sum_{l=0}^{k-1} (-1)^{k-1-l} \lambda_i^l) c_1^k(xy) \matcom & \text{else.}
        \end{cases}
    \end{align*}
\end{prop}
\begin{proof}
    Let $0\ne v \in \gencomp{V}$ and $b \in O$, such that $b \cdot v \in \ker \shapoendo_V \cap O \cdot \gencomp{V}$. Assume $q \ne (-1)^{m+1}$. Moreover let $\mu_0,\ldots,\mu_{m} \in \fK$, such that $b = \sum_{k=0}^{m} \mu_k c_1^k(xy)$.  Then by \Cref{lem2_ker_shapo_deg2}
    \begin{align*}
        0 =& (1-\lambda_i) \sum_{k=0}^{m} \mu_k (c_1^k(xy) - \lambda_i c_1^{k+1}(xy)) \ot v
        \\ =& (1-\lambda_i) \left( ( \mu_0 - \mu_{m} \lambda_i q)xy + \sum_{k=1}^{m} (\mu_k -\mu_{k-1}\lambda_{i}) c_1^k(xy)  \right) \ot v.
    \end{align*}
    This implies $ \mu_k = \mu_{k-1} \lambda_i$ for all $1 \le k \le m$, hence $b \in \fK a$.
    
    Now assume $q = (-1)^{m+1}$. Moreover let $\mu_1,\ldots,\mu_{m} \in \fK$, such that $b = \sum_{k=1}^{m} \mu_k c_1^k(xy)$ (exists by \Cref{lem1_ker_shapo_deg2}). Again by \Cref{lem2_ker_shapo_deg2} we obtain
    \begin{align*}
        0 = (1-\lambda_i) \sum_{k=1}^{m} \mu_k (c_1^k(xy) - \lambda_i c_1^{k+1}(xy)) \ot v.
    \end{align*}
    Then we have using also \Cref{lem1_ker_shapo_deg2}
    \begin{align*}
     0 & =\sum_{k=1}^{m} \mu_k (c_1^k(xy) - \lambda_i c_1^{k+1}(xy))
    \\&= \mu_1 c_1(xy) + (-1)^{m} \mu_{m} \lambda_i xy + \sum_{k=2}^{m} (\mu_k -\mu_{k-1}\lambda_{i}) c_1^k(xy) 
    \\&= \mu_1 c_1(xy) + \mu_{m} \lambda_i \sum_{k=1}^m (-1)^{m+k+1} c_1^k(xy) + \sum_{k=2}^{m} (\mu_k -\mu_{k-1}\lambda_{i}) c_1^k(xy).
    \end{align*}
    For all $1 \le k \le m$ the coefficient of $c_1^k(xy)$ must be $0$, hence we obtain $\mu_1 + \mu_{m} \lambda_i (-1)^{m} = 0$ and
    \begin{align*}
         \mu_{m} \lambda_i (-1)^{m+k+1} +\mu_k -\mu_{k-1}\lambda_{i} = 0
    \end{align*}
    for all $2 \le k \le m$. This gives
    \begin{align*}
        \mu_k = \lambda_i \mu_{k-1} + (-1)^{k-1} \mu_{1},
    \end{align*}
    for all $2 \le k \le m$, which inductively implies $b = \mu_1 a \in \fK a$.
    
    In both cases we get $\ker \shapoendo_V \cap O \cdot \gencomp{V} \subset a \cdot \gencomp{V}$. Hence either we have equality, or $\ker \shapoendo_V \cap O \cdot \gencomp{V} = 0$, depending on the conditions in \Cref{lem3_ker_shapo_deg2} and \Cref{lem4_ker_shapo_deg2}.
\end{proof}

\section{Examples} \label{sect_nich_sys_group_examples}

We will now calculate the kernel of the shapovalov morphism for specific examples of Nichols algebras. The results hint that some results of the previous section should also hold in degree $\ge 3$, see \Cref{conj_kernel_shapo}.

Similar to the previous section let $\scr{N}$ be a pre-Nichols system, $Q := \nsalg{\scr{N}}$ and $N_j := \scr{N}_j$ for all $j \in \mathbb{I}$. 
Let $\Gamma$ be an abelian group, such that $\ndZ^\theta \subset \Gamma$.
Let $V \in \ydcat{Q}{\scr{C}}$ be a homogeneously generated $\Gamma$-graded object. 
We fix $t \in \mathbb{I}$, assume that $\subalgQ{N_t}$ is strictly graded and that there exists $\lambda_t \in \fK^{\times}$, such that \begin{align*} 
    \brd^{\scr{C}}_{\gencomp{V}, N_t} \brd^{\scr{C}}_{N_t, \gencomp{V}} = \lambda_t \, \id_{N_t \ot \gencomp{V}} \matdot
\end{align*}

\begin{exa}
    Let $n\ge 3$, $X= \lbrace (ij) \mid 1\le i<j \le n \rbrace \subset S_n$ the quandle of transpositions with quandle action $\pi \quact \sigma = \pi \sigma \pi^{-1}$ for all $\pi, \sigma \in X$. Let $q : X \times X \rightarrow \fK^{\times}$ be the cocycle defined by
    \begin{align*}
        q(\pi, (ij)) = \begin{cases}
            1, & \text{if $\pi(i) < \pi(j)$}, \\
            -1, & \text{else.}
        \end{cases}
    \end{align*}
    for all $\pi \in X$, $1 \le i< j\le n$.
    Assume that $N_t$ is defined by $X$ and $q$ like in \Cref{exa_rack_cocycle}. Let $x_{\pi}, \pi \in X$ be the corresponding basis elements of $N_t$.
    For $n \le 5$ the algebras $\subalgQ{N_t}$ are known as the Fomin-Kirillov algebras (\cite{FominKirillovAlgebra}). For $n \in \lbrace 3,4,5 \rbrace$ the dimension of $\subalgQ{N_t}$ is $12$, $576$, $8294400$, respectively. Assume $\lambda_t \ne 1$.
    First we want to calculate $\ker \shapoendo_V \cap N_t^2 \cdot \gencomp{V}$.
    For $1\le i,j,k,l \le n$, $\# \lbrace i,j,k \rbrace = 3$ and $\# \lbrace i,j,k,l \rbrace = 4$ we have
    \begin{align*}
        (c_1^m( x_{(ij)} \ot x_{(ij)}))_{0\le m \le 1} =& (x_{(ij)} \ot x_{(ij)}, -x_{(ij)} \ot x_{(ij)} ), \\
        (c_1^m( x_{(ij)} \ot x_{(jk)}))_{0\le m \le 3} =& (x_{(ij)} \ot x_{(jk)}, q((ij),(jk)) x_{(ik)} \ot x_{(ij)}, \\
        & q((ij),(jk)) q((ik),(ij)) x_{(jk)} \ot x_{(ik)}, -  x_{(ij)} \ot x_{(jk)}),
        \\
        (c_1^m( x_{(ij)} \ot x_{(kl)}))_{0\le m \le 2} =& (x_{(ij)} \ot x_{(kl)}, x_{(kl)} \ot x_{(ij)}, x_{ij} \ot x_{kl} ).
    \end{align*}
    Hence we can associate each orbit to a $2$-cycle $(ij)$, a $3$-cycle $(ijk)$ or to a permutation of type $(ij)(kl)$.
    In the notation from before for these orbits we get $m_{(ij)} = 0$, $q_{(ij)} = 1$, $m_{(ijk)} = 2$, $q_{(ijk)} = -1$ and $m_{(ij)(kl)} = 1$, $q_{(ij)(kl)} = 1$. By \Cref{lem1_ker_shapo_deg2} we obtain that the orbits of $N_t^2$ under the action of $\fK \brB_2$ are $O_{(ij)} = 0$ and
    \begin{align*}
        O_{(ijk)} &= \mathrm{Span}_{\fK} \lbrace x_{(ik)} x_{(ij)}, x_{(jk)} x_{(ik)}\rbrace, \\
        O_{(ij)(kl)} &= \mathrm{Span}_{\fK} \lbrace x_{(kl)} x_{(ij)} \rbrace .
    \end{align*}
    \Cref{prop_ker_shapo_deg2} yields: Since $\lambda_t \ne 1$ we get $\ker \shapoendo_V \cap O_{(ij)(kl)} \cdot \gencomp{V} = 0$. Moreover if $1-\lambda_t+\lambda_t^2 \ne 0$, then $\ker \shapoendo_V \cap O_{(ijk)} \cdot \gencomp{V} = 0$ and $\ker \shapoendo_V \cap N_t^2 \cdot \gencomp{V} = 0$. If $1-\lambda_t+\lambda_t^2 = 0$, that is if $\lambda_t=-\omega$, where $\omega$ is a primitive third root of $1$, then
    \begin{align} \label{equ_example_foki}
        \ker \shapoendo_V \cap O_{(ijk)} \cdot \gencomp{V} = \left(  x_{(ik)} x_{(ij)} + (\lambda_t-1)  q((ik),(ij)) x_{(jk)} x_{(ik)} \right) \cdot \gencomp{V}.
    \end{align}
    Observe that we have $\frac{n(n-1)(n-2)}{3}$ of those components.
    
    Now using $GAP$ and its package $GBNP$ (\cite{GAP}, \cite{GBNP}), for $n=3$ and $n=4$ we can calculate the entire $\ker \shapoendo_V \cap \subalgQ{N_t} \cdot \gencomp{V}$ using Gröbner basis.
    For $n=3$ the results are: If $1-\lambda_t+\lambda_t^2 \ne 0$, then $\ker \shapoendo_V \cap \subalgQ{N_t } \cdot \gencomp{V} = 0$ and if $1-\lambda_t+\lambda_t^2 = 0$, then $\ker \shapoendo_V \cap \subalgQ{N_t } \cdot \gencomp{V} = A \cdot \gencomp{V}$, where $A$ is the left ideal of $N_t$ generated by the $2$ elements of degree $2$ that we have seen in \cref{equ_example_foki}, that is
    \begin{align*}
        &x_{(13)} x_{(12)} - (\lambda_t -1) x_{(23)} x_{(13)},
        \\
        &x_{(12)} x_{(13)} + (\lambda_t -1) x_{(23)} x_{(12)}.
    \end{align*}
    $A$ has dimension $6$.
    Now for $n=4$ the result differs, since we might get a new kernel in degree $3$: First similar to before, if $(1-\lambda_t+\lambda_t^2)( \lambda_t^2 + 1 )  \ne 0$, then $\ker \shapoendo_V \cap \subalgQ{N_t } \cdot \gencomp{V} = 0$ and if $1-\lambda_t+\lambda_t^2 = 0$, then $\ker \shapoendo_V \cap \subalgQ{N_t } \cdot \gencomp{V} = A \cdot \gencomp{V}$, where $A$ is the left ideal of $N_t$ generated by the $8$ elements of degree $2$ in \cref{equ_example_foki}. Here $A$ has dimension $528$. Now if $\lambda_t^2 + 1 = 0$, we obtain $\ker \shapoendo_V \cap \subalgQ{N_t } \cdot \gencomp{V} = A \cdot \gencomp{V}$, where $A$ is a $384$-dimensional left ideal of $N_t$ generated by $6$ elements of degree $3$ that we obtain by acting on
    \begin{align*}
        &x_{(14)} x_{(13)} x_{(23)} + (1 - \lambda_t) x_{(14)} x_{(12)} x_{(13)} + x_{(13)} x_{(24)} x_{(34)} \\ &- \lambda_t x_{(13)} x_{(23)} x_{(24)} - x_{(12)} x_{(13)} x_{(24)}.
    \end{align*}
    Observe that $\lambda_t^2 + 1 = 0$ if and only if $1-\lambda_t + \lambda_t^2 - \lambda_t^3 = 0$, since $\lambda_t \ne 1$, which hints at some richer structure.
    
    We conclude that if $n=3$, then $V$ is $t$-well-graded if and only if $\subalgQ{N_t} \cdot \gencomp{V}$ is irreducible, which is if and only if $1-\lambda_t+\lambda_t^2 \ne 0$, see \Cref{prop_shapoendo_max_subobject} or \Cref{cor_i_well_graded_shapo_kernel}.
    If $n=4$ then $V$ is $t$-well-graded if and only if \begin{align*}
        (1-\lambda_t+\lambda_t^2)( 1-\lambda_t + \lambda_t^2 - \lambda_t^3)  \ne 0.
    \end{align*}
\end{exa}

We give another example with a slightly different behaviour.

\begin{exa}
     Let $X = \ndZ / 5\ndZ$ be the quandle given by 
     \begin{align*}
         i \quact j = -i + 2j
     \end{align*}
     for all $i,j \in X$. Let $q : X \times X \rightarrow \fK^{\times}$,  $(i,j)\mapsto -1$ be the constant cocycle.
     Assume that $N_t$ is defined by $X$ and $q$ like in \Cref{exa_rack_cocycle}. $\subalgQ{N_t}$ has dimension $1280$ and appears in the list of \cite{Zoo}. Let $x_{i}, i \in X$ be the corresponding basis elements of $N_t$. Assume $\lambda_t \ne 1$.
     We first calculate $\ker \shapoendo_V \cap N_t^2 \cdot \gencomp{V}$.
    For $0\le i \le 4$ we have
    \begin{align*}
        (c_1^m( x_{i} \ot x_{i}))_{0\le m \le 1} =& (x_{i} \ot x_{i}, -x_{i} \ot x_{i} ), \\
        (c_1^m( x_{i} \ot x_{i+1}))_{0\le m \le 4} =& (x_{i} \ot x_{i+1}, - x_{i+2} \ot x_{i}, x_{i+3} \ot x_{i+2}, -x_{i+1} \ot x_{i+3} \\
        & x_{i} \ot x_{i+1} ).
    \end{align*}
    Hence the first orbit vanishes in $N_t^2$.
    Observe that in the second orbit for each $0 \le k \le 4$ we have precisely one element of the form $x_j \ot x_{j+k}$. Hence for each $0\le i \le 4$ we get a different orbit, when starting with $x_i \ot x_{i+1}$, and those are all possible orbits. So in the notation from before the orbits of $N_t^2$ under the action of $\fK \brB_2$ are the five orbits $O_i$, $0 \le i \le 4$, generated by $x_i x_{i+1}$.
    We have that $m_{i} = 3$, $q_{i} = 1$ and by \Cref{lem1_ker_shapo_deg2} 
    \begin{align*}
        O_{i} &= \mathrm{Span}_{\fK} \lbrace x_{i+2} x_{i}, x_{i+3} x_{i+2}, x_{i+1} x_{i+3}\rbrace.
    \end{align*}
    \Cref{prop_ker_shapo_deg2} yields: If $1-\lambda_t+\lambda_t^2 - \lambda_t^3 \ne 0$, then $\ker \shapoendo_V \cap O_{(ijk)} \cdot \gencomp{V} = 0$ and $\ker \shapoendo_V \cap N_t^2 \cdot \gencomp{V} = 0$. If $1-\lambda_t+\lambda_t^2 - \lambda_t^3 = 0$, that is if $\lambda_t^2=-1$, then
    \begin{align} \label{equ_example_nichv5}
    \begin{split}
        &\ker \shapoendo_V \cap O_{(ijk)} \cdot \gencomp{V} \\ =& \left(  x_{i+2} x_{i} + (1-\lambda_t) x_{i+3} x_{i+2} + (1-\lambda_t+\lambda_t^2) x_{i+1} x_{i+3} \right) \cdot \gencomp{V}.
    \end{split}\end{align}
    Calculating the entire $\ker \shapoendo_V \cap \subalgQ{N_t} \cdot \gencomp{V}$ with the help of GAP and Gröbner basis, gives the following: If $ (\lambda_t^2 + 1 ) (1-\lambda_t+\lambda_t^2 - \lambda_t^3 + \lambda_t^4)  \ne 0$, then $\ker \shapoendo_V \cap \subalgQ{N_t } \cdot \gencomp{V} = 0$ and if $\lambda_t^2  =-1$, then $\ker \shapoendo_V \cap \subalgQ{N_t } \cdot \gencomp{V} = A \cdot \gencomp{V}$, where $A$ is the left ideal of $N_t$ generated by the $5$ elements of degree $2$ in \cref{equ_example_nichv5}. Here $A$ has dimension $1202$. Now if $1-\lambda_t+\lambda_t^2 - \lambda_t^3 + \lambda_t^4 = 0$, that is $\lambda_t = - \omega$, where $\omega$ is a primitive fifth root of $1$, we obtain $\ker \shapoendo_V \cap \subalgQ{N_t } \cdot \gencomp{V} = A \cdot \gencomp{V}$, where $A$ is a $640$-dimensional left ideal of $N_t$ generated by $4$ elements of degree $4$ that we obtain by acting on
    \begin{align*}
        &x_3 x_0 x_3 x_4 + (-\lambda_t^2 + \lambda_t^3 - \lambda_t^4) x_2x_3x_0x_3 + \lambda_t^2 x_2 x_0 x_3 x_1 \\ &+ (\lambda_t^3 -\lambda_t^4)x_1x_2x_3x_0 + \lambda_t^3 x_1 x_2 x_0 x_4 + (-\lambda_t + \lambda_t^2 - \lambda_t^3) x_0 x_3 x_1 x_4 \\
        & - \lambda_t^3 x_0 x_2 x_1 x_3 - \lambda_t^2 x_0 x_2 x_0 x_1 + (\lambda_t- \lambda_t^2 + 2 \lambda_t^3 - \lambda_t^4) x_0 x_1 x_3 x_1.
    \end{align*}
   
    We conclude that $V$ is $t$-well-graded if and only if $\subalgQ{N_t} \cdot \gencomp{V}$ is irreducible, which is if and only if  \begin{align*}
        (1-\lambda_t+\lambda_t^2-\lambda_t^3)( 1-\lambda_t + \lambda_t^2 - \lambda_t^3+\lambda_t^4)  \ne 0.
    \end{align*}
\end{exa}

\begin{rema}
    The last section only calculates $\ker \shapoendo_V \cap \subalgQ{N_t} \cdot \gencomp{V}$ explicitly in degree $2$. The method used to prove the main result \Cref{prop_ker_shapo_deg2} relied on \Cref{lem1_ker_shapo_deg2} and \Cref{lem2_ker_shapo_deg2}, where the former makes use of the simple form of the braided symmetrizer $S_2 = 1 + c_1$ and the latter makes use of the simple form of $g_{1,1} = c_1$. To prove similar statements in degree $\ge 3$ will therefor be either a lot more difficult, or require a different approach.
    However the examples discussed above hint heavily, that the statements in degree $\ge 3$ should be similar. In particular, they justify the following conjecture.
\end{rema}

\begin{conjecture} \label{conj_kernel_shapo}
    There is a class of objects $N_t$, such that the following holds:
    Let $k \ge 2$ and $O \subset N_t^k$ be an orbit of $N_i^k$ under the action of $\fK \brB_k$. Then we either have $\ker \shapoendo_V \cap O \cdot \gencomp{V} = 0$ or there exists an $a \in O$, such that $\ker \shapoendo_V \cap O \cdot \gencomp{V} = a \cdot \gencomp{V} \ne 0$. Also there exists some $l \ge 2$, such that the latter holds if and only if $\sum_{i=0}^l (-1)^i \lambda_t^i = 0$.
    
    In particular, there exist some numbers $l_1, \ldots, l_s \ge 2$, $s \ge 1$, such that $V$ is $t$-well-graded (i.e. $\subalgQ{N_t} \cdot \gencomp{V}$ is irreducible) if and only if
    \begin{align*}
        \prod_{j = 1}^s \left( \sum_{i=0}^{l_j} (-1)^i \lambda_t^i \right) \ne 0.
    \end{align*}
\end{conjecture}
\chapter{Nichols systems of diagonal type} \label{chapt_nich_sys_diag}

In this final chapter we are interested in using the previous theory to describe $\ydmod{}$-modules over Nichols systems of diagonal type. To such a Nichols system we can associate a Dynkin diagram, which we therefor study first in \cref{sect_nich_sys_diag_dynkin}. We will also study the reflection theory of such Dynkin diagrams, to then show in \cref{sect_nich_sys_diag_main} that this is consistent with reflecting the $\ydmod{}$-modules (\Cref{prop_dynk_refl_equals_refl}). To do so, we make important use of the previous results about the Shapovalov morphism, most notably \Cref{thm_shapoendo_via_gnk}. We will also apply the results to the induced $\ydmod{}$-modules introduced in \cref{sect_ydmod_nich_sys_induced_yd_module}. Most notably we find, that the irreducibility of such an object can be characterized by a polynomial that is given by the positive roots of the Nichols system. In fact, this polynomial was already found in a similar context in \cite{MR2840165}, where it appeared as the Shapovalov determinant for bicharacters of finite root systems. Here, in a more general context, we obtain this polynomial without much effort from the Shapovalov morphism.
Finally in \cref{sect_nich_sys_diag_examples} we apply the theory to some explicit examples of Nichols systems of diagonal type.

Let $\theta \in \ndN$ and $\mathbb{I}=\lbrace 1,\ldots, \theta \rbrace$.
Let $\fK$ be a field, $H$ a Hopf algebra over $\fK$ with bijective antipode and let $\scr{C}=\ydmod{H}$.

\section{Dynkin diagrams and their reflections} \label{sect_nich_sys_diag_dynkin}

Besides defining and discussing Dynkin diagrams and their reflection theory, we will also associate a set of roots to a Dynkin diagram, which will later coincide with the set of roots of the Nichols system. As it turns out the reflections of the Dynkin diagrams are really only dependent on these roots, see \Cref{prop_refl_dynkin_relies_on_si}. At the end of the chapter in \Cref{defi_shapovalov_determinant} we will already define the polynomial known as the Shapovalov determinant, without giving much context yet. In \cref{sect_nich_sys_diag_main} we will see how such Dynkin diagrams appear in the theory of Nichols systems of diagonal type and what role the Shapovalov determinant plays.

\begin{defi}
    Let $\dynkcat{\theta}$ be the category where objects $D \in \dynkcat{\theta}$ are a collection $\dynk_1,\ldots,\dynk_{\theta} \in \fK^{\times}$ of vertices and edges $\dynk_{jk} \in \fK^{\times}$ for all $1\le j<k \le \theta$,
    and where a morphisms $\dynk \rightarrow \dynk'$ for $\dynk,\dynk' \in \dynkcat{\theta}$ is a triple $(\dynk', f, \dynk)$, where $f \in \Endo(\ndZ^\theta)$.
    For readability we also denote $\dynk_{kj}:=\dynk_{jk}$ for all $1\le j<k\le \theta$, $\dynk\in \dynkcat{\theta}$ and $\dynk_{jj} = \dynk_j^2$ for all $j \in \mathbb{I}$.
    
    Let $V_1,\ldots , V_\theta \in \scr{C}$ be one-dimensional objects and for $1 \le j \le \theta$ let $v_j \in V_j \setminus \lbrace 0 \rbrace$. Moreover for $1 \le j<k \le \theta$ let $r_j,r_{jk} \in \fK^{\times}$ such that
    \begin{align*}
        \brd_{V_j,V_j}^{\scr{C}}(v_j \ot v_j) = r_j v_j \ot v_j \matcom && \brd_{V_k,V_j}^{\scr{C}}\brd_{V_j,V_k}^{\scr{C}}(v_j \ot v_k) = r_{jk} v_j \ot v_k \matdot
    \end{align*}
    The \textbf{Dynkin diagram of $V_1,\ldots, V_\theta$} is the object $\dynk(V_1,\ldots,V_\theta) \in \dynkcat{\theta}$, where $\dynk(V_1,\ldots,V_\theta)_j := r_j$ for all $1\le j \le n$ and $\dynk(V_1,\ldots,V_\theta)_{jk} := r_{jk}$ for all $1\le j<k \le n$.
\end{defi}

\begin{rema}
    Basically an object $\dynk \in \dynkcat{\theta}$ is a labelled complete graph, where labels are in $\fK^\times$. When we visualize $\dynk$, we will omit all edges where $\dynk_{jk} = 1$ for $1 \le j<k \le \theta$.
\end{rema}

\begin{exa}
    Let $\theta=3$ and let $\omega$ be an third root of unity and assume that $\dynk \in \dynkcat{3}$ is the object where $\dynk_1 = \dynk_2 = \omega$, $\dynk_3=-1$, $\dynk_{12}=\omega^{-1}$, $\dynk_{13}=1$ and $\dynk_{23}=\omega^{-2}$. Then the visualisation of $\dynk$ is
\begin{align*}
\begin{tikzpicture}
    \node[dynnode, label={\tiny$\omega$}] (1) at (0,0) {};
    \node[dynnode, label={\tiny$\omega$}] (2) at (1.5,0) {};
    \node[dynnode, label={\tiny$-1$}] (3) at (3,0) {};
    \path[draw,thick]
    (1) edge node[above]{\tiny$\omega^{-1}$} (2)
    (2) edge node[above]{\tiny$\omega^{-2}$} (3);
\end{tikzpicture}
\matdot
\end{align*}
\end{exa}

\begin{defi} \label{defi_qbinom}
    For $n \in \ndN_0$, $0 \le k \le n$ define the following rational functions in $\ndQ(t)$ with indeterminate $t$
    \begin{align*}
        (n)_t &:= \sum_{i=0}^{n-1} t^i \matcom & (n)_t^! &:= \prod_{i=1}^{n} (i)_t \matcom  &\binom{n}{k}_t &:= \frac{(n)_t^! }{(k)_t^! (n-k)_t^! } \matdot
    \end{align*}
    Moreover for $i < 0$ or $i>n$ we let $\binom{n}{i}_t = 0$.
\end{defi}
\begin{rema}
    By \cite{HeSch}, Lemma 1.9.3(3) we obtain $\binom{n}{k}_t \in \ndZ[t]$ for all $n \in \ndN_0$, $0 \le k \le n$.
\end{rema}

\begin{defi}
    Let $n,k \in \ndN_0$ and $a \in \fK$. Define $(n)_a$, $(n)_a^!$ and $\binom{n}{k}_a$ to be the images of  $(n)_t$, $(n)_t^!$ and $\binom{n}{k}_t$, respectively, under the ring homomorphism $\ndZ[t] \rightarrow \fK$ given by mapping $t$ to $a$.
\end{defi}
    
    
    
\begin{rema} \label{rema_binomq_root_of_one}
    Let $m \in \ndN_0$ and $a \in \fK$.
    Since $(1-a) (m)_a = 1-a^m$ we obtain $a^m = 1$ if and only if $a=1$ or $(m)_a = 0$.
\end{rema}

Regarding the reflection theory of Dynkin diagrams we will establish a notion that is consistent with our previous notions.

\begin{defi}
    Let $\dynk \in \dynkcat{\theta}$ and $i \in \mathbb{I}$. For $j \in \mathbb{I}$ denote 
    \begin{align*}
        m_j^{\dynk} := \min \lbrace m \in \ndN_0 \mid (m+1)_{\dynk_{j}} = 0 \rbrace \matdot
    \end{align*}
    Moreover denote $a_{ii}^{D}=2$ and for all $j \in \mathbb{I} \setminus \lbrace i \rbrace$ denote
    \begin{align*}
        a_{ij}^{D} = - \min \lbrace m \in \ndN_0 \, | \, (m + 1)_{\dynk_{i}} (\dynk_{i}^m \dynk_{ij} - 1) = 0 \rbrace
        \matdot
    \end{align*}
    We call $\dynk$ \textbf{$i$-finite} if for all $j \in \mathbb{I} \setminus \lbrace i \rbrace$ we have $-a_{ij}^{D} \in \ndN_0$.
    If $\dynk$ is $i$-finite, then let $s^{\dynk}_{i} \in \Aut(\ndZ^\theta)$ be such that for all $j \in \mathbb{I}$
	\begin{align*}
		s^{\dynk}_{i} ( \alpha_j) = \alpha_j - a^{\dynk}_{ij} \alpha_{i} \matdot
	\end{align*}
\end{defi}


\begin{defi}
    Let $i \in \mathbb{I}$ and $\dynk \in \dynkcat{\theta}$ be $i$-finite.
    We define the \textbf{$i$-th reflection $\refl_i(\dynk)$ of $\dynk$} as follows:
    For $j,k \in \mathbb{I}$ let
    \begin{align*}
        \refl_i(\dynk)_j &= \dynk_j \dynk_{ij}^{-a^{\dynk}_{ij}} \dynk_i^{{a^{\dynk}_{ij}}^2} 
    \end{align*}
    and
    \begin{align*}
        \refl_i(\dynk)_{jk} &= \dynk_{jk} \dynk_{ik}^{-a^{\dynk}_{ij}} \dynk_{ij}^{-a^{\dynk}_{ik}} \dynk_{i}^{2a^{\dynk}_{ij}a^{\dynk}_{ik}} \matdot
    \end{align*}
\end{defi}

\begin{prop}
    Let $i \in \mathbb{I}$ and $\dynk \in \dynkcat{\theta}$ be $i$-finite. The following hold:
    \begin{enumerate}
        \item $m_i^\dynk = m_i^{\refl_i(\dynk)}$, $a_{ij}^\dynk = a_{ij}^{\refl_i(\dynk)}$ for all $j \in \mathbb{I}$ and $s_i^{\dynk} = s_i^{\refl_i(\dynk)}$.
        \item $\refl_i(\dynk)$ is $i$-finite.
        \item $\refl_i \refl_i(\dynk) = \dynk$.
    \end{enumerate}
\end{prop}
\begin{proof}
    (1): Since $\refl_i(\dynk)_i = \dynk_i$ we obtain $m_i^\dynk = m_i^{\refl_i(\dynk)}$ and for $m \in \ndN_0$, $j \in \mathbb{I} \setminus \lbrace i \rbrace$ we have
    \begin{align*}
        \refl_i(\dynk)_i^m \refl_i(\dynk)_{ij} = \dynk_i^{m + 2 a_{ij}^{\dynk}} \dynk_{ij}^{-1} \matdot
    \end{align*}
    By definition we have $(-a_{ij}^\dynk + 1)_{\dynk_i} ( \dynk_i^{-a_{ij}^{\dynk}} \dynk_{ij} - 1 ) = 0$, hence
    \begin{align*}
        &(-a_{ij}^\dynk + 1)_{\refl_i(\dynk)_i} ( \refl_i(\dynk)_i^{-a_{ij}^\dynk} \refl_i(\dynk)_{ij} - 1 ) 
        \\=& (-a_{ij}^\dynk + 1)_{\dynk_i} ( \dynk_i^{a_{ij}^{\dynk}} \dynk_{ij}^{-1} - 1 ) = 0 \matdot
    \end{align*}
    This implies $-a_{ij}^{\refl_i(\dynk)} \le -a_{ij}^{\dynk}$, in particular (2) holds. If $-a_{ij}^{\dynk}=0$, then (1) holds, so assume $-a_{ij}^{\dynk} \ge 1$, i.e. $D_{ij} \ne 1$. 
    We distinguish three cases: 
    
    First case: $(-a_{ij}^{\dynk} + 1)_{\dynk_i} = 0$.
    By definition for $0 \le m < -a_{ij}^\dynk$ we have $(m+1)_{\dynk_i} \ne 0$ and $\dynk_i^m \dynk_{ij} \ne 1$. Since $\dynk_{ij} \ne 1$ and $\dynk_i^{-a_{ij}^\dynk+1} = 1$, see \Cref{rema_binomq_root_of_one}, we also obtain $\dynk_i^{m} \dynk_{ij} \ne 1$ for all $m \in \ndZ \setminus -1+(-a_{ij}^\dynk+1)\ndZ$. Then for all $0 \le m \le -a_{ij}^\dynk-1$ we have $ \dynk_i^{-m-2a_{ij}^{\dynk}} \dynk_{ij} \ne 1$ and thus
    \begin{align*}
        &(m + 1)_{\refl_i(\dynk)_i} ( \refl_i(\dynk)_i^{m} \refl_i(\dynk)_{ij} - 1 ) 
        \\=& (m + 1)_{\dynk_i} ( \dynk_i^{m+2a_{ij}^{\dynk}} \dynk_{ij}^{-1} - 1 ) \ne 0 \matcom
    \end{align*}
    hence $-a_{ij}^{\refl_i(\dynk)} \ge -a_{ij}^{\dynk}$.
    
    Second case: $(-a_{ij}^{\refl_i(\dynk)} + 1)_{\dynk_i} = 0$. Then $-a_{ij}^{\refl_i(\dynk)} \ge -a_{ij}^{\dynk}$ by definition.
    
    Third case: $(-a_{ij}^{\dynk} + 1)_{\dynk_i} \ne 0$ and $(-a_{ij}^{\refl_i(\dynk)} + 1)_{\dynk_i} \ne 0$. By definition we have $\dynk_i^{a_{ij}^\dynk} = \dynk_{ij}$ (in particular $\dynk_i \ne 1$) and
    \begin{align*}
         0=\refl_i(\dynk)_i^{-a_{ij}^{\refl_i(\dynk)}} \refl_i(\dynk)_{ij}-1=\dynk_i^{-a_{ij}^{\refl_i(\dynk)} + 2 a_{ij}^{\dynk}} \dynk_{ij}^{-1} - 1
    \end{align*}
    i.e. $\dynk_i^{-a_{ij}^{\refl_i(\dynk)} + 2 a_{ij}^{\dynk}} = \dynk_{ij}$. Combining the two we get $\dynk_i^{- a_{ij}^{\dynk}+a_{ij}^{\refl_i(\dynk)} } = 1$. 
    Since $0 \le - a_{ij}^{\dynk}+a_{ij}^{\refl_i(\dynk)} \le -a_{ij}^\dynk$ and $\dynk_i^m \ne 1$ for all $1\le m\le -a_{ij}$, see \Cref{rema_binomq_root_of_one}, we conclude $- a_{ij}^{\dynk}+a_{ij}^{\refl_i(\dynk)} = 0$.
    
    In all cases we conclude $a_{ij}^{\refl_i(\dynk)} = a_{ij}^{\dynk}$ and hence (1) holds.
    Finally to check that~(3) holds is straight forward using~(1).
\end{proof}

\begin{lem} \label{lem_rel_refl_dynk_with_mi}
    Let $i \in \mathbb{I}$ and $\dynk \in \dynkcat{\theta}$ be $i$-finite. We assume that $m_i^\dynk \in \ndN_0$. Then for $j \in \mathbb{I}$ the following hold:
    \begin{align*}
        \dynk_{ij}^{m^\dynk_i + a^\dynk_{ij}} \dynk_i^{-a^\dynk_{ij}(1+2m^\dynk_i+a^\dynk_{ij})} &= 1 \matcom \tag{1}
        \\ \dynk_j \dynk_{ij}^{m^\dynk_i} \dynk_i^{-a^\dynk_{ij}(2m^\dynk_i +1)} &= \refl_i(\dynk)_j  \matdot \tag{2}
    \end{align*}
\end{lem}
\begin{proof}
    (1): 
    Regarding \Cref{rema_binomq_root_of_one}, we obtain that $\dynk_i^{m^\dynk_i +1} = 1$. Hence $ \dynk_i^{-a^\dynk_{ij}(1+2m^\dynk_i+a^\dynk_{ij})}=\dynk_i^{-a^\dynk_{ij}(m^\dynk_i+ a^{\dynk}_{ij})}$ and it is enough to show
    \begin{align*}
        \left( \dynk_{ij} \dynk_i^{-a^\dynk_{ij}} \right)^{m^\dynk_i+ a^{\dynk}_{ij}} = 1 \matdot
    \end{align*}
    This is trivial if $m^\dynk_i = - a^{\dynk}_{ij}$. If $m^\dynk_i \ne - a^{\dynk}_{ij}$, then  $(-a^\dynk_{ij}+1)_{\dynk_i} \ne 0$, hence by definition $\dynk_{ij} \dynk_i^{-a^\dynk_{ij}}=1$.
    
    (2): Using (1) we obtain
    \begin{align*}
        \dynk_j \dynk_{ij}^{m^\dynk_i} \dynk_i^{-a^\dynk_{ij}(2m^\dynk_i +1)} 
        &= \dynk_j \dynk_{ij}^{m^\dynk_i -m^\dynk_i -a^\dynk_{ij}} \dynk_i^{-a^\dynk_{ij}(2m^\dynk_i +1 -1 -2m^\dynk_i -a_{ij})}
        \\ &= \dynk_j \dynk_{ij}^{-a^\dynk_{ij}} \dynk_i^{{a^\dynk_{ij}}^2} \matdot
    \end{align*}
    By definition this is equal to $\refl_i(\dynk)_j$.
\end{proof}

\begin{defi}
	Let $k \in \ndN_0$, $i_1, \ldots, i_k \in \mathbb{I}$ and $D\in \dynkcat{\theta}$.
	\begin{enumerate}
		\item We say $\dynk$ \textbf{admits the reflection sequence $(i_1,\ldots, i_k)$}, if $k=0$ or $\dynk$ is $i_1$-finite and $\refl_{i_1}(\dynk)$ admits the reflection sequence $(i_2,\ldots, i_k)$.
		\item We say $\dynk$ \textbf{admits all reflections}, if for all $n \in \ndN_0$, $i \in \mathbb{I}^n$,
		$\dynk$ admits the reflection sequence $i$.
	\end{enumerate}
\end{defi}

\begin{notation}
	Let $k \in \ndN_0$, $i_1, \ldots, i_k \in \mathbb{I}$ and $D\in \dynkcat{\theta}$.
	Assume $\dynk$ admits the reflection sequence $(i_1,\ldots,i_{k-1})$.
	First denote for all $j \in \mathbb{I}$
	\begin{align*}
	    \dynk_{i} := \refl_{(i_1,\ldots,i_{k-1})}(\dynk)_{i_k} \matcom
	    && \dynk_{i,j} := \refl_{(i_1,\ldots,i_{k-1})}(\dynk)_{i_k j}
	\end{align*}
    Moreover denote $m^{\dynk}_{()}=0$ and if $k \ge  1$ denote
	\begin{align*}
		m^{\dynk}_{i} := m^{\refl_{(i_1,\ldots,i_{k-1})}(\dynk)}_{i_k} \matdot
	\end{align*}
	Finally let $a^{\dynk}_{i,i_k} = 2$ and for $j \in \mathbb{I}\setminus \lbrace i_k \rbrace$ define
	\begin{align*}
		a^{\dynk}_{i,j} :=& \, a^{\refl_{(i_1,\ldots,i_{k-1})}(\dynk)}_{i_k j}
		\matdot
	\end{align*}
	If $\refl_{(i_1,\ldots,i_{k-1})}(\dynk)$ is $i_k$-finite, 
	then let $s^{\dynk}_{(i_1,\ldots,i_{k-1}),i_k} \in \Aut(\ndZ^\theta)$ be such that for all $j \in \mathbb{I}$
	\begin{align*}
		s^{\dynk}_{(i_1,\ldots,i_{k-1}),i_k} ( \alpha_j) = \alpha_j - a^{\dynk}_{i,j} \alpha_{i_k} \matdot
	\end{align*}
	Finally define $s^{\dynk}_i=s^{\dynk}_{(),i_1} s^{\dynk}_{(i_1),i_2} \cdots s^{\dynk}_{(i_1,\ldots,i_{k-1}),i_k}$ and $s^{\dynk}_{()}=\id_{\ndZ^\theta}$.
\end{notation}

\begin{defi}
    Let $k \in \ndN_0$, $i_1, \ldots, i_k \in \mathbb{I}$ and $D\in \dynkcat{\theta}$. If $\dynk$ admits the reflection sequence $i$, then for $j \in \mathbb{I}$ the element $s_i^{\dynk}(\alpha_j) \in \ndZ^\theta$ is called a \textbf{root of~$\dynk$}. The set of all roots is denoted $\roots{\dynk}$. 
    Furthermore we denote $\rootspos{\dynk} = \roots{\dynk} \cap \ndN_0^\theta$ for the set of \textbf{positive roots of $\dynk$}. 
\end{defi}

\begin{notation}
    Let $\dynk \in \dynkcat{\theta}$. Let $\indchar{\dynk} : \ndZ^\theta \rightarrow \fK^\times$ be the character (i.e. group homomorphism) that is given by $\indchar{\dynk}(\alpha_j) = \dynk_j$ for all $j \in \mathbb{I}$. Moreover let  
    \begin{align*}
        \indchar{\dynk} : \ndZ^\theta \times \ndZ^\theta \rightarrow \fK^\times
    \end{align*}
    be the bicharacter (i.e. group homomorphism), such that $\indchar{\dynk}  (\alpha_j, \alpha_j) = \dynk_j$ for all $j \in \mathbb{I}$, as well as $\indchar{\dynk}  (\alpha_j, \alpha_k)= \dynk_{jk}$ and $\indchar{\dynk} (\alpha_k, \alpha_j) = 1$ for all $1\le j < k \le \theta$ (while we are indeed using the symbol $\indchar{\dynk}$ for two different mappings here, it should always be clear which one we are refering to).
    Finally define the bicharacter
    \begin{align*}
        \bichar{\dynk} : \ndZ^\theta \times \ndZ^\theta \rightarrow \fK^\times
    \end{align*}
    that is given by $\bichar{\dynk} (\alpha_j, \alpha_k) = \dynk_{jk}$ for all $j,k \in \mathbb{I}$.
\end{notation}

\begin{lem} \label{lem_bichar_indchar}
    For $\dynk \in \dynkcat{\theta}$, $n=(n_1,\ldots,n_\theta), n'=(n'_1,\ldots,n'_\theta) \in \ndZ^\theta$ we have
    \begin{align*}
        \bichar{\dynk} (n, n') = \indchar{\dynk} (n,n') \indchar{\dynk} (n',n) \matdot
    \end{align*}
\end{lem}
\begin{proof}
    The claim follows from 
    \begin{align*}
        \indchar{\dynk} (n,n') = \left( \prod_{1\le j \le \theta} \dynk_j^{n_j+n'_j}\right) \left( \prod_{1\le j < k \le \theta} \dynk_{jk}^{n_j+n'_k }\right)
    \end{align*}
    as well as $\dynk_j^2 = \dynk_{jj}$ for all $j \in \mathbb{I}$.
\end{proof}

\begin{prop} \label{prop_refl_dynkin_relies_on_si}
    Let $m \in \ndN_0$, $i_1, \ldots, i_m \in \mathbb{I}^m$ and assume $\dynk \in \dynkcat{\theta}$ admits the reflection sequence $i$.
	Then for $j,k \in \mathbb{I}$ we have
	\begin{align*}
	    \refl_i(\dynk)_j &=  \indchar{\dynk} \left( (s_i^{\dynk}(\alpha_j),s_i^{\dynk}(\alpha_j) \right) \matcom 
	    \\ \refl_i(\dynk)_{jk} &= \bichar{\dynk} \left( s_i^{\dynk}(\alpha_j),s_i^{\dynk}(\alpha_k) \right) \matdot 
	\end{align*}
    In particular if $s^{\dynk}_i = \id_{\ndZ^\theta}$, then $\refl_i(\dynk)=\dynk$.
\end{prop}
\begin{proof}
    We do induction on $m$: The claim is trivial for $m=0$. Assume $m\ge 1$ and that the claim holds for $i' = (i_1,\ldots,i_{m-1})$. 
    For $j \in \mathbb{I}$ we have 
    \begin{align*}
        s^{\dynk}_i(\alpha_j) = s_{i'}(\alpha_j) - a_{i,j}^{\dynk} s_{i'}(\alpha_{i_m}) \matdot 
    \end{align*}
    Then by induction hypothesis for $j,k \in \mathbb{I}$ we have 
    \begin{align*}
        \bichar{\dynk} \left(s_i^{\dynk}(\alpha_j),s_{i'}^{\dynk}(\alpha_k) \right) &= \bichar{\dynk} \left(s_{i'}^{\dynk}(\alpha_j),s_{i'}^{\dynk}(\alpha_k) \right) \bichar{\dynk} \left( s_{i'}^{\dynk}(\alpha_{i_m}),s_{i'}^{\dynk}(\alpha_k) \right)^{-a^{\dynk}_{i,j}}
        \\&= \refl_{i'}(\dynk)_{jk} \refl_{i'}(\dynk)_{i_m k}^{-a^{\dynk}_{i,j}}
    \end{align*}
    as well as
     \begin{align*}
        \bichar{\dynk} \left(s_i^{\dynk}(\alpha_j),s_{i'}^{\dynk}(\alpha_{i_m}) \right) 
        &= \refl_{i'}(\dynk)_{j i_m} \refl_{i'}(\dynk)_{i_m}^{-2a^{\dynk}_{i,j}} \matdot
    \end{align*}
    Combing the above we obtain
    \begin{align*}
        \bichar{\dynk} \left(s_i^{\dynk}(\alpha_j),s_i^{\dynk}(\alpha_{k}) \right) 
        &= \bichar{\dynk} \left(s_i^{\dynk}(\alpha_j),s_{i'}^{\dynk}(\alpha_k) \right) \bichar{\dynk} \left(s_i^{\dynk}(\alpha_j),s_{i'}^{\dynk}(\alpha_{i_m}) \right)^{-a_{i,k}^{\dynk}}
        \\ &= \refl_{i'}(\dynk)_{jk} \refl_{i'}(\dynk)_{i_m k}^{-a^{\dynk}_{i,j}} \refl_{i'}(\dynk)_{j i_m}^{-a^{\dynk}_{i,k}} \refl_{i'}(\dynk)_{i_m}^{2a^{\dynk}_{i,j}a^{\dynk}_{i,k}}
        \\ &= \refl_i(\dynk)_{jk} \matdot
    \end{align*}
    This proves the second relation of the claim. The first relation is obtained similarly: By induction hypothesis we have
    \begin{align*}
        \indchar{\dynk} \left(s_i^{\dynk}(\alpha_j),s_{i'}^{\dynk}(\alpha_{j}) \right) 
        &= \refl_{i'}(\dynk)_{j} \, \indchar{\dynk} \left(s_{i'}^{\dynk}(\alpha_{i_m}),s_{i'}^{\dynk}(\alpha_{j}) \right)^{-a^\dynk_{i,j}}
    \end{align*}
    as well as
    \begin{align*}
        \indchar{\dynk} \left(s_i^{\dynk}(\alpha_j),s_{i'}^{\dynk}(\alpha_{i_m}) \right) 
        &= \indchar{\dynk} \left(s_{i'}^{\dynk}(\alpha_j),s_{i'}^{\dynk}(\alpha_{i_m}) \right) \refl_{i'}(\dynk)_{i_m}^{-a^{\dynk}_{i,j}} \matdot
    \end{align*}
    Hence combining the above and using \Cref{lem_bichar_indchar} as well as the second relation from the claim we obtain
    \begin{align*}
        & \indchar{\dynk} \left( (s_i^{\dynk}(\alpha_j),s_i^{\dynk}(\alpha_j) \right)
        \\=& \refl_{i'}(\dynk)_{j}   \Big( \indchar{\dynk} \left(s_{i'}^{\dynk}(\alpha_{i_m}),s_{i'}^{\dynk}(\alpha_{j}) \right) \indchar{\dynk} \left(s_{i'}^{\dynk}(\alpha_j),s_{i'}^{\dynk}(\alpha_{i_m}) \right) \Big)^{-a^\dynk_{i,j}} \refl_{i'}(\dynk)_{i_m}^{{a^{\dynk}_{i,j}}^2}
        \\=& \refl_{i'}(\dynk)_{j} \,  \bichar{\dynk} \left(s_{i'}^{\dynk}(\alpha_{i_m}),s_{i'}^{\dynk}(\alpha_{j}) \right)^{-a^\dynk_{i,j}} \refl_{i'}(\dynk)_{i_m}^{{a^{\dynk}_{i,j}}^2}
        \\=& \refl_{i'}(\dynk)_{j} \refl_{i'}(\dynk)_{i_m j}^{-a^{\dynk}_{i,j}} \refl_{i'}(\dynk)_{i_m}^{{a^{\dynk}_{i,j}}^2} = \refl_i(\dynk)_j \matcom
    \end{align*}
    implying the first relation.
\end{proof}

The following definition will be useful later when we describe Dynkin diagrams of induced $\ydmod{}$-modules.

\begin{defi}
    Let $\dynk \in \dynkcat{\theta}$ and $\fK(t_1,\ldots,t_k)$ the field of rational functions over the indeterminates $t_1,\ldots,t_k$.
    Define $\extdynk{\dynk} \in \dynkcatfld{\fK(t_1,\ldots,t_\theta)}{\theta+1}$ to be the object where $\extdynk{\dynk}_j = \dynk_j$, $\extdynk{\dynk}_{jk} = \dynk_{jk}$ for all $j,k \in \mathbb{I}$ as well as $\extdynk{\dynk}_{\theta+1}=1$ and $\extdynk{\dynk}_{j (\theta+1)}=t_j$.
\end{defi}

\begin{lem} \label{lem_ext_dynk}
    Let $\dynk \in \dynkcat{\theta}$ and $i \in \mathbb{I}$. The following hold:
    \begin{enumerate}
        \item $-a_{i (\theta+1)}^{\extdynk{\dynk}} = m_i^{\dynk}$.
        \item For $i \in \mathbb{I}$ we have that $\extdynk{\dynk}$ is $i$-finite if and only if $\dynk$ is $i$-finite and $m_i^\dynk \in \ndN_0$.
    \end{enumerate}
\end{lem}
\begin{proof}
    We have $\dynk_i^m t_i \ne 1$ for all $m \in \ndN_0$. Hence
    \begin{align*}
        -a_{i (\theta+1)}^{D} = \min \lbrace m \in \ndN_0 \, | \, (m + 1)_{\dynk_{i}} = 0 \rbrace = m_i^\dynk \matdot
    \end{align*}
    This implies (1) and (2).
\end{proof}

We will now do a construction similar to \Cref{nota_bety_sys} and point out its connection with $\extdynk{\dynk}$ for $\dynk \in \dynkcat{\theta}$.

\begin{notation}
    Let $\dynk \in \dynkcat{\theta}$, $k \in \ndN_0$, $i = (i_1,\ldots, i_k) \in \mathbb{I}^k$ and assume $\dynk$ admits the reflection sequence $i$.
	For $m=(m_1,\ldots,m_k) \in \ndN_0^k$ let
	\begin{align*}
	\beta^{\dynk}_{i,m} =& - \sum_{r=1}^{k} m_r s^{\dynk}_{(i_1,\ldots,i_r)}(\alpha_{i_r})   \in \ndZ_0^\theta \matdot
	\end{align*}
	Moreover if $m^{\dynk}_{(i_1,\ldots,i_r)} \in \ndN_0$ for all $1 \le r \le k-1$, then for $0 \le m \le m^{\dynk}_i$ let $\beta^{\dynk}_{i,m} := \beta^{\dynk}_{i,(m^{\dynk}_{(i_1)},\ldots,m^{\dynk}_{(i_1,\ldots,i_{k-1})}),m)}$, where $\beta^{\dynk}_{(),0}=0$. If also $m^{\dynk}_i \in \ndN_0$, then let $\beta^{\dynk}_i := \beta^{\dynk}_{i,m^{\dynk}_i}$.
\end{notation}

\begin{prop} \label{prop_ext_dynk_beta}
    Let $k \in \ndN_0$, $i=(i_1,\ldots,i_k) \in \mathbb{I}^k$, let $\dynk \in \dynkcat{\theta}$, assume $\dynk$ admits the reflection sequence $i$ and that $m^{\dynk}_{(i_1,\ldots,i_r)} \in \ndN_0$ for all $1 \le r \le k$.
    Then we have $s_i^{\extdynk{\dynk}}( \alpha_j) = s_i^{\dynk}(\alpha_j)$ for all $j \in \mathbb{I}$ and
    \begin{align*}
        s_i^{\extdynk{\dynk}} ( \alpha_{\theta+1} ) = \alpha_{\theta+1} + \beta^\dynk_{i} \matdot
    \end{align*}
\end{prop}
\begin{proof}
    We prove the claim by induction on $k$: The claim is trivial for $k=0$. Assume $k \ge 1$ and that the claim holds for $i' = (i_1,\ldots,i_{k-1})$.
    By induction hypothesis we have for $j \in \mathbb{I}$
    \begin{align*}
        s_i^{\extdynk{\dynk}}( \alpha_j) = s^{\extdynk{\dynk}}_{i'} (\alpha_j - a_{i, j}^{\dynk} \alpha_{i_k})= s^{\dynk}_{i'} (\alpha_j - a_{i, j}^{\dynk} \alpha_{i_k}) = s_i^{\dynk}(\alpha_j) \matdot
    \end{align*}
    Thus by induction hypothesis and \Cref{lem_ext_dynk}(2) we have
    \begin{align*}
        s_i^{\extdynk{\dynk}} ( \alpha_{\theta+1} ) = s_{i'}^{\extdynk{\dynk}} ( \alpha_{\theta+1} + m_i^{\dynk} \alpha_{i_k} )
        = \alpha_{\theta+1} + \beta^\dynk_{i'} + m_i^{\dynk} s_{i'}^{\dynk} (\alpha_{i_k})
    \end{align*}
    Since $s_{i'}^{\dynk} (\alpha_{i_k}) = - s_{i}^{\dynk} (\alpha_{i_k})$ we obtain $\beta^\dynk_{i'} + m_i^{\dynk} s_{i'}^{\dynk} (\alpha_{i_k}) = \beta_i^\dynk$.
\end{proof}

\begin{lem} \label{lem_extdynk_only_dep_on_root}
    Let $k \in \ndN_0$, $i=(i_1,\ldots,i_k) \in \mathbb{I}^k$, let $\dynk \in \dynkcat{\theta}$, assume $D$ admits the reflection sequence $i$ and $m^{\dynk}_{(i_1,\ldots,i_r)} \in \ndN_0$ for all $1 \le r \le k$. Then for $j\in \mathbb{I}$ the following holds:
    \begin{align*}
        \bichar{\dynk}\left( \beta_i^\dynk, s_i^\dynk(\alpha_j) \right) \indchar{\dynk}\left( s_i^\dynk(\alpha_j)\right) = \refl_i(\dynk)_j \matdot
    \end{align*}
\end{lem}
\begin{proof}
    We do induction on $k$: The claim is trivial for $k=0$. Assume $k\ge 1$ and that the claim holds for $i' = (i_1,\ldots,i_{k-1})$. Since $\beta_i = \beta_{i'} + m^\dynk_i s^\dynk_{i'} (\alpha_{i_k})$ and using \Cref{prop_refl_dynkin_relies_on_si} we obtain for $j \in \mathbb{I}$
    \begin{align*}
        &\bichar{\dynk}\left( \beta_i^\dynk, s_i^\dynk(\alpha_j) \right) \indchar{\dynk}\left( s_i^\dynk(\alpha_j)\right)
        \\=& \bichar{\dynk}\left( \beta_{i'}^\dynk, s_{i'}^\dynk(\alpha_j) \right) \indchar{\dynk}\left( s_{i'}^\dynk(\alpha_j)\right) \Big( \bichar{\dynk}\left( \beta_{i'}^\dynk, s_{i'}^\dynk(\alpha_{i_k}) \right) \indchar{\dynk}\left( s_{i'}^\dynk(\alpha_{i_k})\right) \Big)^{-a^\dynk_{i,j}} \\ &  \refl_{i'}(\dynk)_{i_k j}^{m^\dynk_i} \refl_{i'}(\dynk)_{i_k}^{-2 a^\dynk_{i,j} m^\dynk_i} 
        \matdot
    \end{align*}
    Hence we obtain by induction hypothesis
    \begin{align*}
        & \bichar{\dynk}\left( \beta_i^\dynk, s_i^\dynk(\alpha_j) \right) \indchar{\dynk}\left( s_i^\dynk(\alpha_j)\right)
        \\ =& \refl_{i'}(\dynk)_j \refl_{i'}(\dynk)_{i_k}^{-a^\dynk_{i,j}} \refl_{i'}(\dynk)_{i_k j}^{m^\dynk_i} \refl_{i'}(\dynk)_{i_k}^{-2 a^\dynk_{i,j} m^\dynk_i} \matdot
    \end{align*}
    By \Cref{lem_rel_refl_dynk_with_mi}(2) we obtain that this coincides with $\refl_i(\dynk)_j$.
\end{proof}

\begin{notation}
    For $v= \sum_{i=1}^{\theta} n_i \alpha_i \in \ndZ^\theta$, $n_1,\ldots,n_\theta \in \ndZ$ we define the monomials $t_v := \prod_{i=1}^{\theta} t_i^{n_i}$ in $\fK(t_1,\ldots,t_\theta)$, such that $t_0=1$.
\end{notation}

\begin{prop} \label{lem_ext_dynk_explicit}
    Let $k \in \ndN_0$, $i=(i_1,\ldots,i_k) \in \mathbb{I}^k$, let $\dynk \in \dynkcat{\theta}$ and assume $\extdynk{\dynk} \in \dynkcatfld{\fK(t_1,\ldots,t_\theta)}{\theta+1}$ admits the reflection sequence $i$. 
    Then for $j \in \mathbb{I}$ the following two relations hold
    \begin{align*}
        \refl_i(\extdynk{\dynk})_{j (\theta+1)} &= t_{s_i^{\dynk}(\alpha_j)} \bichar{\dynk} \left( \beta^\dynk_i, s_i^{\dynk}(\alpha_j) \right) 
        \\ &= t_{s_i^{\dynk}(\alpha_j)}
        \indchar{\dynk}\left(-s^\dynk_i(\alpha_j) \right) \indchar{\dynk} \left( s_i^{\dynk}(\alpha_j),s_i^{\dynk}(\alpha_j) \right) 
        \matdot
    \end{align*}
    In particular $\refl_i(\extdynk{\dynk})_{j (\theta+1)}$ only relies on the root $s_i^{\dynk}(\alpha_j)$.
\end{prop}
\begin{proof}
    Using \Cref{lem_ext_dynk}(2), \Cref{prop_refl_dynkin_relies_on_si} and \Cref{prop_ext_dynk_beta} we get
    \begin{align*}
        \refl_i(\extdynk{\dynk})_{(\theta+1) j} &= \bichar{\extdynk{\dynk}} \left( s_i^{\extdynk{\dynk}}(\alpha_{\theta+1}), s_i^{\extdynk{\dynk}}(\alpha_j) \right)
        \\ 
        &=\bichar{\extdynk{\dynk}} \left( \alpha_{\theta+1}, s_i^{\dynk}(\alpha_j) \right) \bichar{\extdynk{\dynk}} \left( \beta^\dynk_i, s_i^{\dynk}(\alpha_j) \right) 
    \end{align*}
    Now $\extdynk{\dynk}_{(\theta+1) k} = t_k$ for all $k \in \mathbb{I}$, hence
    \begin{align*}
        \bichar{\extdynk{\dynk}} \left( \alpha_{\theta+1}, s_i^{\dynk}(\alpha_j) \right) = t_{s_i^{\dynk}(\alpha_j)} \matdot
    \end{align*}
    Since both $s_i^{\dynk}(\alpha_j)$ and $\beta_i^\dynk$ are linear combinations of $\alpha_1,\ldots,\alpha_\theta$, we obtain using \Cref{lem_extdynk_only_dep_on_root} that
    \begin{align*}
        \bichar{\extdynk{\dynk}} \left( \beta^\dynk_i, s_i^{\dynk}(\alpha_j) \right)
        = \bichar{\dynk} \left( \beta^\dynk_i, s_i^{\dynk}(\alpha_j) \right) 
        = \indchar{\dynk}\left( -s_i^\dynk(\alpha_j)\right) \refl_i(\dynk)_j \matdot
    \end{align*}
    Finally $\refl_i(\dynk)_j = \indchar{\dynk} \left( s_i^{\dynk}(\alpha_j),s_i^{\dynk}(\alpha_j) \right)$ by \Cref{prop_refl_dynkin_relies_on_si}.
\end{proof}

\begin{defi} \label{defi_shapovalov_determinant}
    Let $\dynk \in \dynkcat{\theta}$ and $\gamma \in \ndZ^\theta$. If
    \begin{align*}
        m_\gamma := \min \lbrace m \in \ndN_0 \mid (m+1)_{\indchar{\dynk}(\gamma,\gamma)} = 0 \rbrace \in \ndN_0 \matcom
    \end{align*}
    then we define the rational function
    \begin{align*}
        \shapodiag{\dynk}_\gamma := \prod_{m=1}^{m_\gamma} \left( t_{\gamma} -
        \indchar{\dynk}\left(\gamma \right) \indchar{\dynk} \left( \gamma,\gamma \right)^{-m} \right) \in \fK(t_1,\ldots,t_\theta)
    \end{align*}
    Finally if $\roots{\dynk}$ is finite and $m_\gamma \in \ndN_0$ for all $\gamma \in \rootspos{\dynk}$, then define the polynomial
    \begin{align*}
        \shapodiag{\dynk} := \prod_{\gamma \in \rootspos{\dynk}} \shapodiag{\dynk}_\gamma \in \fK[t_1,\ldots,t_\theta] \matdot
    \end{align*}
\end{defi}

\begin{rema}
    The polynomial in \Cref{defi_shapovalov_determinant} also appears in \cite{MR2840165} as the Shapovalov determinant for bicharacters of finite root systems. We will the importance of this polynomial later in \Cref{prop_ydind_refl_admit_equiv} and \Cref{cor_ydind_refl_admit_equiv}.
\end{rema}

\begin{lem} \label{lem_shapodiag_roots_negative_root}
    Let $\dynk \in \dynkcat{\theta}$, $\gamma \in \rootspos{\dynk}$ and assume
    \begin{align*}
        m_\gamma := \min \lbrace m \in \ndN_0 \mid (m+1)_{\indchar{\dynk}(\gamma,\gamma)} = 0 \rbrace \in \ndN_0 \matdot
    \end{align*}
    Let $r_1, \ldots, r_\theta \in \fK$.
    Then $\shapodiag{\dynk}_\gamma \in \fK[t_1,\ldots,t_\theta]$ and $\shapodiag{\dynk}_\gamma (r_1, \ldots, r_\theta) = 0$ if and only if $\shapodiag{\dynk}_{-\gamma}(r_1, \ldots, r_\theta) = 0$.
\end{lem}
\begin{proof}
    We have that $\shapodiag{\dynk}_{-\gamma}(r_1, \ldots, r_\theta) \ne 0$ if and only if
    \begin{align*}
        t_{-\gamma} (r_1, \ldots, r_\theta) -
        \indchar{\dynk}\left(-\gamma \right) \indchar{\dynk} \left( -\gamma,-\gamma \right)^{-m} \ne 0 
    \end{align*}
    for all $1 \le m \le m_\gamma $. This is equivalent to
    \begin{align*}
        t_{\gamma} (r_1, \ldots, r_\theta) -
        \indchar{\dynk}\left(\gamma \right) \indchar{\dynk} \left( \gamma,\gamma \right)^{m} \ne 0 
    \end{align*}
    for all $1 \le m \le m_\gamma $. By \Cref{rema_binomq_root_of_one} we have $\indchar{\dynk} \left( \gamma,\gamma \right)^{m_\gamma +1} = 1$, hence the above is equivalent to
    \begin{align*}
        t_{\gamma} (r_1, \ldots, r_\theta) -
        \indchar{\dynk}\left(\gamma \right) \indchar{\dynk} \left( \gamma,\gamma \right)^{m-m_\gamma-1} \ne 0 \matdot
    \end{align*}
    for all $1 \le m \le m_\gamma$. This is equivalent to $\shapodiag{\dynk}_\gamma (r_1, \ldots, r_\theta) \ne 0$.
\end{proof}

\begin{prop} \label{prop_shapodiag_root_equivalence}
    Let $\dynk \in \dynkcat{\theta}$, $k \in \ndN_0$, $i \in \mathbb{I}^k$ and assume $\extdynk{\dynk}$ admits the reflection sequence $i$. Let $j \in \mathbb{I}$ and assume $m^{\refl_i(\dynk)}_j \in \ndN_0$, as well as $s_i^\dynk(\alpha_j) \in \ndN_0^\theta \cup -\ndN_0^\theta$. Then for $r_1,\ldots,r_\theta \in \fK$ the following are equivalent:
    \begin{enumerate}
        \item For all $0 \le m \le m_j^{\refl_i(\dynk)}-1$ we have
        \begin{align*}  
            \refl_{i}(\dynk)_{j}^{m} \refl_i(\extdynk{\dynk})_{j (\theta+1)} (r_1, \ldots, r_\theta) \ne 1 \matdot
        \end{align*}
        \item $\shapodiag{\dynk}_{s_i^\dynk(\alpha_j)} (r_1, \ldots, r_\theta) \ne 0$.
        \item $\shapodiag{\dynk}_{-s_i^\dynk(\alpha_j)} (r_1, \ldots, r_\theta) \ne 0$.
    \end{enumerate}
\end{prop}
\begin{proof}
    Denote $\gamma = s_i^\dynk(\alpha_j)$. First observe that by \Cref{prop_refl_dynkin_relies_on_si} we have 
    \begin{align*}
        m^{\refl_i(\dynk)}_j = \min \lbrace m \in \ndN_0 \mid (m+1)_{\indchar{\dynk}(\gamma,\gamma)} = 0 \rbrace \in \ndN_0 \matcom
    \end{align*}
    hence $\shapodiag{\dynk}_{s_i^\dynk(\alpha_j)}$ is well-defined. The equivalence of (2) and (3) is shown in \Cref{lem_shapodiag_roots_negative_root}.
    Now by \Cref{prop_refl_dynkin_relies_on_si} and \Cref{lem_ext_dynk_explicit} we obtain for $0 \le m \le m_j^{\refl_i(\dynk)}-1$ 
    \begin{align*}
        \refl_{i}(\dynk)_{j}^{m} \refl_i(\extdynk{\dynk})_{j (\theta+1)}
        = t_{\gamma}
        \indchar{\dynk}\left(-\gamma \right) \indchar{\dynk} \left( \gamma,\gamma \right)^{m+1} \matdot
    \end{align*}
    Hence $\refl_{i}(\dynk)_{j}^{m} \refl_i(\extdynk{\dynk})_{j (\theta+1)} (r_1, \ldots, r_\theta) \ne 1$ for all $0 \le m \le m_j^{\refl_i(\dynk)}-1$ if and only if
    \begin{align*}
        t_{\gamma} (r_1, \ldots, r_\theta) -
        \indchar{\dynk}\left(\gamma \right) \indchar{\dynk} \left( \gamma,\gamma \right)^{-m} \ne 0 \matdot
    \end{align*} 
    for all $1 \le m \le m_j^{\refl_i(\dynk)}=m_\gamma$. This proves the equivalence of (1) and~(2).
\end{proof}

\begin{cor} \label{cor_shapodiag_equivalence}
    Let $\dynk \in \dynkcat{\theta}$, assume $\extdynk{\dynk}$ admits all reflections, that $\roots{\dynk}$ is finite and $\roots{\dynk} = \rootspos{\dynk} \cup -\rootspos{\dynk}$. 
    Then \begin{align*}
        m_\gamma := \min \lbrace m \in \ndN_0 \mid (m+1)_{\indchar{\dynk}(\gamma,\gamma)} = 0 \rbrace \in \ndN_0
    \end{align*}
    for all $\gamma \in \roots{\dynk}$ and for $r_1,\ldots,r_\theta \in \fK$ the following are equivalent:
    \begin{enumerate}
        \item $\shapodiag{\dynk}(r_1,\ldots,r_\theta) \ne 0$.
        \item For all $k \in \ndN_0$, $i \in \mathbb{I}^k$, $j \in \mathbb{I}$, $0 \le m \le m_j^{\refl_i(\dynk)}-1$ we have
        \begin{align*}  
            \refl_{i}(\dynk)_{j}^{m} \refl_i(\extdynk{\dynk})_{j (\theta+1)} (r_1, \ldots, r_\theta) \ne 1 \matdot
        \end{align*}
    \end{enumerate}
\end{cor}
\begin{proof}
    Let $\gamma \in \roots{\dynk}$, $k \in \ndN_0$, $i \in \mathbb{I}^k$, $j \in \mathbb{I}$, such that $\gamma = s^{\dynk}_i (\alpha_j)$. Then by \Cref{prop_refl_dynkin_relies_on_si} and  \Cref{lem_ext_dynk}(2) we have 
    \begin{align*}
        m_\gamma = m^{\refl_i(\dynk)}_j \in \ndN_0 \matdot
    \end{align*}
    The equivalence of (1) and (2) is implied by \Cref{prop_shapodiag_root_equivalence}.
\end{proof}

\begin{defi}
    Assume $D\in \dynkcat{\theta}$ admits all reflections. The \textbf{reflection graph of $\dynk$} is a graph with vertex set $\lbrace \refl_i(\dynk) \mid k\in \ndN_0, i \in \mathbb{I} \rbrace$ and where for every vertex $D$ and $j \in \mathbb{I}$, such that $D\ne \refl_j(D)$ there is an edge labelled $j$ from $D$ to $\refl_j(D)$.
\end{defi}

\begin{exa} \label{exa_dynk}
    Let $\omega \in \fK$ be a third root of $1$ and assume $\dynk \in \dynkcat{2}$ is given by
    \begin{align*}
    \begin{tikzpicture}
        \node[dynnode, label={\tiny$\omega$}] (1) at (0,0) {};
        \node[dynnode, label={\tiny$-1$}] (2) at (1.5,0) {};
        \path[draw,thick]
        (1) edge node[above]{\tiny$-\omega$} (2);
    \end{tikzpicture}
    \matcom
    \end{align*}
    i.e. $\dynk_1 = \omega$, $\dynk_2=-1$ and $\dynk_{12}=-\omega$.
    We can easily calculate that $\dynk$ is $1$-finite and $2$-finite, that
    $\refl_1(\dynk)=\dynk$ and that $\refl_2(\dynk)$ is given by
    \begin{align*}
    \begin{tikzpicture}
        \node[dynnode, label={\tiny$\omega^{-1}$}] (1) at (0,0) {};
        \node[dynnode, label={\tiny$-1$}] (2) at (1.5,0) {};
        \path[draw,thick]
        (1) edge node[above]{\tiny$-\omega^{-1}$} (2);
    \end{tikzpicture}
    \matdot
    \end{align*}
    We then calculate that $\refl_2(\dynk)$ is $1$-finite and $\refl_1 \refl_2 (\dynk)=\refl_2 (\dynk)$. Hence $\dynk$ admits all reflections and the reflection graph of $\dynk$ is
    \begin{align*}
    \begin{tikzpicture} \tikzset{every loop/.style={}};
        \node (a) at (0,0) {$
            \begin{tikzpicture}
                \node[dynnode, label={\tiny$\omega$}] (1) at (0,0) {};
                \node[dynnode, label={\tiny$-1$}] (2) at (1.5,0) {};
                \path[draw,thick]
                (1) edge node[above]{\tiny$-\omega$} (2);
            \end{tikzpicture}
        $};
        \node (b) at (3.5,0) {$
            \begin{tikzpicture}
                \node[dynnode, label={\tiny$\omega^{-1}$}] (1) at (0,0) {};
                \node[dynnode, label={\tiny$-1$}] (2) at (1.5,0) {};
                \path[draw,thick]
                (1) edge node[above]{\tiny$-\omega^{-1}$} (2);
            \end{tikzpicture}
        $};
        \path[draw,thick] (a) edge node[above]{$2$} (b);
    \end{tikzpicture}
    \matdot
    \end{align*}
    Hence for all $k \in \ndN_0$, $i \in \mathbb{I}^k$
    \begin{align*}
        (a^{\refl_{i}(\dynk)}_{jk})_{j,k \in \mathbb{I}} = \begin{pmatrix}
    2 & -2 \\
    -1 & 2 
    \end{pmatrix} \matcom
        && (m^{\refl_{i} (\dynk)}_j)_{j\in \mathbb{I}} = \begin{pmatrix}
    2 \\1
    \end{pmatrix}
    \matdot
    \end{align*}
\end{exa}

\begin{exa} \label{exa_dynk2}
    Let $a \in \lbrace 3,6 \rbrace$. Moreover let $\omega$ be an $a$-th root of unity and assume that $\dynk \in \dynkcat{3}$ is given by
\begin{align*}
\begin{tikzpicture}
    \node[dynnode, label={\tiny$\omega$}] (1) at (0,0) {};
    \node[dynnode, label={\tiny$\omega$}] (2) at (1.5,0) {};
    \node[dynnode, label={\tiny$-1$}] (3) at (3,0) {};
    \path[draw,thick]
    (1) edge node[above]{\tiny$\omega^{-1}$} (2)
    (2) edge node[above]{\tiny$\omega^{-2}$} (3);
\end{tikzpicture}
\matcom
\end{align*}
i.e. $\dynk_1 = \dynk_2 = \omega$, $\dynk_3=-1$, $\dynk_{12}=\omega^{-1}$, $\dynk_{13}=1$ and $\dynk_{23}=\omega^{-2}$.
We can calculate that $\dynk$ admits all reflections and that the reflection graph of $\dynk$ is
\begin{align*}
\begin{tikzpicture} \tikzset{every loop/.style={}};
    \node (a) at (0,0) {$
        \begin{tikzpicture}
            \node[dynnode, label={\tiny$\omega$}] (1) at (0,0) {};
            \node[dynnode, label={\tiny$\omega$}] (2) at (1.4,0) {};
            \node[dynnode, label={\tiny$-1$}] (3) at (2.8,0) {};
            \path[draw,thick]
            (1) edge node[above]{\tiny$\omega^{-1}$} (2)
            (2) edge node[above]{\tiny$\omega^{-2}$} (3);
        \end{tikzpicture}
    $};
    \node (b) at (4.5,0) {$
        \begin{tikzpicture}
            \node[dynnode, label={\tiny$\omega$}] (1) at (0,0) {};
            \node[dynnode, label={\tiny$-\omega^{-1}$}] (2) at (1.4,0) {};
            \node[dynnode, label={\tiny$-1$}] (3) at (2.8,0) {};
            \path[draw,thick]
            (1) edge node[above]{\tiny$\omega^{-1}$} (2)
            (2) edge node[above]{\tiny$\omega^{2}$} (3);
        \end{tikzpicture}
    $};
    \path[draw,thick] (a) edge node[above]{$3$} (b);
\end{tikzpicture}
\matdot
\end{align*}
Let $a' \in \lbrace 3,6 \rbrace$, such that $a \ne a'$. Then for $k \in \ndN_0$, $i=(i_1,\ldots,i_k) \in \mathbb{I}$ we obtain 
\begin{align*}
    (a^{\refl_{i}(\dynk)}_{jk})_{j,k \in \mathbb{I}}= 
    \begin{cases} 
        \begin{pmatrix}
            2 & -1 & 0 \\
            -1 & 2 & -2 \\
            0 & -1 & 2
        \end{pmatrix}
        & \text{if $\# \lbrace 1\le l \le k \, | \, i_l = 3\rbrace$ is even,} \\
        \begin{pmatrix}
            2 & -1 & 0 \\
            -2 & 2 & -2 \\
            0 & -1 & 2
        \end{pmatrix} & \text{else.}
    \end{cases}
\end{align*}
and $(m^{\refl_{i}(\dynk)}_j)_{j \in \mathbb{I}}=(a-1,a-1,1)$ if $\# \lbrace 1\le l \le k \, | \, i_l = 3\rbrace$ is even and $(m^{\refl_{i}(\dynk)}_j)_{j \in \mathbb{I}}=(a-1,a'-1,1)$ if not.
\end{exa}
\section{Nichols systems of diagonal type} \label{sect_nich_sys_diag_main}

The first half of this section will be about how the reflection theory of Nichols systems of diagonal type and its $\ydmod{}$-modules can be ascribed by the reflection theory of Dynkin diagrams. The core of this is \Cref{lem_ydmod_coact_diagonal_explicit}, where we calculate the Shapovalov morphism in this situation. Here the main results are \Cref{prop_dynk_refl_equals_refl} and its succeeding corollaries. In the second half we will apply these results to the theory of induced $\ydmod{}$-modules. We will see that in the case where the Nichols system is finite-dimensional, the information about what reflection sequences the induced $\ydmod{}$-module admits, and therefore also when it is irreducible, all depends on the roots of the Nichols system and the polynomial of \Cref{defi_shapovalov_determinant}, see \Cref{prop_ydind_refl_admit_equiv} and \Cref{cor_ydind_refl_admit_equiv}. At the end, we will also give an algorithm on how to calculate this polynomial.

Let $\Gamma$ be an abelian group, such that $\ndZ^\theta \subset \Gamma$.

\begin{defi}
    A pre-Nichols system $\scr{N}:=(Q,(N_1,\ldots,N_\theta))$, where $N_1,\ldots, N_\theta  \in \scr{C}$ are one-dimensional objects is called \textbf{pre-Nichols system of diagonal type}.

    The \textbf{Dynkin diagram of $\scr{N}$} is $\dynk^{\scr{N}} := \dynk(N_1,\ldots,N_\theta) \in \dynkcat{\theta}$. 
    For a homogeneously generated $\Gamma$-graded object $V \in \ydcat{Q}{\scr{C}}$, where $\gencomp{V}$ is one-dimensional, the \textbf{Dynkin diagram of $V$} is $\dynk^V := \dynk(N_1,\ldots,N_\theta, \gencomp{V}) \in  \dynkcat{\theta+1}$. 
\end{defi}

Let $\scr{N}:=(Q,(N_1,\ldots,N_\theta))$ be a pre-Nichols system of diagonal type. For $j \in \mathbb{I}$ let $x_j \in N_j \setminus \lbrace 0 \rbrace$, i.e. $N_j = \fK x_j$.

\begin{rema}
    Let $V \in \scr{C}$, $v \in V$. Let $k \in \ndN$, $i_1,\ldots,i_k \in \mathbb{I}$. We have
    \begin{align*}
        & (\id_V \ot \bimu_Q^{k-1}) (\brd^{\scr{C}}_{N_{i_{1}},V} \ot \id_{Q^{\ot k-1 }}) (\id_{Q} \ot \brd^{\scr{C}}_{N_{i_{2}},V} \ot \id_{Q^{\ot k-2 }}) \cdots (\id_{Q^{\ot k-1 }} \ot \brd^{\scr{C}}_{N_{i_k},V}) 
        \\& \left( x_{i_1} \ot \cdots \ot x_{i_k} \ot v \right)
        = \brd^{\scr{C}}_{Q,V} \left( \left( x_{i_1} \cdots x_{i_k} \right) \ot v \right)
    \end{align*}
    and
    \begin{align*}
        & (\bimu_Q^{k-1} \ot \id_V) (\id_{Q^{\ot k-1 }} \ot \brd^{\scr{C}}_{V,N_{i_k}}) (\id_{Q^{\ot k-2 }} \ot \brd^{\scr{C}}_{V,N_{i_{k-1}}} \ot \id_{Q}) \cdots (\brd^{\scr{C}}_{V,N_{i_{1}}} \ot \id_{Q^{\ot k-1 }})
        \\& \left( v \ot x_{i_1} \ot \cdots \ot x_{i_k}\right)
        = \brd^{\scr{C}}_{Q,V} \left( v \ot \left(x_{i_1} \cdots x_{i_k} \right) \right)
        \matdot
    \end{align*}
\end{rema}

\begin{lem} \label{lem_primitive_potency}
    For $i\in \mathbb{I}$ and $m \in \ndN_0$ the following hold:
    \begin{enumerate}
        \item $\bicomu_Q (x_i^m) = \sum_{k=0}^{m} \binom{m}{k}_{\dynk^{\scr{N}}_i} x_i^k \ot x_i^{m-k}$.
        \item $\antip_Q(x_i^m)=(-1)^m {\dynk^{\scr{N}}_i}^{\frac{1}{2}(m-1)m} x_i^m$.
        \item The following are equivalent:
        \begin{enumerate}
            \item $\subalgQ{N_i}$ is strictly graded.
            \item $m_i^{\scr{N}} = m_i^{\dynk^{\scr{N}}}$.
        \end{enumerate}
    \end{enumerate} 
\end{lem}
\begin{proof}
    (1): We do induction on $m$: Clearly the claim holds for $m = 0$. Assume the claim holds for $m \ge 0$.
    Then
    \begin{align*}
        &\bicomu_Q (x_i^{m+1})
        = \bimu_{Q \ot Q} (\bicomu_Q (x_i) \ot \bicomu_Q (x_i^m))
        \\=& \bimu_{Q \ot Q} \left( \left( x_i \ot 1 + 1 \ot x_i \right) \ot \left( \sum_{k=0}^{m} \binom{m}{k}_{\dynk^{\scr{N}}_i} x_i^k \ot x_i^{m-k} \right) \right)
        \\=& \sum_{k=0}^{m} \binom{m}{k}_{\dynk^{\scr{N}}_i} x_i^{k+1} \ot x_i^{m-k} + {\dynk^{\scr{N}}_i}^{k} \binom{m}{k}_{\dynk^{\scr{N}}_i} x_i^k \ot x_i^{m-k+1} 
        \\=& \sum_{k=0}^{m+1} \left(  \binom{m}{k-1}_{\dynk^{\scr{N}}_i}  + {\dynk^{\scr{N}}_i}^{k} \binom{m}{k}_{\dynk^{\scr{N}}_i} \right) x_i^k \ot x_i^{m-k+1} 
        \matdot
    \end{align*}
    Observe that
    \begin{align*}
        \binom{m+1}{k}_{\dynk^{\scr{N}}_i} = \binom{m}{k-1}_{\dynk^{\scr{N}}_i} + {\dynk^{\scr{N}}_i}^{k} \binom{m}{k}_{\dynk^{\scr{N}}_i} \matcom
    \end{align*}
    hence the above calculation yields
    \begin{align*}
        \bicomu_Q (x_i^{m+1}) = \sum_{k=0}^{m+1} \binom{m+1}{k}_{\dynk^{\scr{N}}_i} x_i^k \ot x_i^{m+1-k} \matdot
    \end{align*}
    
    (2): We do induction on $m$: Clearly the claim holds for $m = 0$. Assume the claim holds for $m \ge 0$. Then
    \begin{align*}
        \antip_Q(x_i^{m+1})
        =& \bimu_Q (\antip_Q \ot \antip_Q) \brd^{\scr{C}}_{Q,Q} (x_i^m \ot x_i)
        = {\dynk^{\scr{N}}_i}^m \antip_Q(x_i^m) \antip_Q(x_i) 
        \\=& (-1)^m {\dynk^{\scr{N}}_i}^{\frac{1}{2}(m-1)m+m} x_i^m (-x_i)
        =(-1)^{m+1} {\dynk^{\scr{N}}_i}^{\frac{1}{2}m(m+1)} x_i^{m+1} 
        \matdot
    \end{align*}
    
    (3): Assume (a) holds. Assume $m^{\dynk^{\scr{N}}}_i \in \ndN_0$, i.e. $(m^{\dynk^{\scr{N}}}_i+1)_{\dynk^{\scr{N}}_i} = 0$. Then (1) implies
    \begin{align*}
        \bicomu_Q (x_i^{m^{\dynk^{\scr{N}}}_i+1}) = 1 \ot x_i^{m^{\dynk^{\scr{N}}}_i+1} + x_i^{m^{\dynk^{\scr{N}}}_i+1} \ot 1 \matdot
    \end{align*}
    Hence $x_i^{m^{\dynk^{\scr{N}}}_i+1}=0$ by (a), since $m^{\dynk^{\scr{N}}}_i \ge 1$. This implies $m^{\scr{N}}_i \in \ndN_0$ and $m^{\scr{N}}_i \le m^{\dynk^{\scr{N}}}_i$. 
    On the other hand if $m^{\scr{N}}_i \in \ndN_0$ we have $x_i^{m^{\scr{N}}_i} \ne 0$ and $ x_i^{m^{\scr{N}}_i+1} = 0$. Hence (1) implies
    \begin{align*}
        0 = \bicomu_Q (x_i^{m^{\scr{N}}_i+1}) = \sum_{k=0}^{m^{\scr{N}}_i+1} \binom{m^{\scr{N}}_i+1}{k}_{\dynk^{\scr{N}}_i} x_i^k \ot x_i^{m^{\scr{N}}_i+1-k} \matcom
    \end{align*}
    in particular $\binom{m^{\scr{N}}_i+1}{1}_{\dynk^{\scr{N}}_i} x_i \ot x_i^{m^{\scr{N}}_i}=0$. Since $x_i^{m^{\scr{N}}_i} \ne 0$ we can conclude that $(m^{\scr{N}}_i+1)_{\dynk^{\scr{N}}_i}=\binom{m^{\scr{N}}_i+1}{1}_{\dynk^{\scr{N}}_i}=0$. Hence $m^{\dynk^{\scr{N}}}_i \in \ndN_0$ and $m^{\scr{N}}_i \ge m^{\dynk^{\scr{N}}}_i$. 
    So altogether either both $m^{\scr{N}}_i \notin \ndN_0$ and $m^{\dynk^{\scr{N}}}_i \notin \ndN_0$, or $m^{\scr{N}}_i  = m^{\dynk^{\scr{N}}}_i \in \ndN_0$, i.e. (b) holds.
    
    Now conversely assume (b) holds. Let $m \le m_i^{\scr{N}}$, $\lambda_0,\ldots,\lambda_m \in \fK$, such that $\lambda_m \ne 0$ and such that $x := \sum_{l=0}^m \lambda_l x_i^l$ is primitive. 
    Since $m \le m_i^{\dynk^{\scr{N}}}$ by (b), we obtain $(m)_{\dynk^{\scr{N}}_i} \ne 0$.
    Now (1) implies
    \begin{align*}
        \sum_{l=0}^m \lambda_l \left( x_i^l \ot 1 + 1 \ot x_i^l \right) = \bicomu_Q(x) = \sum_{l=0}^m \lambda_l \sum_{k=0}^{l} \binom{l}{k}_{\dynk^{\scr{N}}_i} x_i^k \ot x_i^{l-k} \matdot
    \end{align*}
    Comparing the coefficients for $1 \ot 1$ yields $2 \lambda_0 = \lambda_0$, hence $\lambda_0 = 0$. If $m \ge 2$, then comparing the coefficients for 
    $x_i \ot x_i^{m-1}$ yields
    \begin{align*}
        0 = \lambda_m \binom{m}{1}_{\dynk^{\scr{N}}_i} = \lambda_m (m)_{\dynk^{\scr{N}}_i} \ne 0 \matcom
    \end{align*}
    hence $m\le 1$, i.e. $m=1$ and $x = \lambda_1 x_i$. Thus (a) holds.
\end{proof}

\begin{lem} \label{lem_ydmod_coact_diagonal_explicit}
    Let $i \in \mathbb{I}$ and assume $\subalgQ{N_i}$ is strictly graded. Moreover let $V \in \ydcat{\subalgQ{N_i}}{\scr{C}}$ be a homogeneously $\Gamma$-graded object. Assume there exists $r \in \fK$ such that for all $v \in \gencomp{V}$ we have
    \begin{align*}
        \brd^{\scr{C}}_{V,Q} \brd^{\scr{C}}_{Q,V}(x_i \ot v) = r x_i \ot v 
        \matdot
    \end{align*}
    Then for $m \in \ndN_0$ and $v \in \gencomp{V}$ we have
        \begin{align*}
            & \shapoendo_V(x_i^{m} \cdot v) = \left(\prod_{l=0}^{m-1} (1-{\dynk^{\scr{N}}_i}^l r) \right) x_i^{m} \ot v 
        \matdot
        \end{align*}
\end{lem}
\begin{proof}
    We do induction on $m$. The case $m=0$ is trivial, since $\coact_V^Q$ is graded. Moreover with \Cref{lem_shapovalov_endo_product_rack_case} (by assumption we have $\lambda_i = r$), we obtain
    \begin{align*}
        \shapoendo_V(x_i^{m+1} \cdot v) = \left(
        \bimu_Q \ot \id_{\gencomp{V}}
        - r (\bimu_Q \brd^{\scr{C}}_{Q, Q}  \ot \id_{\gencomp{V}}) \right) \left(x_i \ot \shapoendo_V(x_i^m \cdot v) \right).
        \end{align*}
    By induction hypothesis this is equal to
    \begin{align*}
        \left(\prod_{l=0}^{m-1} (1-{\dynk^{\scr{N}}_i}^l r) \right) \left(
        \bimu_Q \ot \id_{\gencomp{V}}
        - r (\bimu_Q \brd^{\scr{C}}_{Q, Q} \ot \id_{\gencomp{V}}) \right) \left(x_i \ot x_i^m \ot v \right),
    \end{align*}
    which simplifies to $ \left(\prod_{l=0}^{m} (1-{\dynk^{\scr{N}}_i}^l r) \right) x_i^{m+1} \ot v $.
\end{proof}

\begin{prop} \label{prop_diag_well_graded_miV}
    Let $i \in \mathbb{I}$ and assume $\subalgQ{N_i}$ is strictly graded. Moreover let $V \in \ydcat{\subalgQ{N_i}}{\scr{C}}$ be a homogeneously $\Gamma$-graded object. Assume there exists $r \in \fK$ such that for all $v \in \gencomp{V}$ we have
    \begin{align*}
        \brd^{\scr{C}}_{V,Q} \brd^{\scr{C}}_{Q,V}(x_i \ot v) = r x_i \ot v 
        \matdot
    \end{align*}
    Let $v \in \gencomp{V}$ and $m_i^v = \max \lbrace m \in \ndN_0 \, | \, x_i^m \cdot v \ne 0 \rbrace$.
    The following are equivalent
    \begin{enumerate}
        \item $\subalgQ{N_i} \cdot v \in \ydcat{\subalgQ{N_i}}{\scr{C}}$ is well graded.
        \item $m_i^v= \min \lbrace m \in \ndN_0 \, | \, (m+1)_{\dynk^{\scr{N}}_{i}} ({\dynk^{\scr{N}}_{i}}^{m} r - 1) = 0 \rbrace$.
    \end{enumerate}
\end{prop}
\begin{proof}
First by \Cref{lem_induced_Q_comodule_is_YD_module} we know that $\subalgQ{N_i} \cdot v$ is indeed an $\ndN_0$-graded object in $\ydcat{\subalgQ{N_i}}{\scr{C}}$.
By \Cref{cor_reflection_sequence_shapo_equiv} we obtain that (1) is equivalent to 
\begin{align*}
    \ker \shapoendo_{\subalgQ{N_i} \cdot v} (N_i^m \cdot v) = 0
\end{align*}
for all $0 \le m \le m_i^{\subalgQ{N_i} \cdot v}=m_i^v$. This is equivalent to
\begin{align*}
    \shapoendo_{V} (x_i^m) \ne 0
\end{align*}
for all $0 \le m \le m_i^v$. By \Cref{lem_ydmod_coact_diagonal_explicit} this is equivalent to
\begin{align*}
    \prod_{l=0}^{m-1} (1-{\dynk^{\scr{N}}_{i}}^l r) \ne 0 .
\end{align*}
for all $0 \le m \le m_i^v$.
Hence (2) implies (1). Since $x_i^m \ne 0$ for all $0\le m \le m_i^v$ we get $(m)_{\dynk^{\scr{N}}_{i}} \ne 0$ by \Cref{lem_primitive_potency}(3). So assuming (1), we get 
\begin{align*}
    m_i^v \le \min \lbrace m \in \ndN_0 \, | \, (m+1)_{\dynk^{\scr{N}}_{i}} ({\dynk^{\scr{N}}_{i}}^{m} r - 1) = 0 \rbrace.
\end{align*}
If $m_i^v \notin \ndN_0$, then there is nothing more to show. Assuming $m_i^v \in \ndN_0$, again by \Cref{lem_ydmod_coact_diagonal_explicit} we know
\begin{align*}
    0 = \shapoendo(0) = \shapoendo(x_i^{m_i^v + 1}\cdot v) = \prod_{l=0}^{m_i^v}(1-{\dynk^{\scr{N}}_{i}}^{l} r) x_i^{m_i^v+1} \ot v ,
\end{align*}
So $ 1-{\dynk^{\scr{N}}_{i}}^{m_i^v} r  = 0$ or $x_i^{m_i^v+1} = 0$. With \Cref{lem_primitive_potency}(3) the latter implies $(m_i^v+1)_{{\dynk^{\scr{N}}_{i}}}= 0$, hence we obtain equality in (2).
\end{proof}

\begin{cor} \label{cor_diag_well_graded_miV}
    For $i \in \mathbb{I}$ the following are equivalent:
    \begin{enumerate}
        \item $\scr{N}$ is a Nichols system over $i$.
        \item The following hold:
        \begin{enumerate}
            \item $m^{\scr{N}}_i = m_i^{\dynk^{\scr{N}}}$.
            \item $a_{ij}^{\scr{N}} = a_{ij}^{\dynk^{\scr{N}}}$ for all $j \in \mathbb{I}$.
        \end{enumerate}
    \end{enumerate}
    In this case $\scr{N}$ is $i$-finite if and only if $\dynk^{\scr{N}}$ is $i$-finite.
\end{cor}
\begin{proof}
    By definition (1) holds if and only if $\subalgQ{N_i}$ is strictly graded and $(\ad \subalgQ{N_i})(N_j)=(\ad \subalgQ{x_i})(x_j) \in \ydcat{\subalgQ{N_i}}{\scr{C}}$ is well-graded for all $j \in \mathbb{I} \setminus \lbrace i \rbrace$.
    By \Cref{lem_primitive_potency}(3) the former one is equivalent to (a) and by \Cref{prop_diag_well_graded_miV} the latter one is equivalent to
    \begin{align*}
        &\max \lbrace n \in \ndN_0 \, | \, (\ad \, x_i)^m (x_j) \ne 0 \rbrace 
        \\=& \min \lbrace m \in \ndN_0 \, | \, (m+1)_{\dynk^{\scr{N}}_{i}} ({\dynk^{\scr{N}}_{i}}^{m} \dynk^{\scr{N}}_{ij} - 1) = 0 \rbrace
    \end{align*}
    for all $j \in \mathbb{I} \setminus \lbrace i \rbrace$, which is equivalent to (b).
    
    Now $\scr{N}$ is $i$-finite if and only if $-a_{ij}^{\scr{N}} \in \ndN_0$ for all $j \in \mathbb{I} \setminus \lbrace i \rbrace$ and $\dynk^{\scr{N}}$ is $i$-finite if and only if $-a_{ij}^{\dynk^{\scr{N}}} \in \ndN_0$ for all $j \in \mathbb{I} \setminus \lbrace i \rbrace$.
\end{proof}

\begin{cor} \label{cor_diag_well_graded_miV2}
     Let $i \in \mathbb{I}$ and assume $\subalgQ{N_i}$ is strictly graded. Let $V \in \ydcat{Q}{\scr{C}}$ be a homogeneously generated $\Gamma$-graded object, such that $\gencomp{V}$ is one-dimensional.
    The following are equivalent:
    \begin{enumerate}
        \item $V$ is $i$-well graded.
        \item $m_i^V = -a^{\dynk^V}_{i (\theta+1)}$. 
    \end{enumerate}
\end{cor}
\begin{proof}
    Let $v \in \gencomp{V} \setminus \lbrace 0 \rbrace$. Recall that by definition $V$ is $i$-well graded, if and only if $\subalgQ{N_i} \cdot \gencomp{V} = \subalgQ{x_i} \cdot v \in \ydcat{\subalgQ{N_i}}{\scr{C}}$ is well graded.
    Moreover we have $\brd^{\scr{C}}_{V,Q} \brd^{\scr{C}}_{Q,V}(x_i \ot v) = \dynk^V_{i(\theta+1)} x_i \ot v$, $\dynk^V_i = \dynk^{\scr{N}}_i$. Hence the claim is a special case of \Cref{prop_diag_well_graded_miV}.
\end{proof}

\begin{rema} \label{rema_cor_diag_well_graded_miV2}
    Let $i \in \mathbb{I}$ and assume $\subalgQ{N_i}$ is strictly graded. Let $V \in \ydcat{Q}{\scr{C}}$ be a homogeneously generated $\Gamma$-graded object, such that $\gencomp{V}$ is one-dimensional.
    By \Cref{cor_diag_well_graded_miV2} for all $j \in \mathbb{I}$ we have 
    \begin{align*}
       \refl_i(\dynk^{V})_{j(\theta+1)} &= {\dynk^{V}_{j(\theta+1)}} {\dynk^{V}_{i(\theta+1)}}^{-a^{\dynk^{\scr{N}}}_{ij}} {\dynk^{\scr{N}}_{ij}}^{m_i^V} {\dynk^{\scr{N}}_{i}}^{-2 a^{\dynk^{\scr{N}}}_{ij} m_i^V}
       \matdot
    \end{align*}
    In particular for all $j \in \mathbb{I}$ we have that $\refl_i(\dynk^{V})_{j(\theta+1)}$ only relies on $\dynk^{\scr{N}}$, $\dynk^{V}_{i(\theta+1)}$ and $\dynk^{V}_{j(\theta+1)}$.
\end{rema}

We will now discuss the relationship between reflections in $\dynkcat{\theta}$ and reflections of $\scr{N}$ or objects in $\ydcat{Q}{\scr{C}}_{\rat}$.

\begin{lem} \label{lem_diag_ad_explicit}
    Let $i,j \in \mathbb{I}$ such that $j \ne i$. Then
    \begin{align*}
        (\ad \, x_i)^m (x_j) = \sum_{k=0}^m \lambda_k x_i^{m-k} x_j x_i^k
    \end{align*}
    for some $\lambda_0,\ldots,\lambda_m \in \fK$.
\end{lem}
\begin{proof}
    We know that
    \begin{align*}
        (\id_Q \ot \antip_Q) (\id_{Q} \ot \brd^{\scr{C}}_{Q,Q}) (\bicomu_{Q}(x_i^m) \ot x_j) \in \oplus_{k=0}^{m} N_i^{m-k} \ot N_j \ot N_i^k \matdot
    \end{align*}
    Applying $\bimu_Q (\bimu_Q \ot \id_Q)$ to this gives $(\ad \, x_i)^m (x_j)$.
\end{proof}

\begin{rema}
    Let $i\in \mathbb{I}$ and let $y_i \in N_i^*$ be defined by $y_i(x_i)=1$. For $j \in \mathbb{I}$ let $g_j \in H$, such that $\coact^H_Q(x_j) = g_j \ot x_j$. Then $g_j$ is an invertible, group-like element. 
    For $j \in \mathbb{I}$ let $\chi_j \in H^*$ be defined by $h \cdot x_j = \chi_j(h) x_j $ for all $h \in H$.
    Observe that $\chi_j(g_j) = \dynk^{\scr{N}}_j$ and $\chi_j(g_k) \chi_k(g_j) = {\dynk^{\scr{N}}_{jk}}$ for all $j,k \in \mathbb{I}$, $j\ne k$.
    Recall that then we have $N_i^* \in \scr{C}$ with $H$-action and $H$-coaction given by
    \begin{align*}
        h \cdot y_i = \chi_i(\antip_H (h)) y_i \matcom && \coact^H_{N_i^*}(y_i) = g_i^{-1} \ot y_i
    \end{align*}
    for all $h\in H$, for details refer to \cite{HeSch}, Lemma~4.2.2. 
    Hence we obtain
    \begin{align} \label{equa_rema_diag_dual1}
        \brd^{\scr{C}}_{N_i^*,N_i^*} (y_i \ot y_i) = g_i^{-1} \cdot y_i \ot y_i = \chi_i(g_i) y_i \ot y_i = \dynk^{\scr{N}}_i y_i \ot y_i \matdot
    \end{align}
    Moreover for $j \in \mathbb{I}$
    \begin{align*}
        \brd^{\scr{C}}_{N_i^*,N_j} (y_i \ot x_j) = \chi_j(g_i^{-1}) x_j \ot y_i
    \end{align*}
    and 
    \begin{align*}
        \brd^{\scr{C}}_{N_j,N_i^*} (x_j \ot y_i) = g_j \cdot y_i \ot x_j = \chi_i(g_j^{-1}) y_i \ot x_j
        \matdot
    \end{align*}
    Hence
    \begin{align} \label{equa_rema_diag_dual2}
        \brd^{\scr{C}}_{N_i,N_i^*} \brd^{\scr{C}}_{N_i^*,N_i} (y_i \ot x_i) = \chi_i(g_i^{-1}) \chi_i(g_i^{-1}) y_i \ot x_i = {\dynk^{\scr{N}}_{i}}^{-2} y_i \ot x_i
    \end{align}
    and for $j \in \mathbb{I}$, $j \ne i$
    \begin{align} \label{equa_rema_diag_dual3}
        \brd^{\scr{C}}_{N_j,N_i^*} \brd^{\scr{C}}_{N_i^*,N_j} (y_i \ot x_j) = \chi_i(g_j^{-1}) \chi_j(g_i^{-1}) y_i \ot x_j = {\dynk^{\scr{N}}_{ij}}^{-1} y_i \ot x_j
        \matdot
    \end{align}
    Finally let $V \in \ydcat{\subalgQ{N_i}}{\scr{C}}$. Moreover assume there exists $v \in V \setminus \lbrace 0 \rbrace$, such that $\coact^Q_V(v)=1 \ot v$ and $r \in \fK$ such that
    \begin{align*}
        \brd^{\scr{C}}_{V,Q} \brd^{\scr{C}}_{Q,V}(x_i \ot v) = r x_i \ot v 
        \matdot
    \end{align*}
    Then similar to above we obtain
    \begin{align} \label{equa_rema_diag_dual4}
        \brd^{\scr{C}}_{V,Q} \brd^{\scr{C}}_{Q,V}(y_i \ot v) = r^{-1} y_i \ot v \matdot
    \end{align}
\end{rema}

\begin{prop} \label{prop_dynk_refl_equals_refl}
    Let $i \in \mathbb{I}$ and assume $\scr{N}$ is an $i$-finite Nichols system over $i$. Then $\refl_i(\scr{N})$ is a Nichols system of diagonal type and we have
    \begin{align*}
        \refl_i(\dynk^{\scr{N}}) = \dynk^{\refl_i(\scr{N})}
        \matdot
    \end{align*}
    If $V \in \ydcat{Q}{\scr{C}}_{\rat}$ is homogeneously generated $i$-well $\Gamma$-graded, such that $\gencomp{V}$ is one-dimensional, then $\dynk^{V}$ is $i$-finite, $\gencomp{\refl_i(V)}$ is one-dimensional and we have
    \begin{align*}
       \refl_i(\dynk^{V}) = \dynk^{\refl_i(V)}
        \matdot
    \end{align*}
\end{prop}
\begin{proof}
    Let $y_i \in N_i^*$, such that $y_i(x_i)=1$, let $a_{ij}:=a_{ij}^{\scr{N}}=a_{ij}^{\dynk^{\scr{N}}}$ and let $y_j = (\ad \, x_i)^{-a_{ij}} (x_j) $ for all $j \in \mathbb{I}$, $j\ne i$. Then $y_j$ is a basis of $\refl_i(\scr{N})_j$ for all $j \in \mathbb{I}$ and thus $\refl_i(\scr{N})$ is a Nichols system of diagonal type.
    First $\dynk^{\refl_i(\scr{N})}_i = \dynk^{\scr{N}}_i = \refl_i(\dynk^{\scr{N}})_i$ was shown in (\ref{equa_rema_diag_dual1}). By \Cref{lem_diag_ad_explicit} for $j,k \in \mathbb{I} \setminus \lbrace i \rbrace$, $j\ne k$ there exist $\lambda_0,\ldots,\lambda_{-a_{ij}} \in \fK$ and $\mu_0,\ldots,\mu_{-a_{ik}} \in \fK$, such that 
    \begin{align*}
        y_j = \sum_{l=0}^{-a_{ij}} \lambda_l x_i^{-a_{ij}-l} x_j x_i^l \matcom && y_k = \sum_{l=0}^{-a_{ik}} \mu_l x_i^{-a_{ik}-l} x_k x_i^l \matdot
    \end{align*}
    Now for $0\le l \le -a_{ij}$ we have
    \begin{align*}
        &\brd^{\scr{C}}_{\refl_i(Q),\refl_i(Q)} ( x_i^{-a_{ij}-l} x_j x_i^l \ot x_i^{-a_{ij}-l} x_j x_i^l) 
        \\=& \dynk^{\scr{N}}_j {\dynk^{\scr{N}}_{ij}}^{-a_{ij}} {\dynk^{\scr{N}}_i}^{{a_{ij}}^2}  x_i^{-a_{ij}-l} x_j x_i^l \ot x_i^{-a_{ij}-l} x_j x_i^l
        \matdot
    \end{align*}
    This implies $\dynk^{\refl_i(\scr{N})}_j = \dynk^{\scr{N}}_j {\dynk^{\scr{N}}_{ij}}^{-a_{ij}} {\dynk^{\scr{N}}_i}^{{a_{ij}}^2} = \refl_i(\dynk^{\scr{N}})_{j}$. 
    Considering (\ref{equa_rema_diag_dual2}) and~(\ref{equa_rema_diag_dual3}), we obtain
    \begin{align*}
        \brd^{\scr{C}}_{\refl_i(Q),\refl_i(Q)}\brd^{\scr{C}}_{\refl_i(Q),\refl_i(Q)} ( y_i \ot x_i^{-a_{ij}-l} x_j x_i^l) 
        = {\dynk^{\scr{N}}_{ij}}^{-1} {\dynk^{\scr{N}}_{i}}^{2a_{ij}}  y_i \ot x_i^{-a_{ij}-l} x_j x_i^l \matcom
    \end{align*}
    hence $\dynk^{\refl_i(\scr{N})}_{ij} = {\dynk^{\scr{N}}_{ij}}^{-1} {\dynk^{\scr{N}}_{i}}^{2a_{ij}} = \refl_i(\dynk^{\scr{N}})_{ij}$.
    Similarly for $0 \le l' \le -a_{ik}$ we obtain 
    \begin{align*}
        & \brd^{\scr{C}}_{\refl_i(Q),\refl_i(Q)}\brd^{\scr{C}}_{\refl_i(Q),\refl_i(Q)} ( x_i^{-a_{ij}-l} x_j x_i^l \ot x_i^{-a_{ik}-l'} x_k x_i^{l'}) 
        \\=& \dynk^{\scr{N}}_{jk} {\dynk^{\scr{N}}_{ik}}^{-a_{ij}} {\dynk^{\scr{N}}_{ij}}^{-a_{ik}} {\dynk^{\scr{N}}_{i}}^{2a_{ij}a_{ik}} x_i^{-a_{ij}-l} x_j x_i^l \ot x_i^{-a_{ik}-l'} x_k x_i^{l'} \matcom
    \end{align*}
    implying $\dynk^{\refl_i(\scr{N})}_{jk} = \dynk^{\scr{N}}_{jk} {\dynk^{\scr{N}}_{ik}}^{-a_{ij}} {\dynk^{\scr{N}}_{ij}}^{-a_{ik}} {\dynk^{\scr{N}}_{i}}^{2a_{ij}a_{ik}} = \refl_i(\dynk^{\scr{N}})_{jk}$.
    
    Now $a_{ij}^{\dynk^V} = a_{ij}$ and $-a_{i(\theta+1)}^{\dynk^V} = m_i^V$ by \Cref{cor_diag_well_graded_miV2}, hence we obtain that $\dynk^V$ is $i$-finite.
    Let $0\ne v \in \gencomp{V}$, i.e. $\gencomp{V}=\fK v$. Then we have $\gencomp{\refl_i(V)} = N_i^{m_i^V} \cdot \gencomp{V} = \fK x_i^{m_i^V} \cdot v$ by \Cref{prop_funCD_refl_generator}, i.e. $\gencomp{\refl_i(V)}$ is one-dimensional.
    It remains to show that $\refl_i(\dynk^V)_{\theta+1} = \dynk^{\refl_i(V)}_{\theta+1}$, $\refl_i(\dynk^V)_{i(\theta+1)} = \dynk^{\refl_i(V)}_{i(\theta+1)}$ and that $\refl_i(\dynk^V)_{j(\theta+1)} = \dynk^{\refl_i(V)}_{j(\theta+1)}$. The first relation is implied by
    \begin{align*}
        \brd^{\scr{C}}_{V,V} (x_i^{m_i^V} \cdot v \ot x_i^{m_i^V} \cdot v) =\dynk^{V}_{\theta+1} {\dynk^{V}_{i(\theta+1)}}^{m_i^V} {\dynk^{\scr{N}}_{i}}^{{m_i^V}^2} x_i^{m_i^V} \cdot v \ot x_i^{m_i^V} \cdot v \matdot
    \end{align*}
    Now (\ref{equa_rema_diag_dual2}) and (\ref{equa_rema_diag_dual4}) imply
    \begin{align*}
        \brd^{\scr{C}}_{V,\refl_i(Q)}\brd^{\scr{C}}_{\refl_i(Q),V} (y_i \ot x_i^{m_i^V} \cdot v) = {\dynk^{\scr{N}}_{i}}^{-2m_i^V} {\dynk^V_{i(\theta+1)}}^{-1} y_i \ot x_i^{m_i^V} \cdot v
        \matcom 
    \end{align*}
    implying $\dynk^{\refl_i(V)}_{i(\theta+1)} = {\dynk^{\scr{N}}_{i}}^{-2m_i^V} {\dynk^V_{i(\theta+1)}}^{-1}=\refl_i(\dynk^V)_{i(\theta+1)}$. Finally
    \begin{align*}
        &\brd^{\scr{C}}_{V,\refl_i(Q)}\brd^{\scr{C}}_{\refl_i(Q),V} (x_i^{-a_{ij}-l} x_j x_i^l \ot x_i^{m_i^V} \cdot v) \\
        =& {\dynk^{\scr{N}}_{i}}^{-2m_i^V a_{ij}} {\dynk^{\scr{N}}_{ij}}^{m_i^V} {\dynk^V_{i(\theta+1)}}^{-a_{ij}} {\dynk^V_{j(\theta+1)}} x_i^{-a_{ij}-l} x_j x_i^l \ot x_i^{m_i^V} \cdot v
        \matcom
    \end{align*}
    implying $\dynk^{\refl_i(V)}_{j(\theta+1)} = {\dynk^{\scr{N}}_{i}}^{-2m_i^V a_{ij}} {\dynk^{\scr{N}}_{ij}}^{m_i^V} {\dynk^V_{i(\theta+1)}}^{-a_{ij}} {\dynk^V_{j(\theta+1)}}=\refl_i(\dynk^V)_{j(\theta+1)}$. 
\end{proof}

\begin{cor} \label{cor_dynk_refl_equals_refl}
    Let $k \in \ndN$, $i=(i_1,\ldots,i_k) \in \mathbb{I}^k$.
    The following are equivalent
    \begin{enumerate}
        \item $\scr{N}$ admits the reflection sequence $i$.
        \item $\scr{N}$ admits the reflection sequence $(i_1,\ldots,i_{k-1})$ and the following hold
        \begin{enumerate}
            \item $m_{i}^{\scr{N}} = m_{i}^{\dynk^{\scr{N}}}$.
            \item $a_{ij}^{\scr{N}} = a_{i j}^{\dynk^{\scr{N}}}$ for all $j \in \mathbb{I}$.
            \item $\refl_{(i_1,\ldots,i_{k-1})}(\dynk^{\scr{N}})$ is $i_k$-finite.
        \end{enumerate}
    \end{enumerate}
\end{cor}
\begin{proof}
     Implied by \Cref{cor_diag_well_graded_miV} and \Cref{prop_dynk_refl_equals_refl}.
\end{proof}

\begin{cor} \label{cor_dynk_refl_equals_refl2}
    Let $V \in \ydcat{Q}{\scr{C}}_{\rat}$ be a homogeneously generated $\Gamma$-graded object, such that $\gencomp{V}$ is one-dimensional. Moreover let $k \in \ndN$, $i=(i_1,\ldots,i_k) \in \mathbb{I}^k$. Assume $\scr{N}$ admits the reflection sequence $i$. The following are equivalent
    \begin{enumerate}
        \item $V$ admits the reflection sequence $i$.
        \item $V$ admits the reflection sequence $(i_1,\ldots,i_{k-1})$ and \begin{align*}m_{i}^V = -a^{\dynk^{V}}_{i(\theta+1)} \matdot \end{align*}
    \end{enumerate}
\end{cor}
\begin{proof}
    Implied by \Cref{cor_diag_well_graded_miV2} and \Cref{prop_dynk_refl_equals_refl}.
\end{proof}

\begin{cor} \label{cor_dynk_refl_equals_refl_roots}
    Let $k \in \ndN$, $i=(i_1,\ldots,i_k) \in \mathbb{I}^k$ and assume $\scr{N}$ admits the reflection sequence $i$. 
    Then $s_i^{\scr{N}} = s_i^{\dynk^{\scr{N}}}$. In particular if $\scr{N}$ admits all reflections we have
    \begin{align*}
        \roots{\dynk^{\scr{N}}} = \roots{\scr{N}} \matcom && \roots{\dynk^{\scr{N}}}= \rootspos{\dynk^{\scr{N}}} \cup - \rootspos{\dynk^{\scr{N}}} \matdot
    \end{align*}
    Moreover if $m_{(i_1,\ldots,i_r)}^{\scr{N}} \in \ndN_0$ for all $1 \le r \le k$, then for $1\le m \le m_i^{\scr{N}}$ we have
    \begin{align*}
        \beta^{\scr{N}}_{i,m} = \beta^{\dynk^{\scr{N}}}_{i,m} \matdot
    \end{align*}
    Let $V \in \ydcat{Q}{\scr{C}}_{\rat}$ be homogeneously generated $\Gamma$-graded, such that $\gencomp{V}$ is one-dimensional and $V$ admits the reflection sequence $i$. Then for $1\le m \le m_i^{V}$ we have
    \begin{align*}
        \beta^{V}_{i,m} := \beta^{\dynk^{\scr{N}}}_{i,(-a^{\dynk^{V}}_{(i_1),\theta+1},\ldots,-a^{\dynk^{V}}_{(i_1,\ldots,i_{k-1}),\theta+1}),m)} \matdot
    \end{align*}
\end{cor}
\begin{proof}
    Implied by \Cref{cor_dynk_refl_equals_refl}, \Cref{prop_roots_pos_and_neg} and \Cref{cor_dynk_refl_equals_refl2}.
\end{proof}

\begin{rema}
    Assume $\scr{N}$ admits all reflections. \Cref{cor_dynk_refl_equals_refl_roots} and \Cref{prop_support_is_spanned_by_beta_i_sys} imply that we can calculate the edges of $\Sup{Q}$ from $\dynk^{\scr{N}}$.
    Similarly if $V \in \ydcat{Q}{\scr{C}}_{\rat}$ is homogeneously generated $\Gamma$-graded, such that $\gencomp{V}$ is one-dimensional and $V$ admits all reflections, then we can calculate $\bouind{V}$ from $\dynk^V$.
\end{rema}


We will now discuss when $\ydind{Q}{U} \in \ydcat{Q}{\scr{C}}$ is irreducible, where $U \in \comodcat{Q}{\scr{C}}$ is a one-dimensional object.

\begin{lem} \label{lem_ydmod_polynomial_identity}
    Let $k \in \ndN_0$, $i=(i_1,\ldots,i_k) \in \mathbb{I}^k$ and let $U \in \comodcat{Q}{\scr{C}}$ be a one-dimensional object. 
    Assume that $\ydind{Q}{U}$ admits the reflection sequence~$i$. 
    Finally assume $m_{(i_1,\ldots,i_r)}^{\dynk^{\scr{N}}} \in \ndN_0$ for all $1\le r \le k$.
    Then 
    for $j \in \mathbb{I}$ we have
    \begin{align*}
        \refl_i(\dynk^{\ydind{Q}{U}})_{j(\theta+1)} = \refl_i(\extdynk{\dynk^{\scr{N}}})_{j (\theta+1)} 
        ( \dynk^{\ydind{Q}{U}}_{1(\theta+1)}, \ldots, \dynk^{\ydind{Q}{U}}_{\theta(\theta+1)}) \matdot
    \end{align*}
\end{lem}
\begin{proof}
    We prove the claim by induction on $k$.
    Since
    \begin{align*}
        \extdynk{\dynk^{\scr{N}}}_{j (\theta+1)} ( \dynk^{\ydind{Q}{U}}_{1(\theta+1)}, \ldots, \dynk^{\ydind{Q}{U}}_{\theta(\theta+1)}) = (t_j)(\dynk^{\ydind{Q}{U}}_{1(\theta+1)}, \ldots, \dynk^{\ydind{Q}{U}}_{\theta(\theta+1)}) = \dynk^{\ydind{Q}{U}}_{j(\theta+1)} \matcom
    \end{align*}
    the claim holds for $k=0$. Now assume $k \ge 1$ and that the claim holds for $k-1$. Then by definition, \Cref{lem_ext_dynk} and induction hypothesis we have
    \begin{align*}
        &\refl_i(\extdynk{\dynk^{\scr{N}}})_{j (\theta+1)} ( \dynk^{\ydind{Q}{U}}_{1(\theta+1)}, \ldots, \dynk^{\ydind{Q}{U}}_{\theta(\theta+1)}) 
        \\ =&  {\dynk^{\scr{N}}_{i,j}}^{m_{i}^{\dynk^{\scr{N}}}} {\dynk^{\scr{N}}_{i}}^{-2 a_{i, j}^{\dynk^{\scr{N}}} m_{i}^{\dynk^{\scr{N}}}}
        \\& \left( \refl_{(i_1,\ldots,i_{k-1})}(\extdynk{\dynk^{\scr{N}}})_{j (\theta+1)} \refl_{(i_1,\ldots,i_{k-1})}(\extdynk{\dynk^{\scr{N}}})_{i_k (\theta+1)}^{-a_{i, j}^{\dynk}} \right)
        (\dynk^{\ydind{Q}{U}}_{1(\theta+1)}, \ldots, \dynk^{\ydind{Q}{U}}_{\theta(\theta+1)})
        \\ = & {\dynk^{\scr{N}}_{i,j}}^{m_{i}^{\dynk^{\scr{N}}}} {\dynk^{\scr{N}}_{i}}^{-2 a_{i, j}^{\dynk^{\scr{N}}} m_{i}^{\dynk^{\scr{N}}}} \refl_{(i_1,\ldots,i_{k-1})}(\dynk^{\ydind{Q}{U}})_{j (\theta+1)} {\dynk^{\ydind{Q}{U}}_{i(\theta+1)}}^{-a_{i, j}^{\dynk} }
        \matdot
    \end{align*}
    By \Cref{lem_ydind_mi} and \Cref{cor_diag_well_graded_miV} we have 
    \begin{align*}
        m_i^{\ydind{Q}{U}} = m_i^{\scr{N}} = m_i^{\dynk^{\scr{N}}}
    \end{align*} 
    and thus by \Cref{rema_cor_diag_well_graded_miV2}
    \begin{align*}
        & {\dynk^{\scr{N}}_{i,j}}^{m_{i}^{\dynk^{\scr{N}}}} {\dynk^{\scr{N}}_{i}}^{-2 a_{i, j}^{\dynk^{\scr{N}}} m_{i}^{\dynk^{\scr{N}}}} \refl_{(i_1,\ldots,i_{k-1})}(\dynk^{\ydind{Q}{U}})_{j (\theta+1)} {\dynk^{\ydind{Q}{U}}_{i(\theta+1)}}^{-a_{i, j}^{\dynk} }
        \\=&\refl_i(\dynk^{\ydind{Q}{U}})_{j(\theta+1)} \matdot
    \end{align*}
    This finishes the induction.
\end{proof}

\begin{rema}  \label{rema_fin_dim_ydind_dynk_admits_refl}
    Assume that $Q$ is finite-dimensional.
    Let $k \in \ndN_0$, $i=(i_1,\ldots,i_k) \in \mathbb{I}^k$ and assume $\scr{N}$ admits the reflections sequence~$i$. 
    Then by \Cref{cor_diag_well_graded_miV} we obtain $m^{\dynk^{\scr{N}}}_{i} = m^{\scr{N}}_i \in \ndN_0$, since $Q$ is finite-dimensional.
    Hence by \Cref{lem_ext_dynk} we obtain that $\extdynk{\dynk^{\scr{N}}}$ admits the reflection sequence $i$.
    
    Let $U \in \comodcat{Q}{\scr{C}}$ be a one-dimensional object. 
    By induction on $k$ we show that $\dynk^{\ydind{Q}{U}}$ admits the reflection sequence $i$:
    Let $k\ge 1$ and assume $\dynk^{\ydind{Q}{U}}$ admits the reflection sequence $(i_1,\ldots,i_{k-1})$. As $m^{\dynk^{\scr{N}}}_{i} \in \ndN_0$ we have that  $(m)_{\dynk^{\scr{N}}_{i}} = 0$ for some $m \in \ndN$. This implies that for all $1 \le j \le \theta+1$, $j \ne i_k$ we have
    \begin{align*}
        (m)_{\dynk^{\scr{N}}_{i}} ({\dynk^{\scr{N}}_{i}}^{m-1} \dynk^{\ydind{Q}{U}}_{i, j} - 1) = 0 \matcom
    \end{align*}
    hence $\refl_{(i_1,\ldots,i_{k-1})}(\dynk^{\ydind{Q}{U}})$ is $i_k$-finite.
    
    In particular, $\refl_i(\dynk^{\ydind{Q}{U}})_{j(\theta+1)}$ is defined for all $j \in \mathbb{I}$.
\end{rema}

\begin{thm} \label{prop_ydind_refl_admit_equiv}
    Let $k \in \ndN_0$, $i=(i_1,\ldots,i_k) \in \mathbb{I}^k$, assume $\scr{N}$ admits the reflection sequence $i$ and that $Q$ is finite-dimensional.
    Let $U \in \comodcat{Q}{\scr{C}}$ be a one-dimensional object and $j \in \mathbb{I}$. The following are equivalent
    \begin{enumerate}
        \item $\ydind{Q}{U}$ admits the reflection sequence $(i_1,\ldots,i_k,j)$.
        \item $\ydind{Q}{U}$ admits the reflection sequence $i$ and 
        \begin{align*}
            \shapodiag{\dynk^{\scr{N}}}_{s_i^{\dynk^{\scr{N}}}(\alpha_j)} (\dynk^{\ydind{Q}{U}}_{1(\theta+1)}, \ldots, \dynk^{\ydind{Q}{U}}_{\theta(\theta+1)}) \ne 0 \matdot
        \end{align*}
        \item $\ydind{Q}{U}$ admits the reflection sequence $i$ and 
        \begin{align*}
            \shapodiag{\dynk^{\scr{N}}}_{-s_i^{\dynk^{\scr{N}}}(\alpha_j)} (\dynk^{\ydind{Q}{U}}_{1(\theta+1)}, \ldots, \dynk^{\ydind{Q}{U}}_{\theta(\theta+1)}) \ne 0 \matdot
        \end{align*}
    \end{enumerate}
\end{thm}
\begin{proof}
First consider that $\refl_i(\extdynk{\dynk^{\scr{N}}})$ and $\refl_i(\dynk^{\ydind{Q}{U}})_{j(\theta+1)}$ are defined for all $j \in \mathbb{I}$, as discussed in \Cref{rema_fin_dim_ydind_dynk_admits_refl}.
Assume $\ydind{Q}{U}$ admits the reflection sequence $i$ and let $j \in \mathbb{I}$. 
Combining \Cref{cor_dynk_refl_equals_refl}, \Cref{lem_ydind_mi} and \Cref{cor_dynk_refl_equals_refl2} we obtain that $\ydind{Q}{U}$ admits the reflection sequence $(i_1,\ldots,i_k,j)$ if and only if $m_j^{\refl_i(\dynk^{\scr{N}})} = -a^{\refl_i(\dynk^{\ydind{Q}{U}})}_{j(\theta+1)}$.
By definition this is equivalent to
\begin{align*}
     \refl_i(\dynk^{\scr{N}})_{j}^{m} \refl_i(\dynk^{\ydind{Q}{U}})_{j(\theta+1)} - 1 \ne 0
\end{align*}
for all $0 \le m \le m_j^{\refl_i(\dynk^{\scr{N}})}-1$. Since $s_i^{\dynk^{\scr{N}}}(\alpha_j) \in \ndN_0^\theta \cup - \ndN_0^\theta$ by \Cref{cor_dynk_refl_equals_refl_roots}, using \Cref{lem_ydmod_polynomial_identity} and \Cref{prop_shapodiag_root_equivalence} completes the proof.
\end{proof}

\begin{thm} \label{cor_ydind_refl_admit_equiv}
    Assume $\scr{N}$ admits all reflections and that $Q$ is finite-dimensional.
    Let $U \in \comodcat{Q}{\scr{C}}$ be a one-dimensional object. The following are equivalent
    \begin{enumerate}
        \item $\ydind{Q}{U}$ is irreducible in $\ydcat{Q}{\scr{C}}$.
        \item $\shapodiag{\dynk^{\scr{N}}} (\dynk^{\ydind{Q}{U}}_{1(\theta+1)}, \ldots, \dynk^{\ydind{Q}{U}}_{\theta(\theta+1)}) \ne 0$.
    \end{enumerate}
\end{thm}
\begin{proof}
    By \Cref{cor_ydind_irred_all_refl} (1) holds if and only if $\ydind{Q}{U}$ admits all reflections. Since $\roots{\dynk^{\scr{N}}} = \rootspos{\dynk^{\scr{N}}} \cup - \rootspos{\dynk^{\scr{N}}}$ by \Cref{cor_dynk_refl_equals_refl_roots}, the equivalence of (1) and (2) is inductively implied by \Cref{prop_ydind_refl_admit_equiv}.
\end{proof}

    \Cref{cor_ydind_refl_admit_equiv} states that whether $\ydind{Q}{U}$ is irreducible in $\ydcat{Q}{\scr{C}}$ can be determined by calculating some rational functions in $\fK(t_1,\ldots,t_\theta)$ and checking if $(\dynk^{\ydind{Q}{U}}_{1(\theta+1)}, \ldots, \dynk^{\ydind{Q}{U}}_{\theta(\theta+1)})$ is a root of any of those functions. By definition these functions only rely on the roots of the Nichols system.
    The following is an algorithm that calculates these functions and more.

\begin{algo} \label{algo_diag_irred}
    Assume $\scr{N}$ admits all reflections and $m_i^{\scr{N}}\in \ndN_0$ for all $k \in \ndN_0$, $i \in \mathbb{I}^k$.
    Let $S$ be an empty list with entries in $\Aut(\ndZ^\theta)$ and $X$ be an empty list where entries are objects in $\dynkcat{\theta}$. Moreover let $R$ and $B$ be empty lists with entries in $\ndZ_0^\theta$ (actually $B$ has entries in $\ndN_0^\theta$). Finally let $P$ be an empty list with entries in $\fK(t_1,\ldots,t_\theta)$.
    For any list $L$ and $n \in \ndN$ we denote $L[n]$ for the $n$-th entry in $L$ and let $\#L$ be the number of entries in $L$. 
    Now do the following.
    \begin{enumerate}
        \item Add $\id_{\ndZ^\theta}$ to $S$, $\dynk^{\scr{N}}$ to $X$ and $0\in \ndN_0^\theta$ to $B$. Let $n=1$.
        \item While $n \le \#S$ do the following:
        \begin{enumerate}
            \item Let $s=S[n]$, $\dynk = X[n]$ and $\beta=B[n]$.
            \item For all $j \in \mathbb{I}$ do:
            \begin{enumerate}
                \item Define the monomial
                \begin{align*}
                    q' = t_{s(\alpha_j)} \bichar{\dynk^{\scr{N}}} \left( \beta, s(\alpha_j) \right)  \matcom
                \end{align*}
                let $r=q'(1,\ldots,1) \in \fK^\times$ and $q = r^{-1} q'$. (Careful: Here $\dynk^{\scr{N}}$ is used, not $D$.)
                \item For all $1 \le m \le m_j^\dynk$ do:
                \begin{itemize}
                    \item[] If $q - r^{-1} \dynk_j^{1-m} \notin P$ and $q^{-1} - r \dynk_j^{m-1} \notin P$, then add $q - r^{-1} \dynk_j^{1-m}$ to $P$.
                \end{itemize}
                \item Add $s(\alpha_j)$ to $R$, if it is not already contained.
                \item Let $s' = s s^\dynk_j$.
                \item If $s'$ is not contained in $S$, then add $s'$ to $S$, $\refl_j(\dynk)$ to $X$ and $\beta - m_j^\dynk s'(\alpha_j)$ to $B$.
            \end{enumerate}
            \item Increment $n$ by $1$.
        \end{enumerate}
    \end{enumerate}
    If the algorithm terminates, then the following hold:
    \begin{enumerate}[label=(\Alph*)]
        \item $Q$ is finite-dimensional.
        \item $\roots{\scr{N}} = R$.
        \item $B \subset \Sup{Q}$ and $\Sup{Q}$ is contained in the convex hull of $B$.
        \item For any one-dimensional object $U \in \comodcat{Q}{\scr{C}}$ we have that $\ydind{Q}{U}$ is irreducible in $\ydcat{Q}{\scr{C}}$ if and only if for all $p \in P$ we have
        \begin{align*}
            p(\dynk^{\ydind{Q}{U}}_{1(\theta+1)}, \ldots, \dynk^{\ydind{Q}{U}}_{\theta(\theta+1)}) & \ne 0 \matdot
        \end{align*}
    \end{enumerate}
\end{algo}
\begin{proof}
    We show the following:
    \begin{enumerate}[label=(\Roman*)]
        \item For all $1 \le n \le \#S$ there exists $k\in \ndN_0$, $i \in \mathbb{I}^k$ such that $S[n]=s_i^{\dynk^{\scr{N}}}$, $X[n]=\refl_i(\dynk^{\scr{N}})$ and $B[n]=\beta^{\dynk^{\scr{N}}}_i$.
        \item For all $k\in \ndN_0$, $i \in \mathbb{I}^k$ there exists $1 \le n \le \#S$ such that $S[n]=s_i^{\dynk^{\scr{N}}}$, $X[n]=\refl_i(\dynk^{\scr{N}})$ and $B[n]=\beta^{\dynk^{\scr{N}}}_i$.
    \end{enumerate}
    Then (B) and (C) are implied by (II), \Cref{prop_support_is_spanned_by_beta_i_sys} and \Cref{cor_dynk_refl_equals_refl_roots}, (A) is implied by (C) and~(D) is implied by (II),  \Cref{cor_ydind_refl_admit_equiv}, \Cref{prop_shapodiag_root_equivalence} and \Cref{lem_ext_dynk_explicit}.
    
    (I): Assume that $S[n]=s_i^{\dynk^{\scr{N}}}$, $X[n]=\refl_i(\dynk^{\scr{N}})$ and $B[n]=\beta^{\dynk^{\scr{N}}}_i$ for some $1 \le n \le \#S$, $k \in \ndN_0$, $i=(i_1,\ldots,i_k)\in\mathbb{I}^k$ (given for $n=1$, $k=0$ and $i =()$). Then for $j \in \mathbb{I}$ we have 
    \begin{align*}
        s_i^{\dynk^{\scr{N}}} s_j^{\refl_i(\dynk^{\scr{N}})} = s_i^{\dynk^{\scr{N}}} s_{i,j}^{\dynk^{\scr{N}}} = s_{(i_1,\ldots,i_k,j)}^{\dynk^{\scr{N}}}
    \end{align*}
    as well as $\refl_j (\refl_i(\dynk^{\scr{N}}))= \refl_{(i_1,\ldots,i_k,j)}(\dynk^{\scr{N}})$ and 
    \begin{align*}
        \beta_i^{\dynk^{\scr{N}}} - m_j^{\refl_i(\dynk^{\scr{N}})} s_{(i_1,\ldots,i_k,j)}^{\dynk^{\scr{N}}}(\alpha_j) = \beta_{(i_1,\ldots,i_k,j)}^{\dynk^{\scr{N}}} \matdot
    \end{align*}
    Considering (v) we conclude that (I) holds.
    
    (II): We do induction on $k$: For $k=0$ (II) holds via choosing $n=1$. Let $k\in \ndN_0$, $i \in \mathbb{I}^k$ as well as $1 \le n \le \#S$ and assume that $S[n]=s_i^{\dynk^{\scr{N}}}$, $X[n]=\refl_i(\dynk^{\scr{N}})$ and $B[n]=\beta^{\dynk^{\scr{N}}}_i$. Let $j \in \mathbb{I}$.
    We will show that (II) holds for $i'=(i_1,\ldots,i_k,j)$.
    There are two possibilities: Either in (v) $s_i^{\dynk^{\scr{N}}} s_j^{\refl_i(\dynk^{\scr{N}})} = s_{(i_1,\ldots,i_k,j)}^{\dynk^{\scr{N}}}$ was contained in $S$ or not. In the latter case similar to (I) we conclude that (II) holds for $i'$. So assume in (v) $s_{(i_1,\ldots,i_k,j)}^{\dynk^{\scr{N}}}$ was contained in $S$, i.e. that there exists $1 \le n' \le \#S$, such that $s_{(i_1,\ldots,i_k,j)}^{\dynk^{\scr{N}}}=S[n']$. Now by (I) there exists  $l\in \ndN_0$, $a=(a_1,\ldots,a_l) \in \mathbb{I}^l$, such that $S[n']=s_a^{\dynk^{\scr{N}}}$, $X[n']=\refl_a(\dynk^{\scr{N}})$ and $B[n']=\beta^{\dynk^{\scr{N}}}_a$.
    Since $s_{(i_1,\ldots,i_k,j)}^{\dynk^{\scr{N}}}=s_a^{\dynk^{\scr{N}}}$, \Cref{prop_refl_dynkin_relies_on_si} implies $\refl_{(i_1,\ldots,i_k,j)}(\dynk^{\scr{N}}) = \refl_a(\dynk^{\scr{N}})=X[n']$. Finally \Cref{cor_dynk_refl_equals_refl_roots} and \Cref{cor_sup_iterated_reflection2} imply $\beta_{(i_1,\ldots,i_k,j)}^{\dynk^{\scr{N}}}=\beta^{\dynk^{\scr{N}}}_a=B[n']$.
\end{proof}
\section{Examples} \label{sect_nich_sys_diag_examples}

In this last section we will apply \Cref{algo_diag_irred} to two examples of Nichols systems of diagonal type. The examples appear in \cite{MR2379892} and \cite{MR2500361}, respectively.

\begin{exa}
    Assume $\theta=2$, let $\omega \in \fK$ be a third root of $1$ and assume $\scr{N}$
    is a pre-Nichols system that admits all reflections and has Dynkin diagram
    \begin{align*}
    \begin{tikzpicture}
        \node[dynnode, label={\tiny$\omega$}] (1) at (0,0) {};
        \node[dynnode, label={\tiny$-1$}] (2) at (1.5,0) {};
        \path[draw,thick]
        (1) edge node[above]{\tiny$-\omega$} (2);
    \end{tikzpicture}
    \matcom
    \end{align*}
    In \Cref{exa_dynk} we saw that the reflection graph of $\dynk^{\scr{N}}$ is
    \begin{align*}
    \begin{tikzpicture} \tikzset{every loop/.style={}};
        \node (a) at (0,0) {$
            \begin{tikzpicture}
                \node[dynnode, label={\tiny$\omega$}] (1) at (0,0) {};
                \node[dynnode, label={\tiny$-1$}] (2) at (1.5,0) {};
                \path[draw,thick]
                (1) edge node[above]{\tiny$-\omega$} (2);
            \end{tikzpicture}
        $};
        \node (b) at (3.5,0) {$
            \begin{tikzpicture}
                \node[dynnode, label={\tiny$\omega^{-1}$}] (1) at (0,0) {};
                \node[dynnode, label={\tiny$-1$}] (2) at (1.5,0) {};
                \path[draw,thick]
                (1) edge node[above]{\tiny$-\omega^{-1}$} (2);
            \end{tikzpicture}
        $};
        \path[draw,thick] (a) edge node[above]{$2$} (b);
    \end{tikzpicture}
    \matcom
    \end{align*}
    and that for all $k \in \ndN_0$, $i \in \mathbb{I}^k$
    \begin{align*}
        (a^{\refl_{i}(\scr{N})}_{jk})_{j,k \in \mathbb{I}} = \begin{pmatrix}
    2 & -2 \\
    -1 & 2 
    \end{pmatrix} \matcom
        && (m^{\refl_{i} (\scr{N})}_j)_{j\in \mathbb{I}} = \begin{pmatrix}
    2 \\1
    \end{pmatrix}
    \matdot
    \end{align*}
    Then $s_j^{\refl_{i}(\scr{N})} = s_j^{\scr{N}}$ for all $j \in \mathbb{I}$ and we obtain 
    \begin{align*}
        s^{\scr{N}}_{(1212)}=s_1^{\scr{N}} s_2^{\scr{N}}s_1^{\scr{N}} s_2^{\scr{N}} =- \id_{\ndZ^\theta} = s_2^{\scr{N}}s_1^{\scr{N}} s_2^{\scr{N}}s_1^{\scr{N}} = s^{\scr{N}}_{(2121)} \matdot
    \end{align*}
    We do \Cref{algo_diag_irred} step by step:
    \begin{itemize}
        \item Add $s^{\dynk}_{()}=\id_{\ndZ^\theta}$ to $S$, $\dynk^{\scr{N}}$ to $X$ and $\beta_{()}^{\scr{N}}$ to $B$.
        \item Let $s=S[1]=\id_{\ndZ^\theta}$, $\dynk = X[1] = \dynk^{\scr{N}}$ and $\beta=B[1] = 0$.
        \begin{itemize}
            \item For $j=1$ in (i) we obtain $q=t_1$, $r=1$, hence in (ii) we add $t_1 - 1$ and $t_1 - \omega^2$ to $P$.
            \item For $j=1$ in (iii) we add $(1,0)$ to $R$.
            \item For $j=1$ in (v) we add $s_1^{\dynk}$ to $S$, $\refl_1(\dynk)$ to $X$ and $\beta_1^{\dynk} = (2,0)$ to $B$.
            \item For $j=2$ in (i) we obtain $q=t_2$, $r=1$, hence in (ii) we add $t_2 - 1$ to $P$.
            \item For $j=2$ in (iii) we add $(0,1)$ to $R$.
            \item For $j=2$ in (v) we add $s_2^{\dynk}$ to $S$, $\refl_2(\dynk)$ to $X$ and $\beta_2^{\dynk} = (0,1)$ to $B$.
        \end{itemize}
        \item Let $s=S[2]=s_1^{\dynk}$, $\dynk = X[2] = \refl_1(\dynk^{\scr{N}})$ and $\beta=B[2] = \beta_1^{\dynk}$.
        \begin{itemize}
            \item For $j=1$ we obtain $q=t_1^{-1}$, $r=\omega^2$ and $s'=\id_{\ndZ^\theta}$, hence only add $(-1,0)$ to $R$ in (iii).
            \item For $j=2$ in (i) we obtain $q=t_1^2 t_2$, $r=\omega$, hence in (ii) we add $t_1^2 t_2-\omega^2$ to $P$.
            \item For $j=2$ in (iii) we add $(2,1)$ to $R$.
            \item For $j=2$ in (v) we add $s_{(12)}^{\dynk}$ to $S$, $\refl_{(12)}(\dynk)$ to $X$ and $\beta_{(12)}^{\dynk} = (4,1)$ to $B$.
        \end{itemize}
        \item Let $s=S[3]=s_2^{\dynk}$, $\dynk = X[3] = \refl_2(\dynk^{\scr{N}})$ and $\beta=B[3] = \beta_2^{\dynk}$.
        \begin{itemize}
            \item For $j=1$ in (i) we obtain $q=t_1t_2$, $r=-\omega$, hence in (ii) we add $t_1 t_2+\omega^2$ and $t_1t_2+1$ to $P$.
            \item For $j=1$ in (iii) we add $(1,1)$ to $R$.
            \item For $j=1$ in (v) we add $s_{(21)}^{\dynk}$ to $S$, $\refl_{(21)}(\dynk)$ to $X$ and $\beta_{(21)}^{\dynk} = (2,3)$ to $B$.
            \item For $j=2$ we only add $(0,-1)$ to $R$ in (iii).
        \end{itemize}
        
        \item Let $s=S[4]=s_{(12)}^{\dynk}$, $\dynk = X[4] = \refl_{(12)}(\dynk^{\scr{N}})$ and $\beta=B[4] = \beta_{(12)}^{\dynk}$.
        \begin{itemize}
            \item For $j=1$ we do nothing in (ii).
            \item For $j=1$ in (v) we add $s_{(121)}^{\dynk}$ to $S$, $\refl_{(121)}(\dynk)$ to $X$ and $\beta_{(121)}^{\dynk} = (6,3)$ to $B$.
            \item For $j=2$ we only add $(-2,-1)$ to $R$ in (iii).
        \end{itemize}
        \item Let $s=S[5]=s_{(21)}^{\dynk}$, $\dynk = X[5] = \refl_{(21)}(\dynk^{\scr{N}})$ and $\beta=B[5] = \beta_{(21)}^{\dynk}$.
        \begin{itemize}
            \item For $j=1$ we only add $(-1,-1)$ to $R$ in (iii).
            \item For $j=2$ we do nothing in (ii).
            \item For $j=2$ in (v) we add $s_{(212)}^{\dynk}$ to $S$, $\refl_{(212)}(\dynk)$ to $X$ and $\beta_{(212)}^{\dynk} = (4,4)$ to $B$.
        \end{itemize}
        \item Let $s=S[6]=s_{(121)}^{\dynk}$, $\dynk = X[6] = \refl_{(121)}(\dynk^{\scr{N}})$ and $\beta=B[6] = \beta_{(121)}^{\dynk}$.
        \begin{itemize}
            \item For $j=1$ we do nothing.
            \item For $j=2$ we do nothing in (ii).
            \item For $j=2$ in (v) we add $s_{(1212)}^{\dynk}=-\id_{\ndZ^\theta}$ to $S$, $\refl_{(1212)}(\dynk)$ to $X$ and $\beta_{(1212)}^{\dynk} = (6,4)$ to $B$.
        \end{itemize}
        \item Let $s=S[7]=s_{(212)}^{\dynk}$, $\dynk = X[7] = \refl_{(212)}(\dynk^{\scr{N}})$ and $\beta=B[7] = \beta_{(212)}^{\dynk}$.
        \begin{itemize}
            \item For $j=1$ we do nothing.
            \item For $j=2$ we do nothing.
        \end{itemize}
        \item Let $s=S[8]=s_{(1212)}^{\dynk}$, $\dynk = X[8] = \refl_{(1212)}(\dynk^{\scr{N}})$ and $\beta=B[8] = \beta_{(1212)}^{\dynk}$.
        \begin{itemize}
            \item For $j=1$ we do nothing.
            \item For $j=2$ we do nothing.
        \end{itemize}
        \item The algorithm terminates.
    \end{itemize}
    In total we have
    \begin{align*}
        R &= [ ( 1, 0 ), ( 0, 1 ), ( -1, 0 ), ( 2, 1 ), ( 1, 1 ), ( 0, -1 ), ( -2, -1 ), ( -1, -1 ) ] \matcom \\
        B &= [(0,0), (2,0),(0,1), (4,1),(2,3), (6,3),(4,4), (6,4) ] \matcom
        \\ P &= [t_1 - 1, t_1 - \omega^2, t_2 - 1, t_1^2 t_2-\omega^2,t_1 t_2+\omega^2,t_1 t_2+1] \matdot
    \end{align*}
    Hence $Q$ is finite-dimensional, $\Sup{Q}$ is contained in the set drawn in 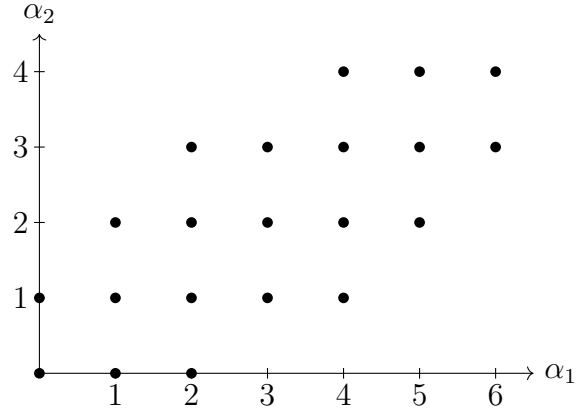
\begin{figure}[t] 
        \begin{tikzpicture}
        \draw[thin,->] (0,0) -- (6.5,0) node[right] {$\alpha_1$};
        \draw[thin,->] (0,0) -- (0,4.5) node[above] {$\alpha_2$};
        \foreach \x [count=\xi starting from 0] in {1,2,3,4}{
            \draw (\x,2pt) -- (\x,-2pt);
            \draw (2pt,\x) -- (-2pt,\x);
                \node[anchor=north] at (\x,0) {$\x$};
                \node[anchor=east] at (0,\x) {$\x$};
        }
        \foreach \x [count=\xi starting from 0] in {5,6}{
            \draw (\x,2pt) -- (\x,-2pt);
            \node[anchor=north] at (\x,0) {$\x$};
        }
        \foreach \point in {(0,0),(1,0),(2,0),(0,1),(1,1),(2,1),(3,1),(4,1), (1,2),(2,2),(3,2),(4,2),(5,2), (2,3),(3,3),(4,3),(5,3), (6,3),(4,4),(5,4), (6,4)}{
            \fill \point circle (2pt);
        }
        \end{tikzpicture}
        \centering
        \caption{A subset of $\ndN_0^2$ in which $\Sup{Q}$ is contained in.}
        \label{figure_supQ_exa1}
    \end{figure} \Cref{figure_supQ_exa1} and for a one-dimensional object $U \in \comodcat{Q}{\scr{C}}$, where $r_1, r_2 \in \fK^\times$ are such that for $u \in U$, $u \ne 0$, $j \in \lbrace 1,2 \rbrace$ we have
    \begin{align*}
        \brd_{U,N_j}^{\scr{C}} \brd_{N_j,U}^{\scr{C}} ( x_j \ot u) = r_j x_j \ot u \matcom
    \end{align*}
    then we obtain that $\ydind{Q}{U}$ is irreducible if and only if
    \begin{align*}
        (r_1-1) (r_1-\omega^2)(r_2-1)(r_1^2 r_2-\omega^2) (r_1 r_2+\omega^2)(r_1 r_2+1) \ne 0 \matdot
    \end{align*} 
\end{exa}

\begin{exa}
    Let $a \in \lbrace 3,6 \rbrace$, assume $\omega \in \fK$ is an $a$-th root of $1$ and assume $\scr{N}$
    is a pre-Nichols system that admits all reflections and has Dynkin diagram
    \begin{align*}
    \begin{tikzpicture}
        \node[dynnode, label={\tiny$\omega$}] (1) at (0,0) {};
        \node[dynnode, label={\tiny$\omega$}] (2) at (1.5,0) {};
        \node[dynnode, label={\tiny$-1$}] (3) at (3,0) {};
        \path[draw,thick]
        (1) edge node[above]{\tiny$\omega^{-1}$} (2)
        (2) edge node[above]{\tiny$\omega^{-2}$} (3);
    \end{tikzpicture}
    \matdot
    \end{align*}
    As seen in \Cref{exa_dynk2} the reflection graph of $\dynk^{\scr{N}}$ is
    \begin{align*}
    \begin{tikzpicture} \tikzset{every loop/.style={}};
        \node (a) at (0,0) {$
            \begin{tikzpicture}
                \node[dynnode, label={\tiny$\omega$}] (1) at (0,0) {};
                \node[dynnode, label={\tiny$\omega$}] (2) at (1.4,0) {};
                \node[dynnode, label={\tiny$-1$}] (3) at (2.8,0) {};
                \path[draw,thick]
                (1) edge node[above]{\tiny$\omega^{-1}$} (2)
                (2) edge node[above]{\tiny$\omega^{-2}$} (3);
            \end{tikzpicture}
        $};
        \node (b) at (4.5,0) {$
            \begin{tikzpicture}
                \node[dynnode, label={\tiny$\omega$}] (1) at (0,0) {};
                \node[dynnode, label={\tiny$-\omega^{-1}$}] (2) at (1.4,0) {};
                \node[dynnode, label={\tiny$-1$}] (3) at (2.8,0) {};
                \path[draw,thick]
                (1) edge node[above]{\tiny$\omega^{-1}$} (2)
                (2) edge node[above]{\tiny$\omega^{2}$} (3);
            \end{tikzpicture}
        $};
        \path[draw,thick] (a) edge node[above]{$3$} (b);
    \end{tikzpicture}
    \matdot
    \end{align*}
    Let $a' \in \lbrace 3,6 \rbrace$, such that $a \ne a'$. For all $k \in \ndN_0$, $i=(i_1,\ldots,i_k) \in \mathbb{I}$ we obtain 
    \begin{align*}
        (a^{\refl_{i}(\dynk)}_{jk})_{j,k \in \mathbb{I}}= 
        \begin{cases} 
            \begin{pmatrix}
                2 & -1 & 0 \\
                -1 & 2 & -2 \\
                0 & -1 & 2
            \end{pmatrix}
            & \text{if $\# \lbrace 1\le l \le k \, | \, i_l = 3\rbrace$ is even,} \\
            \begin{pmatrix}
                2 & -1 & 0 \\
                -2 & 2 & -2 \\
                0 & -1 & 2
            \end{pmatrix} & \text{else.}
        \end{cases}
    \end{align*}
    and $(m^{\refl_{i}(\dynk)}_j)_{j \in \mathbb{I}}=(a-1,a-1,1)$ if $\# \lbrace 1\le l \le k \, | \, i_l = 3\rbrace$ is even and $(m^{\refl_{i}(\dynk)}_j)_{j \in \mathbb{I}}=(a-1,a'-1,1)$ else.
   
    Here \Cref{algo_diag_irred} terminates after $\#S = 96$. We obtain
    \begin{align*}
        R=& [( 1, 0, 0 ), ( 0, 1, 0 ), ( 0, 0, 1 ), ( -1, 0, 0 ), ( 1, 1, 0 ), ( 0, -1, 0 ), ( 0, 2, 1 ), ( 0, 1, 1 ),\\ & ( 0, 0, -1 ), ( -1, -1, 0 ),
  ( 2, 2, 1 ), ( 1, 1, 1 ), ( 0, -2, -1 ), ( 1, 2, 2 ), ( 0, -1, -1 ), \\ & ( -2, -2, -1 ), ( -1, -1, -1 ), ( 1, 2, 1 ), ( 1, 3, 2 ), ( -1, -2, -2 ), ( 2, 3, 2 ), \\ & ( -1, -2, -1 ), ( -1, -3, -2 ), ( 2, 4, 3 ), ( -2, -3, -2 ),
  ( -2, -4, -3 ) ] \matdot
    \end{align*}
    If $a=3$, then we obtain that $B$ and $P$ are given by
    \begin{align*}
        B=&[( 0, 0, 0 ), ( 2, 0, 0 ), ( 0, 2, 0 ), ( 0, 0, 1 ), ( 4, 2, 0 ), ( 2, 0, 1 ), ( 2, 4, 0 ), ( 0, 4, 1 ), 
        \\ & ( 0, 5, 6 ), ( 4, 4, 0 ), ( 6, 4, 1 ), ( 7, 5, 6 ), ( 2, 6, 1 ), ( 0, 9, 6 ), ( 2, 9, 10 ), ( 0, 7, 7 ),
        \\ &   ( 6, 6, 1 ), ( 11, 9, 6 ), ( 9, 9, 10 ), ( 9, 7, 7 ), ( 7, 16, 6 ), ( 2, 15, 10 ), ( 0, 9, 7 ), 
        \\ & ( 7, 14, 15 ),  ( 2, 11, 11 ), ( 11, 16, 6 ), ( 15, 15, 10 ), ( 11, 9, 7 ), ( 9, 14, 15 ), 
        \\ & ( 11, 11, 11 ), ( 9, 22, 10 ),  ( 9, 18, 7 ), ( 7, 25, 15 ), ( 2, 15, 11 ), ( 9, 18, 18 ), 
        \\ & ( 4, 17, 15 ), ( 15, 22, 10 ), ( 11, 18, 7 ),  ( 20, 25, 15 ), ( 15, 15, 11 ), ( 11, 18, 18 ), 
        \\ & ( 15, 17, 15 ), ( 9, 27, 15 ), ( 11, 24, 11 ), ( 9, 29, 18 ), ( 4, 19, 15 ), ( 11, 24, 22 ), 
        \\ & ( 6, 21, 18 ), ( 20, 27, 15 ), ( 15, 24, 11 ), ( 22, 29, 18 ), ( 17, 19, 15 ), ( 15, 24, 22 ),
        \\ & ( 17, 21, 18 ), ( 11, 31, 18 ), ( 15, 30, 15 ), ( 11, 33, 22 ), ( 6, 23, 18 ), ( 15, 30, 26 ), 
        \\ & ( 11, 26, 23 ), ( 22, 31, 18 ), ( 17, 30, 15 ), ( 24, 33, 22 ), ( 19, 23, 18 ), ( 17, 30, 26 ), 
        \\ & ( 17, 26, 23 ),  ( 15, 37, 22 ), ( 17, 34, 18 ), ( 15, 39, 26 ), ( 11, 33, 23 ), ( 15, 32, 27 ), 
        \\ & ( 24, 37, 22 ), ( 19, 34, 18 ), ( 26, 39, 26 ), ( 24, 33, 23 ), ( 19, 32, 27 ), ( 17, 41, 26 ), 
        \\ & ( 17, 39, 23 ), ( 15, 39, 27 ), ( 20, 42, 32 ),  ( 26, 41, 26 ), ( 24, 39, 23 ), ( 26, 39, 27 ), 
        \\ & ( 24, 42, 32 ), ( 19, 43, 27 ), ( 20, 44, 32 ), ( 22, 44, 33 ), ( 26, 43, 27 ), ( 26, 44, 32 ), 
        \\ & ( 24, 44, 33 ), ( 24, 48, 32 ), ( 22, 46, 33 ), ( 26, 48, 32 ), ( 26, 46, 33 ),  ( 24, 48, 33 ), 
        \\ & ( 26, 48, 33 )] 
    \end{align*}
    and
    \begin{align*}
        P=& [t_1-1, t_1- \omega^2, t_2-1, t_2- \omega^2, t_3-1, t_1 t_2- \omega, t_1 t_2-1, t_2^2 t_3- \omega^2, 
        \\ & t_2 t_3- \omega^2, t_2 t_3+1,  t_2 t_3- \omega, t_2 t_3+ \omega^2, t_2 t_3-1, t_1^2 t_2^2 t_3- \omega, t_1 t_2 t_3-1, 
        \\ & t_1 t_2 t_3+ \omega, t_1 t_2 t_3- \omega^2, t_1 t_2 t_3+1, t_1 t_2 t_3- \omega, t_1 t_2^2 t_3^2- \omega^2, t_1 t_2^2 t_3^2- \omega,
        \\ &  t_1 t_2^2 t_3- \omega, t_1 t_2^2 t_3+ \omega^2, t_1 t_2^2 t_3-1, t_1 t_2^2 t_3+ \omega, t_1 t_2^2 t_3- \omega^2, t_1 t_2^3 t_3^2-1,
        \\ &  t_1 t_2^3 t_3^2- \omega^2, t_1^2 t_2^3 t_3^2- \omega, t_1^2 t_2^3 t_3^2-1,  t_1^2 t_2^4 t_3^3-1] \matdot
    \end{align*}
    Moreover if $a=6$, then $B$ and $P$ are given by
    \begin{align*}
        B=&[   ( 0, 0, 0 ), ( 5, 0, 0 ), ( 0, 5, 0 ), ( 0, 0, 1 ), ( 10, 5, 0 ), ( 5, 0, 1 ), ( 5, 10, 0 ), ( 0, 7, 1 ), 
        \\&  ( 0, 2, 3 ), ( 10, 10, 0 ), ( 12, 7, 1 ), ( 7, 2, 3 ), ( 5, 12, 1 ), ( 0, 9, 3 ), ( 5, 12, 13 ), 
        \\& ( 0, 4, 4 ), ( 12, 12, 1 ), ( 14, 9, 3 ), ( 12, 12, 13 ), ( 9, 4, 4 ), ( 7, 16, 3 ), ( 5, 24, 13 ),
        \\& ( 0, 9, 4 ), ( 7, 14, 15 ), ( 5, 14, 14 ), ( 14, 16, 3 ), ( 24, 24, 13 ), ( 14, 9, 4 ), 
        \\& ( 12, 14, 15 ), ( 14, 14, 14 ), ( 12, 31, 13 ),( 9, 18, 4 ), ( 7, 28, 15 ), ( 5, 24, 14 ), 
        \\& ( 9, 18, 18 ),( 10, 29, 24 ), ( 24, 31, 13 ), ( 14, 18, 4 ),( 26, 28, 15 ), ( 24, 24, 14 ), 
        \\& ( 14, 18, 18 ), ( 24, 29, 24 ),( 12, 33, 15 ), ( 14, 33, 14 ), ( 9, 32, 18 ),( 10, 34, 24 ), 
        \\& ( 14, 33, 28 ), ( 12, 33, 27 ), ( 26, 33, 15 ), ( 24, 33, 14 ), ( 28, 32, 18 ), ( 29, 34, 24 ),
        \\& ( 24, 33, 28 ), ( 26, 33, 27 ), ( 14, 37, 18 ), ( 24, 48, 24 ), ( 14, 42, 28 ), ( 12, 38, 27 ),
        \\& ( 24, 48, 38 ), ( 14, 35, 29 ), ( 28, 37, 18 ), ( 29, 48, 24 ), ( 33, 42, 28 ), ( 31, 38, 27 ), 
        \\& ( 29, 48, 38 ), ( 26, 35, 29 ), ( 24, 52, 28 ), ( 26, 52, 27 ), ( 24, 57, 38 ), ( 14, 42, 29 ), 
        \\& ( 24, 50, 39 ), ( 33, 52, 28 ), ( 31, 52, 27 ), ( 38, 57, 38 ), ( 33, 42, 29 ), ( 31, 50, 39 ), 
        \\& ( 29, 62, 38 ), ( 26, 54, 29 ), ( 24, 57, 39 ), ( 26, 54, 41 ), ( 38, 62, 38 ), ( 33, 54, 29 ), 
        \\&( 38, 57, 39 ), ( 33, 54, 41 ), ( 31, 64, 39 ), ( 26, 59, 41 ), ( 28, 56, 42 ), ( 38, 64, 39 ), 
        \\&( 38, 59, 41 ), ( 33, 56, 42 ), ( 33, 66, 41 ), ( 28, 61, 42 ), ( 38, 66, 41 ), ( 38, 61, 42 ),
        \\& ( 33, 66, 42 ), ( 38, 66, 42 ) ]
    \end{align*}
    and
    \begin{align*}
        P=&[ t_1-1, t_1+ \omega^2, t_1- \omega^4, t_1+1, t_1- \omega^2, t_2-1, t_2+ \omega^2, t_2- \omega^4, 
        \\& t_2+1, t_2- \omega^2, t_3-1, t_1 t_2+ \omega^4, t_1 t_2-1, t_1 t_2+ \omega^2, t_1 t_2- \omega^4, 
        \\& t_1 t_2+1, t_2^2 t_3- \omega^2, t_2 t_3- \omega^2, t_2 t_3-1, t_1^2 t_2^2 t_3- \omega^4, t_1 t_2 t_3+1, 
        \\& t_1 t_2 t_3+ \omega^4, t_1 t_2^2 t_3^2- \omega^2, t_1 t_2^2 t_3^2+ \omega^4, t_1 t_2^2 t_3^2-1, t_1 t_2^2 t_3^2+ \omega^2, 
        \\& t_1 t_2^2 t_3^2- \omega^4, t_1 t_2^2 t_3- \omega^4, t_1 t_2^2 t_3- \omega^2, t_1 t_2^3 t_3^2+1, t_1 t_2^3 t_3^2- \omega^2, 
        \\& t_1 t_2^3 t_3^2+ \omega^4, t_1 t_2^3 t_3^2-1, t_1 t_2^3 t_3^2+ \omega^2, t_1^2 t_2^3 t_3^2- \omega^4, t_1^2 t_2^3 t_3^2+1, 
        \\& t_1^2 t_2^3 t_3^2- \omega^2, t_1^2 t_2^3 t_3^2+ \omega^4, t_1^2 t_2^3 t_3^2-1, t_1^2 t_2^4 t_3^3-1 ]
        \matdot
    \end{align*}
    Hence $Q$ is finite-dimensional, $\Sup{Q}$ is contained in the convex hull of $B$ and for a one-dimensional object $U \in \comodcat{Q}{\scr{C}}$, where $r_1, r_2, r_3 \in \fK^\times$ are such that for $u \in U$, $u \ne 0$, $j \in \lbrace 1,2,3 \rbrace$ we have
    \begin{align*}
        \brd_{U,N_j}^{\scr{C}} \brd_{N_j,U}^{\scr{C}} ( x_j \ot u) = r_j x_j \ot u \matcom
    \end{align*}
    then we obtain that $\ydind{Q}{U}$ irreducible if and only if
    \begin{align*}
        \prod_{p \in P} p(r_1,r_2,r_3) \ne 0 \matdot
    \end{align*}
    
\end{exa}

\nocite{MR2840165}
\nocite{MR2490263}
\nocite{MR1684154}
\nocite{MR1938453}
\nocite{MR2462836}
\nocite{MR2596372}

\cleardoublepage
\phantomsection
\addcontentsline{toc}{chapter}{References}

\bibliographystyle{amsalpha}
\bibliography{bibl}

\cfoot{\tiny\hiddenmessage}

\end{document}